\documentclass[11pt]{amsart}
\usepackage[margin=0.9in]{geometry}
\usepackage{amsmath}
\usepackage{amsfonts}
\usepackage{amssymb}
\usepackage{graphicx}
\usepackage{mathrsfs}
\usepackage{graphicx,color}
\usepackage[mathscr]{eucal}
\usepackage{caption}
\usepackage{subcaption}
\usepackage{color}
\usepackage{comment}
\usepackage{tikz-cd}
\usepackage{graphicx,color}
\usepackage{mathrsfs}
\usepackage{marvosym}
\usepackage{mathtools,tikz-cd}
\usepackage{enumitem}
\usepackage{graphicx}

\usepackage{tikz}
\usetikzlibrary{shapes.geometric, arrows}
\tikzstyle{startstop} = [rectangle, rounded corners, minimum width=2.5cm, minimum height=0.5cm,text centered, text width=2cm, draw=black, fill=white!30]
\tikzstyle{startstop2} = [rectangle, rounded corners, minimum width=2.5cm, minimum height=0.5cm,text centered, text width=2cm, draw=black, fill=white!30]
\tikzstyle{startstop3} = [rectangle, rounded corners, minimum width=2.5cm, minimum height=0.5cm,text centered, text width=2.5cm, draw=black, fill=white!30]
\tikzstyle{arrow} = [thick,->,>=stealth]

\newtheorem{theorem}{Theorem}[section]
\newtheorem{lemma}[theorem]{Lemma}

\newtheorem{corollary}[theorem]{Corollary}
\newtheorem{corollary*}[theorem]{Corollary*}
\newtheorem{proposition}[theorem]{Proposition}
\newtheorem{proposition*}[theorem]{Proposition*}
\newtheorem{definition}[theorem]{Definition}
\newtheorem{example}[theorem]{Example}
\newtheorem{remark}[theorem]{Remark}

\newtheorem{claim}[theorem]{Claim}

\numberwithin{equation}{section}
\numberwithin{figure}{section}

\newcommand{\ol}[1]{\overline{#1}}
\newcommand{\ul}[1]{\underline{#1}}

\newcommand{\NN}{\mathbb{N}}
\newcommand{\ZZ}{\mathbb{Z}}

\newcommand{\RR}{\mathbb{R}}

\newcommand{\KK}{\mathbb{K}}

\newcommand{\bH}{\mathbb{H}}

\newcommand{\bc}{\boldsymbol{c}}

\newcommand{\be}{\boldsymbol{e}}

\newcommand{\kg}{\mathfrak{g}}

\newcommand{\CC}{\mathbb{C}}
\newcommand{\CP}{\mathbb{CP}}
\newcommand{\cS}{\mathcal{S}}

\newcommand{\cO}{\mathcal{O}}
\newcommand{\cI}{\mathcal{I}}

\newcommand{\cY}{\mathcal{Y}}

\newcommand{\cJ}{\mathcal{J}}
\newcommand{\cP}{\mathcal{P}}
\newcommand{\cG}{\mathcal{G}}
\newcommand{\cR}{\mathcal{R}}

\newcommand{\cD}{\mathcal{D}}
\newcommand{\cC}{\mathcal{C}}
\newcommand{\cB}{\mathcal{B}}
\newcommand{\cX}{\mathcal{X}}
\newcommand{\cK}{\mathcal{K}}
\newcommand{\cH}{\mathcal{H}}
\newcommand{\cFS}{\mathcal{FS}}

\newcommand{\Cone}{\mathrm{Cone}}
\newcommand{\Ext}{\operatorname{Ext}}

\newcommand{\cA}{\mathcal{A}}

\newcommand{\cL}{\mathcal{L}}

\newcommand{\End}{\mathrm{End}}

\newcommand{\Hom}{\mathrm{Hom}}
\newcommand{\rank}{\mathrm{rank}}

\newcommand{\liesl}{\mathfrak{sl}}

\newcommand{\lieg}{\mathfrak{g}}
\newcommand{\liep}{\mathfrak{p}}

\newcommand{\aff}{\mathrm{aff}}

\newcommand{\im}{\mathrm{im}}
\newcommand{\re}{\mathrm{re}}

\newcommand{\cl}{\operatorname{cl}}
\newcommand{\Conf}{\operatorname{Conf}}
\newcommand{\Hilb}{\operatorname{Hilb}}
\newcommand{\Sym}{\operatorname{Sym}}
\newcommand{\symp}{\operatorname{symp}}
\newcommand{\alg}{\operatorname{alg}}
\newcommand{\Map}{\operatorname{Map}}

\newcommand{\mk}{\operatorname{mk}}
\newcommand{\mult}{\operatorname{mult}}
\newcommand{\virdim}{\operatorname{virdim}}
\newcommand{\length}{\operatorname{length}}
\newcommand{\Spec}{\operatorname{Spec}}

\newcommand{\hs}{\operatorname{hs}}
\newcommand{\coker}{\operatorname{coker}}
\newcommand{\fiber}{\operatorname{fiber}}

\newcommand{\lperf}{\text{-perf}}
\newcommand{\rperf}{\text{perf-}}

\def\eA{\EuScript{A}}

\def\eM{\EuScript{M}}

\begin{document}

\title
[Fukaya-Seidel categories of Hilbert schemes and parabolic category $\cO$ ]
{Fukaya-Seidel categories of Hilbert schemes \\ and parabolic category $\cO$ }
\author{Cheuk Yu Mak and Ivan Smith}

\date{\today}

\address{\tiny Centre for Mathematical Sciences, University of Cambridge, CB3 0WB, UK}
\email{cym22@dpmms.cam.ac.uk}

\address{\tiny Centre for Mathematical Sciences, University of Cambridge, CB3 0WB, UK}
\email{is200@dpmms.cam.ac.uk }
\thanks{Both authors are supported by EPSRC Fellowship EP/N01815X/1.}

\begin{abstract}
We realise Stroppel's extended arc algebra \cite{Stroppel-parabolic, BS11} in the Fukaya-Seidel category of a natural Lefschetz fibration on the 
generic fiber of the adjoint quotient map on a
type $A$ nilpotent slice with two Jordan blocks, and hence obtain a symplectic interpretation of certain parabolic two-block versions of Bernstein-Gelfan'd-Gelfan'd category $\cO$.  
As an application, we give a new geometric construction of the spectral sequence from annular to ordinary Khovanov homology.  The heart of the paper is the development of a cylindrical model to compute Fukaya categories of (affine open subsets of) Hilbert schemes of quasi-projective surfaces, which may be of independent interest.
\end{abstract}

\maketitle

\setcounter{tocdepth}{1}
\tableofcontents

\section{Introduction}\label{s:intro}

\subsection{Summary}
This paper has three parts:
\begin{itemize}
\item We give a `cylindrical' formulation of the Fukaya category of (an affine open subset of) the Hilbert scheme of an affine algebraic surface; (Sections \ref{s:cylindricalFS}-\ref{s:FScategories})
\item We compute the Fukaya-Seidel category of a natural Lefschetz fibration on the generic fiber of the adjoint quotient map on the
nilpotent / Slodowy slice of Jordan type $(n,m-n)$ arising from its inclusion in the $n$-th Hilbert scheme of the Milnor fibre of the $A_{m-1}$-surface singularity,  establishing a Morita equivalence to the principal block $\cO^{n,m}$ of BGG parabolic category $\cO$ associated to the partition $m=n+(m-n)$; (Sections \ref{s:TypeA}-\ref{s:formality})
\item We exploit this equivalence to describe a new semiorthogonal decomposition of $\rperf \cO^{n,2n}$, from which we derive a geometric construction of the spectral sequence from annular to ordinary Khovanov homology; (Sections \ref{s:Dictionary}-\ref{s:Ann2Kh}).
\end{itemize}

\subsection{Context}
Fix a semisimple complex Lie algebra $\lieg$.  The Bernstein-Gelfan'd-Gelfan'd category $\cO$ is an abelian (Noetherian and Artinian) category of finitely generated $\lieg$-modules, 
which plays a central role in many parts of representation theory.  It contains all the highest-weight modules, is closed under the operations of taking submodules, 
quotient modules and under tensoring with finite-dimensional modules, and is ``minimal" with respect to those properties.  
Concretely for $\liesl_m$, fixing the Cartan subalgebra $\mathfrak{h} \subset \liesl_m$ of diagonal matrices with respect to a choice of basis,
it comprises the finitely generated $U(\liesl_m)$-modules which are $\mathfrak{h}$-semisimple and which are locally finite for 
the nilpotent subalgebra $\mathfrak{n}\subset\liesl_m$ given by the upper triangular matrices.  Given a parabolic subalgebra $\mathfrak{p}\subset \mathfrak{g}$ containing the Borel $\mathfrak{n} \oplus \mathfrak{h}$, 
there is a parabolic subcategory $\cO^\liep$ of $\cO$ consisting of those modules on which $\liep$ acts locally finitely.  Again for $\liesl_m$, a partition 
$m=m_1+\cdots+m_k$ determines a parabolic $\liep$; the extreme cases $m=1+\cdots+1$ and $m=m$ respectively define $\cO^\liep=\cO$ and $\cO^{\liep} = \cO^{fd}$, the semisimple category of finite-dimensional $\liesl_m$-modules, so the parabolic subcategories can be viewed as intepolating between these.   
The central characters decompose $\cO^\liep$ into blocks and the block containing the trivial modules is called the principal block $\cO^\liep_0$.

It is known that for any $\liep\subset\lieg$, the category $\cO^\liep_0$ is equivalent to the category of modules over a finite-dimensional associative algebra; there are 
algorithmic descriptions of these algebras via quivers with relations \cite{Vybornov}, but nonetheless working with them concretely is often rather non-trivial.  
For parabolics $\liep  \subset \liesl_m$ associated to two-block partitions $m=n+(m-n)$, the corresponding category $\cO^\liep_0 = \cO^{n,m}$ 
has several other descriptions: it is Morita equivalent to the category of perverse sheaves on the Grassmannian $Gr(n,m)$ constructible with respect to the Schubert 
stratification \cite{Braden}, and the underlying associative algebra\footnote{
Our $K^{\alg}_{n,m}$ is $K_{n}^{m-n}$ in \cite{BS11}.  The superscript $alg$ (for algebraic) distinguishes it from a symplectic sibling which is introduced later in the paper.} 
$K_{n,m}^{\alg}$ has a diagrammatic description due to Stroppel \cite{Stroppel-parabolic} and Brundan-Stroppel \cite{BS11,BS2}, in which context it is known as the `extended arc algebra'. This paper gives a new interpretation of these algebras in terms of Fukaya-Seidel categories associated to natural Lefschetz fibrations on the
generic fiber of the adjoint quotient map of the
nilpotent (Jacobson-Morozov type) slices associated to nilpotent matrices with two Jordan blocks, and hence a symplectic-geometric construction of these particular principal blocks of parabolic category $\cO$.  This fits into the general dictionary between symplectic geometry and aspects of representation theory related to categorification, cf. \cite{Khovanov-Seidel, BLPW}, and simultaneously extends the symplectic viewpoint on Khovanov homology \cite{SeidelSmith, AbouzaidSmith16, AbouzaidSmith19} to a corresponding viewpoint on annular Khovanov homology, complementing recent work of \cite{GLW18, BPW19}. 

\subsection{Main result}

Let $\pi: A_{m-1} \to \CC$ be the standard Lefschetz fibration on the Milnor fibre of the $A_{m-1}$-surface singularity, so $\pi$ 
is an affine conic fibration with $m$ Lefschetz critical points; the symplectic topology of this fibration has been extensively studied, see for 
instance \cite{Khovanov-Seidel, Maydanskiy-Seidel}. 
The Hilbert scheme $\Hilb^n(A_{m-1})$ of zero dimensional length $n$ subschemes has a distinguished divisor $D_r$ of subschemes whose 
projection under $\pi$ has length $< n$, the complement of which is an affine variety $\cY_{n,m}$. 
(When $2n \le m$, this space is isomorphic to the generic fibre of the restriction of the adjoint quotient map $\liesl_{m} \to \mathfrak{h}/W$ to a transverse slice meeting 
the nilpotent cone at a matrix with Jordan blocks of size $\{n,m-n\}$, as studied in classical Springer theory \cite{Slodowy}.) 
The map $\pi$ induces a map $\pi_{n,m}:\cY_{n,m} \to \CC$ which is a Lefschetz fibration in the weak sense that its only critical points are of Lefschetz type 
(it may, however, have critical points at infinity when $n>1$). Nonetheless, there is a well-defined Fukaya-Seidel $A_{\infty}$-category $\cFS(\pi_{n,m})$, governing 
the Floer theory of the Lefschetz thimbles associated to any collection of vanishing paths.  
We write $D^{\pi}(\cC)$ for the derived category of $\cC$ (split-closure of twisted complexes), working over a field $\mathbb{K}$.

\begin{theorem}\label{t:main} 
If $n \le m$ and  $\mathbb{K}$ has characteristic zero, 
 $D^{\pi}\cFS(\pi_{n,m})$ is quasi-equivalent to the dg-category of perfect modules $\rperf K_{n,m}^{alg}$ over $K^{\alg}_{n,m}$, hence Morita equivalent to 
 parabolic category $\cO^{n,m}$.
\end{theorem}

\begin{remark}\label{r:nm}
 For the extreme case $n=m$, we have $D^{\pi}\cFS(\pi_{n,m}) = D^b(\mathbb{K})$ (see Remark \ref{r:nm2}) and $K_{n,m}^{alg}:=\mathbb{K}$.
 When $n>m$, the map $\pi_{n,m}$ has no critical point and $\cFS(\pi_{n,m})$ is an empty category; while $K_{n,m}^{alg}$ is not well-defined.
 When $n=0$, our convention is $\cFS(\pi_{n,m}):=\mathbb{K}$ and $K_{n,m}^{alg}:=\mathbb{K}$.
 Under this convention, Theorem \ref{t:main} is true for all non-negative integers $n,m$.
\end{remark}

\begin{remark}
The hypothesis on the ground field arises in 
two places.
The first one is
a formality theorem for the $A_{\infty}$-structure on the symplectic side, which we prove, following \cite{AbouzaidSmith16}, via methods of non-commutative 
algebra which rely on inverting all primes.   
The second one comes from the construction of an auxiliary $A_{\infty}$-category $\cFS^{cyl}$, which is a cylindrical model of the Fukaya-Seidel category, for which the definition of the $A_{\infty}$-operations involves dividing out by symmetry groups arising from re-ordering marked points on domains.
\emph{We work over a characteristic zero field unless stated otherwise.} All Lagrangians in the paper are assumed to be orientable and to admit (and be equipped with) spin structures. \end{remark}


\begin{remark} 
It is reasonable to expect that symplectic models for the categories $\cO^{\frak{p}}$, for parabolics $\frak{p}$ corresponding to partitions with more parts (or indeed in other simple $\frak{g}$), can also be obtained from the symplectic geometry of nilpotent slices. In general, however, such slices are not known to be re-interpretable in terms of Hilbert schemes of surfaces, so different techniques would be needed to analyse them.  Note that the map $\pi_{n,m}: \cY_{n,m} \to \CC$  is itself obtained from the embedding of $\cY_{n,m}$ into a Hilbert scheme; it would be interesting to characterise $\pi_{n,m}$ intrinsically Lie-theoretically.
\end{remark}

There is a well-known `tautological correspondence' between a holomorphic disc in the symmetric product of a complex manifold $X$ and a map of a branched cover of the disc to $X$ itself, cf. \cite{DonaldsonSmith, Costello, Lipshitz-cylindrical, Auroux-bordered} amongst others. Following \cite{Lipshitz-cylindrical} such correspondences are referred to as `cylindrical models' for computing Floer cohomology. Theorem \ref{t:main} uses the embedding $\cY_{n,m} \hookrightarrow \Hilb^n(A_{m-1})$  \cite{Manolescu06}, and the development of a cylindrical model for computing Fukaya-Seidel categories of Hilbert schemes.  Whilst the tautological correspondence has been widely exploited before, the proof  nonetheless requires numerous technical innovations: we require a model for the whole Fukaya category,  our Lagrangian submanifolds are not compact and not cylindrical at infinity, and the Hilbert scheme is related to the symmetric product by a non-trivial crepant resolution. A brief summary of the relevant issues is given in Section \ref{Sec:FScyl-overview} (our treatment requires only classical transversality theory, but an appeal to virtual perturbation theory would not bypass many of the difficulties).

\subsection{Consequences}

As commented previously, $K^{alg}_{n,m}$ is Morita equivalent to a category of perverse sheaves on the Grassmannian $Gr(m,n)$.  From the (Schubert-compatible) isomorphism $Gr(n,m) = Gr(m-n,m)$, one obtains a geometrically non-trivial:

\begin{corollary}\label{c:projDual}
$D^{\pi}\cFS(\pi_{n,m})$ is quasi-equivalent to $D^{\pi}\cFS(\pi_{m-n,m})$.
\end{corollary}

Note that the categories appearing in Corollary \ref{c:projDual} are associated to Hilbert schemes of $n$ respectively $m-n$ points, so take place in 
symplectic manifolds of different dimension.  For another non-trivial symmetry, the identity $K_{n,m}^{alg} \cong  (K_{n,m}^{alg})^{op}$ of the arc algebra with its 
opposite -- which has an easy diagrammatic proof -- is non-trivial on the geometric side, because of the non-trivial boundary conditions and `wrapping' involved in 
setting up a Fukaya-Seidel category (cf. Section \ref{Sec:consistent}).  

The Fukaya-Seidel category of any Lefschetz fibration has a full exceptional collection.  Additivity of Hochschild homology under semi-orthogonal decomposition \cite{Kuznetsov} immediately yields:

\begin{corollary} \label{c:HHarc}
The Hochschild homology of the extended arc algebra $HH_*(K_{n,m}^{alg})$ vanishes in degree $*\neq 0$, and in degree zero is free of rank $m \choose n$.
\end{corollary}

The Corollary was previously established by Beliakova \emph{et al} \cite{BPW19} by an ingenious and involved  algebraic argument involving a `quantum deformation' of the Hochschild complex. 

Underlying and extending Theorem \ref{t:main}, and the previous Corollaries, there is a rich dictionary between objects in the representation theory and in the symplectic geometry, 
which is useful in both directions.  The blocks of parabolic category $\cO$ are renowned instances of highest-weight categories, i.e. ones with a full exceptional collection  
where the exceptional objects generate rather than just split-generate.
A key point in the proof of Theorem \ref{t:main} 
is that the endomorphism algebra of the collection of Lagrangian submanifolds associated to projective modules is formal. 
Crucially, the `projective' Lagrangians are predicted by the diagrammatic combinatorics on the algebra side; away from Milnor fibres of surface singularities, 
they have no obvious counterpart for Hilbert schemes of general Lefschetz fibrations on quasi-projective surfaces.  

\begin{remark} In contrast to a number of recent formality results which hold (for instance) for degree reasons,  the formality of the symplectic extended arc algebra  holds for non-trivial geometric reasons, and is established, following \cite{AbouzaidSmith16}, by building a non-commutative vector field $b \in HH^1(\cFS(\pi_{n,m}))$  counting certain holomorphic curves with prescribed behaviour at infinity. 
The endomorphism algebra of the Lefschetz thimbles is in general \emph{not} formal (cf. Appendix \ref{Appendix}), and it seems hard to prove Theorem \ref{t:main} by directly comparing the $A_{\infty}$-algebras associated to distinguished bases of exceptional objects.  \end{remark}

In the other direction, the geometry gives rise rather directly to the following:

\begin{theorem}[see Theorem \ref{t:semiOrtho}]\label{t:semiorthogonal}
There is a semi-orthogonal decomposition $\rperf K_{n,m}^{alg} = \langle A_n,\ldots, A_0\rangle$ where $A_j$ is quasi-equivalent to 
$\rperf (K_{j,n}^{\alg} \otimes K_{n-j,m-n}^{\alg})$ for all $j=0,\dots,n$.
\end{theorem}

The equivalence of Theorem \ref{t:semiorthogonal} does not send tensor products of projectives to projectives, which may make it less transparent from the viewpoint of the extended arc algebras themselves.

\subsection{Invariants of braids}

The braid group $Br_m$ acts by symplectomorphisms on $\cY_{n,m}$; a braid $\beta$ defines a symplectomorphism $\phi_{\beta}^{(n)} $ and a corresponding bimodule $P_{\beta}^{(n)}$ over the Fukaya-Seidel category $D^{\pi}\cFS(\pi_{n,m})$. The Hochschild homology $HH_*(P_{\beta}^{(n)})$ is an algebraic analogue of the natural symplectic invariant $HF^*(\phi_{\beta}^{(n)})$ given by taking fixed-point Floer cohomology\footnote{There are numerous results relating fixed-point Floer cohomology and Hochshild homology, see e.g. \cite{Seidel2.5}, which for instance show that the two invariants co-incide when $n=1$. When $n>1$, since $\pi_{n,m}$ has critical points at infinity, the established results do not apply in our case; for simplicity we take our basic invariant of a braid to be the bimodule.}. We define the \emph{symplectic annular Khovanov homology} of $\beta \in Br_m$ by
\[
AKh^{\symp}(\beta) := \oplus_{j=0}^m \, HH_*(P_{\beta}^{(j)}).
\]
This is by definition an invariant of the braid. Recall from \cite{SeidelSmith} that there is a distinguished closed exact Lagrangian submanifold $L_{\wp} \subset \cY_{n,2n}$ with the property that
\[
Kh^{\symp}(\kappa(\beta)) := HF^*(L_{\wp}, \phi_{\beta}^{(n)}(L_{\wp}))
\]
is an invariant of the link closure $\kappa(\beta)$ of $\beta \in Br_n$, known as `symplectic Khovanov cohomology'.  This was introduced in \cite{SeidelSmith} as a singly graded sibling to combinatorial Khovanov homology; working over a characteristic zero field $\mathbb{K}$, the isomorphism $Kh(\kappa(\beta)) \cong Kh^{\symp}(\kappa(\beta))$ of $\ZZ$-graded\footnote{In this paper we will not discuss the grading;  the methods of \cite{AbouzaidSmith16} give rise to a second grading by elements of $\overline{\mathbb{K}}$, which is conjecturally integral, Markov-invariant and lifts the $\ZZ$-graded equivalence to a bigraded equivalence.} 
 vector spaces was established in \cite{AbouzaidSmith16, AbouzaidSmith19}.  

\begin{theorem}\label{t:sseq}
 There is a spectral sequence $AKh^{\symp}(\beta) \Rightarrow Kh^{\symp}(\kappa(\beta))$ from the symplectic annular Khovanov homology of $\beta$ to the symplectic Khovanov cohomology of the link closure $\kappa(\beta)$ of $\beta$.
\end{theorem}

Annular Khovanov homology was introduced by Asaeda, Przytycki and Sikora in \cite{APS04}, via a diagrammatic calculus for links in a solid torus.  Roberts \cite{Roberts13} showed there is a spectral sequence $AKh(\beta) \Rightarrow Kh(\kappa(\beta))$, but the fact that $AKh(\beta)$ splits into summands which can be identified with Hochschild homology groups is non-trivial; this was conjectured by Auroux, Grigsby and Wehrli \cite{AGW} and proven very recently by Beliakova, Putyra and Wehrli \cite{BPW19}.  By contrast, if one had \emph{defined} the annular Khovanov invariant as such a direct sum of Hochschild homologies, the existence of the spectral sequence would seem rather mysterious: the bimodules $P_{\beta}^{(j)}$ over the categories $\cFS(\pi_{j,m})$ do not in themselves have enough information to determine the differentials in the spectral sequence (which do not preserve the decomposition by $j$).   Theorem \ref{t:sseq} gives a geometric explanation for the existence of a spectral sequence from a direct sum of Hochschild homologies to (symplectic) Khovanov cohomology; the crucial input is the semi-orthogonal decomposition from Theorem \ref{t:semiorthogonal}.  We remark that there is an analogous spectral sequence from knot Floer homology, viewed as Hochschild homology of a suitable bimodule via \cite{LOT-bimodules}, to Heegaard Floer homology; it would be interesting to see if that can be derived following the methods of this paper. 

\subsection{Outline of the paper} 
Let $\pi_E: E \to \CC$ be a Lefschetz fibration on a quasi-projective surface. The map $\pi_E$ defines a map $\pi_E^{[n]}: \Hilb^n(E) \to \CC$,  by taking the map induced from the sum of 
copies of $\pi_E$ on the product $E^n$.  The map $\pi_E^{[n]}$ is a Lefschetz fibration when restricted to the affine open subvariety of the Hilbert scheme 
$\Hilb^n(E)$ comprising subschemes whose projection to $\CC$ has length $n$. 
Section \ref{s:cylindricalFS} develops a cylindrical model $\cFS^{cyl}$ for the Fukaya(-Seidel) category of this associated Lefschetz fibration; 
section \ref{s:properties} establishes some basic properties of $\cFS^{cyl}$ analogous to the usual Fukaya(-Seidel) category; 
section \ref{s:FScategories} relates 
the cylindrical model to the usual model $\cFS$, yielding in general an embedding $\cFS \hookrightarrow \cFS^{cyl}$;
and section \ref{s:TypeA} applies this framework to type $A$ Milnor fibres.
Sections \ref{s:ArcAlg}-\ref{s:AlgIsom} construct a symplectic version of the extended arc algebra, which is a priori an $A_{\infty}$-algebra, 
and prove that its cohomology agrees with its combinatorial sibling; sections \ref{s:ncField}-\ref{s:formality} then prove the $A_{\infty}$-structure is actually formal 
in characteristic zero. Section \ref{s:Dictionary} illustrates the dictionary between natural objects in the representation theory of the extended arc algebras and their 
geometric counterparts. Sections \ref{s:Annular} and \ref{s:Ann2Kh} introduce symplectic annular Khovanov homology, and relate this to ordinary Khovanov homology via a 
particular semiorthogonal decomposition of the Fukaya-Seidel category associated to the $(n,2n)$-nilpotent slice (associated to a nilpotent with Jordan blocks of size $(n,n)$). 
The Appendix shows that the $A_{\infty}$-endomorphism algebra of the thimbles in the Fukaya-Seidel category is in general not formal; this underscores the important contribution of the relationship to representation theory, which picks out a different set of generators for the category (as studied in sections \ref{s:TypeA}-\ref{s:formality}) which do have formal endomorphism algebra.


\vspace{0.2cm}
\noindent \textbf{Acknowledgements.} 
The authors are grateful to Mohammed Abouzaid, Robert Lipshitz, Jacob Rasmussen, Catharina Stroppel and Filip Zivanovic for many helpful conversations.
We thank the anonymous referee for their numerous helpful comments and corrections.

\section{The cylindrical Fukaya-Seidel category}\label{s:cylindricalFS}


\subsection{Overview\label{Sec:FScyl-overview}}

The aim of this section is to define, given a Lefschetz fibration $\pi_E$ on a complex surface $E$ equipped with an exact symplectic structure (that satisfies some mild additional hypothesis, see Section \ref{sss:targetspace}) and a positive integer $n$, 
an $A_{\infty}$ category $\cFS^{cyl,n}(\pi_E)$, which we call the {\it n-fold cylindrical Fukaya-Seidel category}. 
Objects in this category are, roughly speaking, unordered $n$-tuples of pairwise disjoint exact Lagrangians $\ul{L}=\{L_1,\dots,L_n\}$, with $L_i \subset E$.
The product of the $L_i$ defines a Lagrangian (which descends to one) in $\Sym^n(E)$,  which can be lifted to a Lagrangian $\Sym(\ul{L})$ in $\Hilb^n(E)$. 

Let $\pi_E^{[n]}:\Hilb^n(E) \to \CC$ be the map induced from the sum of $n$ copies of $\pi_E$.
There is a divisor $D_r$ of $\Hilb^n(E)$ such that $\cY_E:=\Hilb^n(E) \setminus D_r$ admits an exact symplectic structure, and such that 
\begin{align}
\pi_{\cY_E}:=\pi_E^{[n]}|_{\cY_E}:\cY_E \to \CC
\end{align}
is a Lefschetz fibration in the weak sense that all the critical points of $\pi_{\cY_E}$ are of Lefschetz type\footnote{There may be critical points at infinity, so this is not a symplectic Lefschetz fibration in the usual sense.}.  Every Lefschetz thimble of $\pi_{\cY_E}$ is given by $\Sym(\ul{L})$ for some $\ul{L}$.
If $E$ is the $A_{m-1}$-Milnor fiber and $\pi_E$ is the conic fibration with $m$ critical points, then $\cY_E$ and $\pi_{\cY_E}$ coincide with $\cY_{n,m}$ and $\pi_{n,m}$, respectively.

The $A_{\infty}$ structure of $\cFS^{cyl,n}(\pi_E)$ is carefully defined so that it has the property that the category of Lefschetz thimbles of $\pi_{\cY_E}$
embeds into $\cFS^{cyl,n}(\pi_E)$ as a full subcategory, where on the object level, it is given by $\Sym(\ul{L}) \mapsto \ul{L}$.
This reduces calculations in the Fukaya-Seidel category of $\pi_{\cY_E}$ to 
more accessible calculations in $\cFS^{cyl,n}(\pi_E)$, which involve holomorphic curve counts in $E$ itself.

Ideally, we would like to have a choice of perturbation scheme such that there is a bijective correspondence between (perturbed)-holomorphic polygons $u:S \to \cY_E$ contributing to the $A_{\infty}$ structure
of $\cFS(\pi_{\cY_E})$
and pairs $(\pi_\Sigma,v)$ such that $\pi_\Sigma: \Sigma \to S$ is an 
$n$-fold branched covering and $v:\Sigma \to E$ is a solution to a (perturbed)-holomorphic equation.
In this case, if we define the $A_{\infty}$ structure of $\cFS^{cyl,n}(\pi_E)$ as a signed rigid count of the pairs $(\pi_\Sigma,v)$, then it would be tautological that the
$A_{\infty}$ categories $\cFS(\pi_{\cY_E})$ and $\cFS^{cyl,n}(\pi_E)$ would be equivalent. Problems arise when one implements this idea in practice:

\begin{itemize}
 \item The domain-dependent almost complex structure on $E$ has to be complex (and hence induce an almost complex structure on $\Hilb^n(E)$) to have any hope of a bijective correspondence between $u$ and $(\pi_\Sigma,v)$, which puts some restrictions on the perturbation scheme;
 \item Hamiltonian perturbations in $\Hilb^n(E)$ that are simultaneously induced from Hamiltonian perturbations in $E$ and which preserve $D_r$ are not general enough to achieve transversality; 
 strictly speaking, there is no perturbation scheme that can allow us to define the $A_{\infty}$ structure of $\cFS^{cyl,n}(\pi_E)$ by merely counting $(\pi_\Sigma,v)$;
 \item We need to include some $\ul{L}$ for which $\Sym(\ul{L})$ is not cylindrical in  $\cY_E$ (e.g. to obtain formal $A_{\infty}$ Floer cochain algebras in the setting of Theorem \ref{t:main}), so 
 compactness of the moduli involved has to be carefully addressed.
\end{itemize}

Our main contribution to the construction of $\cFS^{cyl,n}(\pi_E)$ is to overcome the above difficulties, roughly as follows:

\begin{itemize}
 \item Instead of using moduli of polygons $\cR^{d+1}$, we use moduli of polygons with ordered interior marked points $\cR^{d+1,h}$ to define the $A_{\infty}$ structure. 
 The interior marked points keep track of the branch points of $\pi_\Sigma:\Sigma \to S$, and the domain-dependent almost complex structure is only required to be  integrable
 near the interior marked points. This gives us more flexibility for the perturbation scheme and at the same time partially recovers  the bijective correspondence between $u$ and $(\pi_\Sigma,v)$.
 \item Whilst the $A_{\infty}$ structure of $\cFS^{cyl,n}(\pi_E)$ is defined by counting certain solutions $u$ mapping to $\Hilb^n(E)\setminus D_r$, rather than pairs  $(\pi_\Sigma,v)$, 
 our set up is sufficiently flexible that the Floer differential and product  can be computed by  counting solutions $(\pi_\Sigma,v)$.
 \item To achieve compactness, we need to avoid that solutions in $\Hilb^n(E)\setminus D_r$ escape to the vertical boundary, horizontal boundary or $D_r$.
 No escape along the vertical boundary is achieved by modifying Seidel's ingenious set-up in \cite{Seidel4.5}; one takes the base of the Lefschetz fibration to be the upper half-plane, and uses hyperbolic isometries to gauge back perturbed holomorphic curves to actual holomorphic curves.
 No escape along the horizontal boundary and into $D_r$ are each achieved by positivity of intersection, which relies on a delicate choice of Hamiltonian perturbation scheme
 and some particular geometric features of $D_r$.  Familiarity with \cite{Seidel4.5} may be helpful.

\end{itemize}

\subsection{Definitions and the setup}\label{ss:DefnsSetup}

\subsubsection{Domain moduli}

Let $\cR^{d+1,h}$ be the moduli space of discs with $d+1$ punctures on the boundary and $h$ ordered and pairwise distinct 
interior marked points. 
Let $\cS^{d+1,h}$ be the universal family of $\cR^{d+1,h}$.
We fix a distinguished puncture $\xi^0$ for the elements in $\cR^{d+1,h}$ consistently, and order the remaining punctures $\xi^1,\dots,\xi^d$ counter-clockwise along the boundary.
For $S \in \cR^{d+1,h}$, 
we denote the ordered interior marked points by $\xi^1_+, \dots \xi^h_+$ and
we use $\partial_j S$ to denote the boundary component of $S$ between $\xi^j$ and $\xi^{j+1}$ for $j=0,\dots,d$ ($\xi^{d+1}$ is understood as $\xi^0$).
We use $\mk(S)$ to denote the set of interior marked points of $S$.

The moduli space $\cR^{d+1,h}$ can be compactified to the moduli of stable discs $\ol{\cR}^{d+1,h}$.
The latter moduli is used to define bulk deformation in \cite{FOOOBook2}, \cite{Siegel}, to which  readers are referred for the details of its construction.
We denote the universal family over $\ol{\cR}^{d+1,h}$ by $\ol{\cS}^{d+1,h}$.

\subsubsection{Strip-like ends and marked-points neighborhoods}\label{sss:EndsAndNeighborhoods}

For each $\cR^{d+1,h}$, we make a choice of strip-like ends $\epsilon=\{\epsilon^0, \dots, \epsilon^d\}$
for elements in $\cR^{d+1,h}$ such that $\xi^0$ is an output and $\xi^j$ is an input for $j=1,\dots,d$.
Thus, for each $S \in \cR^{d+1,h}$, we have holomorphic embeddings varying smoothly respect to $S$
\[   \left\{
\begin{array}{ll}
      \epsilon^0:\RR^{\le 0} \times [0,1] \to S \\
       \epsilon^1, \dots \epsilon^d :\RR^{\ge 0}\times [0,1] \to S \\
      \lim_{s \to  \pm \infty}\epsilon^j(s,\cdot)=\xi^j \\
      (\epsilon^j)^{-1}(\partial S)=\{(s,t):t=0,1\}.
\end{array}
      \right. \]
We denote $\cup_{S \in \cR^{d+1,h}} Im(\epsilon_j)$ in $\cS^{d+1,h}$ by $N_{\epsilon_j}^{d+1,h}$.
      
Furthermore, we choose a `marked-points neighborhood' $\nu(\mk(S))$ for each $S \in \cR^{d+1,h}$, which is a  (possibly disconnected) open subset of $S$ containing $\mk(S)$,
such that 
\begin{align}
 N_{\mk}^{d+1,h}:=\cup_{S \in \cR^{d+1,h}} \ol{\nu(\mk(S))} \label{eq:Nmk}
\end{align} 
is a smooth submanifold (with boundary) in $\cS^{d+1,h}$.
We require that $N_{\mk}^{d+1,h} \cap N_{\epsilon_j}^{d+1,h} =\emptyset$ for all $j$.

We fix a choice of cylindrical ends for the interior marked points of $S \in \cR^{d+1,h}$, and for the interior marked points for the elements in the moduli of spheres with ordered marked points. Given that choice, one obtains a smooth structure on 
$\ol{\cR}^{d+1,h}$ (see \cite[Section 4]{Siegel}, \cite[Section (9g)]{SeidelBook}).

A choice of strip-like ends for all elements of $\cR^{d+1,h}$, for all $d$ and $h$, is called {\it consistent} if it extends to a 
choice of strip-like ends (smooth up to the boundary) for all elements in $\ol{\cR}^{d+1,h}$.
A choice of marked-points neighborhoods for all elements of $\cR^{d+1,h}$, for all $d$ and $h$, is called {\it consistent} if
the closure of their union, denoted by $\ol{N}_{\mk}^{d+1,h}$, is a smooth submanifold with boundary and corners in  $\ol{\cS}^{d+1,h}$.
We fix such a consistent choice of strip-like ends and marked-points neighborhoods.

\begin{remark}\label{r:sphereCom}
Note that $\mk(S) \subset \nu(\mk(S))$ for all $S \in \cR^{d+1,h}$ implies that
if $S \in  \ol{\cR}^{d+1,h}$ has a stable sphere component $Q$, then $Q \subset \nu(\mk(S))$.

\end{remark}

\subsubsection{An isometry group}

Let $G_{\aff}$ be the group of orientation preserving affine transformations of the real line and $\kg_{\aff}$ be its Lie algebra.
Let $\bH$ be the closed upper half-plane, whose interior $\bH^{\circ}$ is equipped with the hyperbolic area form $\omega_{\bH^\circ}=\frac{d \re (w) \wedge d \im(w)}{\im(w)^2}$
and primitive $\theta_{\bH^\circ}=-d^c(\log(\im(w)))$.
We write $\partial_{\infty} \bH:=\bH \setminus \bH^\circ$.

The $G_{\aff}$-action on the real line extends to $\bH$ and we have Lie algebra homomorphism 
\begin{align}
\kg_{\aff} \to C^{\infty}(\bH^\circ, T\bH^\circ) 
\end{align}
which sends $\gamma \in \kg_{\aff}$ to a Hamiltonian vector field $X_\gamma$. We define 
\begin{equation}
 H_\gamma:=\theta_{\bH^\circ}(X_\gamma)
\end{equation}
which is a Hamiltonian on $\bH^\circ$ generating $X_\gamma$.

\subsubsection{Target space}\label{sss:targetspace}

Our setup is modified from \cite{Seidel4.5}.
Let $(E^\rceil,J_{E^\rceil})$ be a complex surface with boundary and let 
\begin{align}
 \pi_{E^\rceil}:E^\rceil: \to \bH
\end{align}
be a proper Lefschetz fibration such that $\partial E^\rceil= \pi_{E^\rceil}^{-1}(\partial \bH)$.
Let $\omega_{\overline{E}}$ be a symplectic form on  $\overline{E}:= E^\rceil \setminus \partial E^\rceil$
which tames $J_{\overline{E}}:=J_{E^\rceil}|_{\overline{E}}$
and makes
\begin{align}
 \pi_{\overline{E}}:=\pi_{E^\rceil}|_{\overline{E}} \to \bH^\circ
\end{align}
into a symplectic Lefschetz fibration.

Let $D_E$ be a smooth and reduced (but possibly disconnected) divisor of $E^\rceil$.
Let $E:=\overline{E} \setminus D_E$, $J_E:=J_{E^\rceil}|_E$, $\omega_E:=\omega_{\overline{E}}|_E$ and $\pi_E:=\pi_{\overline{E}}|_E$.
We assume that there is a primitive $\theta_E$ for $\omega_E$ on $E$ (so, in particular, it implies that $\pi_{E^\rceil}|_{D_E}$ is surjective).

We assume that there is a contractible compact subset $C_{\bH} \subset \bH^\circ $ such that $\pi_{\overline{E}}|_{\pi_{\overline{E}}^{-1}(\bH^\circ \setminus C_{\bH})}$ is symplectically locally trivial.
It means that for all $x \in \pi_{\overline{E}}^{-1}(\bH^\circ \setminus C_{\bH})$
\begin{enumerate}
 \item $x$ is a regular  point of $\pi_{\overline{E}}$, and
 \item the horizontal distribution $T_x^h\overline{E}:=(T^v_x\overline{E})^{\perp_{\omega_{\overline{E}}}}$ is integrable in a neighborhood of $x$, and 
 \item $\omega_{\overline{E}}|_{T_x^h\overline{E}}=(\pi_{\overline{E}})^*\omega_{\bH^{\circ}}|_{T_x^h\overline{E}}$
\end{enumerate}
We also require that for $x \in D_E \cap \pi_{\overline{E}}^{-1}(\bH^\circ \setminus C_{\bH})$, we have $T_x^h\overline{E} = T_xD_E$.


If we pick a point $* \in \bH^\circ \setminus C_{\bH}$ and define $(F,\omega_F, \theta_F):=(\pi_E^{-1}(*), \omega_E|_{\pi_E^{-1}(*)},\theta_E|_{\pi_E^{-1}(*)})$, then
the conditions above imply that for every $z \in \bH^\circ \setminus C_{\bH}$, there is an open neighborhood $U$ of $z$, and a 
symplectomorphism $(\pi_E^{-1}(U),\omega_E) \to (U \times F, \omega_{\bH^\circ}|_U + \omega_F)$ compatible with the projection to $U$.
Therefore, there is a natural way to extend the symplectic form $\omega_{\overline{E}}$ to $\omega_{E^\rceil}$ on $E^\rceil$.

Let $\cJ(E)$ be the space of $\omega_E$-tamed almost complex structures $J$ on $E$ such that 
\begin{align}
 & J=J_E \text{ outside a compact subset in $E$}  \label{eq:pseudoconvex}\\
 &\pi_E \text{ is $(J,j_{\bH^\circ})$-holomorphic}. \label{eq:fiberJ}
\end{align}
Condition \eqref{eq:pseudoconvex} implies that every $J \in \cJ(E)$ can be smoothly extended to an $\omega_{E^\rceil}$-tamed
almost complex structure $J^\rceil$ on $E^\rceil$, and we assume that
\begin{align}
 \text{ every $J^\rceil$-holomorphic map $\CP^1 \to E^\rceil$ has positive algebraic intersection number with $D_E$}. \label{eq:StrictPositiveInt}
\end{align}

We also assume that $c_1(E)=0$ and a trivialization of the canonical bundle is chosen.

For $\gamma \in \kg_\aff$, let 
\begin{align}
 \cH_\gamma(\bH^\circ):=&\{H \in C^{\infty}(\bH^\circ,\RR)| H \text{ is a constant in a neighborhood of }C_{\bH}, \label{eq:HH}\\ 
 &\text{and } H=H_\gamma \text{ outside a compact subset of $\bH^\circ$}\} \nonumber \\
 \cH_\gamma(E):=&\{H \in C^{\infty}(E,\RR)| H=\pi_E^*H' \text{ for some }H' \in \cH_\gamma(\bH^\circ)\} \label{eq:HE}
\end{align}
Then for $H \in \cH_\gamma(\bH^\circ)$, we have
\begin{align}
 X_H|_{C_\bH}=0 \quad \text{ and } \quad X_H=X_{H_\gamma} \text{ outside a compact set}.
\end{align}
Since $\pi_E$ is symplectically locally trivial outside $\pi_E^{-1}(C_\bH)$, for $H \in \cH_\gamma(E)$, $X_H$ is uniquely determined by the property that
\begin{align}
 (\pi_E)_* (X_H|_x)= X_{H'}|_{\pi_E(x)}
\end{align}
for all $x \in E$.
If we extend $X_H$ to a smooth vector field on $\overline{E}$, we also have
\begin{align}
 X_H|_{x} \in T_xD_E \text{ for all } x \in D_E \label{eq:tangentD0}
\end{align}

Next, we consider the class $\cL$ of properly embedded, oriented and spin Lagrangian submanifolds $L \subset E$ such that
\begin{align}
 & \pi_E|_L \text{ is proper and }\partial L=\emptyset \\
 &\text{there is a compact subset $C_L \subset \bH^\circ$ such that $\pi_E(L)\setminus C_L$ is either empty or is
a properly}\\
 &\text{embedded arc $\gamma_L$ in $\bH^\circ$, and if $\pi_E(L)\setminus C_L \neq \emptyset$ 
then we denote the point $\overline{\gamma_L} \setminus \gamma_L$ by $\lambda_L \in \RR$, }\nonumber \\
&\text{where $\overline{\gamma_L}$ is the closure of $\gamma_L$ in $\bH$} \nonumber \\ 
 &L \text{ is exact with respect to } \theta_E.
\end{align}
Note that, if
\begin{align}
 \left\{ 
 \begin{array}{ll}
  A=a_tdt \in \Omega^1([0,1], \kg_{\aff} ) \\
 H=(H_t)_{t \in [0,1]} \text{ such that } H_t \in \cH_{a_t}(E)
 \end{array}
 \right. 
\end{align}
and $\phi_H$ is the associated Hamiltonian diffeomorphism (which is well-defined everywhere because $\phi_H|_{\pi^{-1}_E(C_{\bH})}$ is the identity),
then $\phi_H^{-1}(L) \in \cL$ and $\lambda_{\phi_H^{-1}(L)}=g_A^{-1}\lambda_{L}$,
where $g_A \in G_{\aff}$ is the parallel transport from $0$ to $1$ with respect to $A$ (see Figure \ref{fig:parallelTrans}).

\begin{figure}[ht]
 \includegraphics{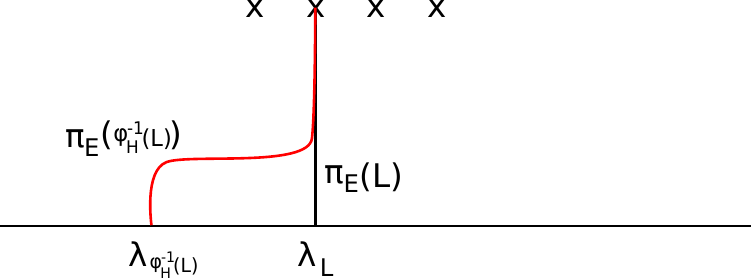}
 \caption{An example of $\pi_E(L)$ (black) and $\pi_E(\phi_H^{-1}(L))$ (red); crosses are critical values.}\label{fig:parallelTrans}
\end{figure}

Finally, we consider the set $\cL^{cyl,n}$ of (unordered) $n$-tuples of Lagrangians
$ \ul{L}=\{L_1,\dots,L_n\}$ such that
\begin{align}
 & L_k \in \cL \text{ for } k=1,\dots,n \\
 &\text{there are pairwise disjoint contractible open sets $U_{L_k}$ in $\bH$ such that} \label{eq:DisjointOpen} \\
 &\text{$\pi_E(L_k) \subset U_{L_k}$ and $U_{L_k} \cap \partial \bH$ is either empty or contractible} \nonumber 
\end{align}
In particular, \eqref{eq:DisjointOpen} implies that the $\{L_k\}$ are pairwise disjoint.
We define
\begin{align}
 \Sym(\ul{L})&:=q_{S_n}(L_1 \times \cdots \times L_n) \subset \Conf^n(E) \subset \Sym^n(E) \label{eq:SymL}\\ 
 \Sym(U_{\ul{L}}) &:=q_{S_n,\bH^\circ}(U_{L_1} \times \dots U_{L_n}) \subset \Sym^n(\bH^\circ) \label{eq:SymU}\\
 \lambda_{\ul{L}}&:=[\min_k \{\lambda_{L_k}\}, \max_k \{\lambda_{L_k}\}] \subset \RR \label{eq:Symlambda}
\end{align}
where $q_{S_n}: E^n \to \Sym^n(E)$ and $q_{S_n, \bH^\circ}: (\bH^\circ)^n \to \Sym^n(\bH^\circ)$ are  the quotient maps by the symmetric group.
When $\lambda_{\ul{L}}=\emptyset$ (i.e. when $\pi_E(L_k)$ is compact for all $k=1,\dots,n$), we define $\lambda_{\ul{L}}:=\{0\}$ so $\lambda_{\ul{L}}$ is a non-empty closed interval for all $\ul{L} \in \cL^{cyl,n}$.


\subsubsection{Hilbert scheme of points}

Let $\Hilb^n(E)$ be the Hilbert scheme of zero dimensional length $n$ subschemes on $E$, defined with respect to the complex structure $J_E$ on $E$.
We denote the Hilbert-Chow divisor by $D_{HC}$ and the relative Hilbert scheme with respect to $\pi_E$ by $D_r$, i.e. this is the divisor of subschemes whose projection under $\pi_E$ has length $<n$. (We will sometimes write $\cY_E$ for the complement $\Hilb^n(E) \backslash D_r$ when the particular value of $n$ is implicit or plays no role.) 
Let $\pi_{HC}:\Hilb^n(E) \to \Sym^n(E)$ be the contraction of $D_{HC}$ and $\Sym^n(\pi_E): \Sym^n(E) \to \Sym^n(\bH^\circ)$ be the natural map induced by $\pi_E$.
Let $\Delta_{\bH^\circ} \subset \Sym^n(\bH^\circ)$ be the big diagonal (i.e. all unordered tuples $\{q_1,\dots,q_n\}$ of points in $\bH^\circ$ such that $q_i = q_j$ for some $i \neq j$).

\begin{lemma}\label{l:DhcDr}
 $(\Sym^n(\pi_E) \circ \pi_{HC})^{-1}(\Delta_{\bH^\circ})=D_{HC} \cup D_r$.
\end{lemma}

\begin{proof}
 It is clear that $\Sym^n(\pi_E) \circ \pi_{HC}(D_{HC} \cup D_r) \subset \Delta_{\bH^\circ}$.
 For the converse, if the support of $z \in \Hilb^n(E) \cap (\Sym^n(\pi_E) \circ \pi_{HC})^{-1}(\Delta_{\bH^\circ})$ is a union of distinct $n$ points, then $z \in D_r$.
 If the support consists of $k < n$ points instead, then $z \in D_{HC}$. 
\end{proof}




Note that $\Hilb^n(E) \setminus D_{HC}=\Conf^n(E)$, and we have a trivialization of the canonical bundle of $\Conf^n(E)$
induced by that of $E$.

\begin{lemma}\label{l:canonicalBundle}
 The trivialization of the canonical bundle of $\Conf^n(E)$ extends smoothly to a trivialization of the canonical bundle of $\Hilb^n(E)$.
\end{lemma}

\begin{proof}
 It follows from the fact that $\pi_{HC}$ is a crepant resolution of $\Sym^n(E)$.
\end{proof}

We equip $\Conf^n(E)$ with the product symplectic form $\omega_{\Conf^n(E)}$ from $(E,\omega_E)$. This is smooth,
but cannot be smoothly extended to a $2$-form on $\Hilb^n(E)$.

\begin{lemma}\label{l:tameSymp}
 For every open neighborhood $U \subset \Hilb^n(E)$ of $D_{HC}$,
 there is a symplectic form on $\Hilb^n(E)$ which tames the complex structure, and coincides with $\omega_{\Conf^n(E)}$ outside $U$.
\end{lemma}

\begin{proof}
  This follows essentially from \cite{Voisin}, see also \cite{Perutz1, Perutz2} in the K\"ahler case.  Let $D_{HC}^\rceil$ be the Hilbert-Chow divisor of $\Hilb^n(E^\rceil)$.
 Let $U,V \subset \Hilb^n(E^\rceil)$ be an open neighborhoods of $D_{HC}^\rceil$
such that $U$ contains the closure of $V$.
 In \cite{Voisin}, Voisin constructed two smooth closed two-forms $\chi$ and $\Psi$ on $\Hilb^n(E^\rceil)$ 
 such that for some $\lambda_0 >0$ and for all $\lambda \in (0,\lambda_0)$, $\chi+ \lambda \Psi$ is a symplectic form on $\Hilb^n(E^\rceil)$
 that tames the complex structure.
 Moreover, $\chi|_{\Hilb^n(E^\rceil) \setminus V}=\omega_{\Conf^n(E^\rceil)}|_{\Conf^n(E^\rceil) \setminus V}$ and $\Psi|_{\Hilb^n(E^\rceil) \setminus D_{HC}^\rceil}=d\Theta$ for some $\Theta \in \Omega^1(\Hilb^n(E^\rceil) \setminus D_{HC}^\rceil)$.
 
 Let $\rho:\Hilb^n(E^\rceil) \to [0,1]$ be a cut off function such that $\rho|_V=1$ and $\rho=0$ outside $U$.
 Then $(\chi+\lambda d(\rho \Theta))|_{\Hilb^n(E^\rceil) \setminus U}= \omega_{\Conf^n(E^\rceil)}|_{\Conf^n(E^\rceil) \setminus U}$
 and $(\chi+\lambda d(\rho \Theta))|_V=(\chi+ \lambda \Psi)|_V$.
 
 Since $\chi$ is non-degenerate and tames the complex structure outside $V$, and being non-degnerate and taming are both open conditions, for sufficiently small 
 $\lambda >0$, we know that $\chi+\lambda d(\rho \Theta)$ is a symplectic form which tames the complex structure outside $V$.
 Moreover, this is also true inside $V$ because $\chi+\lambda d(\rho \Theta)=\chi+ \lambda \Psi$ inside $V$.
 Therefore, we can restrict this symplectic form to $\Hilb^n(E)$ to get the result.
\end{proof}

\begin{lemma}
\label{l:rationalCurveinHilb} If $C$ is the image of a non-constant rational curve in $\Hilb^n(E)$,
then $[C] \cdot [D_r] >0$.
\end{lemma}

\begin{proof}
From Lemma \ref{l:tameSymp}, one sees that the symplectic form on 
$\Hilb^n(E)$ is Poincare dual to the relative cycle $-\epsilon [D_{HC}]$  for some $\epsilon>0$ (recall that $\omega_E$ is exact).
It is proved in \cite[Lemma 5.4]{AbouzaidSmith16} that $[D_r]$ is a positive multiple of $-[D_{HC}]$ so $[C] \cdot [D_{r}] >0$ follows from positivity of the symplectic area of $C$.
\end{proof}

Let $D_r^{\circ}=D_r \cap \Conf^n(E)$.
For $J\in \cJ(E)$, we define $J^{[n]}$ to be the almost complex structure on $\Conf^n(E)$ decended from 
the product almost complex structure $J^n$ on $E^n$.
Note that $J^{[n]}$ is $\omega_{\Conf^n(E)}$-tamed and when $J=J_E$, $J^{[n]}$ extends smoothly to $\Hilb^n(E)$.
We define the following space of  almost complex structures on $\Conf^n(E)$:
\begin{align}
 \cJ^{n}(E):=&\{J | J=(J')^{[n]} \text{ for some } J' \in \cJ(E)\} \label{eq:JnE}
\end{align}

For $H \in C^{\infty}(E,\RR)$, we define $H^{[n]} \in C^{\infty}(\Conf^n(E),\RR)$ to be 
\begin{align}
 H^{[n]}(\ul{z}):=\sum_{i=1}^n H(z_i) \label{eq:sumH}
\end{align}
for $\ul{z}=\{z_1,\dots,z_n\} \in \Conf^n(E)$.
For $\gamma \in \kg_\aff$, let 
\begin{align}
 \cH_\gamma^{n,pre}(E):=&\{H \in C^{\infty}(\Conf^n(E),\RR)|H=(H')^{[n]} \text{ for some }H' \in \cH_\gamma(E)\} \\
 \cH_\gamma^n(E):=&\{H \in C^{\infty}(\Conf^n(E),\RR)| H=H' \text{ outside a compact subset of } \label{eq:HnE}\\
                  & \Conf^n(E) \setminus D_r^{\circ}, \text{ for some $H' \in \cH_\gamma^{n,pre}(E)$}\} \nonumber
\end{align}
Note that, if
\begin{align}
 \left\{ 
 \begin{array}{ll}
  A=a_tdt \in \Omega^1([0,1], \kg_{\aff} ) \\
 H=(H_t)_{t \in [0,1]} \text{ such that } H_t=(H'_t)^{[n]} \in \cH_{a_t}^{n,pre}(E) \text{ for } H'_t \in \cH_{a_t}(E)
 \end{array}
 \right. 
\end{align}
then $\phi_H$ is a well-defined Hamiltonian diffeomorphism of $\Conf^n(E)$ and the Hamiltonian vector field $X_{H_t}$ satisfies
\begin{align}
X_{H_t}|_p \in T_pD_r^{\circ}  \label{eq:preserveDr}
\end{align}
for all $p \in D_r^{\circ}$ and for all $t \in [0,1]$.
It implies that $\phi_H(D_r^{\circ})=D_r^{\circ}$.
Moreover, we have $\phi_H^{-1}(\Sym(\ul{L}))=\Sym(\phi_{H'}^{-1}(\ul{L})) \in \Sym^n(\cL)$ and 
\begin{align}
\lambda_{\phi_{H'}^{-1}(\ul{L})}=g_A^{-1}\lambda_{\ul{L}}. \label{eq:gaugelambda}
\end{align}
As a result, if $H=(H_t)_{t \in [0,1]}$ and $H_t \in \cH_{a_t}^n(E)$ for all $t$, the associated Hamiltonian vector field 
satisfies \eqref{eq:preserveDr} and $\phi_H$ is also a well-defined Hamiltonian diffeomorphism of $\Conf^n(E)$.
Moreover, $\phi_H^{-1}(\Sym(\ul{L}))=\Sym(\ul{L'})$ outside a compact set of $\Conf^n(E)\setminus D_r^{\circ}$ for some $\ul{L'} \in \cL^{cyl,n}$.
Therefore, we can define $\lambda_{\phi_H^{-1}(\Sym(\ul{L}))}=\lambda_{\ul{L'}}$.

\begin{remark}\label{r:FlexH}
 We introduce both $\cH_\gamma^{n,pre}(E)$ and $\cH_\gamma^n(E)$ because, on the one hand, $\cH_\gamma^n(E)$ gives us more freedom to achieve transversality, but on the other, working with 
 $\cH_\gamma^{n,pre}(E)$ simplifies explicit computations for cases that transversality can be achieved within that more restricted class.
\end{remark}

\subsubsection{Floer cochains}\label{sss:FloerCochain}

Let $\cG_{\aff}([0,1]):=C^{\infty}([0,1], G_\aff)$ which is weakly contractible.
Let $\cI_\aff$ be the set of non-empty closed intervals in $\RR$. 
Let 
\begin{align}
\cC_\aff([0,1]):=\{(\lambda_0,\lambda_1) \in \cI_\aff^2|  \lambda_0 > \lambda_1\} \label{eq:ConfInterval} 
\end{align}
where the strict inequality means $\min \lambda_0 > \max \lambda_1$.
For each $\Phi \in \cG_{\aff}([0,1])$ and $(\lambda_0,\lambda_1) \in \cI_\aff^2$, we define 
\begin{align}
 \Phi_*(\lambda_0,\lambda_1):=(\Phi_*(0_\aff), \Phi_0(\lambda_0), \Phi_1(\lambda_1)) \in \Omega^1([0,1],\kg_\aff) \times \cI_\aff^2 \label{eq:gaugeInterval}
\end{align}
where $0_\aff$ is the trivial connection, $\Phi_*(0_\aff)$ is the gauge transformation by $\Phi$ and $\Phi_i:=\Phi(i)$ for $i=0,1$.
We define
\begin{align}
 \cP_\aff([0,1]):=\{\Phi_*(\lambda_0,\lambda_1)| (\lambda_0,\lambda_1) \in \cC_\aff([0,1])\}.
\end{align}

Note that, if $(A,\lambda_0,\lambda_1) \in \cP_\aff([0,1])$, then by \eqref{eq:ConfInterval} and \eqref{eq:gaugeInterval}, we have
\begin{align}
 g_A^{-1} \lambda_{1} < \lambda_{0}. \label{eq:orderinglambda}
\end{align}

\begin{lemma}[cf. Section $(2a)$ of \cite{Seidel4.5}]
 $\cP_\aff([0,1])$ is non-empty and weakly contractible.
 The projection
\begin{align}
 \cP_\aff([0,1]) \to \cI_\aff^2 \label{eq:weakfibration1}
\end{align}
is a surjective weak fibration.
Therefore, the fibers of \eqref{eq:weakfibration1} are also weakly contractible.
\end{lemma}

\begin{proof}
 Since $\Phi_*(\lambda_0,\lambda_1)=\Phi'_*(\lambda'_0,\lambda'_1)$ if and only if $\Phi \circ (\Phi')^{-1}$ is a constant 
and $\lambda_j= \Phi \circ (\Phi')^{-1}(\lambda_j)$ for $j=0,1$,
$\cP_\aff([0,1])$ can be identified with $\cG_{\aff}([0,1]) \times_{G_\aff} \cC_\aff([0,1])$.
That implies that $\cP_\aff([0,1])$ is non-empty and weakly contractible.

On the other hand, given $\lambda_0, \lambda_1 \in \cI_\aff$, there exists $g \in G_\aff$ such that $g^{-1} \lambda_1 < \lambda_0$.
Moreover, there exists $\Phi \in \cG_{\aff}([0,1])$ such that for $A:=\Phi_*(0_\aff)$, we have $g_A=g$.
It follows that \eqref{eq:weakfibration1} is surjective.
We leave to readers to check that \eqref{eq:weakfibration1} is a weak fibration (cf. \cite[Section $(2a)$]{Seidel4.5}).
\end{proof}

We denote the fiber of \eqref{eq:weakfibration1} at $(\lambda_0, \lambda_1)$ by $\cA_\aff([0,1],\lambda_0, \lambda_1)$.

\begin{definition}
For $\ul{L}_0=\{L_{0,k}\}_{k=1}^n$ and $\ul{L}_1=\{L_{1,k}\}_{k=1}^n$ in $\cL^{cyl,n}$, a perturbation pair is a pair $(A_{\ul{L}_0,\ul{L}_1},H_{\ul{L}_0,\ul{L}_1})$
such that
\begin{align}\label{eq:HamData}
\left\{ 
\begin{array}{lll}
 A_{\ul{L}_0,\ul{L}_1}=a_{\ul{L}_0,\ul{L}_1,t}dt \in \cA_\aff([0,1],\lambda_{\ul{L}_0},\lambda_{\ul{L}_1}) \\
 H_{\ul{L}_0,\ul{L}_1}=(H_{\ul{L}_0,\ul{L}_1,t})_{t \in [0,1]}, \text{ } H_{\ul{L}_0,\ul{L}_1,t} \in \cH_{a_{\ul{L}_0,\ul{L}_1,t}}^n(E) \\
\end{array}
\right. 
\end{align}
and 
\begin{align}
 \phi_{H_{\ul{L}_0,\ul{L}_1}}^{-1}(\Sym(\ul{L}_{1})) \pitchfork \Sym(\ul{L}_{0}). \label{eq:transversalIntersect}
\end{align}
\end{definition}

\begin{lemma}
For any $\ul{L}_0, \ul{L}_1 \in \cL^{cyl,n}$, the set of  perturbation pairs is non-empty.
\end{lemma}

\begin{proof}
For any choice of $(A_{\ul{L}_0,\ul{L}_1},H_{\ul{L}_0,\ul{L}_1})$ satisfying \eqref{eq:HamData}, by \eqref{eq:gaugelambda} and \eqref{eq:orderinglambda}, we have
\begin{align}
 \lambda_{\phi_{H_{\ul{L}_0,\ul{L}_1}}^{-1}(\Sym(\ul{L}_{1}))} < \lambda_{\ul{L}_0}
\end{align}
so the transversality condition \eqref{eq:transversalIntersect} is satisfied outside a compact subset.
By definition \eqref{eq:HnE}, we are free to perturb $H_{\ul{L}_0,\ul{L}_1}$ inside a compact subset, so the result follows (cf. Remark \ref{r:FlexH}).
\end{proof}


Let $\cX(H_{\ul{L}_0,\ul{L}_1}, \ul{L}_0,\ul{L}_1)=\phi_{H_{\ul{L}_0,\ul{L}_1}}^{-1}(\Sym(\ul{L}_{1})) \pitchfork \Sym(\ul{L}_{0})$
 which is identified with the set of $X_{H_{\ul{L}_0,\ul{L}_1}}$-chords from $\Sym(\ul{L}_{0})$ to $\Sym(\ul{L}_{1})$.
By \eqref{eq:preserveDr} and the fact that $\Sym(\ul{L}) \subset \Conf^n(E) \setminus D^\circ_r$ for every $\ul{L} \in \cL^{cyl,n}$, we know that
\begin{align}
\text{the $X_{H_{\ul{L}_0,\ul{L}_1}}$-chords from $\Sym(\ul{L}_{0})$ to $\Sym(\ul{L}_{1})$ are disjoint from $D^\circ_r$.} \label{eq:disjointFromDr}
\end{align}
\begin{example}\label{ex:noHperb}
 If \eqref{eq:transversalIntersect} can be achieved by $H_{\ul{L}_0,\ul{L}_1,t}=(H'_{\ul{L}_0,\ul{L}_1,t})^{[n]} \in \cH_{a_{\ul{L}_0,\ul{L}_1,t}}^{n,pre}(E)$, then the set 
 $\cX(H_{\ul{L}_0,\ul{L}_1}, \ul{L}_0,\ul{L}_1)$ can be identified with 
 the set of {\it unordered} $n$-tuples $\ul{x}=(x_{1}, \dots, x_{n})$
such that $x_{k} \in \phi_{H'_{\ul{L}_0,\ul{L}_1}}^{-1}(L_{1,b_k}) \pitchfork L_{0,a_k}$, where $\{a_k:k=1,\dots,n\} =\{b_k:k=1,\dots,n\}=\{1,\dots,n\}$ (see Figure \ref{fig:noH}).
\end{example}

\begin{figure}[ht]
 \includegraphics{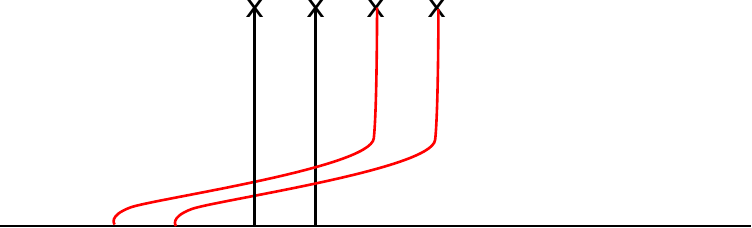}
 \caption{Lagrangian tuples $\pi_E(\ul{L}_0)$ (black) and $\pi_E(\phi_{H'_{\ul{L}_0,\ul{L}_1}}^{-1}(\ul{L}_1))$ (red) for $H'_{\ul{L}_0,\ul{L}_1} \in
 H_{\gamma}(E) $}\label{fig:noH}
\end{figure}





Given $\ul{L}_0,\ul{L}_1 \in \cL^{cyl,n}$, a perturbation pair $(A_{\ul{L}_0,\ul{L}_1}, H_{\ul{L}_0,\ul{L}_1})$ and a smooth family $J=(J_t)_{t \in [0,1]}, J_t \in \cJ^n(E)$,
we define the Floer cochains (as a vector space over a chosen coefficient field of characteristic zero\footnote{Recall that to work in characteristic zero requires that we fix spin structures on the Lagrangians; in the main examples studied later in the paper, the Lagrangians are products $(S^2)^j \times (\mathbb{R}^2)^k$ for $j,k \geq 0$ and admit unique spin structures.}) by 
\begin{align}
CF(\ul{L}_0,\ul{L}_1) :=\oplus_{\ul{x} \in \cX(H_{\ul{L}_0,\ul{L}_1}, \ul{L}_0,\ul{L}_1)}  o_{\ul{x}} 
\end{align}
where $o_{\ul{x}}$ is the orientation line at $\ul{x}$.
In this case, we call $(A_{\ul{L}_0,\ul{L}_1}, H_{\ul{L}_0,\ul{L}_1},J)$ a Floer cochain datum for $\ul{L}_0,\ul{L}_1 \in \cL^{cyl,n}$.

For a fixed choice of grading functions $\eta_{0,k}$ on $L_{0,k}$ and $\eta_{1,k}$ on $L_{1,k}$, for $k=1,\dots,n$,
Lemma \ref{l:canonicalBundle}  implies that the product $\prod_k \eta_{i,k}$
descends to a grading function on $\Sym(\ul{L}_i)$, and
we use that to grade $\cX(H_{\ul{L}_0,\ul{L}_1}, \ul{L}_0,\ul{L}_1)$.

\begin{example}
 In the situation of Example \ref{ex:noHperb},  the grading of $\ul{x}=(x_{1}, \dots, x_{n})$ is 
 \begin{align}
  |\ul{x}|=\sum_{j=1}^n |x_j|
 \end{align}
where $|x_j|$ is the Floer grading of $x_j$ in $E$.
 \end{example}

\subsubsection{Floer data}\label{sss:FloerData}

We define the following group of gauge transformations:
\begin{align}
 \cG_\aff(\ol{\cS}^{d+1,h}):=\{\Phi \in C^{\infty}(\ol{\cS}^{d+1,h},G_\aff) \ :&
 \Phi|_{\ol{N}_{\mk}^{d+1,h}}=id_{G_\aff} \text{ and }  \Phi(\epsilon_j(s,t)) \text{ is independent of }       \label{eq:GaugeS}\\
 &\text{both } s \text{ and }S \in \ol{\cR}^{d+1,h}\}.  \nonumber
\end{align}
where $\ol{N}_{\mk}^{d+1,h}$ is defined in Section \ref{sss:EndsAndNeighborhoods}.

\begin{lemma}
 $\cG_\aff(\ol{\cS}^{d+1,h})$ is weakly contractible.
\end{lemma}

\begin{proof}
Recalling that $G_{\aff}$ is contractible and
$\ol{N}_{\mk}^{d+1,h}$ is a codimension $0$ smooth submanifold of $\ol{\cS}^{d+1,h}$ with boundary and corner.
The result follows from the weak fibration
 \begin{align}
  \cG_\aff(\ol{\cS}^{d+1,h}) \to C^{\infty}(\ol{\cS}^{d+1,h},G_\aff) \to C^{\infty}(\ol{N}_{\mk}^{d+1,h},G_\aff) \times \prod_{j=0}^d C^{\infty}(\epsilon_j(0,[0,1]),G_\aff)
 \end{align}
\end{proof}


Similarly to \eqref{eq:ConfInterval}, we define 
\begin{align}
\cC_\aff(\ol{\cS}^{d+1,h}):=\{(\lambda_0,\dots,\lambda_d) \in \cI_\aff^{d+1}| \lambda_0 > \lambda_1 > \dots > \lambda_d\}.
\end{align}

Consider a trivial $G_\aff$-bundle over $\ol{\cS}^{d+1,h}$ with fiber $\RR$ and equip it with the trivial connection $0_\aff$.
For $\Phi \in \cG_\aff(\ol{\cS}^{d+1,h})$ and $\lambda \in C^{\infty}(\partial \ol{\cS}^{d+1,h}, \cI_\aff)$, we define
\begin{align}
 \Phi_*\lambda:=&(\Phi_*0_\aff, \Phi \circ \lambda) \in \Omega^1(\ol{\cS}^{d+1,h},\kg_\aff) \times C^{\infty}(\partial \ol{\cS}^{d+1,h}, \cI_\aff) \label{eq:PhiLambda} \\
 \cP_\aff(\ol{\cS}^{d+1,h}):=&\{\Phi_*\lambda | \lambda \text{ is locally constant and } (\lambda|_{\partial_0 \ol{\cS}^{d+1,h}}, \dots, \lambda|_{\partial_d \ol{\cS}^{d+1,h}}) \in \cC_\aff(\ol{\cS}^{d+1,h})\}
\end{align}
where $\partial_j \ol{\cS}^{d+1,h}:=\cup_{S \in \ol{\cR}^{d+1,h}} \partial_j S$ for all $j$.
We have the identification
\begin{align}
 \cP_\aff(\ol{\cS}^{d+1,h}) \simeq \cG_\aff(\ol{\cS}^{d+1,h}) \times_{G_\aff} \cC_\aff(\ol{\cS}^{d+1,h})
\end{align}
so $\cP_\aff(\ol{\cS}^{d+1,h})$ is weakly contractible. 
By \eqref{eq:GaugeS}, if $(A,\lambda) \in \cP_\aff(\ol{\cS}^{d+1,h})$, then $A|_{\epsilon_j(s,t)}$ and $\lambda|_{\epsilon_j(s,k)}$ (for $k=0,1$) are independent of $s$
and $S \in \ol{\cR}^{d+1,h}$.
Therefore, over strip-like ends, we have a well-defined projection 
\begin{align}
 \cP_\aff(\ol{\cS}^{d+1,h}) \to \prod_{j=0}^d \cP_\aff([0,1])
\end{align}
which is surjective and is a weak fibration (because $G_\aff$ is contractible so one can extend $\Phi_j \in \cG_\aff([0,1])$ over strip-like ends smoothly to a $\Phi \in \cG_\aff(\ol{\cS}^{d+1,h})$ consistently).
For a choice of $(A_j,\lambda_{j,0}, \lambda_{j,1}) \in \cP_\aff([0,1])$ for each $j=0,\dots,d$, we denote the fiber by
\begin{align}
 \cP_\aff(\ol{\cS}^{d+1,h}, \{(A_j,\lambda_{j,0}, \lambda_{j,1})\}_j)
\end{align}
which is also weakly contractible.
Note also that, if $(A,\lambda) \in \cP_\aff(\ol{\cS}^{d+1,h})$, then by \eqref{eq:GaugeS},
\begin{align}\label{eq:SurfaceConnection}
\text {$A$ is flat everywhere and vanishes in $\ol{N}^{d+1,h}_{\mk}$.} 
\end{align}

Now, for $S \in \cR^{d+1,h}$, we define $\cP_\aff(S, \{(A_j,\lambda_{j,0}, \lambda_{j,1})\}_j)$ by
\begin{align}
 \{(A_S,\lambda):(A_S,\lambda)=(A,\lambda')|_{S} \text{ for some }(A,\lambda') \in  \cP_\aff(\ol{\cS}^{d+1,h}, \{(A_j,\lambda_{j,0}, \lambda_{j,1})\}_j)\}
\end{align}

A {\it cylindrical Lagrangian label} is a choice of an element $\ul{L}_j \in \cL^{cyl,n}$ associated to $\partial_j S$ for all $j$.
We choose a cylindrical Lagrangian label and Floer cochain data $(A_0,H_0,J_0)$ and $(A_j,H_j,J_j)$ for $(\ul{L}_{0}, \ul{L}_{d})$ and $(\ul{L}_{j-1}, \ul{L}_{j})$
for $j=1,\dots, d$, respectively.







Fix $A_S \in \cP_\aff(S, \{(A_j, \lambda_{\ul{L}_{j-1}}, \lambda_{\ul{L}_{j}})\}_{j=0}^d)$, where for $j=0$, 
it should be understood as $(A_j, \lambda_{\ul{L}_{j-1}}, \lambda_{\ul{L}_{j}}):=(A_0,\lambda_{\ul{L}_{0}},\lambda_{\ul{L}_{d}})$.
Recall that we have chosen $J_j=(J_{j,t})_{t \in [0,1]}, J_{j,t} \in \cJ(E)$ for $j=0,\dots,d$.
We have also chosen strip-like ends $\epsilon_0(s,t):(-\infty,0] \times [0,1] \to S$ and $\epsilon_j(s,t):[0,\infty)  \times [0,1] \to S$ for $j=1,\dots,d$.
We equip $S$ with the following additional data
\begin{align}\label{eq:FloerDataJ}
 &\text{a smooth family $J=(J_z)_{z \in S}$, $J_z \in \cJ^n(E)$ such that $J_z=J_E^{[n]}$ in $\nu(\mk(S))$}\\
 &\text{and $J_{\epsilon_j(s,t)}=J_{j,t}$ for all $j$} \nonumber \\
 &\text{$K \in \Omega^1(S,C^{\infty}(\Conf^n(E),\RR))$ such that for each $w \in TS$, $K(w) \in \cH_{A_S(w)}^n(E)$, and}\label{eq:FloerDataK}\\
 &\text{for each $j$ and $r>0$, there is some $c_j>0$ for which $\|(\epsilon_j^*K-H_jdt)e^{c_j|s|}\|_{C^r}$ } \nonumber\\
 &\text{converges to $0$ as $s$ goes to $\pm \infty$. Moreover, we require }K \text{ to vanish in $\nu(\mk(S))$.}  \nonumber 
\end{align}
Let $X_K \in \Omega^1(S,C^{\infty}(\Conf^n(E),TE))$ be the associated one-form with values in Hamiltonian vector fields.

\begin{remark}\label{rmk:exp_convergence}
 We require $\epsilon_j^*K$ to converge to $H_jdt$ exponentially fast, instead of co-inciding with it, because it makes it easier to achieve regularity of the moduli whilst maintaining 
 compatibility with gluing (see \cite[Remark 4.7]{Seidel4.5}).
\end{remark}

\begin{remark}\label{r:JKextension}
 Note that $J$ can be extended smoothly to a family of tamed almost complex structures in $\Conf^n(E^{\rceil})$ and $K$ 
can be extended smoothly to an element in $\Omega^1(S,C^{\infty}(\Conf^n(E^{\rceil}),\RR))$.
\end{remark}


We choose a smooth family of $(A_S,J,K)$ for $S$ varying in $\cR^{d+1,h}$. 
Given $\ul{x}_0 \in  \cX(H_0, \ul{L}_0,\ul{L}_d)$ and $\ul{x}_j \in \cX(H_j, \ul{L}_{j-1},\ul{L}_j)$ for $j=1,\dots,d$,
we define $\cR^{d+1,h}(\ul{x}_0;\ul{x}_d, \dots, \ul{x}_1)_{pre}$ to be the moduli of all maps $u :S \to \Hilb^n(E)$ such that
\begin{align}\label{eq:FloerEquation0}
    \left\{
\begin{array}{ll}
      u^{-1}(D_{HC}) \subset \nu(\mk(S)) \\
      (Du|_z-X_{K}|_{u(z)})^{0,1}=0 \text{ with respect to } (J_z)_{u(z)} \text{ for } z \in S\\
      u(z) \in \Sym(\ul{L}_j) \text{ for } z \in \partial_j S \\
      \lim_{s \to \pm \infty}u( \epsilon^j(s,\cdot))=\ul{x}_j(\cdot) \text{ uniformly}
\end{array}
      \right.
\end{align}

Note that the conditions: $X_{K}$ vanishes in $\nu(\mk(S))$, $J=J_E^{[n]}$ in $\nu(\mk(S))$, 
and $u^{-1}(D_{HC})\subset \nu(\mk(S))$ guarantee that $(Du|_z-X_{K}|_{u(z)})^{0,1}=0$ is a well-defined equation for all $z \in S$, by identifying $\Conf^n(E)$ as a subset of $\Hilb^n(E)$.

Subsequently, we define
$\cR^{d+1,h}(\ul{x}_0;\ul{x}_d, \dots, \ul{x}_1)$ to be the subset of $\cR^{d+1,h}(\ul{x}_0;\ul{x}_d, \dots, \ul{x}_1)_{pre}$ 
consisting of all $u$ such that
\begin{align}\label{eq:FloerEquation1}
u(\xi_+^i) \in D_{HC} \text{ for all }i=1,\dots,h.
\end{align}

\begin{lemma}\label{l:regularity}
For generic $(J,K)$ such that \eqref{eq:FloerDataJ} and \eqref{eq:FloerDataK} are satisfied, every  solution $u\in \cR^{d+1,h}(\ul{x}_0;\ul{x}_d, \dots, \ul{x}_1)_{pre}$ is regular.
\end{lemma}

\begin{proof}
For the ease of exposition, we discuss the case when $S \in \cR^{d+1,h}$ is fixed.
We refer readers to \cite[Section (9k)]{SeidelBook} for the discussion when $S$ is allowed to vary in $\cR^{d+1,h}$, and to \cite[Section 9]{Seidel4} for the role of Remark \ref{rmk:exp_convergence} in achieving regularity compatibly with gluing.

Let $\cB$ be the space of smooth maps $u:S \to \Hilb^n(E)$ such that 
$u^{-1}(D_{HC})\subset \nu(\mk(S))$.
Note that $\cB$ is an open subset of $C^{\infty}(S,\Hilb^n(E))$.
Let $V^1 \supset V^2 \supset \dots$ be a sequence of neighborhoods of $D_{HC}$ such that $\cap_k V^k =D_{HC}$.
Let $\cB^k=\{u \in \cB | Im(u|_{S \setminus \nu(\mk(S))}) \cap V^k =\emptyset\}$.
Note that $\cB= \cup_k \cB^k$.
We want to run the Fredholm theory for appropriate Sobolev completions of $\cB^k$ for each $k$.

For each $k$, using Lemma \ref{l:tameSymp}, we pick a symplectic form $\omega_{k}$ on $\Hilb^n(E)$ which agrees with $\omega_{\Conf^n(E)}$ outside $V^k$
and tames the complex structure on $\Hilb^n(E)$.
This induces a family of Riemannian metrics $g_k=(g_{k,z})_{z \in S}$ on $\Hilb^n(E)$ which agree  with the metric induced  from $\omega_{\Conf^n(E)}$ and $J_z$ outside $V^k$. 
We use $g_k$ to form an appropriate Sobolev completion of $\cB^k$ (or the corresponding function space of the graphs of the maps).

The boundary conditions ensure that every solution $u$ of \eqref{eq:FloerEquation0} has non-empty 
 intersection with a compact subset of $\Conf^n(E) \setminus D^\circ_r$ outside $\mk(S)$. 
 For $u \in \cB^k$, Gromov's graph trick applies, because on the one hand, $X_K|_{S \setminus \nu(\mk(S))}$ is the Hamiltonian field with respect to $\omega_{k}$, and 
 on the other, $X_K|_{\nu(\mk(S))}=0$, so it is tautologically the Hamiltonian field with respect to $\omega_{k}$. Hence regularity of $u$ can be achieved
 by choosing $K$ generically amongst functions satisfying \eqref{eq:FloerDataK}, i.e. although we require that $K|_{\nu(\mk(S))}=0$, the freedom of $K$ outside $\nu(\mk(S))$ is sufficient to achieve regularity.
 
 The outcome is a sequence of residual sets $\cS^k$ in the space of all $K$ satisfying \eqref{eq:FloerDataK}
 such that for every $K \in \cS^k$,
 every $u \in \cB^k$ satisfying \eqref{eq:FloerEquation0} is regular.
 Therefore, we can take $\cS:=\cap_{k} \cS^k$, which is still dense
 and every $u\in \cB$ satisfying \eqref{eq:FloerEquation0} is regular for every $K \in \cS$.
 
 More precisely, it means that for every $K \in \cS$, and for every $u\in \cB$ satisfying \eqref{eq:FloerEquation0},
 there exists $N>0$ such that for all $k>N$, the Fredholm operator $D_u$ at $u$, with domain 
 a Sobolev completion of $\cB^k$ and codomain a Sobolev completion of $\Omega^{0,1}(S,u^*T\Hilb^n(E))$ with respect to the metric $g_k$, is surjective.
 This gives a manifold structure on $\cR^{d+1,h}(\ul{x}_0;\ul{x}_d, \dots, \ul{x}_1)_{pre}$.
 \end{proof}


\begin{lemma}\label{l:regularity2}
For generic $(J,K)$ such that \eqref{eq:FloerDataJ} and \eqref{eq:FloerDataK} are satisfied, every  solution $u\in \cR^{d+1,h}(\ul{x}_0;\ul{x}_d, \dots, \ul{x}_1)$ is regular.
\end{lemma}

\begin{proof}
It suffices to show that for generic $(J,K)$,
the evaluation map
\begin{align}
\cR^{d+1,h}(\ul{x}_0;\ul{x}_d, \dots, \ul{x}_1)_{pre} \ni u \mapsto (u(\xi_+^1),\dots,u(\xi_+^h)) \in (\Hilb^n(E))^h 
\end{align} 
is transversal to $(D_{HC})^h$ (i.e. the algebraic intersection, which is well-defined as $u$ is holomorphic near $\mk(S)$, is $1$ at all the intersection points; note this implies that each such intersection point belongs to the smooth locus of $D_{HC}$).
This can be achieved by combining the argument in Lemma \ref{l:regularity} 
with standard transversality results for evaluation maps (see \cite[Sections 6 \& 7]{McDuffSalamonJ} and note that $D_{HC}$ is the image of a smooth pseudocycle).
\end{proof}

There is a correspondence of maps as follows (see \cite[Lemma 3.6]{OS04} or \cite[Section 13]{Lipshitz-cylindrical}).
An $n$-fold branched covering $\pi_\Sigma:\Sigma \to S$ and a continuous map $v: \Sigma \to E$ together 
 uniquely determine a continuous map $u:S \to \Sym^n(E)$, given by $u(z)=v(\pi_\Sigma^{-1}(z))$, counted with multiplicity.
Conversely, if $u:S \to \Sym^n(E)$ is complex analytic near the big diagonal $\Delta_E$ and $Im(u)$ is not contained in $\Delta_E$, then 
we can form the fiber product
\[
\begin{tikzcd}
\widetilde{\Sigma} \ar[rr,"\tilde{v}"] \ar[d,"\pi_{\widetilde{\Sigma}}"] & & E^n  \ar[d,"q_{S_n}"]\\
S \ar[rr,"u"] & &  \Sym^n(E).
\end{tikzcd}
\]
and the map $\tilde{v}$ is $S_n$-equivariant. 
Let $\pi_1:E^n \to E$ be the projection to the first factor.
Consider the subgroup $S_{n-1}$ of $S_n$ which fixes the first element.
The map $\pi_1 \circ \tilde{v}:\widetilde{\Sigma} \to E$
factors through $\Sigma:= \widetilde{\Sigma}/S_{n-1} \to E$
and we denote the latter map by $v$.
The map $\pi_{\widetilde{\Sigma}}$ also induces a map $\pi_{\Sigma}:\Sigma \to S$.
One can check that $\pi_{\Sigma}$ is an $n$-fold branched covering
such that $u(z)=v(\pi_\Sigma^{-1}(z))$.
Moreover, 
\begin{align}\label{eq:BranchedInDiag}
\text{$z \in S$ is a critical value of $\pi_{\Sigma}$ only if $u(z) \in \Delta_E$.} 
\end{align}
We call this the  `tautological correspondence'.

\begin{remark}\label{r:BranchedInDiag}
 If the algebraic intersection number between $u$ and $\Delta_E$ at $z$ is $1$, then $u(z)$ lies in
the top stratum of $\Delta_E$ and there is exactly $1$ critical point $p$ of $\pi_E$ such that $z$
is its critical value. Moreover, $p$ is Morse.
\end{remark}


\begin{lemma}[Tautological correspondence]\label{l:tautologicalCorr}
Every solution $u$ of \eqref{eq:FloerEquation0}
determines uniquely an $n$-fold branched covering $(\Sigma, \pi_{\Sigma})$ of $S$
and a map $v:\Sigma \to E$ such that $\pi_{HC} \circ u(z)=v(\pi_{\Sigma}^{-1}(z)) \in \Sym^n(E)$ for all $z \in S$. 
\end{lemma}

\begin{proof}
Every solution $u$ of \eqref{eq:FloerEquation0} 
is complex analytic near $D_{HC}$ and $Im(u)$ is not contained in $D_{HC}$.
Therefore, by the tautological correspondence, we get an $n$-fold branched covering $\pi_\Sigma:\Sigma \to S$ and a continuous map $v: \Sigma \to E$
such that $\pi_{HC} \circ  u(z)=v(\pi_\Sigma^{-1}(z))$.
\end{proof}

\begin{remark}\label{r:SplitData}
 Suppose that we are in the situation of Example \ref{ex:noHperb} for all the pairs $(\ul{L}_0,\ul{L}_d)$ and $(\ul{L}_{j-1},\ul{L}_j)$.
 Suppose also that $J_z=(J_z')^{[n]} \in \cJ^{n}(E)$ and $K(w)=(K'(w))^{[n]} \in \cH_{A_S(v)}^{n,pre}(E)$ in \eqref{eq:FloerDataJ} and \eqref{eq:FloerDataK}, respectively.
 Then \eqref{eq:FloerEquation0} and \eqref{eq:FloerEquation1} 
 imply that the maps $v:\Sigma \to E$ and $\pi_{\Sigma}:\Sigma \to S$ satisfy
 \begin{align}
  (Dv|_z- X_{\pi_{\Sigma}^*K'}|_{v(z)})^{0,1}=0 \text{ with respect to } (J'_{\pi_\Sigma(z)})_{v(z)} \text{ for all } z \in \Sigma
 \end{align}
 and the critical values of $\pi_{\Sigma}$ are contained in $\mk(S)$.
 \end{remark}

\subsubsection{Homotopy classes of maps}


Let $B(\ol{S})$ be the real blow-up of $\ol{S}$  at the boundary punctures.
In other words, we replace the punctures of $S$ by closed intervals, which can be identified with $\{\epsilon^j(\pm \infty,t)|t \in [0,1]\}$.
For $u \in \cR^{d+1,h}(\ul{x}_0;\ul{x}_d, \dots, \ul{x}_1)$, 
we define $G(u):=(\Sym^n(\pi_E) \circ \pi_{HC} \circ u,id_S):S \to \Sym^n(\bH^\circ) \times S$, which 
 can be extended continuously to
\begin{align}
 \overline{G}(u):B(\ol{S}) \to \Sym^n(\bH^\circ) \times B(\ol{S})
\end{align}
by sending $\epsilon^j(\pm \infty,t)$ to $(\Sym^n(\pi_E) \circ \pi_{HC} \circ \ul{x}_j(t),\epsilon^j(\pm \infty,t))$ for all $j$.
Note that, $\overline{G}(u)(\partial B(\ol{S}))$ lies inside
\begin{align}
 \partial_{\ul{x}_0;\ul{x}_d,\dots,\ul{x}_1} &:= (\cup_j \Sym(U_{\ul{L}_j}) \times \partial_j S) \cup (\cup_j \{(\Sym^n(\pi_E) \circ \pi_{HC} \circ \ul{x}_j(t),\epsilon^j(\pm \infty,t))| t \in [0,1]\}) \\
 &\subset \Sym^n(\bH^\circ) \times  B(\ol{S})
\end{align}
In particular, $\overline{G}(u)$ descends to a class in the space $\Map(\ul{x}_0;\ul{x}_d,\dots,\ul{x}_1)$ of homotopy class of continuous maps 
from $(B(\ol{S}),\partial B(\ol{S}))$ to $(\Sym^n(\bH^\circ) \times  B(\ol{S}),  \partial_{\ul{x}_0;\ul{x}_d,\dots,\ul{x}_1})$.
In other words,
\begin{align}
 \overline{G}(u) \in \Map(\ul{x}_0;\ul{x}_d,\dots,\ul{x}_1):=[(B(\ol{S}),\partial B(\ol{S})), (\Sym^n(\bH^\circ) \times  B(\ol{S}), \partial_{\ul{x}_0;\ul{x}_d,\dots,\ul{x}_1} ))]
\end{align}

By \eqref{eq:DisjointOpen}, \eqref{eq:disjointFromDr} and Lemma \ref{l:DhcDr}, $ (\Delta_{\bH^{\circ}} \times  B(\ol{S})) \cap \partial_{\ul{x}_0;\ul{x}_d,\dots,\ul{x}_1} =\emptyset$ 
so we have an intersection pairing (with respect to the obvious orientations)
\begin{align}
 \cdot  [(\Delta_{\bH^{\circ}} \times  B(\ol{S}))]: \Map(\ul{x}_0;\ul{x}_d,\dots,\ul{x}_1) \to \ZZ.
\end{align}

\begin{lemma}\label{l:IPairing}
 Given $\ul{x}_0,\dots,\ul{x}_d$, there is $I_{\ul{x}_0;\ul{x}_d,\dots,\ul{x}_1} \in \ZZ$ such that for every $u \in \cR^{d+1,h}(\ul{x}_0;\ul{x}_d, \dots, \ul{x}_1)$, we have
 \begin{equation}
[\overline{G}(u)] \cdot  [(\Delta_{\bH^{\circ}} \times  B(\ol{S}))]=I_{\ul{x}_0;\ul{x}_d,\dots,\ul{x}_1}  
 \end{equation}
 Moreover, $I_{\ul{x}_0;\ul{x}_d,\dots,\ul{x}_1}$ is independent of $h$.
\end{lemma}

\begin{proof}
Since $\Sym(U_{\ul{L}_j})$ is a contractible open set, we have $\Map(\ul{x}_0;\ul{x}_d,\dots,\ul{x}_1)=\pi_2(\CC^{n+1}, S^1)=\pi_1( S^1)= \ZZ$.
In this case, the intersection pairing is a multiple of the winding number along the boundary and the winding number of
$[\overline{G}(u)]$ is $1$, by definition.
\end{proof}


\begin{lemma}[Positivity of intersection]\label{l:PositivityIntersection}
 If $u \in \cR^{d+1,h}(\ul{x}_0;\ul{x}_d, \dots, \ul{x}_1)$, then
  \begin{equation}
[\ol{G}(u)]\cdot  [(\Delta_{\bH^{\circ}} \times  B(\ol{S}))] \ge h 
 \end{equation}
and equality holds if and only if the image of $u$ is disjoint from $D_r$, $u^{-1}(D_{HC})=\mk(S)$
and the multiplicity of intersection between $u$ and $D_{HC}$ is $1$ for all $z \in \mk(S)$. 
\end{lemma}

\begin{proof}
 Let $\ol{u}:=\Sym^n(\pi_E) \circ \pi_{HC} \circ u:S \to \Sym^n(\bH^\circ)$.
 There are two kinds of intersections between $\ol{u}$ and $\Delta_{\bH^{\circ}}$, namely, at or away from $\mk(S)$.
 If $z \in \mk(S)$, then \eqref{eq:FloerDataJ}, \eqref{eq:FloerDataK}, \eqref{eq:FloerEquation0} and \eqref{eq:FloerEquation1} imply that
 $\ol{u}$ is $(j_{\Sym^n(\bH^\circ)},j_S)$-holomorphic near $z$.
 It implies that the contribution of the algebraic intersection at $z$ is at least $1$ and $z$ is the only intersection with $\Delta_{\bH^{\circ}}$ in a small neighborhood of $z$.
 Summing over all $z \in \mk(S)$, the contribution to the algebraic intersection is at least $h$.
 
 We need to show that the contribution from other intersections with $\Delta_{\bH^{\circ}}$ is positive.
 Let $z_0 \in S \setminus \mk(S)$ be such that $\ol{u}(z_0) \in \Delta_{\bH^{\circ}}$.
 Let $B(z_0) \subset S$ be a small disc centred at $z_0$.
 By Lemma \ref{l:DhcDr}, we must have $u(z_0) \in D_{HC}\cup D_r$.
To show that the contribution at $\ol{u}(z_0) $ is positive, it suffices to show that
the algebraic intersection number between $u(B(z_0))$ and $D_{HC}\cup D_r$ at $u(z_0)$ is positive.
If $u(z_0) \in D_{HC}$, this follows from $u^{-1}(D_{HC}) \subset \nu(\mk(S))$ and $u|_{\nu(\mk(S))}$ is complex analytic. 
If $u(z_0) \in D_{r}$, then it can be achieved by Gromov's graph trick, and the fact that our choice of $X_K$ is tangential to $D_r$ at 
$u(z_0)$.

We give a detailed explanation of the last sentence. 
 By the definitions \eqref{eq:JnE} and \eqref{eq:HnE}, near $D_r$, we have $J_{z_0}=(J')^{[n]}$ and $K(w)=(\pi^*_{E}H^w)^{[n]}$ for some $H^w \in \cH_{A_S(w)}(\bH^{\circ})$
 for $w \in TB(z_0)$.
 It means that, by \eqref{eq:fiberJ} and by shrinking $B(z_0)$ if necessary, we have 
 \begin{equation}\label{eq:tangentialPerturbation}
  (X_K(\eta))_\tau \in T_{\tau}D_r
 \end{equation}
 for all $\eta \in TB(z_0)$ and $\tau \in D_r$, where 
 $T_{\tau}D_r$ is understood to be the tangent space of the smallest strata of $D_r$ that contains $\tau$.

 We consider the graph trick. Let $F=(u,id):B(z_0) \to \Conf^n(E) \times B(z_0)$.
 Let $J_{F}$ be the almost complex structure of $\Conf^n(E)  \times B(z_0)$ characterised by
 \begin{align}\label{eq:GraphJF}
 \left\{
 \begin{array}{ll}
  J_{F}|_{\Conf^n(E) \times \{z\}}=J_z\\
  d\pi_{B(z_0)} \circ J_F=j_S \circ d\pi_{B(z_0)}  \\
  J_F(\partial_s+X_K(\partial_s))=\partial_t+X_K(\partial_t)  
 \end{array}
\right.
 \end{align}
 where $\pi_{B(z_0)} : \Conf^n(E) \times B(z_0) \to B(z_0)$ is the projection.
 In holomorphic coordinates $(s,t)$, we have
 \begin{align}
  &(DF+J_F \circ DF \circ j_S)(\partial_s) \\
  =&Du(\partial_s)+\partial_s+J_F(Du(\partial_t)-X_K(\partial_t))+J_F(\partial_t+X_K(\partial_t)) \\
  =&Du(\partial_s)+\partial_s-(Du(\partial_s)-X_K(\partial_s))-(\partial_s+X_K(\partial_s))=0
 \end{align}
where the last line uses $(Du-X_K)^{0,1}=0$.
Similarly, we also have $(DF+J_F \circ DF \circ j_S)(\partial_t)=0$ and hence $(DF)^{0,1}=0$ and $F$ is $(J_F,j_S)$-holomorphic.

Notice that all strata of $D_r^{\circ} \times B(z_0)$ are $J_F$-holomorphic due to \eqref{eq:tangentialPerturbation}, \eqref{eq:GraphJF} and
the fact that $\Delta_{\bH^{\circ}}$ is a $J_z$-holomorphic subvariety.
As a result, the intersection $F(z_0)$ between $F(B(z_0))$ and $D_r^\circ \times B(z_0)$ is positive.
This completes the proof.
\end{proof}

\begin{corollary}\label{c:missingDr}
 If $u \in \cR^{d+1,h}(\ul{x}_0;\ul{x}_d, \dots, \ul{x}_1)$ and $I_{\ul{x}_0;\ul{x}_d,\dots,\ul{x}_1}=h$, then $Im(u) \cap D_r= \emptyset$, $u^{-1}(D_{HC})=\mk(S)$ and the  multiplicity of intersection between $u$ and $D_{HC}$ is $1$ for all $z \in \mk(S)$.
\end{corollary}

\begin{corollary}\label{c:SideBubbling}
 If $I_{\ul{x}}=0$, then $\cR^{0+1,h}(\ul{x})=\emptyset$ for all $h \ge 1$.
\end{corollary}

\begin{lemma}\label{l:NoSideBubbling}
 For every $\ul{L} \in \cL^{cyl,n}$, there is a perturbation pair $(A_{\ul{L}}, H_{\ul{L}})$ for $(\ul{L},\ul{L})$ such that 
 $I_{\ul{x}}=0$ for all $\ul{x} \in \cX(H_{\ul{L}},\ul{L},\ul{L})$.
\end{lemma}

\begin{proof}

Let $A=a_tdt \in \cA_{\aff}([0,1], \lambda_{\ul{L}},\lambda_{\ul{L}} )$.
Since $g_A^{-1}\lambda_{\ul{L}}< \lambda_{\ul{L}}$, we can find
 $H'=(H'_t)_{t \in [0,1]}$, $H'_t \in \cH_{a_t}(E)$ and an ordering of the Lagrangians in $\ul{L}$ such that
$L_i \cap \phi_{H'}(L_j) \neq \emptyset$ only if $i\geq j$. 
In this case, the (possibly non-transversal) $X_{(H')^{[n]}}$-chords $\ul{x}(t)$ from $\Sym(\ul{L})$ to $\Sym(\ul{L})$ 
are given by unordered tuples of $X_{H'}$-chords from $L_i$ to $L_i$ for $i=1,\dots,n$ (see Example \ref{ex:noHperb}).
Moreover, by choosing $H'$ appropriately, we can assume that $\Sym^n(\pi_E)(\pi_{HC}(\ul{x}(t))) \in \Sym(U_{\ul{L}})$ for all $X_{(H')^{[n]}}$-chords $\ul{x}(t)$.

We can pick $H$ to be a compactly supported perturbation of $(H')^{[n]}$ such that every $X_{H}$-chord  $\ul{x}(t)$ from $\Sym(\ul{L})$ to $\Sym(\ul{L})$ 
is transversal and $\Sym^n(\pi_E)(\pi_{HC}(\ul{x}(t))) \in \Sym(U_{\ul{L}})$ for all $t$.
By taking $(A_{\ul{L}}, H_{\ul{L}})$ to be $(A,H)$, the result follows, because the boundary conditions also project to $\Sym(U_{\ul{L}})$, which is contractible.


\end{proof}

\subsubsection{Energy}\label{sss:Energy}

The product symplectic form on $\Conf^n(E)$ cannot be smoothly extended to $\Hilb^n(E)$, and nor can the induced metric.
Therefore, we will define energy and discuss compactness with the help of the corresponding maps $v: \Sigma \to E$ obtained from Lemma \ref{l:tautologicalCorr}.

By \eqref{eq:FloerDataJ} and \eqref{eq:FloerDataK},
there exists a compact subset $C$ of $\Conf^n(E)\setminus D_r^\circ$ and
\begin{align}
 &\text{a smooth family } J'=(J'_z)_{z \in S}, J'_z \in \cJ(E) \text{, and } \\
 &K' \in \Omega^1(S,C^{\infty}(E)), K'(w) \in \cH_{A_S(w)}(E) \text{ for all $w \in TS$}
\end{align}
such that outside $C$, we have (cf. Remark \ref{r:SplitData})
\begin{align}
 &\text{$(J_z')^{[n]}=J_z$ for all $z \in S$, and }  \label{eq:productJ}\\
 &\text{$K'(w)^{[n]}=K(w)$ for all $w \in TS$}  \label{eq:productK}
\end{align}

Let $u$ be a solution of \eqref{eq:FloerEquation0} and \eqref{eq:FloerEquation1} and $v$ be the map obtained by Lemma \ref{l:tautologicalCorr}.
Let $U \subset \Conf^n(E)$ be a relatively compact open neighborhood of $C$, and define
\begin{align}
S^{in}:=u^{-1}(\overline{U}) \qquad S^{out}:= u^{-1}(\Hilb^n(E) \setminus U)
\end{align}
where $\overline{U}$ is the closure of $U$.
By adjusting $U$, we assume that $u(S)$ is transversal to $\partial \overline{U}$ so that $u^{-1}(\partial \overline{U}) \subset S$ is a
smooth manifold with boundary.
Note that, by definition, $\mk(S) \subset S^{out}$ and $S^{in}$ contains small neighborhoods of the punctures.
We also define $\Sigma^{in}=\pi_{\Sigma}^{-1}(S^{in})$ and $\Sigma^{out}=\pi_{\Sigma}^{-1}(S^{out})$. 

\begin{definition}[Energy]\label{d:Energy}
 The energy of $u$ is defined to be
 \begin{align}
  E(u):=\frac{1}{2}\int_{S^{in}} \|Du-X_K\|_g^2 dvol +\frac{1}{2} \int_{\Sigma^{out}} \|Dv- X_{\pi_{\Sigma}^*K'}\|_{g'}^2dvol \label{eq:DefnEnergy}
 \end{align}
where $g$ is the metric induced by $J$ and $\omega_{\Conf^n(E)}$ and 
$g'$ is the metric induced by $J'$ and $\omega_E$. 
\end{definition}

The energy is defined this way, instead of as $\frac{1}{2}\int_{S} \|Du-X_K\|_g^2 dvol$, because $u^*X_K$ is not defined at $u^{-1}(D_{HC}) \subset \mk(S)$, so it is not \emph{a priori} clear that the latter expression is related to the action of the asymptotes of $u$. However, we show the following (see also \eqref{eq:L2_u}).

\begin{lemma}\label{l:EnergyIndep}
 The energy $E(u)$ is independent of the choice of $U$.
\end{lemma}

\begin{proof}
 It suffices to show that for every open subset $G \subset S \setminus \mk(S)$ such that $u(G) \cap C =\emptyset $, we have
 \begin{align}
  \int_{G} \|Du-X_K\|_g^2 dvol=\int_{\pi_\Sigma^{-1}(G)} \|Dv- X_{\pi_{\Sigma}^*K'}\|_{g'}^2dvol \label{eq:u=v}
 \end{align}
 This is in turn clear because by \eqref{eq:productJ} and \eqref{eq:productK}, both $J$ and $X_K$ split as products.
 Since $\omega_{\Conf^n(E)}$ is also a product, the metric $g$ is the product metric.
 
 More precisely, let $G \subset S \setminus \mk(S)$ be a small open set such that $u(G) \cap C =\emptyset $ 
 and $\pi_{\Sigma}^{-1}(G)$ is a disjoint union of open sets $G_1,\dots,G_n \subset \Sigma$.
 For $z \in G$ and $z_j:= G_j \cap \pi_{\Sigma}^{-1}(z)$ for $j=1,\dots,n$, we have canonical identifications
\begin{align}
 T_zG \simeq T_{z_j}G_j \qquad T_{u(z)}\Conf^n(E) \simeq \oplus_{j=1}^n T_{v(z_j)} E
\end{align}
By \eqref{eq:productJ} and \eqref{eq:productK}, both $J_z$ and $X_K$ (and $\omega_{\Conf^n(E)}$) respect this product decomposition 
and every summand is given by $J'_z$ and $X_{{\pi_{\Sigma}}^*K'}$, respectively.
Therefore, \eqref{eq:u=v} is true for $G$.
Now, the result follows by summing over these small open subsets
$G \subset S \setminus \mk(S)$. 
 \end{proof}
By taking a sequence of larger and larger $U$ and applying Lemma \ref{l:EnergyIndep}, we have 
\begin{align}
 E(u)=\frac{1}{2}\int_{S\setminus \mk(S)} \|Du-X_K\|_g^2 dvol. \label{eq:L2_u}
\end{align}

Our next task is to derive a uniform upper bound for $E(u)$ that depends only on $(A_S,J,K)$ and the Lagrangian boundary condition.

Consider again the graph construction.
Let $\hat{v}:=(v, id): \Sigma \to E \times \Sigma$ and define on $E \times \Sigma$ the following $2$-forms: 
\begin{align}
 \omega_{\pi_{\Sigma}^*K'}^{geom}:=&\omega_{E}+\omega_{E}(X_{\pi_{\Sigma}^*K'}(\partial_s),\cdot)\wedge ds+\omega_{E}(X_{\pi_{\Sigma}^*K'}(\partial_t),\cdot)\wedge dt \\
 &-\omega_{E}(X_{\pi_{\Sigma}^*K'}(\partial_s),X_{\pi_{\Sigma}^*K'}(\partial_t)) ds\wedge dt \nonumber \\
 \omega_{\pi_{\Sigma}^*K'}^{top}:=&\omega_{E}-d(\pi_{\Sigma}^*K'(\partial_s) ds)-d(\pi_{\Sigma}^*K'(\partial_t) dt)=\omega_{\pi_{\Sigma}^*K'}^{geom}+R_{\pi_{\Sigma}^*K'} \label{eq:omegaR}
\end{align}
where $R_{\pi_{\Sigma}^*K'}$ is the curvature defined by
\begin{align}
 R_{\pi_{\Sigma}^*K'}:=(\partial_t\pi_{\Sigma}^*K'(\partial_s)-\partial_s\pi_{\Sigma}^*K'(\partial_t)+\{\pi_{\Sigma}^*K'(\partial_s),\pi_{\Sigma}^*K'(\partial_t)\})ds \wedge dt 
\in \Omega^2(\Sigma, C^{\infty}(E)) \label{eq:curvature}
\end{align}
It is clear that (cf. Remark \ref{r:SplitData})
\begin{align}
 (Dv|_{\Sigma^{out}}- X_{\pi_{\Sigma}^*K'}|_{\Sigma^{out}})^{0,1}=0 \label{eq:FloerEquationV}
\end{align}
so, by tameness of $J'$, we have
\begin{align}
 \frac{1}{2} \int_{\Sigma^{out}} \|Dv- X_{\pi_{\Sigma}^*K'}\|_{g'}^2dvol=\int_{\Sigma^{out}} \hat{v}^*\omega_{\pi_{\Sigma}^*K'}^{geom} \label{eq:geomE1}
\end{align}

\begin{lemma}\label{l:boundGeoEnergy1}
 There is a constant $T>0$ such that for any choice of solution $u$ of \eqref{eq:FloerEquation0} and \eqref{eq:FloerEquation1} (and hence the corresponding $v$) and 
 $U$, we have $|\int_{\Sigma^{out}} \hat{v}^*R_{\pi_{\Sigma}^*K'}| <T$.
\end{lemma}

\begin{proof}
 By \eqref{eq:HE}, there exists $K'' \in \Omega^1(S, C^{\infty}(\bH^\circ))$ such that $K'=\pi_E^*K''$ and $K''(w) \in \cH_{A_S(w)}(\bH^\circ)$ for all $w \in TS$.
 Recall also that $K''=0$ near $\mk(S)$.
 It implies that 
 \begin{align}
  R_{\pi_{\Sigma}^*K''}= \pi_\Sigma^*R_{K''} \quad \text{and } \quad R_{\pi_{\Sigma}^*K'}= \pi_E^*R_{\pi_{\Sigma}^*K''} \label{eq:compareCurvatures}
 \end{align}
 where $R_{-}$ is the corresponding curvature defined by $-$ using \eqref{eq:curvature}.
 
 Note that by \eqref{eq:HamData} and \eqref{eq:FloerDataK}, $K''$ converges exponentially fast in any $C^r$ topology with respect to $s$ over strip-like ends of $S$,
 so there is a constant $T'>0$ such that for every $j$ and any section $f:S \to \bH^\circ \times S$, we have $|\int_{Im(\epsilon_j)} f^*R_{K''}| <T'$.
 Moreover, \eqref{eq:HH}, \eqref{eq:HamData} and the flatness of $A_S$ implies that 
 $R_{K''}$ takes values in functions of $\bH^\circ$ that are supported in a compact subset of $\bH^\circ$.
 Therefore, by \eqref{eq:compareCurvatures}, $R_{\pi_{\Sigma}^*K'}$ 
 takes values in functions of $E$ that are uniformly bounded (the bound only depends on $K''$ but not $u$),
 and there is a constant $T''>0$ such that for every $j$ and any section $f:\Sigma \to E \times \Sigma$, we have $|\int_{\pi_{\Sigma^{-1}}(Im(\epsilon_j))} f^*R_{\pi_{\Sigma}^*K'}| <T''$.
 
 Let $\Sigma^{e} \subset \Sigma$ be the closure of the complement of the strip-like ends.  
 The discussion in the previous paragraph implies that $|\int_{\Sigma^{out}} \hat{v}^*R_{\pi_{\Sigma}^*K'}|$ is bounded above by $(d+1)T''$ plus 
 the integral of a bounded function over $\Sigma^{e}$,
 where the bound on that function depends only on $K''$. The result follows.
\end{proof}

Similarly, 
Let $\hat{u}:=(u, id): S \to \Conf^n(E) \times S$ and define on $\Conf^n(E) \times S$ the following $2$-forms
\begin{align}
 \omega_K^{geom}:=&\omega_{\Conf^n(E)}+\omega_{\Conf^n(E)}(X_K(\partial_s),\cdot)\wedge ds+\omega_{\Conf^n(E)}(X_K(\partial_t),\cdot)\wedge dt\\
 &-\omega_{\Conf^n(E)}(X_K(\partial_s),X_K(\partial_t)) ds\wedge dt ^{in}\nonumber \\
 \omega_K^{top}:=&\omega_{\Conf^n(E)}-d(K(\partial_s) ds)-d(K(\partial_t) dt)=\omega_K^{geom}+R_K \label{eq:omegaR2}
\end{align}
where $R_K \in \Omega^2(S, C^{\infty}(\Conf^n(E)))$ is the curvature of $K$.
We have
\begin{align}
 \frac{1}{2}\int_{S^{in}} \|Du-X_K\|_g^2 dvol= \int_{S^{in}} \hat{u}^*\omega_K^{geom} \label{eq:geomE2}
\end{align}

We have the parallel lemma.

\begin{lemma}\label{l:boundGeoEnergy2}
 There is a constant $T>0$ such that for any choice of solution $u$ of \eqref{eq:FloerEquation0} and \eqref{eq:FloerEquation1}, and any choice of 
 $U$, we have $|\int_{S^{in}} \hat{u}^*R_K| <T$.
\end{lemma}

\begin{proof}
 As before, there exists $T'>0$ such that for every section $S \to \Conf^n(E) \times S$ and every $j$, we have $|\int_{Im(\epsilon_j)} f^* R_K|<T'$.
 Ouside $U$, $R_K$ takes values in bounded functions on $\Hilb^n(E) \setminus U$ with bound determined by $K'$.
 Since $U$ is relatively compact, there is also a bound for the function-values of $R_K$ inside $U$ that is independent of $u$.
 Over all, if we let $S^e$ be the closure of the complement of strip-like ends in $S$, then
 $|\int_{S^{in}} \hat{u}^*R_K|$ is bounded above by $(d+1)T'$ plus an integration of a bounded function over $S^e$, so the result follows.
\end{proof}

The primitive one form $\theta_E$ for $\omega_E$ induces a primitive one form $\theta_{\Conf^n(E)}$
for $\omega_{\Conf^n(E)}$.
It is clear that 
\begin{align}
 \theta_{\pi_{\Sigma}^*K'}^{top}:=\theta_E- \pi_{\Sigma}^*K' \quad \text{ and }\quad \theta_{K}^{top}:=\theta_{\Conf^n(E)}- K
\end{align}
are primitives of $\omega_{\pi_{\Sigma}^*K'}^{top}$ and $\omega_{K}^{top}$, respectively.

\begin{lemma}\label{l:EnergyBound}
 We have
 \begin{align}
  \int_{\Sigma^{out}} \hat{v}^*\omega_{\pi_{\Sigma}^*K'}^{top}+\int_{S^{in}} \hat{u}^*\omega_K^{top}= \int_{\partial S} \hat{u}^*\theta_{K}^{top} \label{eq:TopEnergy}
 \end{align}
so there is a constant $T>0$ such that for all $u$ satisfying \eqref{eq:FloerEquation0} and \eqref{eq:FloerEquation1},
we have $E(u)<T$.
\end{lemma}

\begin{proof}
 Apply Stokes theorem to the terms on the left hand side of \eqref{eq:TopEnergy}.
 Note that,
 \begin{align}
  \int_{\pi_{\Sigma}^{-1}(u^{-1}(\partial \overline{U}))} \hat{v}^*\theta_{\pi_{\Sigma}^*K'}^{top} + \int_{u^{-1}(\partial \overline{U})} \hat{u}^*\theta_K^{top}=0
 \end{align}
because the Floer data splits into a product outside $U$ and the orientation of the curves in the two summands are opposite to one another.
Similarly,
 \begin{align}
  \int_{\pi_{\Sigma}^{-1}(\partial S^{out} \setminus u^{-1}(\partial \overline{U}))} \hat{v}^*\theta_{\pi_{\Sigma}^*K'}^{top} =\int_{\partial S^{out} \setminus u^{-1}(\partial \overline{U})} \hat{u}^*\theta_K^{top}.
 \end{align}
Therefore, we get \eqref{eq:TopEnergy}.

Moreover, there is a constant $T'>0$ (independent of $u$) such that
$\int_{\partial S} \hat{u}^*\theta_K^{top} <T'$ (see \cite[Lemma $4.8$]{Seidel4.5}) so the result follows
from Lemma \ref{l:boundGeoEnergy1} and \ref{l:boundGeoEnergy2}, and the equalities \eqref{eq:omegaR}, \eqref{eq:geomE1}, \eqref{eq:omegaR2}, \eqref{eq:geomE2} and \eqref{eq:TopEnergy}.
\end{proof}

Finally, we address the $L^2$-norm of $v$ when the domain is not restricted to $\Sigma^{out}$.

\begin{lemma}\label{l:comparableMetric}
 For every relatively compact open subset $C_{\Sigma} \subset \Sigma$,
 there is a constant $T_{C_{\Sigma}}>0$ such that for all $v$ arising from applying Lemma \ref{l:tautologicalCorr} to $u$ satisfying \eqref{eq:FloerEquation0} and \eqref{eq:FloerEquation1}, we have
 \begin{align}
  \frac{1}{2} \int_{C_{\Sigma}} \|Dv- X_{\pi_{\Sigma}^*K'}\|_{g'}^2dvol <T_{C_{\Sigma}}.
 \end{align}

\end{lemma}

\begin{proof}
Let $G \subset S$ be a relatively  compact open subset such that $C_{\Sigma} \subset \pi_{\Sigma}^{-1}(G)$.
Since $J_z=(J_z')^{[n]}$ and $X_K=X_{(K')^{[n]}}$ outside a compact subset of $\Conf^n(E) \setminus D^\circ_r$, 
there is a constant $T_G>0$ (independent of $u$) such that
\begin{align}
 \int_{\pi_{\Sigma}^{-1}(G)} \|Dv- X_{\pi_{\Sigma}^*K'}\|_{g'}^2dvol < T_G \int_G  \|Du-X_K\|_g dvol.
\end{align}
The right hand sided is in turn bounded above by a constant independent of $u$, by Lemma \ref{l:EnergyBound}.
\end{proof}

\begin{remark}
 The term $\frac{1}{2} \int_{\Sigma} \|Dv- X_{\pi_{\Sigma}^*K'}\|_{g'}^2dvol$ may be infinite, because 
  $X_K \neq X_{(K')^{[n]}}$ everywhere, so the symmetric product of the asymptotes of $v$
corresponding to a fixed puncture of $S$ is not necessarily a $X_K$-Hamiltonian chord.
This implies that the integral over the strip-like ends might diverge.
\end{remark}

\subsubsection{Compactness}\label{sss:Compactness}

Let $\ul{x}_0,\dots,\ul{x}_d$ be as before.
We next discuss  compactness of the solution spaces $\cR^{d+1,h}(\ul{x}_0;\ul{x}_d, \dots, \ul{x}_1)$.

\begin{lemma}\label{l:nobubble}
 Fix $S \in \cR^{d+1,h}$.
 Let $u_k:S_k \to \Hilb^n(E)$ be a sequence in $\cR^{d+1,h}(\ul{x}_0;\ul{x}_d, \dots, \ul{x}_1)$ such that $S_k$ converges to $S \in \cR^{d+1,h}$.
 If there exists $z_k \in S_k \setminus \nu(\mk(S_k))$ such that $\|Du_k(z_k)-X_K(z_k)\|_g$ diverges to infinity, then for any 
 compact subset $C_S$ in the universal family over $\cR^{d+1,h}$ and $N>0$, there exists $k>N$ such that $z_k \notin C_S$. 
\end{lemma}

\begin{proof}
 For simplicity, we assume $S_k=S$ for all $k$; the reasoning below adapts directly to the general case.
 Suppose the lemma were false, then there is a subsequence of $z_k$ which converges to $z_\infty \in S$.
 
 First assume that $z_\infty$ is an interior point of $S$.
 Let $B \subset S$ be a small ball centered at $z_\infty$ and 
 conformally identify $B$ with a $3\epsilon$-ball in $\CC$ centered at the origin.
 Under this identification, we can assume that for all $k$, $z_k$ lies in the $\epsilon$-ball centered at $z_\infty$.
 By assumption and the uniform bound of $X_K$ near $z_{\infty}$, $\|Du_k(z_k)\|_g$ diverges to infinity.
 
 We apply a rescaling trick to the $\epsilon$-ball $B_k$ centered at $z_k$ in the following sense (see \cite[Section 4.2]{McDuffSalamonJ}).
 Let $r_k=\sup_{z \in B_k} \|Du_k(z)\|_g$ and $r_kB_k$ be the $\epsilon r_k$-ball centered at the origin.
 Let $\phi_k:r_kB_k \to B_k$ be $\phi_k(z)=z_k+\frac{z}{r_k}$.
 The sequence of maps $u_k \circ \phi_k: r_kB_k \to \Conf^n(E)$ satisfies $\sup_{r_kB_k} \|D(u_k \circ \phi_k)\|<2$ for all $k$.
Since $\pi_{\Sigma}^{-1}(B)$ is a disjoint union of $n$ discs in $\Sigma$,
 the $u_k \circ \phi_k$ induce corresponding maps $V_{k,j}:  r_kB_k \to E$ for $j=1,\dots,n$ as in Lemma \ref{l:tautologicalCorr}.
 By the same reasoning in Lemma \ref{l:comparableMetric}, we know that there exists $T>0$ such that $\sup_{r_kB_k}\|DV_{k,j}\|_{g'} <T $ for all $k$.
 
 Since the metric on $E$ extends to a metric on $E^{\rceil}$ and $V_{k,j} \in W^{1,p}$ (for $p>2$) with uniformly bounded $W^{1,p}$-norm,
 by applying compactness of $W^{1,p}$ in $C^0$ to $V_{k,j}$ for each $j$ in turn, 
 we get a subsequence which converges uniformly on compact subsets to continuous functions $V_{\infty,j}: \CC \to E^{\rceil}$ for all $j=1,\dots,n$.
 Moreover, $V_{\infty,j}$ satisfies the $(j_\CC,J'_{z_\infty})$-holomorphic equation for all $j$, because of the stipulation that  $J_z=(J'_z)^{[n]}$ for all $z$.
 Elliptic bootstrapping shows that $V_{\infty,j}$ is smooth.
 Moreover, since $\|Du_k(z_k)-X_K(z_k)\|_g$ diverges to infinity, at least one of $V_{\infty,j}$ is not a constant map.
 By Lemma \ref{l:comparableMetric}, $DV_{\infty,j}$ has finite $L^2$-norm so we can apply removal of singularities to conclude that every $V_{\infty,j}$ extends to a  $J'_{z_\infty}$-holomorphic
 map $\CP^1 \to E^{\rceil}$ (and at least one of them is not a constant).
 
 By assumption \eqref{eq:StrictPositiveInt}, every non-constant  $J'_{z_\infty}$-holomorphic map $\CP^1 \to E^{\rceil}$
 has strictly positive algebraic intersection with $D_E$, which in turn implies that this is true for $v_k$ for large $k$ (because all other sphere bubbles, if any, 
 also contribute positively to the algebraic intersection).
 However, $Im(v_k)$ is contained in $E$, giving a contradiction. 
 
 Now, if instead  $z_\infty \in \partial S$, then we can apply the same rescaling trick and the outcome is a 
 $J'_{z_\infty}$-holomorphic disc with appropriate Lagrangian boundary for each $j$, and at least one of them is non-constant.
 By exactness of the Lagrangian boundary, the  $J'_{z_\infty}$-holomorphic disc
 has strictly positive algebraic intersection with $D_E$, which in turn implies that that is true for $v_k$ for large $k$, yielding the same contradiction.
 \end{proof}

Next, we consider the case that we have a uniform bound on $\sup_{S_k \setminus \nu(\mk(S_k))} \|Du_k(z_k)-X_K(z_k)\|_g$.

\begin{proposition}\label{p:Compactness}
 Fix $S \in \cR^{d+1,h}$.
 Let $u_k:S_k \to \Hilb^n(E)$ be a sequence in $\cR^{d+1,h}(\ul{x}_0;\ul{x}_d, \dots, \ul{x}_1)$ such that $S_k$ converges to $S \in \cR^{d+1,h}$.
 We assume that $h=I_{\ul{x}_0;\ul{x}_d, \dots, \ul{x}_1}$.
 If there exists $T>0$ such that 
 \begin{align}
\sup_{S_k \setminus \nu(\mk(S_k))} \|Du_k(z_k)-X_K(z_k)\|_g <T  \label{eq:EnergyNoConcentration}
 \end{align}
 for all $k$,
 then there exists $u_\infty \in \cR^{d+1,h}(\ul{x}_0;\ul{x}_d, \dots, \ul{x}_1)$ such that a subsequence of $u_k$ converges (uniformly on compact subsets) to $u_\infty$.

\end{proposition}

\begin{proof}
 Without loss of generality, assume $S_k=S$ for all $k$.
 Let $v_k:\Sigma_k \to E$ be the corresponding maps.
Br Corollary \ref{c:missingDr},  \eqref{eq:BranchedInDiag} and Remark \ref{r:BranchedInDiag}, $\pi_{\Sigma_k}:\Sigma_k \to S$
is an $n$-fold branched covering such that its critical values are precisely $\mk(S)$ and all of its critical points are Morse.
 Since there are only finitely many such $n$-fold branched coverings of $S$, by passing to a subsequence, we can assume $\Sigma_k=\Sigma$ for all $k$.
 
 First, Lemma \ref{l:comparableMetric} and \eqref{eq:EnergyNoConcentration} imply that there is no energy concentration outside $\pi_{\Sigma}^{-1}(\nu(\mk(S)))$.
 Since $v_k|_{\pi_{\Sigma}^{-1}(\nu(\mk(S)))}$ is $(J'_{\pi_{\Sigma}(z)})_{z \in \pi_{\Sigma}^{-1}(\nu(\mk(S)))}$-holomorphic, energy concentration inside 
 $\pi_{\Sigma}^{-1}(\nu(\mk(S)))$ would result in a non-constant map $\CP^1 \to E^{\rceil}$, which is impossible.
 Therefore, there exists $T>0$ such that 
 \begin{align}
\sup_{\Sigma} \|Dv_k(z_k)-X_{\pi_{\Sigma}^*K}(z_k)\|_{g'} <T  \label{eq:EnergyNoConcentration2}
 \end{align}
 for all $k$.
 
 In particular, for every relatively compact open subset $G$ of $S$, we have $v_k|_{\pi_\Sigma^{-1}(G)} \in W^{1,p}$ for $p>2$.
 Therefore, $v_k|_{\pi_\Sigma^{-1}(G)}$ converges uniformly to a continuous function $v_\infty: \pi_\Sigma^{-1}(G) \to E^{\rceil}$.
 
 There exists a compact set $C_E$ of $E$ such that outside $C_E$, $v_k|_{\pi_\Sigma^{-1}(G)}$ satisfies \eqref{eq:FloerEquationV} (cf. Remark \ref{r:SplitData}).
 Therefore, so does $v_\infty$, and elliptic bootstrapping implies that $v_\infty$ is smooth outside $C_E$ and satisfies \eqref{eq:FloerEquationV}.
 Let $u'_\infty:G \to \Sym^n(E^\rceil)$ be the corresponding continuous map induced from $v_\infty$.
 We can apply elliptic bootstrapping away from $\Delta_{E^\rceil}$ so $u'_\infty$ is actually smooth away from $\Delta_{E^\rceil}$ , so $v_\infty$ is smooth everywhere.
 
 We will prove that
 \begin{align}
  &Im(v_\infty) \cap D_E= \emptyset  \label{eq:NoEscapeHor}\\
  &Im(v_\infty) \cap \pi^{-1}_E(\partial \bH)= \emptyset  \label{eq:NoEscapeVer}\\
  &(u'_\infty)^{-1}(\Delta_{E^\rceil}) = G \cap \nu(\mk(S)). \label{eq:NoEscapeDiag}
 \end{align}
 Given these, $u'_\infty$ uniquely lifts to a map $u_\infty:G \to \Hilb^n(E)$ such that $u_\infty^{-1}(D_{HC}) \subset G \cap \nu(\mk(S))$.
 Note that this will be true for for every relatively compact open subset $G$ of $S$.
 
 On the other hand, we will also prove that, by possibly passing to a subsequence, 
 \begin{align}
\text{there is a compact subset $C$ of $\Conf^n(E)$ such that $u_k(Im(\epsilon_j)) \subset C$ for all $k$ and $j$} \label{eq:NoEscapeEnd}
 \end{align}
 so we can apply compactness of $u_k$ over strip-like ends inside $C$.
 Combining all this information, and by a diagonal subsequence argument, we obtain $u_\infty \in \cR^{d+1,h}(\ul{x}_0;\ul{x}_d, \dots, \ul{x}_1)$ such that $u_k$ has a subsequence 
 that is uniformly converging to $u_\infty$ over $S \setminus \nu(\mk(S))$.
 Note that we cannot guarantee the convergence is also uniform over $\nu(\mk(S))$, because $u_\infty$ is lifted from $u'_\infty$.
 
 However, we can argue as follows.
 If there is a subsequence of $u_k$ which converges uniformly over $\nu(\mk(S))$, then it converges pointwise to $u_\infty$ over $\nu(\mk(S)) \setminus \mk(S)$.
 By continuity, the limit must be $u_\infty$ and we are done.
 If there is no such subsequence, then for all $T>0$ and $N>0$, there exists $k>N$ such that $\sup_{\nu(\mk(S))} \| Du_k\|_{g''} >T$, where $g''$ is the induced metric 
 from a choice of a symplectic
 form on $\Hilb^n(E)$ that tames $J_E^{[n]}$.
 In this case, bubbling occurs and results in a $J_E^{[n]}$-holomorphic sphere mapping to $\Hilb^n(E)$.
 However, every $J_E^{[n]}$-holomorphic sphere intersects $D_r$ strictly positively (see Lemma \ref{l:rationalCurveinHilb}) so $Im(u_k) \cap D_r  \neq \emptyset$ for large $k$. This contradicts our assumption that
 $h=I_{\ul{x}_0;\ul{x}_d, \dots, \ul{x}_1}$, by Corollary \ref{c:missingDr}.
\end{proof}

Now, we need to justify the claims in the previous proof.

\begin{lemma}\label{l:1true}
 \eqref{eq:NoEscapeHor} is true.
\end{lemma}

\begin{proof}
 Suppose not, then there exists $z_0 \in \pi_{\Sigma}^{-1}(G)$ such that $v_\infty(z_0) \in D_E$.
 
 First, if $\pi_{E^\rceil}\circ v_\infty(z_0) \in C_{\bH}$, then $v_\infty$ is $J'$-holomorphic near $z_0$ so it has strictly positive algebraic intersection with $D_E$.
 Contradiction.
 
 If $\pi_{E^\rceil}\circ v_\infty(z_0) \notin C_{\bH}$, then \eqref{eq:FloerEquationV} implies that 
 $v_\infty$ satisfies a perturbed pseudo-holomorphic equation near $z_0$ with perturbation term having values in Hamiltonian vector fields tangent to $D_E$ (see 
 \eqref{eq:tangentD0}).
 It means that the intersection also contributes strictly positively to the algebraic intersection by the graph trick (cf. the proof of Lemma \ref{l:PositivityIntersection}).
 Contradiction.
\end{proof}

The following is a modification of the corresponding result of Seidel in \cite[Lemma 4.9]{Seidel4.5}.

\begin{proposition}\label{p:containment}
 \eqref{eq:NoEscapeVer} is true.
\end{proposition}

\begin{proof}
 We can make \eqref{eq:FloerEquationV} more akin to the situation in \cite{Seidel4.5}, namely, if we define 
$A_{\Sigma}=\pi_{\Sigma}^*A_S$,  $K_{\Sigma}=\pi_{\Sigma}^*K'$ and $(J_{\Sigma})_z=J_{\pi_\Sigma(z)}'$ for $z \in \Sigma$, then 
outside a compact set of $E$, \eqref{eq:FloerEquationV} becomes
 \begin{equation}
  (Dv_k-X_{K_\Sigma})^{0,1}=0 \text{ with respect to } ((J_\Sigma)_z)_{u(z)}.
 \end{equation}
 Note also that our choice of Floer data $(A_S,K,J)$ (see \eqref{eq:Symlambda} and \eqref{eq:orderinglambda}) 
 and the induced Floer data $(A_{\Sigma},K_{\Sigma},J_{\Sigma})$ guarantee that over each strip-like end, we have
 \begin{align}
  g_{A_{\Sigma}}^{-1}(\lambda_{L_{j,b_i}}) =g_{A_{j}}^{-1}(\lambda_{L_{j,b_i}}) < \lambda_{L_{j-1,a_i}} \label{eq:ordering1}
 \end{align}
 for all $a_i$ and $b_i$ and all $j=1,\dots,d$. For $j=0$, we have
 \begin{align}
g_{A_{\Sigma}}^{-1}(\lambda_{L_{d,b_i}})=g_{A_{0}}^{-1}(\lambda_{L_{d,b_i}}) < \lambda_{L_{0,a_i}}  \label{eq:ordering2}
 \end{align}
for all $a_i$ and $b_i$,
 
 In \cite[Lemma $4.9$]{Seidel4.5}, Seidel shows that if $Im(v_\infty) \cap \pi^{-1}_E(\partial \bH) \neq \emptyset$, then $v_\infty$
 lies entirely in $\pi^{-1}_E(\partial \bH)$ . In that case it will satisfy
 \begin{align}
 \left\{
 \begin{array}{ll}
  \partial_s (\pi_E \circ v_{\infty} \circ \epsilon_0)=0 \\
  \partial_t (\pi_E \circ v_{\infty} \circ \epsilon_0)=X_{a_{\ul{L}_{0},\ul{L}_{d},t}} \\
  \pi_E \circ v_{\infty} \circ \epsilon_0(s,0) \in \lambda_{\ul{L}_{0}} \\
  \pi_E \circ v_{\infty} \circ \epsilon_0(s,1)\in \lambda_{\ul{L}_{d}}.
 \end{array}
 \right.
 \end{align}
 These conditions imply $g_{A_{0}}^{-1}(\lambda_{\ul{L}_{d}}) \cap \lambda_{\ul{L}_{0}} \neq \emptyset$, violating \eqref{eq:ordering2}. Therefore $Im(v_\infty) \cap \pi^{-1}_E(\partial \bH) = \emptyset$.
\end{proof}

\begin{lemma}\label{l:3true}
 \eqref{eq:NoEscapeDiag} is true.
\end{lemma}

\begin{proof}
 Suppose not. Then there is a point $z_0 \in G \setminus \mk(S)$ such that $u'_\infty(z_0) \in \Delta_E$ 
 (here, we have applied Lemma \ref{l:1true} and Proposition \ref{p:containment} to replace $\Delta_{E^\rceil}$ by $\Delta_E$).
 
 Then, by the same reasoning as in Lemma \ref{l:PositivityIntersection}, $\Sym^n(\pi_E) \circ u'_\infty(G)$ and $\Delta_\bH$ intersect positively at $z_0$.
 That would violate the assumption that $h=I_{\ul{x}_0;\ul{x}_d, \dots, \ul{x}_1}$.
\end{proof}

\begin{lemma}
 \eqref{eq:NoEscapeEnd} is true.
\end{lemma}

\begin{proof}
 Suppose not, then there exists $j \in \{0,\dots,d\}$ and a sequence $z_k=\epsilon_j(s_k,t_k)$ such that $u_k(z_k)$ goes to the infinite end of $\Conf^n(E)$.
 Moreover, we know that $t_k$ has a subsequence converging to some $t_\infty \in [0,1]$ and $s_k$ goes to infinity. 
 Pick a small neighborhood $D \subset \CC$ of $it_\infty$ and consider the maps $\hat{u}_k:D \to \Conf^n(E)$ given by $\hat{u}_k(s,t)=u_k(\epsilon_j(s+s_k,t))$.
 Let $\hat{v}_{k,a}: D \to E$ be the corresponding maps for $a=1,\dots,n$.
 
 Since there is no energy concentration, there is a subsequence of $k$ such that for each $a=1,\dots,n$, $\hat{v}_{k,a}$ converges uniformly
 to a continuous map $\hat{v}_{\infty,a}:D \to E^{\rceil}$;  these in turn induce a continuous map $\hat{u}_{\infty}':D \to \Sym^n(E^\rceil)$.
 By applying Lemma \ref{l:1true}, \ref{l:3true} and Proposition \ref{p:containment}, we get the corresponding results \eqref{eq:NoEscapeHor}, \eqref{eq:NoEscapeVer}
 and \eqref{eq:NoEscapeDiag} for $\hat{v}_{\infty,a}$ and $\hat{u}_{\infty}'$.
 This contradicts  the assumption that $u_k(z_k)$ goes to infinity in $\Conf^n(E)$.
 
\end{proof}

We now assume that $(A_S,J,K)$ is chosen to vary smoothly in $\ol{\cR}^{d+1,h}$.
That implies that $J=J^{[n]}_E$ and $K=0$ on the sphere components of $S \in \ol{\cR}^{d+1,h}$ (see Remark \ref{r:sphereCom}).
We call the (interior/boundary) special points of $S$ which connect different irreducible components of $S$  (interior/boundary) nodes.
For an interior node $z$ on a disc component, we define the multiplicity $\mult(z)$ of $z$ to be the total number of marked points (not including nodes) on the tree of sphere 
components to which $z$ is connected.

Lemma \ref{l:nobubble} can be directly generalized by replacing $\cR^{d+1,h}$ with $\ol{\cR}^{d+1,h}$.
The generalization of Proposition \ref{p:Compactness} is as follows.

\begin{proposition}\label{p:Compactness2}
 Fix $S \in \ol{\cR}^{d+1,h}$.
 Let $u_k:S_k \to \Hilb^n(E)$ be a sequence in $\cR^{d+1,h}(\ul{x}_0;\ul{x}_d, \dots, \ul{x}_1)$ such that $S_k$ converges to $S \in \ol{\cR}^{d+1,h}$.
 We assume that $h=I_{\ul{x}_0;\ul{x}_d, \dots, \ul{x}_1}$.
 If there exists $T>0$ such that 
 \begin{align}
\sup_{S_k \setminus \nu(\mk(S_k))} \|Du_k(z_k)-X_K(z_k)\|_g <T  \label{eq:EnergyNoConcentration}
 \end{align}
 for all $k$,
 then there exists a subsequence of $u_k$ which converges  (uniformly on compact subsets) to 
a stable map $u_\infty:S \to  \Hilb^n(E)$ such that $u|_Q$ is a constant map for every sphere component $Q$ of $S$.
Moreover, interior nodes on disc components of $S$ are mapped to $D_{HC}$ under $u_{\infty}$ and the algebraic intersection number between $u_{\infty}$ and $D_{HC}$ at an interior
node $z$ is given by $\mult(z)$.

\end{proposition}

\begin{proof}
For ease of exposition, we only consider the case
$S=S_{\alpha} \cup S_{\beta}$, where $S_\alpha$ is a disc component and $S_{\beta}$ is a sphere component.
Let $z_{\alpha \beta} \in S_{\alpha}$ 
and $z_{\beta \alpha} \in S_{\beta}$  be the nodes.
There exists an open subset $S_{\alpha,k}$ of $S_k$
such that $S_{\alpha,k}$ converges to $S_{\alpha} \setminus z_{\alpha \beta}$.
The proof of Proposition \ref{p:Compactness} can be applied 
to $u_k|_{S_{\alpha,k}}$ to conclude that there is a subsequence which converges
(uniformly on compact subsets) to a map 
$u_{\infty}^{\alpha}:S_{\alpha} \setminus z_{\alpha \beta} \to \Hilb^n(E)$ which satisfies \eqref{eq:FloerEquation0}. By 
removal of singularities, we can extend this to a map  
$u_{\infty}^{\alpha}:S_{\alpha}  \to \Hilb^n(E)$.

On the other hand, there exists an open subset $S_{\beta,k}$ of $S_k$
such that $S_{\beta,k}$ converges to $S_{\beta} \setminus z_{\beta \alpha}$.
Gromov compactness  can be applied to $u_k|_{S_{\beta,k}}$ and after
removal of singularities, we obtain a stable $J_E^{[n]}$-holomorphic map $u_{\infty}^{\beta}$ from a tree of spheres
to $\Hilb^n(E)$.
 
We know that every non-constant $J_E^{[n]}$-holomorphic sphere intersects $D_r$ strictly positively, by Lemma \ref{l:rationalCurveinHilb}, so 
if $u_{\infty}^{\beta}$ is not a constant map, then we have
 $Im(u_k) \cap D_r  \neq \emptyset$ for large $k$ (cf. the last paragraph in the proof of Proposition \ref{p:Compactness}).
This is a contradiction, so $u_{\infty}^{\beta}$ is a constant map to a point in $D_{HC}$.
That implies that 
$u_{\infty}^{\alpha}(z_{\alpha \beta}) \in D_{HC}$, and furthermore that 
the algebraic intersection number with
$D_{HC}$ at this point is precisely $\mult(z_{\alpha \beta})$.
 \end{proof}

\begin{corollary}\label{c:v1}
Suppose $\cR^{d+1,h}(\ul{x}_0;\ul{x}_d, \dots, \ul{x}_1)$  has virtual dimension $1$ and 
$h=I_{\ul{x}_0;\ul{x}_d, \dots, \ul{x}_1}$. For generic $(J,K)$ satisfying \eqref{eq:FloerDataJ}
and \eqref{eq:FloerDataK}, the moduli
 $\cR^{d+1,h}(\ul{x}_0;\ul{x}_d, \dots, \ul{x}_1)$
can be compactified by adding stable maps $u:S \to  \Hilb^n(E)$ for $S\in \ol{\cR}^{d+1,h}$
for which $S$ has no sphere components.
\end{corollary}

\begin{proof}
By Lemma \ref{l:nobubble} and 
Proposition \ref{p:Compactness2}, $\cR^{d+1,h}(\ul{x}_0;\ul{x}_d, \dots, \ul{x}_1)$ can be compactified by stable maps 
 $u:S \to  \Hilb^n(E)$
such that every sphere component of $u$ is mapped to a constant.
If $S$ has a sphere component, then some disc component would have algebraic intersection number $>1$ with $D_{HC}$ at an interior node.
However, this is a phenomenon of codimension at least $2$, so it 
has virtual dimension at most $-1$ and hence
can be avoided for generic $(J,K)$.

More precisely, just as  $\cR^{d+1,h}(\ul{x}_0;\ul{x}_d, \dots, \ul{x}_1)$ is the fiber product between 
(a smooth pseudocycle replacement for) the inclusion $D_{HC}^h \to (\Hilb^n(E))^h$ and the evaluation map $\cR^{d+1,h}(\ul{x}_0;\ul{x}_d, \dots, \ul{x}_1)_{pre} \to (\Hilb^n(E))^h$, we can define 
a subset of $\cR^{d+1,h}(\ul{x}_0;\ul{x}_d, \dots, \ul{x}_1)_{pre}$ consisting of those $u$
for which the interior marked points have prescribed algebraic intersection numbers with $D_{HC}$.
In this case, some factors of the target of the evaluation map should be taken to be the appropriate jet bundles of $\Hilb^n(E)$ (cf. \cite[Section 6]{CieMoh}).
The regularity argument in Lemma \ref{l:regularity}, \ref{l:regularity2} applies to this case
to conclude that, whenever there is an interior node on a disc component that is required to be mapped to $D_{HC}$ with algebraic intersection number $>1$, the restriction of $u$ to the disc component 
does not exist for generic $(J,K)$.
Therefore, 
 $\cR^{d+1,h}(\ul{x}_0;\ul{x}_d, \dots, \ul{x}_1)$ can be compactified by stable maps without sphere components.
\end{proof}

\subsubsection{The definition}

An object of $\cFS^{cyl,n}(\pi_{E})$ is an element $\ul{L} \in \cL^{cyl,n}$.
Given two objects $\ul{L}_0=\{L_{0,k}\}_{k=1}^n$ and $\ul{L}_1=\{L_{1,k}\}_{k=1}^n$, we choose a Floer chain datum 
$(A_{\ul{L}_0,\ul{L}_1},H_{\ul{L}_0,\ul{L}_1}, J_{\ul{L}_0,\ul{L}_1})$ under the additional assumption that
if $\ul{L}_0=\ul{L}_1$, the perturbation pair chosen is given by Lemma \ref{l:NoSideBubbling}.
The corresponding morphism space is
\begin{align}
hom(\ul{L}_0,\ul{L}_1) :=CF(\ul{L}_0,\ul{L}_1)=\oplus_{\ul{x} \in \cX(H_{\ul{L}_0,\ul{L}_1}, \ul{L}_0,\ul{L}_1)} o_{\ul{x}} 
\end{align}


Recall that we have chosen strip-like ends, cylindrical ends and marked-points neighborhoods in $\ol{\cS}^{d+1,h}$ that are smooth up 
to the boundary and corners.
 Following \cite[Section (9i)]{SeidelBook}, see also \cite[Section 9]{Seidel4}, we call a choice of Floer data smooth and consistent if for any  Lagrangian labels 
$\ul{L}_0,\dots,\ul{L}_d$, the Floer data depends smoothly on $\ol{\cR}^{d+1,h}$ up to the boundary and corner.
Note that smoothness near the boundary and corner strata is defined with respect to a collar neighborhood obtained from gluing the lower dimensional strata.
Therefore, a smooth and consistent  choice of Floer data is obtained from an inductive procedure from the lower dimensional strata. In our case, we proceed as follows:
\begin{enumerate}
\item We equip the unique element $S \in \cR^{1+1,0}$ with the ($s$-invariant) data $(A_{\ul{L}_0,\ul{L}_1},J_{\ul{L}_0,\ul{L}_1},H_{\ul{L}_0,\ul{L}_1})$.
 \item For every element of $\cR^{d+1,0}$ with Lagrangian labels being $\ul{L}_0,\dots,\ul{L}_d$, we equip with it $(A_S,J,K)$ satisfying 
 \eqref{eq:SurfaceConnection}, \eqref{eq:FloerDataJ} and \eqref{eq:FloerDataK} such that $(A_S,J,K)$ varies smoothly and consistently with respect to gluing 
 (this is done inductively in $d$);
 \item Equip any sphere with any number of interior marked points with $(A_S,J,K)=(0,J_E^{[n]},0)$.
 \item Equip the element of $\cR^{0+1,1}$ with Lagrangian label being $\ul{L}_0$ or $\ul{L}_1$ with $(A_S,J,K)$ satisfying 
 \eqref{eq:SurfaceConnection}, \eqref{eq:FloerDataJ} and \eqref{eq:FloerDataK}; 
 \item For every element of $\cR^{d+1,1}$ with Lagrangian labels being $\ul{L}_0,\dots,\ul{L}_d$, we equip  it with $(A_S,J,K)$ satisfying 
 \eqref{eq:SurfaceConnection}, \eqref{eq:FloerDataJ} and \eqref{eq:FloerDataK} such that $(A_S,J,K)$ varies smoothly and consistently with respect to gluing; 
 \item repeat the procedure with increasing $h$.
\end{enumerate}
For  a generic consistent choice of Floer data, all elements in $\cR^{d+1,h}(\ul{x}_0;\ul{x}_d, \dots, \ul{x}_1)$ are regular (here, the genericity is with respect to the Floer's $C^{\infty}_{\epsilon}$-topology, see \cite[Remark 9.9]{SeidelBook}).



The $A_{\infty}$ structure is now defined by
\begin{align}
 \mu^d(\ul{x}_d,\dots,\ul{x}_1)=\sum_{\ul{x}_0 \in \cX(H_{\ul{L}_0,\ul{L}_d}, \ul{L}_0,\ul{L}_d)} \frac{1}{I_{\ul{x}_0;\ul{x}_d,\dots,\ul{x}_1} !}
 (\# \cR^{d+1,I_{\ul{x}_0;\ul{x}_d,\dots,\ul{x}_1}}(\ul{x}_0;\ul{x}_d,\dots,\ul{x}_1)) \ul{x}_0
\end{align}
where, when $h=0$ and $d=1$, $\# \cR^{d+1,h}(\ul{x}_0;\ul{x}_d,\dots,\ul{x}_1)$ is understood as the signed count after dividing out by the $\RR$-symmetry (and where we suppress the discussion of signs).

\begin{lemma}
 The collection of maps $\{\mu^d\}$ satisfies the $A_{\infty}$ relation.
\end{lemma}

\begin{proof}
 By Corollary \ref{c:v1}, the relevant $1$-dimensional moduli spaces can be compactified by stable broken maps without sphere components.
 By Lemma \ref{l:NoSideBubbling} and exactness of the individual Lagrangians, every component of such a stable broken map has at least two boundary punctures.
 The rest follows from a well-established argument (see e.g. \cite[Section 4.1.8]{Siegel}) together with the additivity property of $I_{\ul{x}_0;\ul{x}_d,\dots,\ul{x}_1}$ under gluing of maps.
\end{proof}


\section{Properties of the cylindrical Fukaya-Seidel category}\label{s:properties}

At this point, given  a four-dimensional exact Lefschetz fibration $\pi_E: E \to \bH$, we have constructed for each $n\geq 1$  an $A_{\infty}$ category $\cFS^{cyl,n}(\pi_E)$ which captures certain Floer-theoretic computations in the Hilbert scheme of $E$, or rather its affine subset $\cY_E$. Unsurprisingly, the resulting $A_{\infty}$ category satisfies typical properties of Fukaya-type categories, with proofs which are minor modifications of those which pertain in more familiar settings. This short section briefly elaborates upon some of these; to keep the exposition of manageable length, the proofs are only sketched.


\subsection{Unitality and Hamiltonian invariance}

\begin{lemma}[Cohomological unit]
 $\cFS^{cyl,n}(\pi_{E})$ is cohomologically unital.
\end{lemma}

\begin{proof}[Sketch of proof]
To give a cochain representative of the cohomological unit of $HF(\ul{L},\ul{L})$, we need to consider $S \in \cR^{0+1,0}$.
We put
the Floer cochain datum
$(A_{\ul{L},\ul{L}},H_{\ul{L},\ul{L}}, J_{\ul{L},\ul{L}})$
on the outgoing strip-like end and stabilize the disc $S$ by picking $(A_S,K,J)$ satisfying  \eqref{eq:SurfaceConnection}, \eqref{eq:FloerDataJ} and \eqref{eq:FloerDataK}.
The Lagrangian boundary condition is a moving one from $\phi_{H_{\ul{L},\ul{L}}}(\Sym(\ul{L}))$ to $\Sym(\ul{L})$.

We then form the moduli space of maps $u: S \to \Hilb^n(E) \setminus (D_r \cup D_{HC})$ such that $u$ is asymptotic to $\ul{x}_0 \in \cX(H_{\ul{L},\ul{L}},\ul{L},\ul{L})$
near the strip-like end.
The signed rigid count of this moduli gives a cochain level representative of the cohomological unit.

Note that we only need to consider $\cR^{0+1,h}$ for $h=0$ because, by Lemma \ref{l:NoSideBubbling}, the corresponding moduli space for $h>0$ is necessarily empty.
\end{proof}

\begin{lemma}\label{l:chainRep}
For any $\ul{L} \in \cFS^{cyl,n}(\pi_{E})$, there is a choice of Floer cochain datum such that 
$CF^*(\ul{L},\ul{L})$ is non-negatively graded and $CF^0(\ul{L},\ul{L})$ is rank one.
As a result, the Hamiltonian chord generating  $CF^0(\ul{L},\ul{L})$ gives a chain level representative of the cohomological unit.
\end{lemma}

\begin{proof}
As in the proof of Lemma \ref{l:NoSideBubbling}, we can choose $A$ and $H'$
such that 
the (possibly non-transversal) $X_{(H')^{[n]}}$-chord from $\Sym(\ul{L})$ to itself is given by the unordered tuple of $X_{H'}$-chords $x(t)$ from $L_i$ to itself for $i=1,\dots,n$.
Recall that $H'$ is chosen such that $\pi_E(x(t)) \in U_{L_i}$ for all $x(t)$.
For each $i$, we pick $H_i'$ to be a perturbation of $H'$ that is supported in a compact subset of $\pi_E^{-1}(U_{L_i})$ and 
such that the $X_{H'_i}$-chords from $L_i$ to itself satisfy: (i) they are transversal, (ii) they are non-negatively graded,  and (iii) there exists exactly one such chord of grading $0$.


Recall the definition of $q_{S_n}$ in \eqref{eq:SymL}.
Note that there is a compact subset $K$ of $q_{S_n}(\pi_E^{-1}(U_{L_1}) \times \dots \times \pi_E^{-1}(U_{L_n}))$  which contains all
the $X_{(H')^{[n]}}$-chords from $\Sym(\ul{L})$ to itself.
Since $q_{S_n}(\pi_E^{-1}(U_{L_1}) \times \dots \times \pi_E^{-1}(U_{L_n}))$ is
symplectomorphic to the product symplectic manifold $\pi_E^{-1}(U_{L_1}) \times \dots \times \pi_E^{-1}(U_{L_n})$,
we can find a  perturbation $H=(H_t)_{t \in [0,1]}$ of $(H')^{[n]}$ such that 
$H_t \in \cH^n_{a_t}(E)$ for all $t$ and $H$ is the product type Hamiltonian $H(\ul{z})=H'_1(z_1)+\dots +H_n'(z_n)$
locally near $K$,
where $\ul{z} \in q_{S_n}(\pi_E^{-1}(U_{L_1}) \times \dots \times \pi_E^{-1}(U_{L_n}))$
is identified with $(z_1,\dots,z_n) \in \pi_E^{-1}(U_{L_1}) \times \dots \times \pi_E^{-1}(U_{L_n})$.

In particular, it means that the $X_H$ Hamiltonian chords from $\Sym(\ul{L})$ to itself are precisely 
given by unordered tuples of $X_{H'_i}$-chord from $L_i$ to itself for $i=1,\dots,n$.
Therefore, the $X_H$-chords are
non-negatively graded 
and precisely one of them is in grading zero.
\end{proof}

\begin{lemma}[Hamiltonian invariance]\label{l:HamInvariance}
 If $\ul{L}_0, \ul{L}_1 \in \cL^{cyl,n}$ can be connected by a smooth family of elements $(\ul{L}_t)_{t \in [0,1]}$ in $\cL^{cyl,n}$,
 then $\ul{L}_0$ is quasi-isomorphic to $\ul{L}_1$ in $\cFS^{cyl,n}(\pi_{E})$.
\end{lemma}

\begin{proof}[Sketch of proof]
Without loss of generality, we can assume that $\ul{L}_1 $ is $C^1$-close  to $\ul{L}_0$.
In particular, we can assume that $\pi_E(L_{0,i}) \cap \pi_E(L_{1,j}) =\emptyset$ if $j \neq i$.

Pick a Floer cochain datum $(A_{\ul{L}_0,\ul{L}_1},H_{\ul{L}_0,\ul{L}_1}, J_{\ul{L}_0,\ul{L}_1})$  for $(\ul{L}_0, \ul{L}_1)$.
 For each integer $h \geq 0$ and $S \in \cR^{0+1,h}$, we equip $S$  with the moving Lagrangian boundary label $\ul{L}_t$ on $\partial S$.
 Then, we pick $(A_{S},K,J)$ such that \eqref{eq:SurfaceConnection}, \eqref{eq:FloerDataJ} and \eqref{eq:FloerDataK}  are satisfied
 and compatible with gluing with elements in $\cR^{1+1,k}$ for all $k$.
 The rigid count of the corresponding solutions of \eqref{eq:FloerEquation0} and \eqref{eq:FloerEquation1} give us an element $\alpha$ in $CF^0(\ul{L}_0, \ul{L}_1)$.
 
  We claim that $\alpha$ is a cocycle. Indeed, by the same argument as in Lemma \ref{l:NoSideBubbling}, when the Floer cochain datum is carefully chosen, the rigid count of solutions for $h>0$ vanishes.
In this case, $\alpha$  is a sum of continuation elements $\alpha_i \in CF^0(L_{0,i},L_{1,i})$.
 
 Similarly, we can construct a cocycle $\beta$ in $CF^0(\ul{L}_1, \ul{L}_0)$.
 By gluing, one checks that $\mu^2(\alpha,\beta)$ and $\mu^2(\beta,\alpha)$ are cohomological units (for suitable choices of Floer cochain data as above, these cohomological units agree with the classical units). \end{proof}

\subsection{Canonical embeddings}\label{sss:CanEmbed}

Let $W$ be a contractible (hence connected) open subset of $\bH$ such that
\begin{align}
 \text{$W \cap \partial \bH$ equals $(R,\infty)$ or $(-\infty,R)$ for some $R \in \mathbb{R}$} \label{eq:Wcondition}
\end{align}

Let $W^{\circ}:=W \cap \bH^\circ$.
We can define an $A_{\infty}$ full subcategory $\cFS^{cyl,n}_W(\pi_E)$ of $\cFS^{cyl,n}(\pi_E)$ with objects given by 
$\ul{L}=\{L_1,\dots,L_n\}$ such that $\pi_E(L_j) \subset W^{\circ}$ for all $j$.

The conditions that $W$ is contractible and satisfies \eqref{eq:Wcondition} imply that $W^c:=Int(\bH \setminus W)$ is also contractible and satifies \eqref{eq:Wcondition}, where $Int(-)$ stands for the interior.

Let $\ul{K}$ be an object of $\cFS^{cyl,k}_{W^c}(\pi_E)$. 
Each object $\ul{L}$ of $\cFS^{cyl,n}_W(\pi_E)$ determine a object $\ul{L}\sqcup\ul{K}$ in $\cFS^{cyl,n+k}(\pi_E)$ given by adding to 
$\ul{L}$ the Lagrangians in $\ul{K}$.

\begin{lemma}
 There is a cohomologically faithful $A_{\infty}$ functor $\sqcup \ul{K}: \cFS^{cyl,n}_W(\pi_E) \to \cFS^{cyl,n+k}(\pi_E)$,
  which on objects is given by sending $\ul{L}$ to $\ul{L}\sqcup\ul{K}$.
\end{lemma}

\begin{proof}
 By \eqref{eq:Wcondition}, $W$ contains a neighborhood of $[R,\infty)$ or $(-\infty,R]$.
 Therefore, we can choose the Hamiltonian term in the Floer data such that
 each element in $\cX(\ul{L_0}\sqcup \ul{K},\ul{L_1}\sqcup \ul{K})$ is a tuple  given by adjoining an element in 
 $\cX(\ul{L_0},\ul{L_1})$ with another element in $\cX(\ul{K},\ul{K})$.
 It establishes a bijective correspondence between the generators in $CF(\ul{L_0}\sqcup \ul{K},\ul{L_1}\sqcup \ul{K})$
 and $CF(\ul{L_0},\ul{L_1}) \otimes CF(\ul{K},\ul{K})$.
By Lemma \ref{l:chainRep}, we can arrange the Floer data such that  $CF^0(\ul{K},\ul{K})$ is rank one and 
generated by a chain level representative $e$ of the cohomological unit.

The $A_{\infty}$ functor  $\sqcup \ul{K}$ has first order term sending $x \in CF(\ul{L_0},\ul{L_1})$
to the generator in $CF(\ul{L_0}\sqcup \ul{K},\ul{L_1}\sqcup \ul{K})$ corresponding to $x\otimes e$, and has no higher order terms an an $A_{\infty}$ functor. 
To check that $\sqcup \ul{K}$ is indeed an  $A_{\infty}$ functor, it suffices to show that there is a bijective correspondence between the moduli governing the $A_{\infty}$ structures in $\cFS^{cyl,n}_W(\pi_E)$ and the moduli 
 governing the $A_{\infty}$ structures in $ \cFS^{cyl,n+k}(\pi_E)$ with all inputs and output being of the form $x\otimes e$.
 
 Let $u:S \to \Hilb^n(E)$ be a solution contributing to the $A_{\infty}$ structure of  $\cFS^{cyl,n}_W(\pi_E) $.
 Let $v:\Sigma \to E$ be the map which tautologically corresponds to $u$.
 By projecting to $\bH^\circ $ via $\pi_E$ and appealing to the open mapping theorem, we know that the image of $v$ lies inside $W$.

The proof of Lemma \ref{l:chainRep} shows
 that $e$ is given by the unordered tuple of grading zero $X_{H_i'}$ Hamiltonian chords $e_i$ from $K_i$ to itself, for $i=1,\dots,k$.
 Moreover, we can assume that $e_i$ is a constant chord (i.e. $X_{H_i'}(e_i(t))=0$ for all $t$).
 
 We now separate the discussion into two cases, namely, the stable case 
 $(d,h) \neq (1,0)$ and the semi-stable case $(d,h) = (1,0)$.
 We start with $(d,h) \neq (1,0)$.
 
 In this case, for each $S \in \cR^{d+1,h}$, the constant map $v_i$ from $S$ to $e_i$ satisfies
 \begin{align}\label{eq:constantSol}
 \left\{
 \begin{array}{ll}
  (Dv_i-X_{H_j'})^{0,1}=0 \\
  v_i(\partial S) \subset L_i' \\
  \lim_{s \to \pm \infty} v_i(\epsilon_j(s,\cdot))=e_i \text{ uniformly for all }j=0,\dots,d.
 \end{array}
 \right.
 \end{align}
 The moduli of solutions to \eqref{eq:constantSol} has virtual dimension $0$ (because the conformal structure of $S$ is fixed).
 Moreover, $v_i$ is a regular solution to \eqref{eq:constantSol}.
 Note also that $\pi_E \circ v_i \notin W$ for all $i$.
 
 Now we define $\widetilde{\Sigma}=\Sigma \sqcup  (\sqcup_{i=1}^k S)$, and we define $\pi_{\widetilde{\Sigma}}:\widetilde{\Sigma} \to S$
 to be $\pi_{\widetilde{\Sigma}}=\pi_\Sigma \sqcup (\sqcup_{i=1}^k Id_S)$, where $Id_S:S \to S$ is the identity map.
 Let $\widetilde{v}=v \sqcup (\sqcup_{i=1}^k v_i)$ and $\widetilde{u}:S \to \Hilb^{n+k}(E)$ be the map 
 tautologically corresponding to $\widetilde{v}$.
 Notice that $Im(v)$ and $Im(v_i)$ are pairwise disjoint.
This means that the Fredholm operator associated to $\widetilde{u}$
splits into the direct sum of those associated to $u$ and $v_i$.
Therefore, $\widetilde{u}$ is a regular solution contributing to the $A_{\infty}$ structure of  $\cFS^{cyl,n+k}(\pi_E) $
 with all inputs and output being of the form $x\otimes e$.
 Conversely, by the Lagrangian boundary conditions, every such solution is of the form $\widetilde{u}$ for some $u$ contributing to the $A_{\infty}$ structure of  $\cFS^{cyl,n}_W(\pi_E) $.
 
 Next, we consider the case $(d,h) = (1,0)$.
 In this case, $\Sigma=\sqcup_{i=1}^n S$ and $\pi_{\Sigma}=\sqcup_{i=1}^n Id_S$.
 Since $u$ is rigid, it means that the moduli containing $u$ is $1$ dimensional before dividing out by the $\RR$-translation-symmetry.
 We define $\widetilde{\Sigma}$, $\pi_{\widetilde{\Sigma}}$, $\widetilde{v}$ and $\widetilde{u}$ as above.
 The moduli containing $\widetilde{u}$ still has virtual dimension $1$ before dividing by the $\RR$-symmetry, because the constant maps have virtual dimension $0$.
 Moreover, $\widetilde{u}$ is regular because one can split the Fredholm operator 
 of $\widetilde{u}$ to those of $u$ and the constant maps, which are all regular.
 Conversely, rigid solutions contributing to the $A_{\infty}$ structure of  $\cFS^{cyl,n+k}(\pi_E) $ with $(d,h)=(1,0)$ and 
 both input and output being of the form $x\otimes e$ are necessarily of the form $\widetilde{u}$ for some $u$.
 
  This finishes the proof.
 \end{proof}

\begin{corollary}\label{c:Triangle}
 If $\ul{L}_0,\ul{L}_1,\ul{L}_2$ form an exact triangle in $\cFS^{cyl,n}_W(\pi_E)$, then  $\ul{L}_0 \sqcup \ul{K},\ul{L}_1 \sqcup \ul{K},\ul{L}_2 \sqcup \ul{K}$ form one in $\cFS^{cyl,n+k}(\pi_E)$.
\end{corollary}

Using Corollary \ref{c:Triangle}, we can inductively construct a plethora of exact triangles.
For example, we will prove in Section \ref{ss:AmMilnorFiber} that our setup applies to the standard Lefschetz fibration $\pi_E$ on the $A_{m-1}$-Milnor fiber $E$. In that case, 
there are well-known exact triangles \cite[Lemma 18.20]{SeidelBook} relating matching spheres and thimbles in $\cFS^{cyl,1}(\pi_E)=\cFS(\pi_E)$.
We can obtain exact triangles in $\cFS^{cyl,n}(\pi_E)$ by adjoining matching spheres/thimbles to the exact triangles in $\cFS^{cyl,1}(\pi_E)$ (see Figure \ref{fig:pushTriangle}).

\begin{figure}[ht]
 \includegraphics{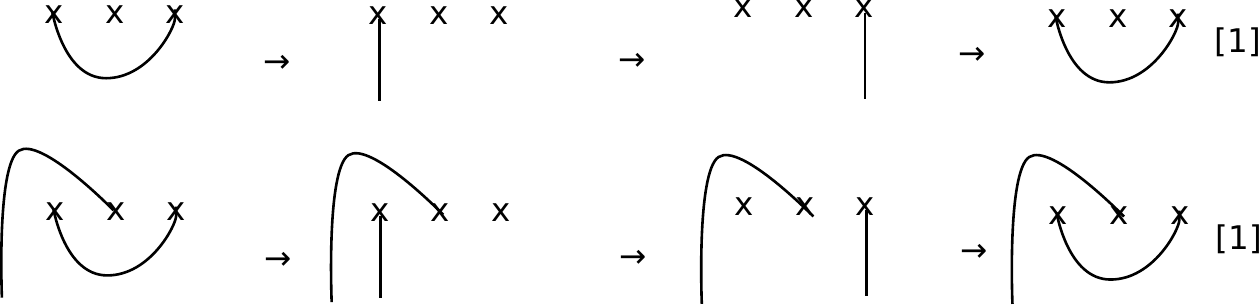}
 \caption{The top row represents an exact triangle in $\cFS(\pi_E)$ when $E$ is the $A_2$-Milnor fiber. The bottom row is exact in $\cFS^{cyl,2}(\pi_E)$ by Corollary \ref{c:Triangle}.}\label{fig:pushTriangle}
\end{figure}



\subsection{Serre functor}\label{ss:SerreFunctor}

An important observation due to Kontsevich and Seidel is that the global monodromy $\tau$ should induce the Serre functor on the Fukaya-Seidel cayegory up to degree shift 
\cite{Seidel1, SeidelSHHH, Seidel2, Hanlon}.
In our context, the global monodromy of $\pi_{E}$ induces an auto-equivalence of $\cFS^{cyl,n}(E)$ for each $n$.
By formally the same argument, we:

\begin{claim}\label{p:Serre}
 The global monodromy of $\pi_{E}$ induces the Serre functor on $\cFS^{cyl,n}(E)$ up to a degree shift by $-2n$.
\end{claim}

Even for Lefschetz fibrations, a complete proof that 
the auto-equivalence induced by global monodromy agrees with the
Serre functor (up to shift) does not seem to appear in the literature\footnote{Forthcoming work of Abouzaid and Ganatra lays the general foundations for the treatment of Serre functors in the context of Fukaya categories for Landau-Ginzburg models, which generalise Fukaya-Seidel categories.}.  
Any argument is likely to apply to our case.
For the reader's convenience, we outline the essential geometric input underlying the Claim.

\begin{proof}[Sketch of proof]

\begin{figure}[ht]
 \includegraphics{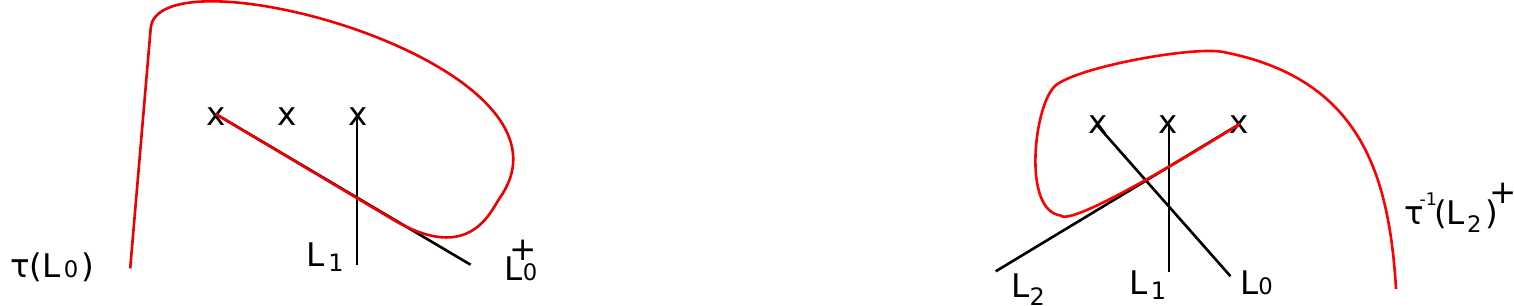}
 \caption{}\label{fig:GlobalMonodromy}
\end{figure}

As a graded symplectomorphism, we require that $\tau$ acts as the identity on the trivialization of the bicanonical bundle 
in a compact region containing the critical points.
It means that for each compact exact graded Lagrangian $L$ in $E$, $\tau(L)=L$ as graded objects.
 On the chain level, there is a canonical isomorphism (see the left of Figure \ref{fig:GlobalMonodromy})
\begin{align}
CF^i(L_0,L_1) \simeq (CF^{-i}(L_1, \tau(L_0)[-2]))^{\vee}=(CF^{-i}(\tau^{-1}(L_1)[2], L_0))^{\vee} \label{eq:SerreCan}
\end{align}
given by sending an intersection of $L_0^+ \cap L_1$ to the corresponding intersection point of $\tau(L_0) \cap L_1^+$
but viewing the latter as a generator in the dual group $CF(L_1, \tau(L_0))$.
Here, for two Lagrangians $L,K$, we use $L^+ \cap K$ (or $K \cap L^+$)
to mean 
that we pick a Hamiltonian diffeomorphism $\phi$ such that $\lambda_K <\lambda_{\phi(L)}$, $\phi(L) \pitchfork K$
and we use $L^+ \cap K$ to denote $\phi(L) \pitchfork K$.

In particular, when $L_0$ is compact, a generator $x$ of $CF^i(L_0,L_1)$ is mapped to the linear dual of 
the corresponding generator
$x^\vee \in CF^{2-i}(L_1, L_0)=CF^{-i}(L_1,L_0[-2])=CF^{-i}(L_1,\tau(L_0)[-2])$.
The canonical isomorphism \eqref{eq:SerreCan} lifts to the canonical isomorphism for Lagrangian tuples 
\begin{align}
 CF^i(\ul{L}_0,\ul{L}_1) \simeq (CF^{-i}(\ul{L}_1, \tau(\ul{L}_0)[-2n]))^{\vee}=(CF^{-i}(\tau^{-1}(\ul{L}_1)[2n], \ul{L}_0))^{\vee} \label{eq:SerreCan2}
\end{align}

A key claim is that, under the canonical isomorphism \eqref{eq:SerreCan2}, one can arrange to have a bijective correspondence between the moduli computing higher $A_{\infty}$ operations:
\begin{align}
CF(\ul{L}_{d-1},\ul{L}_d) \times CF(\ul{L}_{d-2},\ul{L}_{d-1}) \times \dots \times CF(\ul{L}_0,\ul{L}_1) \to CF(\ul{L}_0,\ul{L}_d), \text{ and }\\
CF(\tau^{-1}(\ul{L}_d)[2n], \ul{L}_{d-1})^{\vee} \times CF(\ul{L}_{d-2},\ul{L}_{d-1}) \times \dots \times CF(\ul{L}_0,\ul{L}_1) \to CF(\tau^{-1}(\ul{L}_d)[2n], \ul{L}_0)^{\vee} \label{eq:dualize}
\end{align}
where \eqref{eq:dualize} is obtained by dualizing $CF(\tau^{-1}(\ul{L}_d)[2n], \ul{L}_{d-1})$ and $CF(\tau^{-1}(\ul{L}_d)[2n], \ul{L}_0)$ in the structural map
\begin{align}
CF(\ul{L}_{d-2},\ul{L}_{d-1}) \times \dots \times CF(\ul{L}_0,\ul{L}_1) \times CF(\tau^{-1}(\ul{L}_d)[2n], \ul{L}_0) \to CF(\tau^{-1}(\ul{L}_d)[2n], \ul{L}_{d-1})
\end{align}
The right  side of Figure \ref{fig:GlobalMonodromy} gives a schematic indication of why such a bijection exists, in a simple case in which the Lagrangians are pairwise distinct (and the $A_{\infty}$-products are governed by the same set of holomorphic curves projecting to the unique triangle in the base). 
Together, these claims imply that $CF(\tau^{-1}(\ul{L})[2n],-)^\vee$ is isomorphic as an $A_{\infty}$ right module to $CF(-,\ul{L})$.
In other words, on the object level, the global monodromy $\tau[-2n]$ sends 
$\ul{L}$ to $\tau(\ul{L})[-2n]$, whose Yoneda image is in turn isomorphic to $CF(\ul{L},-)^\vee$. This is the first piece of geometric information that enters into Seidel's argument.
\end{proof}

\begin{remark}\label{r:claim}
There is an embedding from an appropriate Fukaya-Seidel category to the cylindrical version $\cFS^{cyl,n}$ (see Proposition \ref{p:embedding}).
 Claim \ref{p:Serre} is only used to prove that this embedding is essentially surjective 
 for type $A$-Milnor fiber (see Proposition \ref{p:generation}), and to compare certain Lagrangian tuples with certain modules over the extended arc algebra in Section \ref{s:Dictionary}.
 It is not needed to derive any of the results mentioned in Section \ref{s:intro}.
\end{remark}

\section{Comparing Fukaya-Seidel categories}\label{s:FScategories}

In this section, we discuss the relation between $\cFS^{cyl,n}(\pi_E)$ and the Fukaya-Seidel category of the Lefschetz fibration on $\Hilb^n(E)\setminus D_r$ induced by $\pi_E$ (we recall that this is a `weak' Lefschetz fibration, in the sense that it may have critical points at infinity, but it still has a Fukaya-Seidel category of Lagrangians proper over the base, as constructed in \cite{SeidelBook}).
This allows us to translate our subsequent study of $\cFS^{cyl,n}(\pi_E)$ back to the usual Fukaya-Seidel category.

\subsection{A directed subcategory}\label{ss:directedsub}

Let $\bc_1,\dots, \bc_m$ be the set of critical values of $\pi_E$.
By applying a diffeomorphism of $E$ covering a compactly supported diffeomorphism of $\bH^\circ$
and using push forward $J_{E^\rceil}$, $\omega_{E^\rceil}$, etc, we assume that $\bc_k:=k+\sqrt{-1} \in \bH^\circ$ for $k=1,\dots,m$.
For simplicity, we assume that there is exactly one critical point lying above a critical value.
We also assume that the symplectic parallel transport is well-defined everywhere.
This holds for type $A$-Milnor fiber (see \cite{Khovanov-Seidel} or the subsequent discussion in Section \ref{s:TypeA}).
In fact we only apply parallel transport to Lagrangians that are proper over $\bH^\circ$
and one can avoid this hypothesis at the cost of having more notations.

A matching path $\gamma:[0,1] \to \bH^{\circ}$ of $\pi_E: E \to \bH^{\circ}$ is a smooth path from $\bc_a$ to $\bc_b$, for some $a, b \in \{1,\dots,m\} $ with $a \neq b$,
such that $\gamma(t)$ is not a critical value of $\pi_E$ for all $t \neq 0,1$, 
and such that the vanishing cycles from the critical points lying above $\bc_a$ to $\bc_b$ match up under symplectic parallel transport along $\gamma$ to give a Lagrangian matching sphere $L_\gamma$ in $E$.
A thimble path $\gamma:[0,1] \to \bH$ of $\pi_E: E \to \bH^{\circ}$ is a smooth path from $\bc_a$, 
for some $a \in \{1,\dots,m\} $, to a point on the real line such that $\gamma(t) \in \bH^{\circ}\setminus \{\bc_b|b=1,\dots,m\}$ for $t \neq 0,1$, and the symplectic parallel transport of the vanishing cycle from $\bc_a$
gives a Lagrangian disc (thimble) $L_{\gamma}$ in $E$.

\begin{definition}\label{d:admissibletuple}
For an $n$-tuple $\Gamma=\{\gamma_1,\dots,\gamma_n\}$ of pairwise disjoint embedded curves in $\bH^{\circ}$ such that
each curve is either a matching path or a thimble path of $\pi_E: E \to \bH^{\circ}$, we can define the corresponding $n$-tuple of Lagrangians
\begin{align}
 \ul{L}_{\Gamma}:=\{L_{\gamma_1}, \dots, L_{\gamma_n}\} \in \cL^{cyl,n}
\end{align}
In this case, we call $\Gamma$ an {\it admissible tuple}. 
\end{definition}

For $r \in \RR$ and $k \in \{1,\dots, m\}$, let $l_{r,k}$ be the straight line joining $r$ and $\bc_k$. 
Every $l_{r,k}$ is a thimble path.

\begin{figure}[h]
 \includegraphics{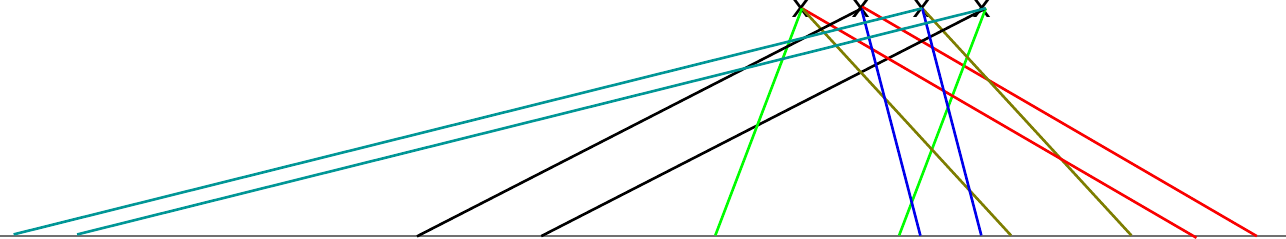}
 \caption{The case when $s=4$ and $n=2$}\label{fig:Thimbles}
\end{figure}

Let $\cI$ be the set of cardinality $n$ subsets of $\{1,\dots,m\}$.
We define a partial ordering on $\cI$, called the Bruhat order, as follows:
For $I_0,I_1 \in \cI$, $I_0 \le I_1$ if and only if there is a bijection $f: I_0 \to I_1$ such that $x\le f(x)$ for all $x \in I_0$.
Strict inequality $I_0<I_1$ is defined by $I_0 \le I_1$ and $I_0 \neq I_1$.

Let $f:\cI \to \cI_{\aff}$ (see Section \ref{sss:FloerCochain}) be a function such that
\begin{align}
 &f(I_0) \cap f(I_1) = \emptyset \text{ for all }I_0 \neq I_1 \in \cI \\
 &f(I_0) > f(I_1) \text{ if }I_0 < I_1  \label{eq:lambdaThimbleOrdering}\\
 &\length(f(I))=m
\end{align}
where $\length([a,b]):=b-a$.

Given $f$ and $I=\{i_1< \dots < i_n\} \in \cI$, we define
\begin{align}
 \gamma^{I,k}:=l_{\min(f(I))+i_k, i_k}
\end{align}
so $\Gamma^I:=\{\gamma^{I,k}|k=1,\dots,n\}$ is a collection of parallel lines in $\bH^\circ$.
For generic $f$, no three pairwise distinct lines in $\cup_{I \in \cI} \Gamma^I$ intersect at the same point in $\{z \in \bH^\circ |im(z)<1\}$ (see Figure \ref{fig:Thimbles}).

Let $\ul{T}^I:=\ul{L}_{\Gamma^I} \in \cL^{cyl,n}$.
For a generic perturbation of $\omega_E$ inside a compact subset, which changes the symplectic connection but keeps the symplectic 
Lefschetz fibration structure, we have
\begin{align}
 L_{\gamma^{I,k}} \pitchfork  L_{\gamma^{I',k'}} \text{ for all }(I,k) \neq (I',k').
\end{align}
Moreover, by applying a diffeomorphism of $E$ covering a compactly supported diffeomorphism of $\bH^\circ$ again (and using the push-forward $J_{E^\rceil}$, $\omega_{E^\rceil}$, etc), we can assume that there exists $0< \eta <1$ such that
$\pi_E$ is symplectically locally trivial in $\pi_E^{-1}(\{z \in \bH^\circ |im(z)<1-\eta\})$ and for any $\gamma^{I,k} \neq \gamma^{I',k'} \in \Gamma$,
we have $\gamma^{I,k} \cap \gamma^{I',k'} \subset \{z \in \bH^\circ |im(z)<1-\eta \text{ or }im(z)=1\}$.
We are interested in the subcategory of $\cFS^{cyl,n}(\pi_E)$ that is (split-)generated by $\{\ul{T}^I|I \in \cI\}$.

\begin{lemma}\label{l:CohVanishing}
 Let $I_0, I_1 \in \cI$. Then $HF(\ul{T}^{I_0},\ul{T}^{I_1}) \neq 0$ only if $I_0 < I_1$.
 Moreover, if $I_0=I_1$, then $HF(\ul{T}^{I_0},\ul{T}^{I_1})$ is generated by the identity element.
\end{lemma}

\begin{proof}
 For the first statement, it suffices to note that if $I_0 \nless I_1$ then we can find an isotopy of thimbles $\ul{T}^{I_0}_t \in \cL^{cyl,n}$
 such that $\ul{T}^{I_0}_0=\ul{T}^{I_0}$, $\lambda_{\ul{T}^{I_0}_1} > \lambda_{\ul{T}^{I_1}}$
 and $\Sym(\ul{T}^{I_0}_1) \cap \Sym(\ul{T}^{I_1})=\emptyset$ (see Figure \ref{fig:NoIntersections}).
 It implies that $CF(\ul{T}_1^{I_0},\ul{T}^{I_1})=HF(\ul{T}_1^{I_0},\ul{T}^{I_1})=0$, but by Lemma \ref{l:HamInvariance}, $\ul{T}_1^{I_0}$ is quasi-isomorphic to $\ul{T}^{I_0}$ as objects in $\cFS^{cyl,n}(\pi_E)$, 
 so the first statement follows.
 
 For the second one, when  $I_0=I_1=I$, we can find an isotopy of thimbles $\ul{T}^{I}_t \in \cL^{cyl,n}$
 such that $\ul{T}^{I}_0=\ul{T}^{I}$, $\lambda_{\ul{T}^{I}_1} > \lambda_{\ul{T}^{I}}$
 and $\Sym(\ul{T}^{I}_1) \cap \Sym(\ul{T}^{I})$ is a singleton.
 By Lemma \ref{l:HamInvariance} again, it implies that $HF(\ul{T}^{I},\ul{T}^{I})$ has rank one, so is generated by the cohomological unit.
\end{proof}

\begin{figure}[ht]
 \includegraphics{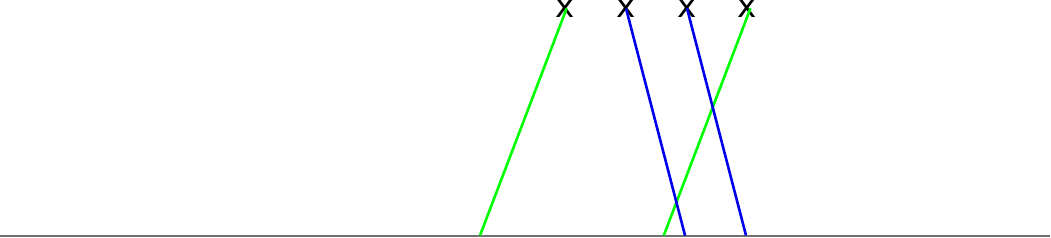}
 \caption{$\Sym(\ul{T}^{I_0}_1) \cap \Sym(\ul{T}^{I_1})=\emptyset$}\label{fig:NoIntersections}
\end{figure}


\begin{lemma}\label{l:directedSubCat}
 Let $I_0 < I_1 < \dots < I_d$. 
 For the pairs $(\ul{T}^{I_0},\ul{T}^{I_d})$, and $(\ul{T}^{I_{j-1}},\ul{T}^{I_j})$ for $j=1,\dots,d$, we can choose  Floer cochain data 
 $(A,H,J)$ such that $A=0$ and $H \equiv 0$.
 Moreover, for $\ul{x}_0 \in \cX(\ul{T}^{I_0},\ul{T}^{I_d})$ and $\ul{x}_j \in \cX(\ul{T}^{I_{j-1}},\ul{T}^{I_j})$ for $j=1,\dots,d$,
 we can also choose the Floer data $(A_S,K,J)$ such that $A_S=0$ and $K\equiv 0$.
 In consequence, when $ \cR^{d+1,h}(\ul{x}_0;\ul{x}_d, \dots, \ul{x}_1)$ has virtual dimension zero,
 its regularity can be achieved by generic $J$.
\end{lemma}

\begin{proof}
 First note that, if $i<j<k$ and $\ul{x}=\{x_1,\dots,x_n\} \in \Sym(\ul{T}^{I_i}) \cap \Sym(\ul{T}^{I_j}) \cap \Sym(\ul{T}^{I_k})$, then since there is no triple intersection in 
 $\{z \in \bH^\circ |im(z)<1\}$,
 it is necessary that for all $t=1,\dots, n$, $\pi_E(x_t)=\bc_{l_t}$ for some $l_t \in \{1,\dots,m\}$.
 This in turn implies that $i=j=k$, which is a contradiction.
 Therefore, if $i<j<k$ then $\Sym(\ul{T}^{I_i}) \cap \Sym(\ul{T}^{I_j}) \cap \Sym(\ul{T}^{I_k}) = \emptyset$.

 
 We first discuss how to achieve regularity of elements in $\cR^{d+1,h}(\ul{x}_0;\ul{x}_d, \dots, \ul{x}_1)_{pre}$.
 Let $u:S \to \Hilb^n(E)$ be an element in this space.
 Let $v:\Sigma \to E$ be the map associated to $u$ by the tautological correspondence.
 Let $\Sigma_1,\Sigma_2,\dots,\Sigma_m$ be the connected components of $\Sigma$ and $v_j=v|_{\Sigma_j}$.
By reordering $\Sigma_1, \dots, \Sigma_m$ if necessary, we can assume that 
there is $0 \le a \le m$ such that
$\pi_E \circ v_i$ is a constant for $i \le a$, and is a non-constant map otherwise.

By the boundary conditions, 
 $\pi_\Sigma|_{\Sigma_i}$ must have degree $1$ for $i \le a$, so $\Sigma_i=S$ is a disc with $d+1$ boundary punctures.
We want to discuss the Fredholm operator associated to $v_i:S \to E$ for $i \le a$.
Let assume $a \neq 0$ and consider $v_1$.
 We denote the Lagrangian boundary label on $\partial \Sigma_1$ by $L_0,\dots,L_d$ so that $L_j \in \ul{T}^{I_j}$ for $j=0,\dots,d$.
The map $\pi_E \circ v_1$ is either a constant map to a point in $\{z \in \bH^\circ |im(z)<1-\eta\}$
 or to a point $\bc_k$ for some $k$.

We first suppose that $\pi_E \circ v_1$ is a constant map to a point in $\{z \in \bH^\circ |im(z)<1-\eta\}$.
 Then we must have $d=1$ because no three pairwise distinct lines in $\Gamma$ intersect at the same point in $\{z \in \bH^\circ |im(z)<1\}$.
 If we view $\pi_E \circ v_1$ as a holomorphic map with boundary on $\pi_E(L_j)$ and both asymptotic conditions are given by uniform convergence to the intersection point, then
 $\pi_E \circ v_1$ is a regular rigid solution because the input and output are the same (so have the same gradings).
The cokernel of the Fredholm operator $D_{v_1}$ of $v_1$ sits inside a short exact sequence
\begin{align}
\coker(D_{\fiber}) \to \coker(D_{v_1}) \to \coker(D_{\pi_E \circ v_1}) \to 0
\end{align}
where $D_{\pi_E \circ v_1}$ is the Fredholm operator for $\pi_E \circ v_1$ and $D_{\fiber}$ is the  Fredholm operator of $v_1$ viewed as a map to the fiber.
By regularity of $\pi_E \circ v_1$, we have $\coker(D_{\pi_E \circ v_1})=0$.
On the other hand, depending on the virtual dimension of $v_1$,
either $\coker(D_{\fiber})$ can also be made $0$ by a generic choice of $J$, or $v_1$ does not exist for generic $J$.
In the former case, $D_{v_1}$ is surjective.
 
 Next, we consider the case that $\pi_E \circ v_1$ is a constant map to a point $\bc_k$ for some $k$.
 In this case, $d$ is not necessarily $1$ but, by \eqref{eq:lambdaThimbleOrdering}, we have
  \begin{align}
 \lambda_{L_0}>\dots>\lambda_{L_d}. \label{eq:lorder}
 \end{align}
 We can assume that $L_i \cap L_j=p$ for all $i,j$, where $p$ is the critical point lying above $\bc_k$.
 Let $B_p$ be a Darboux ball centered at $p$.
 We can assume that $J$ is integrable near $p$ so that $B_p$ can be identified with a ball in $\CC^2$.
 Moreover, by \eqref{eq:lorder}, we can assume that $T_pL_i \cap T_pB_p$ are pairwise transversally intersecting Lagrangian planes in $T_pB_p$
 with strictly decreasing K\"ahler angles.
 More explicitly, a local model is given by
 \begin{align}
  &\pi_E|_{B_p}(z_1,z_2)=z_1z_2, \quad z_1,z_2 \in \CC \\
  &L_i \cap B_p=\{(z_1,z_2)=(re^{i\theta_i+t},re^{i\theta_i-t}) \in B_p|r \ge 0, t \in [0,2\pi]\} \text{ for some }\theta_i \in [0,2\pi)
 \end{align}
 and \eqref{eq:lorder} translates to $\theta_0 >\dots>\theta_d$.
 As a result, there are choices of grading functions on $\{L_i\}_{i=0}^d$ such that 
 for all $i$, the point $p$ as a generator of $CF(L_{i-1},L_i)$ has grading $0$ (when $i=0$, $CF(L_{i-1},L_i)$ should be understood as $CF(L_{0},L_d)$).
 Therefore, for a fixed $S\in \cR^{d+1}$, the moduli of solutions to the equation
 \begin{align}
 \left\{
 \begin{array}{ll}
  w:S \to E \\
  (Dw)^{0,1}=0\\
  w(\partial_j S) \subset L_j \text{ for all }j\\
  \lim_{s \to \pm \infty}w(\epsilon_j(s,\cdot))=p \text{ uniformly for all } j
 \end{array}
 \right.
 \end{align}
 has virtual dimension $0$.
 Moreover, the constant map from $S$ to $p$ is regular and rigid.
 On the other hand, $v_1$ must be the constant map from $S$ to $p$ so $v_1$ is regular.

We can now address the regularity of $u$. Recall that the key point is to show that, 
if $\eta$ is an element in an appropriate Sobolev completion of
 $\Omega^{0,1}(S,u^*T\Hilb^n(E))$ which annihilates the image of the Fredholm operator associated to $u$,
then $\eta$ vanishes identically. Since $\eta$ lies inside the kernel of the adjoint operator, $\eta$ satisfies the 
unique continuation property, so it suffices to show that $\eta$ vanishes on an open subset $G$ of $S$.

 Let $z \in S \setminus \nu(\mk(S))$ such that $u(z) \in \Conf^n(E)$.
 Let $G$ be an open neighborhood of $z$ such that $\pi_{\Sigma}^{-1}(G)$ consists of $n$ disjoint open sets $G_1,\dots,G_n$.
 There is a neighborhood $U$ of $u(z)$ that is symplectomorphic to a product $U_1 \times \dots \times U_n$, for open sets $U_i \subset E$
 satisfying $U_i \cap U_j =\emptyset$ if $i \neq j$.
Moreover, we can assume the image of $v|_{G_i}$ lies inside $U_i$, so we have
$u|_G(z)=(v|_{G_1}(z),\dots,v|_{G_n}(z))$ under the identifications between $U$ and $U_1 \times \dots \times U_n$, and between $G$ and $G_i$.
Having this product type local model for $u|_G$, we can write $\eta|_G$
as $(\eta_1,\dots,\eta_n)$, where $\eta_i \in \Omega^{0,1}(G_i,v|_{G_i}^*TU_i)$.
If $\eta \neq 0$, then at least one $\eta_i \neq 0$; relabel so $\eta_1 \neq 0$.

If $\pi_E \circ v|_{G_1}$ is not a constant map, then there is an infinitesimal deformation $Y$ of $J$
in the space of almost complex structures  satisfying \eqref{eq:FloerDataJ}
 (which only deforms in the first factor
of the product $U_1 \times \dots \times U_n$) such that 
\begin{align}
\int_G \langle \eta, Y \circ (du- X_K) \circ j_S \rangle =\int_{G_1} \langle \eta_1, Y \circ (dv_1- X_{\pi_{\Sigma}^*K})  \circ j_S \rangle \neq 0
\end{align}
where $j_S$ is the complex structure on $S$.
That contradicts the assumption that $\eta$  annihilates the image of the Fredholm operator associated to $u$.

If $\pi_E \circ v|_{G_1} $ is a constant, then we must have $G_1 \subset \Sigma_i$ for some $i \le a$.
Without loss of generality, we assume $G_1 \subset \Sigma_1$ and 
$v|_{G_1}=v_1|_{G_1}$.
From the discussion of the surjectivity of the Fredholm operator associated to $v_1$ above, 
we know that there exists an infinitesimal deformation $Y$ of $J$ supported in $G_1$, and an element $\xi$
in an appropriate Sobolev completion of $C^{\infty}(S,v_1^*TE)$ supported in $G_1$ such that
\begin{align}
\int_{G_1} \langle \eta_1, D_{v_1} \xi +\frac{1}{2} Y \circ (dv_1- X_{\pi_{\Sigma}^*K})  \circ j_S \rangle \neq 0.
\end{align}
Therefore, if we
identify $G_1$ with $G$ and 
 think of $v_1|_{G_1}^*TE=v_1|_{G_1}^*TU_1$ 
as a component of $u|_G^*TU$, then we have 
\begin{align}
\int_G \langle \eta, D_{u} \xi +\frac{1}{2}Y \circ (du- X_K) \circ j_S \rangle  \neq 0.
\end{align}
Therefore, we have $\eta=0$ and hence the regularity of $u$ can be achieved by generic $J$.

This proves the regularity of elements in $\cR^{d+1,h}(\ul{x}_0;\ul{x}_d, \dots, \ul{x}_1)_{pre}$.
The transversality  between the evaluation map $\cR^{d+1,h}(\ul{x}_0;\ul{x}_d, \dots, \ul{x}_1)_{pre} \to (\Hilb^n(E))^h$
and a pseudocycle representing the inclusion $(D_{HC})^h \to (\Hilb^n(E))^h$ can be addressed by the same reasoning.
This finishes the proof.
\end{proof}

 

By Lemma \ref{l:directedSubCat}, we can define the $A_{\infty}$ operations
\begin{align}
 CF(\ul{T}^{I_{d-1}},\ul{T}^{I_d}) \times \dots \times CF(\ul{T}^{I_0},\ul{T}^{I_1}) \to CF(\ul{T}^{I_0},\ul{T}^{I_d})
\end{align}
without introducing a Hamiltonian term $K$, whenever $I_0 < I_1 < \dots < I_d$.
In a standard way, we can extend the $A_{\infty}$ operations by adding idempotents.

\begin{corollary}\label{c:directedSubCat}
 Let $R=\oplus_{I \in \cI} \KK e_I$ and $e_I^2=e_I$.
 We have a strictly unital $A_\infty$ algebra
 \begin{align}
R \oplus (\oplus_{I_0<I_1}CF(\ul{T}^{I_0},\ul{T}^{I_1})) \label{eq:Dir1}
 \end{align}
 with unit $\sum_{I \in \cI} e_I$. 
\end{corollary}

\begin{proof}
See \cite[Section 7]{Seidel1} for how the $A_{\infty}$ structure is defined after adjoining the strict units $e_I$ of the idempotents.
\end{proof}

\begin{corollary}\label{c:directedSubCat2}
 The $A_{\infty}$ algebra \eqref{eq:Dir1} is quasi-isomorphic to $\End_{\cFS^{cyl,n}}(\oplus_{I \in \cI} \ul{T}^{I})$.
\end{corollary}

\begin{proof}
 This follows from Lemmas \ref{l:CohVanishing} and \ref{l:directedSubCat}, Corollary \ref{c:directedSubCat}
 and homological perturbation.
\end{proof}

\subsection{Fukaya-Seidel embeds into cylindrical Fukaya-Seidel}\label{ss:FSEmbed}


In this section, we show that 
\begin{align}
 \pi_{\cY_E}:=\pi_E^{[n]}|_{\cY_E}:\cY_E:=\Hilb^n(E) \setminus D_r \to \CC \label{eq:LF}
\end{align}
is a Lefschetz fibration,
and identify the  subcategory of the associated Fukaya-Seidel category $D^\pi\cFS(\pi_{\cY_E})$ generated by thimbles with the category of perfect modules over the $A_{\infty}$-algebra \eqref{eq:Dir1}. 

\begin{lemma}\label{l:ConcentratedAtCritical}
 If $f: \CC^2 \to \CC$ is a holomorphic map without critical points, then so is the induced map $f^{[n]}:\Hilb^n(\CC^2) \to \CC$.
\end{lemma}

\begin{proof}
 Without loss of generality, it suffices to show that $f^{[n]}$ is regular at a $0$-dimensional length $n$
 subscheme $z$ supported at $0$.
 We can also assume that $f(x,y)=x$.
 
 Let $z_t$ be a family of $0$-dimensional length $n$ subschemes such that $z_0=z$ and $z_t$ is a length $n$
 subscheme supported at $(t,0)$.
 Then $f^{[n]}(z_t)=nt$ so $f^{[n]}$ is regular at $0$.
\end{proof}

\begin{corollary}\label{c:CritAtCrit}
 Every critical point of $\pi_E^{[n]}: \Hilb^n(E) \to \bH^\circ$ has support lying inside the union of critical points of $\pi_E$.
\end{corollary}

\begin{proof}
 If the support of $z \in \Hilb^n(E)$ contains a point $p$ that is not a critical point of $\pi_E$, then we can apply Lemma \ref{l:ConcentratedAtCritical} locally near $p$
 to show that $z$ is not a critical point of $\pi_E^{[n]}$.
\end{proof}

If $z$ is a critical point of $\pi_E^{[n]}$, consider whether the support is a disjoint union of $n$ points or not.
We consider the former case first.
For $I \in \cI$, we use $z_I \in \Conf^n(E) \subset \Hilb^n(E)$ to denote the subscheme with support being the union of the critical points lying above $\{\bc_i|i \in I\}$.

\begin{lemma}[Proposition $2.1$ and Lemma $2.2$ of \cite{Auroux-bordered}]\label{l:HilbThimble}
For each $I \in \cI$, the point $z_I$ is a Lefschetz critical point of $\pi_E^{[n]}$ and $\Sym(\ul{T}^I)$ is a Lefschetz thimble.
\end{lemma}

\begin{proof}
 Since $z_I$ consists of pairwise distinct points, near $z_I$, $\pi_E^{[n]}$ is locally given by 
 $(u_1,v_1,\dots,u_n,v_n) \mapsto u_1^2+v_1^2+\dots u_n^2+v_n^2$.
 Therefore, it admits a Lefschetz critical point at $z_I$.
 
 Since $\{\pi_E(T^{I,k})\}_{k=1}^n$ are parallel lines, $\pi_E^{[n]}(\Sym(\ul{T}^I))$ is also a straight line so it is a Lefschetz thimble.
\end{proof}

For the other case, we have a local lemma.

\begin{lemma}\label{l:thrownAway}
 Let $\pi: \CC^2 \to \CC$ be $\pi(u,v)=u^2+v^2$. 
 For $n \ge 2$ and $z \in \Hilb^n(\CC^2)$ supported at the origin of $\CC^2$, the length of the projection of $z$ by $\pi$ is strictly less than $n$.
\end{lemma}

\begin{proof}
Let $\CC=\Spec(\CC[w])$ and $f:\CC[w] \to \CC[u,v]$ be the algebra homomorphism sending $w$ to $u^2+v^2$.
Let $I$ be a length $n$ ideal of $\CC[u,v]$ supported at the origin, which corresponds to the point $z$ in the statement.
We want to show that the ideal $f^{-1}(I)$ in $\CC[w]$ has length strictly less than $n$.

Since $\dim_\CC(\CC[u,v]/I)=n$, we know that $u^av^b \in I$ if $a,b, \ge 0$ and $a+b \ge n$.
It means that when $n$ is even (resp. odd), we have $(u^2+v^2)^{\frac{n}{2}} \in I$ (resp.  $(u^2+v^2)^{\frac{n+1}{2}} \in I$).
Therefore, $z^{\frac{n}{2}} \in f^{-1}(I)$ when $n$ is even (resp. $z^{\frac{n+1}{2}} \in f^{-1}(I)$ when $n$ is odd)
which in turn implies that $f^{-1}(I)$ has length no greater than $\frac{n+1}{2}$.

\end{proof}

\begin{corollary}\label{c:AllCrit}
 The critical points of \eqref{eq:LF} are precisely $\{z_I\}_{I \in \cI}$ and they are Lefschetz. 
\end{corollary}

\begin{proof}
By Lemma \ref{l:thrownAway}, if the support of $z$ has multiplicity $2$ at some critical point of $\pi_E$, then $z \in D_r$ which is not in $\cY_E = \Hilb^n(E)\setminus D_r$.
The result then follows from Corollary \ref{c:CritAtCrit} and Lemma \ref{l:HilbThimble}.
\end{proof}

\begin{remark}\label{r:horBoundaryNonTrivial}
 Even though \eqref{eq:LF} has only Lefschetz critical points, it is in general not symplectically locally trivial
 near the horizontal boundary, because there are critical points of $ \pi_E^{[n]}:\Hilb^n(E) \to \CC$ lying in $D_r$ (when $n>1$ so $D_r \neq \emptyset$).
\end{remark}

In light of Remark \ref{r:horBoundaryNonTrivial}, we should clarify what we mean by $\cFS(\pi_\cY)$.
Objects of the Fukaya-Seidel category $\cFS(\pi_\cY)$ 
are restricted to be Lefschetz thimbles $\Sym(\ul{T}^{I})$ for $I \in \cI$.
As in Corollary \eqref{c:directedSubCat2} (see \cite[Section 7]{Seidel1}), the endomorphism $A_{\infty}$ algebra of the direct sum of the objects is defined to be
\begin{align}
R \oplus (\oplus_{I_0<I_1}CF(\Sym(\ul{T}^{I_0}),\Sym(\ul{T}^{I_1}))) \label{eq:Dir2} 
\end{align}
where the Floer cochains are taken in $\Hilb^n(E) \setminus D_r$ and no Hamiltonian perturbation is put on the Floer equations defining the  $A_{\infty}$ structure.
As in \cite{Seidel1}, this definition is quasi-isomorphic to the directed subcategory of the vanishing cycles in the distinguished fiber.

\begin{proposition}\label{p:embedding}
There is a quasi-isomorphism between \eqref{eq:Dir1} and \eqref{eq:Dir2}.
As a result, there is a cohomologically full and faithful embedding $D^{\pi}\cFS(\pi_\cY) \to D^{\pi}\cFS^{cyl,n}(\pi_E)$.
\end{proposition}

\begin{proof}
There is an obvious bijective correspondence of objects and generators between \eqref{eq:Dir1} and \eqref{eq:Dir2}.
Note that all pseudo-holomorphic maps involved in defining the $\mu^d$ operations in \eqref{eq:Dir1} are contained in $\Hilb^n(E) \setminus D_r$.

For the definition of the $\mu^d$ operations in \eqref{eq:Dir2}, we can use moduli of pseudo-holomorphic maps with no 
Hamiltonian perturbation term as in \eqref{eq:Dir1} (see Lemma \ref{l:directedSubCat}) but we need to use the domain moduli $\cR^{d+1}$ instead of $\cR^{d+1,h}$.
It means that we equip $\cup_d \overline{\cR}^{d+1}$ with a consistent choice of domain-dependent almost complex structures that are 
equal to $J_E^{[n]}$ outside a compact subset of $\Hilb^n(E) \setminus D_r$ 
and generic inside the compact subset, and we count the corresponding moduli of maps to define the $\mu^d$ operations in \eqref{eq:Dir2}.

Given the difference of domain moduli for \eqref{eq:Dir2} and in \eqref{eq:Dir1}, we cannot directly compare the $\mu^d$ operations.
The standard method to circumvent this is to use the `total Fukaya category' trick from \cite[Section 10a]{SeidelBook}.

Briefly, one can show that  
$\mu^d$ operations in \eqref{eq:Dir1} yield a category quasi-isomorphic to another one in which the operations $\tilde{\mu}^d$ are defined by the domain moduli $\cup_{d,h} \overline{\cR}^{d+1,h}$
but with the domain dependent $J$ chosen \emph{independent} of the interior marked points, equal to $J_E^{[n]}$ outside a compact subset of $\Hilb^n(E) \setminus (D_{HC} \cup D_r)$ 
and generic inside a compact subset.
Moreover, for generic $J$, we can assume the the universal evaluation map to $(\Hilb^n(E))^h$ is transversal to $D_{HC}^h$.
With these data, 
$\tilde{\mu}^d$ is defined by counting the corresponding moduli of maps $u$ such that $u(z) \in D_{HC}$ for  $z \in \mk(S)$.
In this modification, the interior marked points are merely decorative, and one can canonically identify the  moduli spaces with the ones 
defining the $\mu^d$ operations in \eqref{eq:Dir2}.
\end{proof}

 It is natural to ask when the thimbles $\{\ul{T}^{I}\}_{I \in \cI}$ split-generate $D^{\pi}\cFS^{cyl,n}(\pi_E)$.
We next show that this holds when $\pi_E$ is the standard Lefschetz fibration on the type-$A$ Milnor fibre.

\section{Type $A$ geometry}\label{s:TypeA}

In this section, we  apply the results in Sections \ref{s:cylindricalFS}--\ref{s:FScategories} to the case in which $E=A_{m-1}$ is a  type A Milnor fiber. 
When $2n \le m$, the corresponding $\cY_E = \Hilb^n(E) \setminus D_r$ is the generic fiber of the adjoint quotient map restricted to
a nilpotent (or Slodowy) slice associated to a nilpotent with two Jordan blocks, as studied 
in \cite{SeidelSmith, Manolescu-link} in the context of symplectic Khovanov cohomology. 
At the end of this section, we show that if Claim \ref{p:Serre} holds, then the embedding in 
Proposition \ref{p:embedding} is essentially surjective in this case.

\subsection{$A_{m-1}$ Milnor fibers}\label{ss:AmMilnorFiber}

We recall the symplectic geometry of type $A$ Milnor fibers, with an emphasis on verifying the assumptions made in the setup in Section \ref{s:cylindricalFS}.
The following model is a slight modification of the one in \cite[Section 7]{EvansThesis}.
Let $m>1$ be an integer.
Let
\begin{align}
 X:=\CC \times \CP^1 \times \prod_{i=1}^{m} \CP^1_i
\end{align}
with coordinates $(x, [Y_0:Y_1], [a_1,b_1], \dots, [a_{m},b_{m}])$.
Let $M$ be the subvariety given by the equations
\begin{align}
 a_iY_0=b_iY_1(x-i) \text{ for }i=1,\dots,m
\end{align}
Let $\pi_M: M \to \CC$ be the projection to the $x$ coordinate.
For $x \neq 1,\dots,m$, the fiber at $x$ is given by
\begin{align}
 P_x:=\pi_E^{-1}(x)=\{(x,[Y_0:Y_1], [Y_1(x-1):Y_0], \dots, [Y_1(x-m):Y_0])\}
\end{align}
which is a smooth rational curve.
For $x=i \in \{1,\dots,m\}$, 
\begin{align}
 P_i:=\pi_E^{-1}(i)=&\{(i,[Y_0:Y_1], [Y_1(i-1):Y_0], \dots, [Y_1(i-m):Y_0]) | Y_0 \neq 0\} \\
     &\cup \{(i,[0:1], [1:0], \dots,[a_i:b_i], \dots, [1:0])|[a_i:b_i] \in \CP^1_i\}
\end{align}
is a union of two irreducible smooth rational curves.
One can check that $\pi_M$ is a Lefschetz fibration. 
Consider the following sections
\begin{align}
D_{\infty}:=&\{(x,[1:0],[0:1],\dots,[0:1]) |x \in \CC\} \\
D_{0}:=&\{(x,[0:1],[1:0],\dots,[1:0]) |x \in \CC\}
\end{align}
and define $D:=D_{0} \cup D_{\infty}$.

\begin{lemma}[cf. Lemma 7.1 of \cite{EvansThesis}]\label{l:typeA}
 $M \setminus D$ is biholomorphic to $\{a^2+b^2+(c-1)\dots(c-m)=0\}$.
\end{lemma}

\begin{proof}
For a dense subset, the identification is given by
\begin{align}
 (a,b,c) \mapsto (c,[a+\sqrt{-1}b:1], [c-1:a+\sqrt{-1}b], \dots, [c-m:a+\sqrt{-1}b])
\end{align}
We leave the rest to the reader.
\end{proof}

We call a smooth affine variety of the form $\{a^2+b^2+(c-c_1)\dots(c-c_m)=0\}$ with pairwise distinct 
$c_1,\dots,c_m \in \CC$ an $A_{m-1}$ Milnor fiber.
In particular, $M \setminus D$ is an $A_{m-1}$ Milnor fiber by Lemma \ref{l:typeA}.

\begin{remark}
By projecting to the $c$ co-ordinate, $\{a^2+b^2+(c-c_1)\dots(c-c_m)=0\}\subset \CC^3$ is the total space of a Lefschetz fibration with general fibre $\CC^*$ and with $m$ nodal fibers (over the $c_i$). Since we are primarily interested in the symplectic geometry and not complex geometry of $M\backslash D$, we will also refer to any such Lefschetz fibration as an $A_{m-1}$-Milnor fiber; in particular, the restriction of this Lefschetz fibration of $M\backslash D$ to a disc in the $c$-plane containing all the critical values is a Milnor fiber. 
\end{remark}

\begin{lemma}[cf. Lemma 7.2 of \cite{EvansThesis}]\label{l:omegaOrthogonal}
 Let $\omega_X$ be a product symplectic form on $X$ that tames the complex structure.
 Then $\omega_M:=\omega_X|_M$ is a symplectic form and $P_x$ is $\omega_M$-orthogonal to $D$ for all $x$.
\end{lemma}

\begin{proof}
 Both assertions are clear. For the second one, observe that $TD$ and $TP_x$ lie in the first and second factor of $T\CC \oplus (T(\CP^1)^{m+1}) =TX$, respectively.
\end{proof}

Let $\omega_X$ be the standard symplectic form on $X$ and define $\omega_M:=\omega_X|_M$.
For $R>0$, let $B_R \subset \CC$ be the open disc of radius $R$ and $M_R=\pi_M^{-1}(B_R)$.
Let $\pi_{M_R}: \frac{1}{R} \pi_M|_{M_R}:M_R \to B_1$ which is still a holomorphic Lefschetz fibration, and a symplectic Lefschetz fibration
with respect to $\omega_{M_R}:=\omega_M|_{M_R}$.
We equip the base unit disc $B_1$ with the hyperbolic area form $\omega_{hyp}$.

\begin{lemma}\label{l:omegaR}
 Let $R>m$.
 There exists a symplectic form $\omega_R$ on $M_{R}$ and a compact subset $C_B \subset B_1$ such that
 \begin{align}
  &\omega_R \text{ tames the complex structure}  \label{eq:adjOmegaTame}\\
  &\omega_R \text{ is symplectically locally trivial in }\pi_{M_R}^{-1}(B_1 \setminus C_B) \label{eq:adjOmegaTrivial} \\
  &D \cap M_R \text{ is $\omega_R$-symplectic-orthogonal to } \pi_{M_R}^{-1}(x) \text{ for all }x \in B_1 \setminus C_B \label{eq:adjOmegaOrtho}
 \end{align}
\end{lemma}

\begin{proof}

For $\epsilon>0$ small, we have a holomorphic embedding
\begin{align}
 \Phi: (B_{1} \setminus B_{1-\epsilon}) \times \CP^1 &\to M_R \\
 (x,[Y_0:Y_1]) &\mapsto (Rx,[Y_0:Y_1], [Y_1(x-1):Y_0], \dots, [Y_1(x-m):Y_0]) \nonumber
\end{align}
which is compatible with the projection to $B_{1} \setminus B_{1-\epsilon}$.

Let $g:(1-\epsilon,1] \to (1-\epsilon,1]$ be an increasing function such that there exists $\delta>0$ so that 
$g(r)=r$ for $r$ near $1-\epsilon$, $g$ is strictly increasing when $r \in (1-\epsilon,1-\delta)$, and $g(r)=1$ for $r \ge 1-\delta$.
It induces a smooth map $\phi_g: M_R \to M_R$ such that 
\begin{align}
\phi_g(z)&:=z \text{ for } z \notin Im(\Phi) \\
\phi_g(\Phi(re^{\sqrt{-1}\theta},y))&:=\Phi(g(r)e^{\sqrt{-1}\theta},y) \text{, where }(x,[Y_0:Y_1])=(re^{\sqrt{-1}\theta},y)
\end{align}
Then $\phi_g^*\omega_{M_R}$ is a closed $2$-form, non-degenerate in $\pi_{M_R}^{-1}(B_{1-\epsilon})$ and fiberwise symplectic.

A direct calculation shows that there is $A>0$ such that for all $a>A$,
$\frac{1}{a}\phi_g^*\omega_{M_R} + \pi_{M_R}^* \omega_{hyp}$ is a symplectic form that satisfies all the conditions \eqref{eq:adjOmegaTame}, \eqref{eq:adjOmegaTrivial}
and \eqref{eq:adjOmegaOrtho} for $C_B:=\overline{B_{1-\delta}}$.
\end{proof}

Now, we give a dictionary to the setup in Section \ref{sss:targetspace}.
Let $f:\ol{B}_1 \to \bH$ be a hyperbolic isometry.
Let 
\begin{align}
 E^{\rceil}&:=\ol{M}_R 
  \qquad \pi_{E^{\rceil}}:= f \circ \frac{1}{R}\pi_M|_{\ol{M}_R} \qquad D_E:=D \cap \ol{M}_R  \label{eq:dictionary}\\
 \overline{E}&:=M_R  \qquad \pi_{\overline{E}}:=f \circ \pi_{M_R} \qquad \omega_{\ol{E}}:=\omega_R \qquad C_{\bH}:=f(C_B)  \nonumber
\end{align}
It is straight forward to check that the assumptions made in Section \ref{sss:targetspace} are satisfied.
Most notably, $\omega_{\ol{E}}$ tames the complex structure, $\pi_{\ol{E}}|_{\pi_{\ol{E}}^{-1}(\bH^\circ \setminus C_\bH)}$ is symplectically locally trivial
and every holomorphic map $\mathbb{CP}^1 \to E^{\rceil}$ has positive algebraic intersection number with $D_E$.

\begin{remark}
 There is an additional feature in this setting that is not assumed in Section \ref{sss:targetspace}, namely, for all $x \in \pi_{\ol{E}}|_{\pi_{\ol{E}}^{-1}(\bH^\circ \setminus C_\bH)}$, the
 complex structure at $x$ respects the symplectic decomposition $T_x\ol{E}=T^v_x\ol{E} \oplus T^h_x\ol{E}$.
\end{remark}

\subsection{Nilpotent slices}\label{ss:nilSlice}

Let $m,n \in \NN$ such that $2n \le m$. Let $G=GL_{m}(\CC)$ and $ \mathfrak{g}$ be its Lie algebra.
The adjoint quotient map $\chi:\kg \to \mathfrak{h}/W=\Sym^{m}(\CC)$ take an element $A \in \mathfrak{g}$ to the coefficients of its characteristic polynomial, i.e. the symmetric functions of the  
eigenvalues (with multiplicities) of $A$.
The set of regular values is given by $\Conf^{m}(\CC) \subset \Sym^{m}(\CC)=\CC^{m}$.
Let $S_{n,m} \subset \kg$ be the affine subspace consisting of matrices
\[ A= \left( \begin{array}{cccccccccccccccccccccccc}
a_1   & 1 &         &  & b_1  &    &  & \\
a_2   & 0 & 1       &  & b_2  &    &  & \\
\dots &   & \dots   &   & \dots &  &    \\
a_{n-1}& &         & 1 & b_{n-1}&  &  & \\
a_n    & &         & 0 & b_n    &  &  & \\
0      &   &      &    & d_1     & 1  &   & \\
\dots &     &     &    & d_2    &  0 &1    & \\
0     &  &       &    & \dots & & &\dots & \\
c_1 &   &       &     & \dots & & &\dots & \\
\dots & &      &      & d_{m-n-1} & &     & & 1\\
c_n   & &       &     & d_{m-n}& &  & &0
\end{array} 
\right)\]
such that $a_i,b_i,c_i,d_i \in \CC$ for all $i$.
Note that, for the first column of $A$ to be well-defined, we used $2n \le m$.
The affine subspace $S_{n,m}$ is a nilpotent (Slodowy) slice of the nilpotent element of Jordan type $(n,m-n)$.
The restriction $\chi|_{S_{n,m}}: S_{n,m} \to \Sym^{m}(\CC)$ admits a simultaneous resolution by Grothendieck (see \cite{Slodowy}, \cite{CG97}), and $\chi|_{S_{n,m}}^{-1}(\Conf^{m}(\CC)) \to \Conf^{m}(\CC)$ is a differentiable fiber bundle.
For $ \tau \in \Conf^{m}(\CC)$, we define
\begin{align}
\cY_{n,\tau}:= \chi|_{S_{n,m}}^{-1}(\tau)
\end{align}

\begin{example}
 When $n=1$,
 \begin{align}
  \cY_{1,\tau}=\{(b,c,z) \in \CC^3|bc+P_{\tau}(z)=0\}
 \end{align}
where $P_{\tau}(t)$ is the degree $m$ monic polynomial with roots given by elements in $\tau$.
Note that, if $\tau=\{1,\dots,m\}$, then $\cY_{1,\tau}$ is biholomorphic to $\{a^2+b^2+(c-1)\dots(c-m)=0\}$ (cf. Lemma \ref{l:typeA}).

\end{example}

We denote the projection to the $z$ co-ordinate by $\pi_{1,\tau}: \cY_{1,\tau} \to \CC$.

\begin{lemma}[\cite{Manolescu06}]\label{l:Manolescu}
 When $2n \le m$, there is a holomorphic open embedding $j:\cY_{n,\tau} \to \Hilb^{n}(\cY_{1,\tau})$.
 The complement of the image is the relative Hilbert scheme of the projection $\pi_{1,\tau}$, i.e. the subschemes whose projection to $\CC_z$ does not have length $n$.
\end{lemma}

\begin{proof}
 Although Manolescu \cite{Manolescu06} only considers the case $m$ even, the same proof goes through in general.
 More precisely, for $A \in S_{n,m}$, it is easy to check that
 \begin{align}
  \det(tI-A)=A(t)D(t)-B(t)C(t)
 \end{align}
where 
\begin{align*}
 &A(t)=t^n-a_1t^{n-1}+\dots+(-1)^na_n, \\
 &B(t)=b_1t^{n-1}-b_2t^{n-2}+ \dots+(-1)^{n-1}b_n, \\
 &C(t)=c_1t^{n-1}-c_2t^{n-2}+ \dots+(-1)^{n-1}c_n, \\
 &D(t)=t^{m-n}-d_1t^{m-n-1}+\dots+(-1)^{m-n}d_{m-n}
\end{align*}
Points in $\cY_{n,\tau}$ are identified with 4-tuples of polynomials $(A(t),B(t),C(t),D(t))$ in $\CC[t]$
such that $A(t)D(t)-B(t)C(t)=P_{\tau}(t)$.
Points in $\Hilb^{n}(\cY_{1,\tau})$ are identified with ideals $\cI$ in $\cO:=\CC[b,c,z]/(bc+P_{\tau}(z))$ such that $\dim \cO/\cI=n$. 
The holomorphic embedding $j:\cY_{n,\tau} \to \Hilb^{n}(\cY_{1,\tau})$ is given by 
\begin{align}
 j(A(t),B(t),C(t),D(t))=\{Q(b,c,z) \ | \ A(t) \text{ divides } Q(B(t),C(t),t)\}
\end{align}

\end{proof}


Let $\tau=\{1,\dots,m\}$ and identify $\cY_{1,\tau}$ and $\pi_{1,\tau}$ with $M$ and $\pi_M$ from Section \ref{ss:AmMilnorFiber}, respectively.
For $E^\rceil=\ol{M}_R$ as in the dictionary \eqref{eq:dictionary}, we know that, by Lemma \ref{l:Manolescu}, 
$\cY_{n,m}:=\Hilb^n(E)\setminus D_r$ is an open subset of $\cY_{n,\tau}$ when $2n \le m$.
Moreover, $\cY_{n,m}$ exhausts $\cY_{n,\tau}$ as $R$ goes to infinity.

\begin{lemma}
When $2n \le m$, the map $\pi_{1,\tau}^{[n]}|_{j(\cY_{n,\tau})}$ is given by sending a matrix $A \in \cY_{n,\tau}$ to the top left entry $a_1$.
\end{lemma}

\begin{proof}
By \cite[Remark 2.8]{Manolescu06}, the composition
\begin{align}
 \cY_{n,\tau} \hookrightarrow \Hilb^{n}(\cY_{1,\tau}) \xrightarrow{\text{$\pi_{HC}$}} \Sym^n(\cY_{1,\tau}) 
\xrightarrow{\Sym^n(\pi_E)} \Sym^n(\CC)
\end{align}
is given by the roots of $A(t)$ (in the notation of the proof of Lemma \ref{l:Manolescu}).
Therefore the map $\pi_{1,\tau}^{[n]} \circ j: \cY_{n,\tau} \to \CC$ is the sum of the roots of $A(t)$, which is $a_1$.

\end{proof}

\subsection{Sliding invariance}

If $\ul{L}_{0}$ and $\ul{L}_1$ are not connected by a path $\ul{L}_t \in \cL^{cyl,n}$, then 
the quasi-isomorphism types of $\ul{L}_0$ and $\ul{L}_1$ in $\cFS^{cyl,n}(\pi_E)$ are in general different.
However, we present three special cases where the quasi-isomorphism type is unchanged by a move not arising from an isotopy through admissible tuples (recall Definition \ref{d:admissibletuple}).

The first case is the analogue of \cite[Lemma 49]{SeidelSmith}.

\begin{lemma}\label{l:sliding1}
 Let $\Gamma:=\{\gamma_{1},\dots,\gamma_{n}\}$ be an admissible tuple.
 Suppose $\gamma_{i}$ and $\gamma_{j}$ are matching paths and $i \neq j$.
 Let $\gamma'$ be the matching path obtained by sliding $\gamma_{i}$ across $\gamma_{j}$.
 If $\Gamma'$ is the admissible tuple obtained by replacing $\gamma_{i}$ by $\gamma'$, then $\ul{L}_{\Gamma}$
 is quasi-isomorphic to  $\ul{L}_{\Gamma'}$.
\end{lemma}

\begin{proof}
 First consider the case that $n=2$, $i=1$, $j=2$ and $m \ge 2n$.
 In this case, it is proved in \cite[Lemma 49]{SeidelSmith} that 
 $\Sym(\ul{L}_{\Gamma})$ and $\Sym(\ul{L}_{\Gamma'})$ are Hamiltonian isotopic in $\cY_{n,m}$ (they treat the case $m=2n$ but the proof works for all $m \ge 2n$).
 In particular, we can find
 $a \in HF(\Sym(\ul{L}_{\Gamma}),\Sym(\ul{L}_{\Gamma'}))$ and $b \in HF(\Sym(\ul{L}_{\Gamma'}),\Sym(\ul{L}_{\Gamma}))$
 such that 
 \begin{align}
  \mu^2(b,a)=1_{HF(\Sym(\ul{L}_{\Gamma}),\Sym(\ul{L}_{\Gamma}))} \text{ and } \mu^2(a,b)=1_{HF(\Sym(\ul{L}_{\Gamma'}),\Sym(\ul{L}_{\Gamma'}))}
 \end{align}
 
 As in the proof of Proposition \ref{p:embedding}, when we turn off the Hamiltonian perturbation,
 there is a bijective correspondence between the $J$-holomorphic curves in $\cY_{n,m}$
 and the corresponding curves in $E$.
 In particular, we can find
 $a' \in HF(\ul{L}_{\Gamma},\ul{L}_{\Gamma'})$ and $b' \in HF(\ul{L}_{\Gamma'},\ul{L}_{\Gamma})$
 such that 
 \begin{align}
  \mu^2(b',a')=1_{HF(\ul{L}_{\Gamma},\ul{L}_{\Gamma})} \text{ and } \mu^2(a',b')=1_{HF(\ul{L}_{\Gamma'},\ul{L}_{\Gamma'})} \label{eq:equivalentmu2}
 \end{align}
 which implies that $\ul{L}_{\Gamma}$ and $\ul{L}_{\Gamma'}$ are quasi-isomorphic objects in $\cFS^{cyl,n}(E)$.
 Moreover, by the open mapping theorem, the projection to $\bH^{\circ}$ of the curves contributing to $\mu^2$ in \eqref{eq:equivalentmu2}
 is contained in the disc bounded by $\gamma_1$ and $\gamma'$ (see Figure \ref{fig:equivalentmu2}).
 
Now, we consider the general $n$ with $2n \le m$.
Without loss of generality, we continue to assume $i=1$ and $j=2$.
If $\Gamma$ has no thimble path, then we do not need to impose Hamiltonian perturbations in the Floer equation when we compute the differential and product.
In this case, the Floer solutions contributing to $\mu^2$ in \eqref{eq:equivalentmu2} persist because they lie above the disc bounded by $\gamma_1$ and $\gamma'$, which is not altered when adding the other Lagrangian components. Together with the constant triangles on the other Lagrangian components representing $\mu^2(e,e)=e$, we conclude that there exists
 $a' \in HF(\ul{L}_{\Gamma},\ul{L}_{\Gamma'})$ and $b' \in HF(\ul{L}_{\Gamma'},\ul{L}_{\Gamma})$
 such that \eqref{eq:equivalentmu2} holds, and hence $\ul{L}_{\Gamma}$ is quasi-isomorphic to $\ul{L}_{\Gamma'}$ (see Figure \ref{fig:equivalentmu2}).

If $\Gamma$ contains some thimble paths, then Hamiltonian perturbation terms in the Floer equation are necessary. 
However, the Hamiltonian terms can be taken to be zero over the disc bound by  $\gamma_1$ and $\gamma'$ and hence the Floer solutions contributing to $\mu^2$ in \eqref{eq:equivalentmu2} persist. 
Again, together with the Floer triangles on the other Lagrangian components representing $\mu^2(e,e)=e$ (all of which are constant for an appropriate choice of Hamiltonian perturbation, cf. the proof of Lemma \ref{l:chainRep}), we conclude that
$\ul{L}_{\Gamma}$ is quasi-isomorphic to $\ul{L}_{\Gamma'}$.

Finally, we explain the remaining case where $2n > m$.
To deal with this case, we want to compare  $\cY_{n,m}$ and $\cY_{n,2n}$.
Let $E$ be as above so  that  $\cY_{n,m}=\Hilb^n(E)\setminus D_r$.
We denote the corresponding $E$ for $\cY_{n,2n}$ by $E_+$ (i.e. $\cY_{n,2n}=\Hilb^n(E_+)\setminus D_r$).
Similarly, we denote the Lefschetz fibration $E_+ \to \bH^\circ$ by $\pi_{E_+}$.
Let $W \subset \bH$ be $\{ Re(z)>2n-m+\frac{1}{2}\}$ so $W$ contains exactly $m$ of the $2n$ critical values of $\pi_{E_+}$.
Pick a diffeomorphism from $W$ to $\bH$ which sends $\{2n-m+1,\dots,2n\}$ to $\{1,\dots,m\}$.
The pre-images of $\Gamma$ and $\Gamma'$ under this diffeomorphism define two  
admissible tuples, denoted by $\Gamma_+$
and $\Gamma'_+$, respectively.
From the discussion above, we know that 
 $\ul{L}_{\Gamma_+}$ is quasi-isomorphic to $\ul{L}_{\Gamma'_+}$ in $\cFS_W^{cyl,n}(\pi_{E_+})$
because we can find $a',b'$ such that the corresponding \eqref{eq:equivalentmu2} holds for $\ul{L}_{\Gamma_+}$ and $\ul{L}_{\Gamma'_+}$. The pairs $\ul{L}_{\Gamma}$, $\ul{L}_{\Gamma'}$ in $E$ and
 $\ul{L}_{\Gamma_+}$, $\ul{L}_{\Gamma'_+}$ in $\pi_{E_+}^{-1}(W)$ are isotopic through a family of Lagrangians associated to a family of admissible tuples in Lefschetz fibrations over the disc with varying symplectic form and almost complex structure, and where the isotopy does not create or cancel intersection points. There is no bifurcation in the moduli spaces of constant holomorphic triangles in such a deformation, so 
 one can find the corresponding $a',b'$ 
for  $\ul{L}_{\Gamma}$ and $\ul{L}_{\Gamma'}$ such that \eqref{eq:equivalentmu2} holds.
Therefore, the result follows.
\end{proof}

\begin{figure}[h]
 \includegraphics{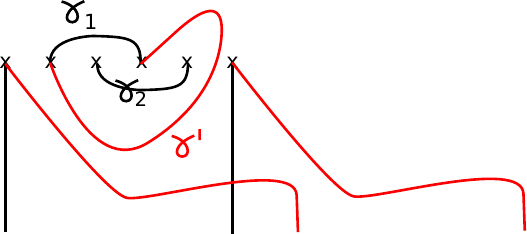}
 \caption{Floer cochain $CF(\ul{L}_{\Gamma'},\ul{L}_{\Gamma})$, where $\ul{L}_{\Gamma'}$ is in red and $\ul{L}_{\Gamma}$ is in black}\label{fig:equivalentmu2}
\end{figure}

\begin{lemma}\label{l:sliding2}
 Let $\Gamma:=\{\gamma_{1},\dots,\gamma_{n}\}$ be an admissible tuple.
 Suppose $\gamma_{i}$ is a thimble path and $\gamma_{j}$ is a matching path.
 Let $\gamma'$ be the thimble path obtained by sliding $\gamma_{i}$ across $\gamma_{j}$.
 If $\Gamma'$ is the admissible tuple obtained by replacing $\gamma_{i}$ with $\gamma'$, then $\ul{L}_{\Gamma}$
 is quasi-isomorphic to  $\ul{L}_{\Gamma'}$.
\end{lemma}

\begin{proof}

We first consider the case that $n=2$. 
Let $\gamma_1$ be the thimble path and $\gamma_2$ be the matching path.
We can assume that the thimble path $\gamma'$ only intersects $\gamma_1$
at the starting point, which we denote by $c$.
Along with $\gamma'$, we consider two more auxiliary thimble paths $\gamma_{1,w}$
and $\gamma_w'$, which are obtained by positively wrapping $\gamma_1$
and $\gamma'$ along the real line, respectively.
Without loss of generality (by possibly switching the role of $\gamma_1$ and $\gamma'$), we can assume that $\gamma_{1,w}$ intersects $\gamma'$ at two points $c$ and $c'$, and any other pair of thimble paths amongst $\gamma_1,\gamma',\gamma_{1,w},\gamma'_w$ intersect only at $c$ (see Figure \ref{fig:equivalentmu2_2}).

\begin{figure}[h]
 \includegraphics{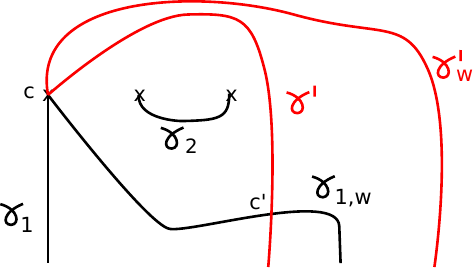}
 \caption{}\label{fig:equivalentmu2_2}
\end{figure}

Let $\Gamma_w:=\{\gamma_{1,w},\gamma_2\}$ and $\Gamma'_w:=\{\gamma'_w,\gamma_2\}$. It is clear that $\ul{L}_{\Gamma}$ is quasi-isomorphic to $\ul{L}_{\Gamma_w}$
and $\ul{L}_{\Gamma'}$ is quasi-isomorphic to $\ul{L}_{\Gamma'_w}$.
To show that $\ul{L}_{\Gamma}$ is quasi-isomorphic to $\ul{L}_{\Gamma'}$, it suffices to find
$a \in HF(\ul{L}_{\Gamma'},\ul{L}_{\Gamma})$,
$b \in HF(\ul{L}_{\Gamma_w},\ul{L}_{\Gamma'})$, and
$a' \in HF(\ul{L}_{\Gamma'_w},\ul{L}_{\Gamma_w})$ such that
\begin{align}
\mu^2(a,b)=1_{HF(\ul{L}_{\Gamma_w},\ul{L}_{\Gamma})} \text{ and } 
\mu^2(b,a')=1_{HF(\ul{L}_{\Gamma'_w},\ul{L}_{\Gamma'})} \label{eq:restrictionmu2}
\end{align}
and such that $a$ is identified with $a'$ under the continuation map (with respect to positive wrapping along the real line).
In the given positions of the Lagrangians, we do not need to use Hamiltonian terms on the Floer multiplication maps
\begin{align}
&HF(\ul{L}_{\Gamma'},\ul{L}_{\Gamma}) \times HF(\ul{L}_{\Gamma_w},\ul{L}_{\Gamma'})  \to HF(\ul{L}_{\Gamma_w},\ul{L}_{\Gamma})\\
& HF(\ul{L}_{\Gamma_w},\ul{L}_{\Gamma'}) \times HF(\ul{L}_{\Gamma'_w},\ul{L}_{\Gamma_w}) \to
 HF(\ul{L}_{\Gamma'_w},\ul{L}_{\Gamma'})
\end{align}
To compute these maps and hence verify \eqref{eq:restrictionmu2}, we rely on \eqref{eq:equivalentmu2} and a restriction argument as follows.

We add a critial value $\hat{c}$ and extend the thimble paths $\gamma_1,\gamma',\gamma_{1,w},\gamma'_w$ to run into $\hat{c}$.
The thimble paths become matching paths $\hat{\gamma}_1,\hat{\gamma}', \hat{\gamma}_{1,w},\hat{\gamma}'_w$ and we define $\hat{\Gamma}:=\{\hat{\gamma}_1,\gamma_2\}$, $\hat{\Gamma}':=\{\hat{\gamma}',\gamma_2\}$, $\hat{\Gamma}_{1,w}=\{\hat{\gamma}_{1,w},\gamma_2\}$ and $\hat{\Gamma}'_w:=\{\hat{\gamma}'_w,\gamma_2\}$ (see Figure \ref{fig:equivalentmu2_3}). 
By Lemma \ref{l:sliding1}, $\ul{L}_{\hat{\Gamma}}$ , $\ul{L}_{\hat{\Gamma}'}$, $\ul{L}_{\hat{\Gamma}_{1,w}}$ and $\ul{L}_{\hat{\Gamma}'_w}$ are all quasi-isomorphic objects. 
Therefore, there are
$\hat{a} \in HF(\ul{L}_{\hat{\Gamma}'},\ul{L}_{\hat{\Gamma}})$,
$\hat{b} \in HF(\ul{L}_{\hat{\Gamma}_w},\ul{L}_{\hat{\Gamma}'})$, and
$\hat{a}' \in HF(\ul{L}_{\hat{\Gamma}'_w},\ul{L}_{\hat{\Gamma}_w})$ such that
\begin{align}
\mu^2(\hat{a},\hat{b})=1_{HF(\ul{L}_{\hat{\Gamma}_w},\ul{L}_{\hat{\Gamma}})} \text{ and } 
\mu^2(\hat{b},\hat{a}')=1_{HF(\ul{L}_{\hat{\Gamma}'_w},\ul{L}_{\hat{\Gamma}'}).} \label{eq:restrictingmu2}
\end{align}
The degree $0$ generator for each of the Floer cochain groups above is given by an intersection point lying above either $c$ or $c'$, tensored with the degree $0$ generator of $CF(L_{\gamma_2},L_{\gamma_2})$.
For example, the cochain complex $CF(L_{\hat{\gamma}_{1,w}},L_{\hat{\gamma}'})$ is generated by one degree $0$ and one degree $1$ generator lying over $c'$ (arising from Morsifying the clean $S^1$-intersection locus \cite{Pozniak}),  and  transverse degree $2$ intersection points lying over each of $c$ and $\hat{c}$. 
\begin{figure}[h]
 \includegraphics{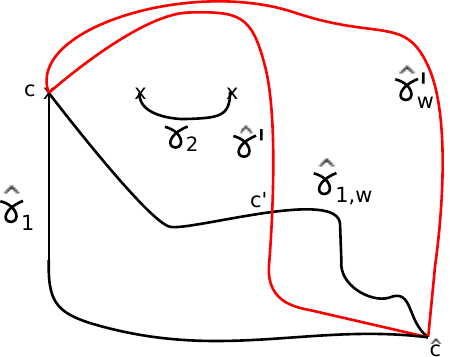}
 \caption{}\label{fig:equivalentmu2_3}
\end{figure}
Since each of the $CF^0$ groups involved has rank one, $\hat{a}$, $\hat{b}$ and $\hat{a}'$ are the unique degree $0$ generators (up to sign).  The curves contributing to \eqref{eq:restrictingmu2} cannot hit the generator lying above $\hat{c}$, so all the curves lie away from the fiber above $\hat{c}$.
By the open mapping theorem, the projection of these curves are contained in the disc bounded by $\gamma'$ and $\gamma_{1,w}$, which shows that \eqref{eq:restrictionmu2} holds before adding the critical value $\hat{c}$.
This verifies the case in which $n=2$.

Since the previous computation is local, the general case where $n>2$ can be treated as in the proof of Lemma \ref{l:sliding1}.
\end{proof}

The last case involves sliding a thimble path across another thimble path.
The proof is again an adaption of Lemma \ref{l:sliding1}, extending the thimble paths in order to apply the same kind of restriction argument as in the proof of Lemma \ref{l:sliding2},  so we omit it.

\begin{lemma}\label{l:sliding3}
 Let $\Gamma:=\{\gamma_{1},\dots,\gamma_{n}\}$ be an admissible tuple.
Suppose $\gamma_{i}$ and $\gamma_{j}$ are thimble paths and $i \neq j$.
 Let $\gamma'$ be the thimble path obtained by sliding $\gamma_{i}$ across $\gamma_{j}$.
 If $\Gamma'$ is the admissible tuple obtained by replacing $\gamma_{i}$ with $\gamma'$, then $\ul{L}_{\Gamma}$
 is quasi-isomorphic to  $\ul{L}_{\Gamma'}$.
\end{lemma}

\subsection{Generation}

We now show that when $E$ is the $A_{m-1}$-Milnor fiber,  the embedding in 
Proposition \ref{p:embedding} is essentially surjective.
In other words, we want to show that the split-closure $\cA$ of the thimbles $\ul{T}^I$ is the entire $D^{\pi}\cFS^{cyl,n}(\pi_E)$. 
We first recall some general facts for $A_{\infty}$/triangulated categories.

\begin{lemma}
 The full subcategory generated by an exceptional collection is admissible (i.e. admits right and left adjoints).
 In particular, $\cA$ is an admissible subcategory of $D^{\pi}\cFS^{cyl,n}(\pi_E)$.
\end{lemma}

\begin{proof}
Recall that a category admitting a full exceptional collection is already split-closed \cite[Remark 5.14]{SeidelBook}. The result then follows from e.g. \cite[Lemma 1.58]{FMBook}.
\end{proof}

For a subcategory $\cA$ of a triangulated category $\cC$, the right orthogonal $\cA^{\perp}$ of $\cA$ 
is the full subcategory of objects $\{X \in \cC \, | \, \mathrm{Hom}_{\cC}(A,X) = 0 \ \forall  A\in \cA\}$. The left orthogonal is defined similarly.

\begin{lemma}\label{l:directSum}
 Let $\cS$ be the Serre functor of $D^{\pi}\cFS^{cyl,n}(\pi_E)$.
 If $\cS(X), \cS^{-1}(X) \in \cA$ for all $X \in \cA$, then $D^{\pi}\cFS^{cyl,n}(\pi_E)=\cA \oplus \cA^{\perp}$.
\end{lemma}

\begin{proof}
 If $Y \in D^{\pi}\cFS^{cyl,n}(\pi_E)$ satisfies $\Hom(Y,X)=0$ for all $X \in \cA$, then
 $\Hom(X,\cS(Y))^\vee=0$ for all $X \in \cA$.
 It means that $\Hom(\cS^{-1}(X),Y)=0$ for all $X \in \cA$.
 By assumption, this is equivalent to $\Hom(X,Y)=0$ for all $X \in \cA$.
 Similarly, $\Hom(X,Y)=0$ for all $X \in \cA$ implies that $\Hom(Y,X)=0$ for all $X \in \cA$.
 As a result, the left-orthogonal of $\cA$ coincides with its right-orthogonal, so $D^{\pi}\cFS^{cyl,n}(\pi_E)$ splits as a direct sum $\cA \oplus \cA^{\perp}$.
\end{proof}

\begin{lemma}\label{l:allThimbles}
 Suppose $\ul{T}=\{T_1,\dots,T_n\} \in \cL^{cyl,n}$ is a Lagrangian tuple such that each $T_i$ is a thimble of $\pi_E$.
 Then $\ul{T}$ is generated by $\{\ul{T}^I\}_{I \in \cI}$.
\end{lemma}

\begin{proof}
 First note that the braid group acts transitively (up to isotopy) on all the $\ul{T}=\{T_1,\dots,T_n\}$ such that $T_i$ is a thimble for each $i$.
 Therefore, it suffices to show that for each simple braid $\sigma$ 
 and the associated symplectomorphism $\phi_{\sigma}$, the images  $\phi_{\sigma}(\ul{T}^I)$ and $\phi_{\sigma}^{-1}(\ul{T}^I)$
 are generated by $\{\ul{T}^I\}_{I \in \cI}$.
 
 Let $\sigma$ be the positive half twist swapping $\bc_j$ and $\bc_{j+1}$.
 There are 4 cases of $\ul{T}^I$ to consider, namely, whether $j$ and/or $j+1$ is contained in $I$ or not (recall that $I$ is a cardinality $n$ subset of $\{1,\dots,m\}$ and $\bc_i=i+\sqrt{-1}$ for all $i=1,\dots,m$).
 For each of these 4 cases, one can apply the exact triangles from Corollary \ref{c:Triangle} and the sliding invariance property in Lemma \ref{l:sliding3} to show that for all $I \in \cI$,
 $\phi_{\sigma}(\ul{T}^I)$ and $\phi_{\sigma}^{-1}(\ul{T}^I)$
 are generated by $\{\ul{T}^I\}_{I \in \cI}$.

More precisely, if $j,j+1 \notin I$, then $\phi_{\sigma}(\ul{T}^I)=\phi_{\sigma}^{-1}(\ul{T}^I)=\ul{T}^I$.
If $j+1 \in I$ and $j \notin I$, then $\phi_{\sigma}^{-1}(\ul{T}^I)=\ul{T}^{I'}$, where $I'=(I \setminus \{j+1\}) \cup \{j\}$.
On the other hand,  $\phi_{\sigma}(\ul{T}^I)$ can be obtained from applying iterated exact triangles to $\ul{T}^I$
and $\ul{T}^{I'}$ (see the first row of Figure \ref{fig:braidimage} where $n=4$, $j=2$
and $\phi_{\sigma}(\ul{T}^I)$ corresponds to the third term in the exact triangle).
If $j \in I$ and $j+1 \notin I$, it is similar to the previous case.
The last case is $j,j+1 \in I$. In this cae, both $\phi_{\sigma}(\ul{T}^I)$ and $\phi_{\sigma}^{-1}(\ul{T}^I)$
can be obtained from applying sliding invariance property to $\ul{T}^I$ (see the second row of Fgure \ref{fig:braidimage}).
 \begin{figure}[h]
 \includegraphics{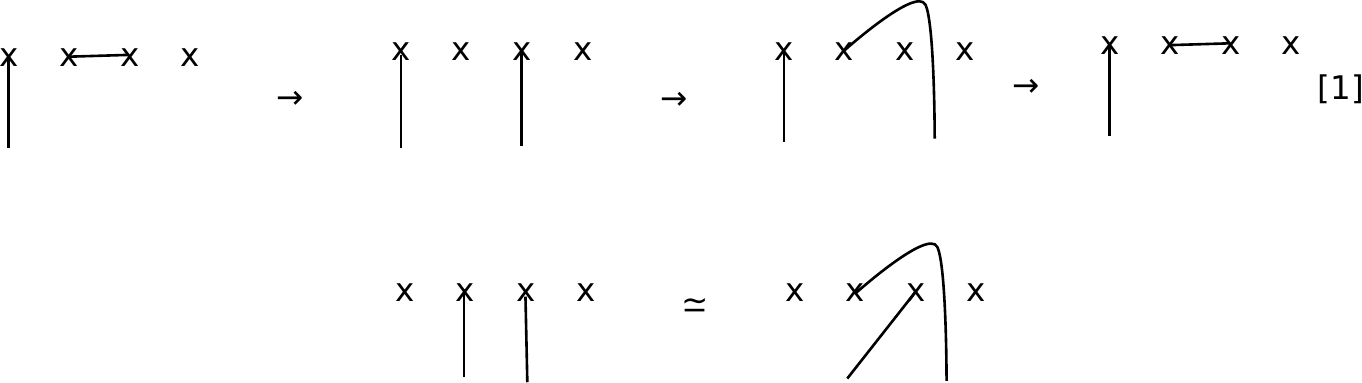}
 \caption{The top row is an exact triangle so the third term is generated by the first and second terms, which are in turn generated by $\{\ul{T}^I\}_{I \in \cI}$.
 The second row represents two quasi-isomorphic objects.}\label{fig:braidimage}
\end{figure}
\end{proof}


\begin{corollary}
Assume that Claim \ref{p:Serre} holds. Then the assumption in Lemma \ref{l:directSum} holds, i.e. $\cS(X), \cS^{-1}(X) \in \cA$ for all $X \in \cA$.
\end{corollary}

\begin{proof}
 By Claim \ref{p:Serre}, the Serre functor of $D^{\pi}\cFS^{cyl,n}(\pi_E)$ is given by the global monodromy $\tau$ up to grading shift.
 Therefore, it suffices to prove that for each thimble $\ul{T}^I$, the images $\tau(\ul{T}^I)$ and $\tau^{-1}(\ul{T}^I)$ are split-generated by 
 the collection of thimbles $\{\ul{T}^I\}_{I \in \cI}$. This is the content of Lemma \ref{l:allThimbles}.
 \end{proof}

\begin{proposition}\label{p:generation}
If Claim \ref{p:Serre} holds, then $\cA^\perp=0$, so $D^{\pi}\cFS^{cyl,n}(\pi_E)=\cA$.
\end{proposition}

\begin{proof}
 It suffices to show that for each $\ul{L} \in \cL^{cyl,n}$, there is an object $\ul{T}$ in $\cA$ such that $\Hom(\ul{T},\ul{L}) \neq 0$.
 By definition, we have $\ul{L}=\{L_1,\dots,L_n\}$ and $\pi_E(L_i) \subset U_i$ for some contractible $U_i$.
 We can assume that $\ol{U_i} \cap \partial \bH$ is connected and non-empty so that $\bH^{\circ} \setminus U_i$ is also contractible.
 
 For each $i$, there is a thimble $T_i$ of $\pi_E$ such that $\pi_E(T_i) \subset U_i$ and $HF(T_i,L_i) \neq 0$.
 This follows from the fact that thimbles generate $\cFS_{U_i}^{cyl,1}(\pi_E)=\cFS(\pi_E|_{\pi_E^{-1}(U_i)})$, the usual Fukaya-Seidel category of the Milnor fibre, see \cite{SeidelBook} (note that when $n=1$ there are no critical points at infinity).
 Let $\ul{T}=\{T_1,\dots,T_n\}$.
 It is clear that there is a cochain isomorphism
 \begin{align}
  CF(\ul{T},\ul{L}) \simeq \otimes_{i=1}^n CF(T_i,L_i)
 \end{align}
 which implies that $HF(\ul{T},\ul{L}) \neq 0$. By Lemma \ref{l:allThimbles}, we have $\ul{T} \in \cA$, concluding the proof.
 \end{proof}

\begin{corollary}\label{c:nm}
 Given Claim \ref{p:Serre}, when $n=m$ we have $\cA=\mathbb{K}$ so $D^{\pi}\cFS^{cyl,n}(\pi_E)=D^b(\mathbb{K})$.
\end{corollary}

\begin{proof}
 When $n=m$, there is only one cardinality $n$ subset in $\{1,\dots,m\}$ so
there is only one object in $\cA$ up to quasi-isomorphism (recall that $\cA$ is defined to be the full subcategory of $\cFS^{cyl,n}(\pi_E)$ with objects $\ul{T}^I$).
Moreover, this object is an exceptional object.
 Therefore, the result follows from Proposition \ref{p:generation}.
\end{proof}

This special case is not very important at this point but we will come back to it in Section \ref{s:Annular} when we define the symplectic annular Khovanov homology
and compare it to the algebraically defined annular Khovanov homology.

\begin{remark}\label{r:nm2}
 Even though Corollary \ref{c:nm} depends on  Claim \ref{p:Serre}, the fact that $D^{\pi}\cFS(\pi_\cY)=D^b(\mathbb{K})$ when $n=m$ follows from the definitions.
\end{remark}

\section{The extended symplectic arc algebra}\label{s:ArcAlg}

In this section, we will introduce a particular collection of admissible tuples, and hence
the corresponding collection of objects in $\cL^{cyl,n}$; these objects are motivated by the diagrammatics in \cite{Stroppel-parabolic, BS11}. We will prove that the cohomological Floer endomorphism algebra of these objects recovers the algebraic extended arc algebra as a graded vector space, cf. Lemma \ref{l:gvspIsom}.  
The $A_{\infty}$ endomorphism algebra of this collection of objects will be the `extended symplectic arc algebra', and will contain the symplectic arc algebra from \cite{AbouzaidSmith16} as a subalgebra. The corresponding Lagrangian products in $\Hilb^n(A_{m-1})\backslash D_r$ would not be cones over Legendrian submanifolds at infinity, which is why it is important to be able to study these Lagrangians in the cylindrical model.  The later parts of the section begin the study of the algebra structure on, and formality of, the extended symplectic arc algebra; these studies continue in Sections \ref{s:AlgIsom} and \ref{s:ncField} respectively.

\subsection{Weights and projective Lagrangians}\label{ss:weights}
To introduce the collection of admissible tuples, we start with some terminology.
Without loss of generality, we assume that the critical values are $\bc_k:=k+\sqrt{-1}$ for $k=1,\dots,m$.
A weight of type $(n,m)$ is a function $\lambda: \{1,\dots, m\} \to \{\wedge,\vee\}$ such that $|\lambda^{-1}(\vee)|=n$.
Let $\Lambda_{n,m}$ be the set of all weights of type $(n,m)$.
For $\lambda \in \Lambda_{n,m}$, let $c_{\lambda,1}< \dots < c_{\lambda,n}$ be the integers such that $\lambda(c_{\lambda,j})=\vee$.
For each $c_{\lambda,j}$, if there exists $c' \in \lambda^{-1}(\wedge)$ with $c'>c_{\lambda,j}$ and such that 
\begin{align}
 |\{c \in \{1,\dots,m\}| \lambda(c)=\vee, c_{\lambda,j}< c < c'\}| = |\{c \in \{1,\dots,m\}| \lambda(c)=\wedge, c_{\lambda,j}< c < c'\}| \label{eq:balancing}
\end{align}
then we call $c_{\lambda,j}$ a good point of $\lambda$; 
the minimum of all $c'$ satisfying \eqref{eq:balancing} is denoted by $c_{\lambda,j}^{\wedge}$.
If $c_{\lambda,j}$ is not a good point of $\lambda$, then we call it a bad point (see Figure \ref{fig:weight} for an example).

For $\lambda \in \Lambda_{n,m}$, we choose $n$ pairwise disjoint embedded curves
$\ul{\gamma}_{\lambda,1}, \dots, \ul{\gamma}_{\lambda,n}$ in $\{z \in \bH^\circ| im(z) \le 1, re(z)<2m\}$ such that 
\begin{align}
 &\text{if $c_{\lambda,j}$ is a good point, then $\ul{\gamma}_{\lambda,j}$ is a matching path joining $c_{\lambda,j}+\sqrt{-1}$ and $c_{\lambda,j}^{\wedge}+\sqrt{-1}$} \label{eq:lower1}\\
 &\text{if $c_{\lambda,j}$ is a bad point, then $\ul{\gamma}_{\lambda,j}$ is a thimble path from $c_{\lambda,j}+\sqrt{-1}$ to $c_{\lambda,j}$} \label{eq:lower2}
\end{align}
We define $\ul{\lambda}:=\{\ul{\gamma}_{\lambda,1}, \dots, \ul{\gamma}_{\lambda,n}\}$ and
\begin{align}
 \ul{L}_{\ul{\lambda}}:=\{L_{\ul{\gamma}_{\lambda,1}}, \dots, L_{\ul{\gamma}_{\lambda,n}}\} \in \cL^{cyl,n}
\end{align}
The quasi-isomorphism type of $\ul{L}_{\ul{\lambda}}$ is independent of the choice of $\ul{\lambda}$ (i.e. of the particular choice of paths) by Lemma \ref{l:HamInvariance}.

\begin{figure}[h]
 \includegraphics{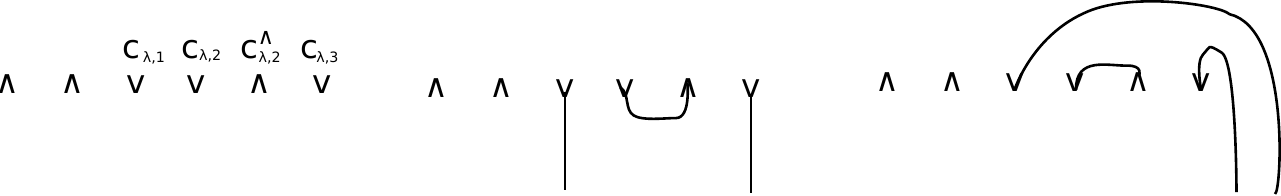}
 \caption{The left figure represents a weight $\lambda \in \Lambda_{3,6}$ with a good point $c_{\lambda,2}$ and two bad points $c_{\lambda,1}, c_{\lambda,3}$.
 The middle figure is $\ul{\lambda}$ and the right figure is $\ol{\lambda}$.}\label{fig:weight}
\end{figure}

%

Similarly, for $\lambda \in \Lambda_{n,m}$, we choose $n$ pairwise disjoint embedded curves
$\ol{\gamma}_{\lambda,1}, \dots, \ol{\gamma}_{\lambda,n}$ in $\{z \in \bH^\circ| im(z) \ge 1 \text{ or } re(z)>2m\}$ such that 
\begin{align}
 &\text{if $c_{\lambda,j}$ is a good point, then $\ol{\gamma}_{\lambda,j}$ is a matching path joining $c_{\lambda,j}+\sqrt{-1}$ and $c_{\lambda,j}^{\wedge}+\sqrt{-1}$} \label{eq:upper1}\\
 &\text{if $c_{\lambda,j}$ is a bad point, then $\ol{\gamma}_{\lambda,j}$ is a thimble path from $c_{\lambda,j}+\sqrt{-1}$ to $6m-c_{\lambda,j}$} \label{eq:upper2}
\end{align}
We define $\ol{\lambda}:=\{\ol{\gamma}_{\lambda,1}, \dots, \ol{\gamma}_{\lambda,n}\}$ and
\begin{align}
 \ul{L}_{\ol{\lambda}}:=\{L_{\ol{\gamma}_{\lambda,1}}, \dots, L_{\ol{\gamma}_{\lambda,n}}\} \in \cL^{cyl,n}
\end{align}
The quasi-isomorphism type of $\ul{L}_{\ol{\lambda}}$ is again independent of the choices of paths made in defining $\ol{\lambda}$,  by Lemma \ref{l:HamInvariance}.

\begin{figure}[h]
 \includegraphics{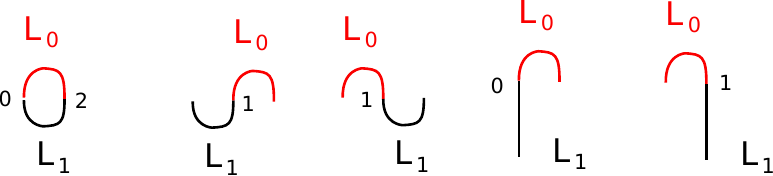}
 \caption{The integer near an intersection point labels its degree as an element of $CF(L_0,L_1)$.} \label{fig:grading}
\end{figure}

We want to choose a grading function on each $L_{\ul{\gamma},j}$ and $L_{\ol{\gamma},k}$ to induce a grading on 
$\ul{L}_{\ul{\lambda}}$ and $\ul{L}_{\ol{\lambda}}$.
These grading functions are chosen so that (see Figure \ref{fig:grading})
\begin{align}
 \text{$x  \in  CF(L_{\ol{\gamma},k},L_{\ul{\gamma},j})$ has degree $a$ if $x$ is the right end point of $a \in \NN$ matching spheres;} \label{eq:FloerGrading}
\end{align}
in particular, $\deg(x) \in \{0,1,2\}$ for all $x \in CF(L_{\ol{\gamma},k},L_{\ul{\gamma},j})$.

By iteratively applying Lemmas \ref{l:sliding1}, \ref{l:sliding2} and \ref{l:sliding3}, we obtain the following (see Figures  \ref{fig:sliding1}, \ref{fig:sliding2}, \ref{fig:sliding3}):

\begin{proposition}\label{p:slidinginvariance}
 $\ul{L}_{\ul{\lambda}}$ is quasi-isomorphic to $\ul{L}_{\overline{\lambda}}$  in $\cFS^{cyl,n}$. 
\end{proposition}


When only the quasi-isomorphism type is important, we denote either of $\ul{L}_{\ul{\lambda}}$ or $\ul{L}_{\overline{\lambda}}$ by $\ul{L}_{\lambda}$.

\begin{figure}[h]
\includegraphics{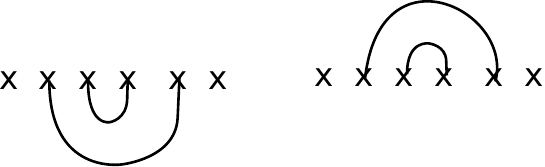}
\caption{Quasi-isomorphic compact objects}\label{fig:sliding1}
\end{figure}

\begin{figure}[h]
\includegraphics{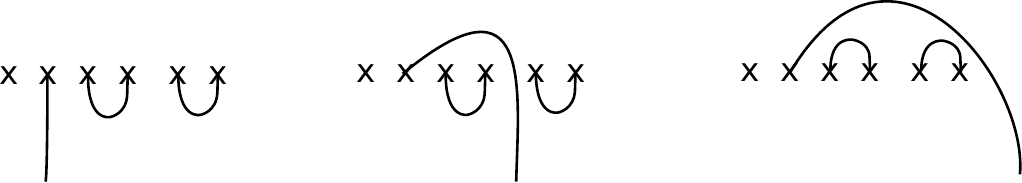}
\caption{Quasi-isomorphic `mixed' objects}\label{fig:sliding2}
\end{figure}

\begin{figure}[h]
\includegraphics{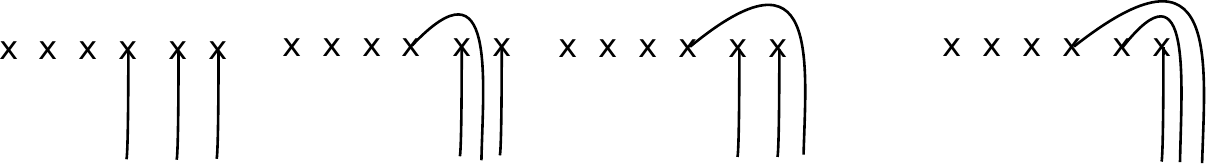}
\caption{Quasi-isomorphic thimble objects}\label{fig:sliding3}
\end{figure}

\begin{definition}
 The \emph{extended symplectic arc algebra} is the $A_{\infty}$-algebra
 \begin{align}
  \cK_{n,m}^{\symp}:=\oplus_{\lambda, \lambda' \in \Lambda_{n,m}} CF(\ul{L}_{\lambda},\ul{L}_{\lambda'}),
 \end{align}
 which is well-defined up to quasi-isomorphism.
\end{definition}

We want to choose a basis for the cohomology of $\cK_{n,m}^{\symp}$ as follows.
Let $\ol{\lambda} \cup \ul{\lambda}'$ be the union of all the paths in $\ol{\lambda}$ and $\ul{\lambda}'$.
By definition, $\ol{\lambda} \cup \ul{\lambda}'$ is a union of embedded circles and arcs; 
 some circles might be nested inside one another. It will be helpful to 
consider alternative admissible tuples which avoid such nesting (but for which the quasi-isomorphism type of the associated Lagrangian tuple is unchanged).

\begin{lemma}[cf. Lemma 5.15 of \cite{AbouzaidSmith16}]
There is an admissible tuple $\tilde{\lambda}'$ 
such that 
\begin{itemize}
 \item if $\gamma \in \ul{\lambda}'$ is not contained in a circle of  $\ol{\lambda} \cup \ul{\lambda}'$ (for example, when $\gamma$ is a  thimble path), then $\gamma \in \tilde{\lambda}'$;
 \item if $\gamma \in \ul{\lambda}'$ is contained in a circle $C$ of  $\ol{\lambda} \cup \ul{\lambda}'$, then there is a matching path $\tilde{\gamma} \in \tilde{\lambda}'$ with the same end points as $\gamma$
such that $\tilde{\gamma}$ is enclosed in $C$;
 \item $\ol{\lambda} \cup \tilde{\lambda}'$ is a union of embedded circles and arcs such that none are nested.
\end{itemize}
\end{lemma}

\begin{proof}

The proof is analogous to the proof of Lemma 5.15 of \cite{AbouzaidSmith16}.
The only difference for our case is that we could have some thimble paths in admissible tuples.

More precisely, if $\gamma$ is contained in a circle $C$, then each critical value enclosed in $C$ is an end point of a matching path of $\ul{\lambda}'$ (and also a matching path of $\ol{\lambda}$), directly from the definitions \eqref{eq:lower1}, \eqref{eq:lower2} (resp. \eqref{eq:upper1}, \eqref{eq:upper2}).
Thus, a suitable $\tilde{\lambda}'$  can be obtained by iteratively applying Lemma \ref{l:sliding1} to $ \ul{\lambda}'$, ensuring that $\ul{L}_{\ul{\lambda'}}$ is quasi-isomorphic to $\ul{L}_{\tilde{\lambda}'}$ (see Figure \ref{fig:tildeGeom}).  Note that thimble paths are never contained in a circle, so don't need to be changed.
\end{proof}

Note that the definition of $\tilde{\lambda}'$ depends on the pair $(\lambda, \lambda')$, and not just on $\lambda'$.

\begin{figure}[h]
 \includegraphics{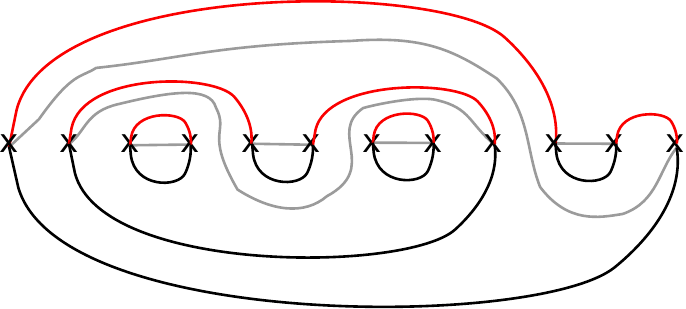}
 \caption{Here $\ol{\lambda}$ is red, $\tilde{\lambda}'$ is grey, and $\ul{\lambda'}$ is black.}\label{fig:tildeGeom}
\end{figure}

\begin{lemma}
The cohomology of $\cK_{n,m}^{\symp}$, denoted by $K_{n,m}^{\symp}$,  is given by
\begin{align}
 K_{n,m}^{\symp}=\oplus_{\lambda, \lambda' \in \Lambda_{n,m}} HF(\ul{L}_{\ol{\lambda}},\ul{L}_{\tilde{\lambda'}})
 = \oplus_{\lambda, \lambda' \in \Lambda_{n,m}} CF(\ul{L}_{\ol{\lambda}},\ul{L}_{\tilde{\lambda'}}) \label{eq:CohBasis}
\end{align}
as a graded vector space.
\end{lemma}

\begin{proof}
With the grading conventions of Figure \ref{fig:grading}, 
the pure degree elements in $CF(\ul{L}_{\ol{\lambda}},\ul{L}_{\tilde{\lambda'}})$ are concentrated either in odd or even degree,  so the Floer differential vanishes (compare to \cite[Proposition 5.12]{AbouzaidSmith16}).
\end{proof}

We call a basis $\cB^{\symp}$ for $K_{n,m}^{\symp}$ `geometric' 
if each basis element in $\cB^{\symp}$ is concentrated at a single intersection point in $\ul{L}_{\ol{\lambda}} \cap \ul{L}_{\tilde{\lambda'}}$ under the isomorphism \eqref{eq:CohBasis}.
Once a $\tilde{\lambda'}$ has been chosen for each pair $(\lambda, \lambda')$, two different geometric bases can differ only by signs.

\subsection{Extended arc algebra}

We briefly recall the diagrammatic extended arc algebra $K_{n,m}^{\alg}$. Details can be found in \cite{BS11, Stroppel-parabolic}, to which we defer for many details.
For each weight $\lambda \in \Lambda_{n,m}$, there is an associated cup diagram
$\ul{\lambda}^{\alg}$ and a cap diagram $\ol{\lambda}^{\alg}$ as follows.
The cup diagram $\ul{\lambda}^{\alg}$ can be obtained by adding 
a thimble path to $\ul{\lambda}$ from $a+\sqrt{-1}$ to $a$, for each 
$a \in \{1,\dots,m\}$, such that $a+\sqrt{-1}$ is not contained in any of the paths in
$\ul{\lambda}$ (see Figure \ref{fig:alg_weight}).
For the cap diagram  $\ol{\lambda}^{\alg}$, we need to replace the thimble paths from $c+\sqrt{-1}$ to $6m-c$ in $\ol{\lambda}$ by vertical rays from $c+\sqrt{-1}$ to 
$c+\sqrt{-1}\infty$, and in addition, for each 
$a \in \{1,\dots,m\}$ such that $a+\sqrt{-1}$ is not contained in any of the paths in
$\ol{\lambda}$, add a thimble path from  $a+\sqrt{-1}$ to $a+\sqrt{-1}\infty$ (see Figure \ref{fig:alg_weight}).
In this paper, the only cup and cap diagrams we will encounter are given by $\ul{\lambda}^{\alg}$ or 
$\ol{\lambda}^{\alg}$ for some $\lambda \in \Lambda_{n,m}$.

\begin{figure}[h]
 \includegraphics{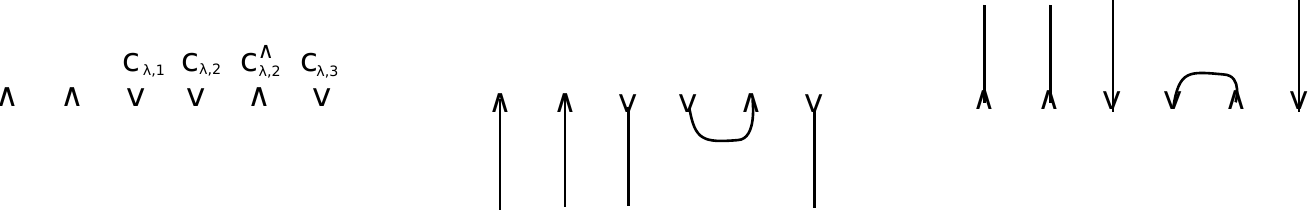}
 \caption{From left to right:  a weight $\lambda$, 
 the cup diagram  $\ul{\lambda}^{\alg}$ and the cap diagram  $\ol{\lambda}^{\alg}$.}\label{fig:alg_weight}
\end{figure}

If $\beta$ is a cup diagram and $\lambda$ is a weight, then we say that $\beta \lambda$
is an {\it oriented cup diagram} if 
\begin{align}
&\text{the $\lambda$-values of the two ends of every matching path in $\beta$ are different} \label{eq:orientations1} \\
&\text{if $\gamma_a$ and $\gamma_b$ are thimble paths in $\beta$ containing $a+\sqrt{-1}$ and $b+\sqrt{-1}$, respectively, } \label{eq:orientations2}\\
&\text{such that $a<b$ and $\lambda(a)=\vee$, then $\lambda(b)=\vee$.} \nonumber
\end{align}
A cap diagram is defined to be 
the reflection $\beta^{r}$ of a cup diagram $\beta$ along the line $\{im(z)=1\}$.
If $\alpha$ is a cap diagram and $\lambda$ is a weight, then we say that $ \lambda \alpha$ is
an oriented cap diagram if $\alpha^r \lambda $ is an oriented cup diagram.

The union of a cup diagram $\beta$ and a cap diagram $\alpha$ is denoted by $\beta \cup \alpha$  and is
called a circle diagram. This is a union of  embedded circles and arcs in the upper half-plane.
An orientation of a circle diagram $\beta\cup\alpha  $ is a weight $\lambda$
such that $\lambda \alpha $ and $\beta \lambda $ is an oriented cap diagram and oriented cup diagram, respectively.
Given such a $\lambda$, we denote the resulting oriented circle diagram by $\beta\lambda\alpha  $.

A clockwise cap (resp. cup) of an oriented cap (resp. cup) diagram $ \lambda \alpha$ (resp. $\alpha \lambda $) 
is a matching path $\gamma \in \alpha$ such that the 
$\lambda$-value of the left end point is $\wedge$, and hence the $\lambda$-value of the right end point is $\vee$.
The degree (or grading) of an oriented cap/cup/circle diagram is defined to be the number of clockwise cups in it.
As a result, we have
\begin{align}
 \deg(\beta\lambda\alpha  )=\deg(\lambda \alpha) + \deg (\beta \lambda ). \label{eq:AlgGrading}
\end{align}

As a graded vector space, $K_{n,m}^{\alg}$ is generated by oriented circle diagrams of the form $\ul{\lambda}^{\alg}_b \lambda \ol{\lambda}^{\alg}_a$, for  $\lambda, \lambda_a,\lambda_b \in \Lambda_{n,m}$,
and the grading of an oriented circle diagram is given by its degree.

\begin{lemma}\label{l:gvspIsom}
 There is a graded vector space isomorphism $\Phi:K_{n,m}^{\symp} \to K_{n,m}^{\alg}$ of the cohomological symplectic extended arc algebra with its algebraic counterpart.
\end{lemma}

\begin{proof}
 Let $\lambda_0,\lambda_1 \in \Lambda_{n,m}$.
 On the symplectic side, we consider the Floer cochains $CF(\ul{L}_{\ol{\lambda}_0},\ul{L}_{\tilde{\lambda}_1})$.
 On the diagrammatic side, we consider the graded vector space $S(\lambda_0,\lambda_1)$ generated by the orientations of the circle diagram $\ul{\lambda}^{\alg}_1 \cup \ol{\lambda}^{\alg}_0  $.
 By \eqref{eq:CohBasis}, $K_{n,m}^{\symp}=\oplus_{\lambda_0, \lambda_1 \in \Lambda_{n,m}} CF(\ul{L}_{\ol{\lambda}_0},\ul{L}_{\tilde{\lambda}_1})$
 so it suffices to find a graded vector space isomorphism between 
 $CF(\ul{L}_{\ol{\lambda}_0},\ul{L}_{\tilde{\lambda}_1})$ and $S(\lambda_0,\lambda_1)$ for all $\lambda_0, \lambda_1$.
 
 Each generator $\ul{x}=\{x_1,\dots,x_n\}$ of $CF(\ul{L}_{\ol{\lambda}_0},\ul{L}_{\tilde{\lambda}_1})$ projects to
 an $n$-tuple of pairwise distinct points $\pi_E(\ul{x})$ in $\{1,\dots, m\}+\sqrt{-1}$, 
 such that each Lagrangian component in $\ol{\lambda}_0$ and $\tilde{\lambda}_1$
 contains exactly one $\pi_E(x_i)$.
 Conversely, every $n$-tuple of pairwise distinct points in $\{1,\dots, m\}+\sqrt{-1}$ satisfying this property uniquely determines a  generator of $CF(\ul{L}_{\ol{\lambda}_0},\ul{L}_{\tilde{\lambda}_1})$.
 Let $\lambda_{\ul{x}}$ be the weight given by $\lambda_{\ul{x}}(a)= \vee$ if and only if $a+\sqrt{-1} \in \pi_E(\ul{x})$.
 We claim that the linear map $\Phi_{\lambda_0,\lambda_1}:CF(\ul{L}_{\ol{\lambda}_0},\ul{L}_{\tilde{\lambda}_1}) \to S(\lambda_0,\lambda_1)$
 given by
 \begin{align}
  \ul{x}=\{x_1,\dots,x_n\} \mapsto \lambda_{\ul{x}} \label{eq:identifyGen}
 \end{align}
 is a graded vector space isomorphism.
 
 To see that $\Phi_{\lambda_0,\lambda_1}$ is well-defined, we observe that all the thimble paths contained in $\ul{\lambda}^{\alg}$ but not in $\ul{\lambda}$ are on the left of 
 the thimble paths (if any) in $\ul{\lambda}$, and the same is true for $\ol{\lambda}^{\alg}$ and $\ol{\lambda}$. 
 Therefore, $\lambda_{\ul{x}}$ satisfies \eqref{eq:orientations2}.
 On the other hand, since each Lagrangian component contains exactly one of the $x_i$, it means that $\lambda_{\ul{x}}$ also satisfies \eqref{eq:orientations1}.
 
 It is routine to check that $\Phi_{\lambda_0,\lambda_1}$ is bijective and preserves the grading, using \eqref{eq:FloerGrading} and \eqref{eq:AlgGrading}.
\end{proof}

The algebra structure on $K_{n,m}^{\alg}$ is defined by applying an appropriate diagrammatic TQFT.  We briefly recall one possible definition of this algebra, and some crucial properties that we will use later, in Section \ref{Sec:tqft_multiplication} below, and refer the readers to \cite[Sections 3 \& 4]{BS11} for a detailed exposition. 

\subsection{Compact subalgebra}

We call a weight $\lambda \in \Lambda_{n,m}$ a \emph{compact weight} if $\ol{\lambda}$ (and hence $\ul{\lambda}$) consists only of matching paths.
Let $\Lambda_{n,m}^c \subset \Lambda_{n,m}$ be the subset of compact weights.
We define 
\begin{align}
 \cH_{n,m}^{\symp}:=\oplus_{\lambda,\lambda' \in \Lambda_{n,m}^c} CF(\ul{L}_{\lambda}, \ul{L}_{\lambda'})
\end{align}
which is well-defined up to quasi-isomorphism.
As in \eqref{eq:CohBasis}, its cohomology is given by
\begin{align}
 H_{n,m}^{\symp}=\oplus_{\lambda,\lambda' \in \Lambda_{n,m}^c} HF(\ul{L}_{\lambda}, \ul{L}_{\lambda'})=\oplus_{\lambda,\lambda' \in \Lambda_{n,m}^c} CF(\ul{L}_{\ol{\lambda}}, \ul{L}_{\tilde{\lambda}'}). \label{eq:CohBasisCpt}
\end{align}
A basis of $H_{n,m}^{\symp}$ is called geometric if it is given by the geometric intersection points in $CF(\ul{L}_{\ol{\lambda}}, \ul{L}_{\tilde{\lambda}'})$ under the isomorphism
\eqref{eq:CohBasisCpt} (again such bases are well-defined up to sign). 

On the diagrammatic side, we can define the corresponding subalgebra $H_{n,m}^{\alg}$ of $K_{n,m}^{\alg}$,
which is generated by oriented circle diagrams such that the underlying cap and cup diagrams are  given by $\ol{\lambda}^{\alg}$
and $\ul{\lambda'}^{\alg}$ for some $\lambda, \lambda' \in \Lambda_{n,m}^c$.
It is clear that the graded vector space isomorphism in Lemma \ref{l:gvspIsom} induces a graded vector space isomorphism between $H_{n,m}^{\symp}$
and $H_{n,m}^{\alg}$.

Previous study has focussed on the case $m=2n$.
In this case,  $H_{n,2n}^{\alg}$ is generated by oriented circle diagrams whose underlying diagram only contains circles (and the corresponding Lagrangian submanifolds of $\Hilb^n(A_{2n-1})$ are compact, being products of spheres rather than products of spheres and thimbles). 
The algebra $H_{n,2n}^{\alg}$ is also known as Khovanov's \emph{arc algebra}.

\begin{theorem}[\!\cite{AbouzaidSmith16, AbouzaidSmith19}]\label{t:AbouzaidSmith}
 The $A_{\infty}$ algebra $H_{n,2n}^{\symp}$ is formal.
 Moreover, there is an isomorphism between $H_{n,2n}^{\alg}$ and $H_{n,2n}^{\symp}$, sending the oriented circle diagram basis 
 of $H_{n,2n}^{\alg}$ to a geometric basis of $H_{n,2n}^{\symp}$.
\end{theorem}

\begin{proof}[Sketch of proof]
 Formality of $H_{n,2n}^{\symp}$ is the main result of \cite{AbouzaidSmith16}, whilst a basis-preserving algebra  isomorphism between $H_{n,2n}^{\alg}$ and $H_{n,2n}^{\symp}$
 is one of the main results of \cite{AbouzaidSmith19}.
 We now recall the basis for $H_{n,2n}^{\symp}$ chosen in \cite{AbouzaidSmith19}, and explain why it is a geometric basis in our sense. 
 The essential point \cite[Corollary 5.5]{AbouzaidSmith19} is the compatibility of the basis with various K\"unneth-type functors and decompositions (of the cohomology of products of spheres and thimbles with the cohomologies of the constituent factors).

 The natural action of the braid group $Br_{2n}$ on $A_{2n-1}$ and on $\cY_{n,2n}$ factors through an action of the symmetric group $Sym_{2n}$ on cohomology.  Furthermore, there are inclusions
 \[
 A_{2n-1} = \cY_{1,2n} \subset \cY_{n,2n} \subset (\mathbb{P}^1)^{2n}
 \]
 which induce $Sym_{2n}$-equivariant cohomology isomorphisms 
 \begin{align}
  H^2(A_{2n-1})=H^2(\cY_{1,2n}) \simeq H^2(\cY_{n,2n}) \simeq \mathbb{Z}\langle e_1,\ldots, e_{2n}\rangle / \langle \sum e_i \rangle \leftarrow \mathbb{Z} \langle e_1,\ldots, 
  e_{2n}\rangle = H^2((\mathbb{P}^1)^{2n}) \label{eq:geombasis}
 \end{align}
 In particular, the image of $\{e_i|1 \le i \le 2n-1\}$ in $H^2(\cY_{n,2n})$ forms a basis.
 The restriction map $H^2(\cY_{n,2n}) \to H^2(\Sym(\ul{L}_{\lambda}))$ is surjective for each $\lambda \in \Lambda^c_{n,2n}$, and by identifying the range with the cohomology of suitable multi-diagonals in $(\mathbb{P}^1)^{2n}$, 
 it is proved in \cite{AbouzaidSmith19} that 
 the chosen basis has the property that 
 it induces a well-defined basis for $H^2(\Sym(\ul{L}_{\lambda}))$, i.e. there is a basis of the latter such that each $e_i$ maps to a basis element or to zero.
 
 Now, we translate the basis of $H^2(\cY_{n,2n})$ to a basis of $H^2(A_{2n-1})$. 
 Denote by $s_1,\ldots, s_{2n-1}$ the homology classes of the $2n-1$ standard matching spheres in $A_{2n-1}$, 
 i.e. the ones that lie above $\{im(z)=1\}$; 
here they are labelled from left to right, and the spheres are oriented as the complex curves in the resolution of the $A_{2n-1}$-surface singularity.  
We denote the image of $e_i$ under \eqref{eq:geombasis} by $v_i \in H^2(A_{2n-1})$ for $1\le i \le 2n-1$.
They satisfy
\begin{align}
 v_1=s_1^* \text{ and } v_j=(-1)^{j}(s_{j-1}^*-s_j^*) \text{ for } 1<j \le 2n-1
\end{align}
where $*$ stands  for linear dual.
 The set $\{v_j\}_{j=1}^{2n-1}$ is the corresponding basis for $H^2(A_{2n-1})$.
 It induces a well-defined basis for $HF^2(\ul{L}_{\lambda},\ul{L}_{\lambda})$ via restriction 
 \begin{align}
  H^2(A_{2n-1}) \to H^2(\ul{L}_{\lambda})=HF^2(\ul{L}_{\lambda},\ul{L}_{\lambda})
 \end{align}
  where the first map is defined by regarding $\ul{L}_{\lambda}$ as the disjoint union of the matching spheres it contains, and hence as a submanifold of $A_{2n-1}$.
  This basis for $H^2(\ul{L}_{\lambda})$ naturally corresponds to the basis of $H^2(\Sym(\ul{L}_{\lambda}))$.
 Monomial products of the resulting elements give a basis for $HF^*(\ul{L}_{\lambda},\ul{L}_{\lambda})$.
 
 When $HF(\ul{L}_{\lambda},\ul{L}_{\lambda'}) \neq \{0\}$, it has rank one in minimal degree.  Let $a_{\min}$ be a minimal degree generator of $HF(\ul{L}_{\lambda},\ul{L}_{\lambda'})$ over the integers $\mathbb{Z}$.
 Then $HF(\ul{L}_{\lambda},\ul{L}_{\lambda'})$ is generated by $a_{\min}$
 as a module over each of $HF(\ul{L}_{\lambda},\ul{L}_{\lambda})$ and $HF(\ul{L}_{\lambda'},\ul{L}_{\lambda'})$.
A basis in $HF(\ul{L}_{\lambda},\ul{L}_{\lambda'})$ can therefore be defined by taking products of $a_{\min}$ with the bases for either of  
 $HF(\ul{L}_{\lambda},\ul{L}_{\lambda})$ or $HF(\ul{L}_{\lambda'},\ul{L}_{\lambda'})$; these act via the action of $H^*(\cY_{n,2n})$, which acts centrally, so there is no ambiguity (aside from a choice of sign of $a_{min}$).
 This finishes recalling the basis for $H_{n,2n}^{\symp}$, which we call a `convenient'  basis.
 
 We claim this convenient  basis 
  is geometric. 
 Let $S_{\lambda,1},\dots,S_{\lambda,n}$ be the matching spheres that $\ul{L}_{\lambda}$ contains.
 The first observation is that $v_j(S_{\lambda,k}) \in \{0,\pm 1\}$ for each $j$, and furthermore, for each $j$ there is exactly one $k$ for which  $v_j(S_{\lambda,k}) \neq 0$.
 This means that the convenient basis of $H^2(\Sym(\ul{L}_{\lambda}))$ induced from $H^2(\cY_{n,2n})$
 coincides with the product basis of the cohomology of $\Sym(\ul{L}_{\lambda})$ as a product of matching spheres.
 The geometric basis in $CF(\ul{L}_{\ol{\lambda}}, \ul{L}_{\tilde{\lambda}})$ also coincides with the product basis,  because $\tilde{\lambda}$ is  a 
 Hamiltonian push-off of $\ol{\lambda}$, and $\ol{\lambda} \cup \tilde{\lambda}$ consists of $n$ embedded pairwise non-nested circles. 
 That suffices to show that the convenient basis on $HF(\ul{L}_{\lambda},\ul{L}_{\lambda})$ is geometric.
 
 On the other hand, the minimal degree generator of  $HF(\ul{L}_{\lambda},\ul{L}_{\lambda'})$ is forced to be geometric, because the minimal degree subspace has rank $1$.
Since $\ol{\lambda} \cup \tilde{\lambda}'$ is a union of embedded and non-nested circles, the product map
 \begin{align}
  HF(\ul{L}_{\ol{\lambda}},\ul{L}_{\tilde{\lambda}'}) \otimes HF(\ul{L}_{\ol{\lambda}},\ul{L}_{\ol{\lambda}}) \to HF(\ul{L}_{\ol{\lambda}},\ul{L}_{\tilde{\lambda}'}) \label{eq:piece}
 \end{align}
 decomposes into pieces, one for each circle $C$ in $\ol{\lambda} \cup \tilde{\lambda}'$.
More precisely, we have isomorphisms
\begin{align}
HF(\ul{L}_{\ol{\lambda}},\ul{L}_{\tilde{\lambda}'}) &=\otimes_C H^{*}(S^2) \\
HF(\ul{L}_{\ol{\lambda}},\ul{L}_{\ol{\lambda}}) &= \otimes_C  (H^*(S^2))^{\otimes m_C}
\end{align}
where the tensor product is over all circles $C$ in $\ol{\lambda} \cup \tilde{\lambda}'$, and $m_C$ is the number of paths of $\ol{\lambda} $ that lie in $C$.
The product map \eqref{eq:piece} decomposes into the tensor product over all circles $C$ of maps
\begin{align}
H^{*}(S^2) \otimes  (H^*(S^2))^{\otimes m_C} \to H^{*}(S^2) \label{eq:piece2}
\end{align}
for which a typical local model is given by Figure \ref{fig:local_model}.
The product of the degree $0$ geometric generator in $H^{*}(S^2)$ and a degree $2$ geometric generator in 
$(H^*(S^2))^{\otimes m_C}$
 is, up to sign, the degree $2$ geometric generator in $H^{*}(S^2)$, because  $H^{2}(S^2)$ has rank $1$, the product is (by definition) a basis element of the convenient basis, and  \cite[Section 5]{AbouzaidSmith19} proves that the convenient basis is a basis for the cohomology groups over $\mathbb{Z}$.
 We conclude that the convenient basis is geometric.
 \end{proof}

\begin{figure}[h]
 \includegraphics{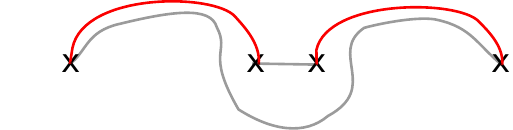}
 \caption{A circle $C$ in $\ol{\lambda} \cup \tilde{\lambda}'$ with $m_C=2$ (see Figure \ref{fig:tildeGeom})}\label{fig:local_model}
\end{figure}

\begin{corollary}\label{c:IdentifyingBasis}
After possibly changing $\Phi$ by sign on certain basis elements, the resulting graded vector space isomorphism $\Phi|_{H_{n,2n}^{\symp}}$ is an algebra isomorphism between $H_{n,2n}^{\symp}$ and $H_{n,2n}^{\alg}$. 
\end{corollary}

From now on, we assume $\Phi|_{H_{n,2n}^{\symp}}$ has been chosen to be an algebra isomorphism.

\subsection{From the compact arc algebra to the extended arc algebra\label{Sec:tqft_multiplication}}

In this section, we summarise the multiplication rule 
for $K_{n,m}^{\alg}$ and $H_{n,m}^{\alg}$
from \cite[Sections 3 \& 4]{BS11}.  Our approach is dictated by two considerations. 
First, the easiest (though not the original) description of the algebra structure on $K_{n,m}^{\alg}$ is to 
realise $K_{n,m}^{\alg}$ as a quotient algebra of $H_{m,2m}^{\alg}$. 
Second, in the next section, we precisely compute $K_{n,m}^{\symp}$ by realising it as a quotient of $H_{m,2m}^{\symp}$ by an ideal which, under the isomorphism $H_{m,2m}^{\alg}=H_{m,2m}^{\symp}$ (Theorem \ref{t:AbouzaidSmith}), we identify with the corresponding ideal on the algebraic side.  The main purpose of this section is therefore to recall this quotient description of $K_{n,m}^{alg}$.

For completeness, we briefly recall the multiplication rule for $H_{m,2m}^{\alg}$ from  \cite[Section 3, multiplication]{BS11}.
Let $\ul{d} \lambda \ol{c}$ and $\ul{b} \mu \ol{a}$ be two oriented
circle diagrams with $a,b,c,d \in \Lambda_{n,m}^c$. If $b \neq c$,  then the
product is defined to be zero.
If $b=c$, then we put the diagram $\ul{b} \mu \ol{a}$ above
 the diagram $\ul{d} \lambda \ol{c}$
and apply the TQFT surgery procedure -- based on the Frobenius algebra underlying $H^*(S^2;\mathbb{Z})$, and  described e.g. in \cite[Section 3, the surgery procedure]{BS11}
iteratively to convert it into
a disjoint union of diagrams each of which  has no cups/caps in the `middle section'.
After this, the middle sections are  unions of line segments; we shrink 
each such line segment to a point to obtain a disjoint union of some new
oriented circle diagrams. The product $(\ul{d} \lambda \ol{c})(\ul{b} \mu \ol{a})$ is then defined to be the sum
of the corresponding basis vectors of $H_{m,2m}^{\alg}$.

One distinguished feature of this TQFT surgery procedure  is that
the output circle (resp. disjoint union of two circles) 
is (resp. are) oriented according to the following rules, where $1 \Leftrightarrow $ anti-clockwise orientation, and $x \Leftrightarrow$  clockwise orientation:
\begin{align}
 1 \otimes 1 \mapsto 1, \quad 1 \otimes x \mapsto x, &\quad x \otimes 1 \mapsto x, \quad x \otimes x \mapsto 0 \label{eq:TQFT1} \\
 1 \mapsto 1 \otimes x + x \otimes 1 \quad& \quad x \mapsto x \otimes x. \label{eq:TQFT2} 
\end{align}
Each TQFT operation yields a disjoint union of zero, one or two new oriented diagrams replacing
the old diagram (the iterative application may yield larger linear combinations of diagrams).
This completes our r\'esum\'e  of $H_{m,2m}^{\alg}$ as an algebra.

We now recall how to realise $K_{n,m}^{\alg}$ as a quotient algebra of $H_{m,2m}^{\alg}$; for an independent definition of $K_{n,m}^{alg}$, see \cite{BS11}.
For $\lambda \in \Lambda_{n,m}$, we define its {\it closure} $\cl(\lambda) \in \Lambda_{m,2m}^c$ by
\begin{align}
 \cl(\lambda)=
 \left\{
 \begin{array}{ll}
  \vee  &\text{ if }a \le m-n \\
  \lambda(a-(m-n)) &\text{ if } m-n< a \le 2m-n \\
  \wedge &\text{ if } a > 2m-n.
 \end{array}
 \right.
\end{align}
Intuitively, $\cl(\lambda)$ can be regarded as putting $m-n$ many $\vee$'s and $n$ many $\wedge$'s to the left and right of $\lambda$, respectively, 
to make $\cl(\lambda)$ a compact weight.
For any oriented circle diagram $\ul{b} \lambda \ol{a} \in K_{n,m}^{\alg}$, we define
\begin{align}
 \cl(\ul{b} \lambda \ol{a}):= \ul{\cl(b)} \cl(\lambda) \ol{\cl(a)} \label{eq:cl}
\end{align}
which is an oriented circle diagram of $H_{m,2m}^{\alg}$.

\begin{lemma}\cite[Lemma 4.2]{BS11}\label{l:qAlg}
 The map $\ul{b} \lambda \ol{a} \mapsto  \cl(\ul{b} \lambda \ol{a})$ is a degree preserving bijection between
the set of oriented circle diagrams in $K_{n,m}^{\alg}$ with underlying weight $\lambda \in \Lambda_{n,m}$ and
the set of oriented circle diagrams
in $H_{m,2m}^{\alg}$ with underlying weight $\cl(\lambda)$.
\end{lemma}

Let $I_{\Lambda_{n,m}}$ be the subspace of $H_{m,2m}^{\alg}$ spanned by the vectors
\begin{align}
 \{\ul{b} \lambda \ol{a} \in H_{m,2m}^{\alg}|  \text{ oriented circle diagram } \ul{b} \lambda \ol{a} \text{ with} \lambda \in \Lambda_{m,2m}\setminus \cl(\Lambda_{n,m})\} \label{eq:Inm}
\end{align}
which is a two-sided ideal. 
The condition $ \lambda \in \Lambda_{m,2m}\setminus \cl(\Lambda_{n,m})$ is equivalent to 
\begin{align}
\{1,\dots,m-n\} \nsubseteq \lambda^{-1}(\vee) \text{ or } \{2m-n+1,\dots,2m\} \nsubseteq \lambda^{-1}(\wedge) \label{eq:2sidedIdeal}
\end{align}
 In view of Lemma \ref{l:qAlg}, the vectors
\begin{align}
 \{  \cl(\ul{b} \lambda \ol{a})+I_{\Lambda_{n,m}}|\text{ oriented circle diagrams }\ul{b} \lambda \ol{a} \text{ with }\lambda \in \Lambda_{n,m}\}
\end{align}
give a basis for the quotient algebra $H_{m,2m}^{\alg}/I_{\Lambda_{n,m}}$. We deduce that the map
\begin{align}
 cl: K_{n,m}^{\alg} \to H_{m,2m}^{\alg}/I_{\Lambda_{n,m}} & &\ul{b} \lambda \ol{a} \mapsto \cl(\ul{b} \lambda \ol{a})+I_{\Lambda_{n,m}} \label{eq:cl2}
\end{align}
is an isomorphism of graded vector spaces. We use this to transport the algebra structure on $H_{m,2m}^{\alg}/I_{\Lambda_{n,m}}$ to that on $K_{n,m}^{\alg}$.

We want to extract from \eqref{eq:cl2} two algebra isomorphisms
\begin{align}
H^{\alg}_{n,m} &= H^{\alg}_{m-n,2(m-n)}/I \text{ for }2n \le m \label{eq:1quotient}\\
K^{\alg}_{n,m} &= H^{\alg}_{n,m+n}/J \text{ for all }n \le m  \label{eq:2quotient}
\end{align}
where $I$ and $J$ are certain ideals  to be specified. Note that \eqref{eq:1quotient}
is an empty statement when $2n > m$ because, in this case, $H^{\alg}_{n,m}=0$, so one can simply take $I$ to be 
$H^{\alg}_{m-n,2(m-n)}$.

By definition, an oriented circle diagram
$\ul{b}\lambda \ol{a} \in K^{\alg}_{n,m}$
lies in $H^{\alg}_{n,m}$
if and only if  $a$ and $b$ have the property that all the $c_{a,j}$
and $c_{b,j}$ are good points (see \eqref{eq:balancing}).
That is equivalent to asking
\begin{align}
c^{\wedge}_{\cl(a),k}=c^{\wedge}_{\cl(b),k}=2m+1-k \text{ for all }k=1,\dots,n.
\end{align}
It means that for $\ul{b}\lambda \ol{a} \in H^{\alg}_{n,m}$, $\cl(\ul{b}\lambda \ol{a})$ 
has $n$ counterclockwise circles enclosing an oriented circle diagram, which we call $\bc(\ul{b}\lambda \ol{a})$, and which defines an element in 
$H^{\alg}_{m-n,2(m-n)}$.
Let $f:H^{\alg}_{m-n,2(m-n)} \to H_{m,2m}^{\alg}$ be the algebra embedding given by adding to an oriented circle diagram in $H^{\alg}_{m-n,2(m-n)}$ precisely $n$ 
counterclockwise circles enclosing it, so we have $f \circ \bc=\cl$. 

\begin{lemma}
We have a commutative diagram of algebra maps
\[
\begin{tikzcd}
H^{\alg}_{n,m} \arrow[dr,"\bc",hook] \arrow[rr,"\cl", hook]& &   H_{m,2m}^{\alg}/I_{\Lambda_{n,m}}\\
    & H^{\alg}_{m-n,2(m-n)}/f^{-1}(I_{\Lambda_{n,m}}) \arrow[ur,"f",hook]&
\end{tikzcd}
\]
\end{lemma}

\begin{proof}
The diagram is commutative by construction, and we  know that $f$ and $\cl$ are algebra maps.
It remains to check is that $\bc$ is also an algebra map, which follows from the injectivity of $f$ and the fact that $f$ and $\cl$ are algebra maps.
\end{proof}

By \eqref{eq:Inm}, \eqref{eq:2sidedIdeal} and the description of $f$, one checks
the ideal $I:=f^{-1}(I_{\Lambda_{n,m}})$ is the subspace of $H^{\alg}_{m-n,2(m-n)}$  spanned by
\begin{align}
 \{\ul{b} \lambda \ol{a} \in H^{\alg}_{m-n,2(m-n)}| \text{ oriented circle diagrams }\ul{b} \lambda \ol{a} \text{ such that } \{1,\dots,m-2n\} \nsubseteq \lambda^{-1}(\vee) \} \label{eq:ILeft}
\end{align}
From \eqref{eq:ILeft}, it is obvious that $\bc$ is an isomorphism (of vector spaces, and hence algebras)
\begin{align}
\bc:H^{\alg}_{n,m} \simeq  H^{\alg}_{m-n,2(m-n)}/I. \label{eq:1Q}
\end{align}

We given an equivalent reformulation of \eqref{eq:1Q}.
For $2n \le m$, define
 $\bc : \Lambda_{n,m} \to \Lambda_{m-n,2(m-n)}$ by
\begin{align}
 \bc(\lambda)(a):=
 \left\{
 \begin{array}{ll}
  \vee  &\text{ if }a \le m-2n \\
  \lambda(a-(m-2n)) &\text{ if }a> m-2n.
 \end{array}
 \right. \label{eq:bcFull}
\end{align}
This has the property that $\bc(\Lambda_{n,m}^c) \subset  \Lambda_{m-n,2(m-n)}^c$.
For $\ul{b}\lambda \ol{a} \in H^{\alg}_{n,m}$, we define
\begin{align}
\bc(\ul{b}\lambda \ol{a} )=\ul{\bc(b)} \bc(\lambda) \ol{\bc(a)}  \in  H^{\alg}_{m-n,2(m-n)}.
\end{align}
Let $I$ be as in \eqref{eq:ILeft}. Then
\begin{align}
 \bc: H_{n,m}^{\alg} \to H^{\alg}_{m-n,2(m-n)}/I &  &\ul{b} \lambda \ol{a} \mapsto \bc(\ul{b} \lambda \ol{a})+I \label{eq:bc}
\end{align}
is an algebra isomorphism.

Now we explain the definition of the ideal $J$ appearing in \eqref{eq:2quotient}.
Let $I_{n,m}^{(1)}$ and $I_{n,m}^{(2)}$ be the ideals of $H^{\alg}_{m,2m}$ spanned by
\begin{align}
 \{\ul{b} \lambda \ol{a} \in H_{m,2m}^{\alg}| \{1,\dots,m-n\} \nsubseteq \lambda^{-1}(\vee) )\} \text{ and } \label{eq:Inm1}\\
 \{\ul{b} \lambda \ol{a} \in H_{m,2m}^{\alg}| \{2m-n+1,\dots,2m\} \nsubseteq \lambda^{-1}(\wedge)\} \label{eq:Inm2}
\end{align}
respectively (see Figure \ref{fig:Ideals}).
\begin{figure}[h]
 \includegraphics{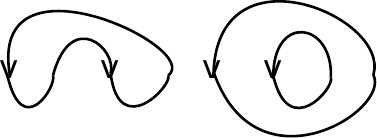}
 \caption{When $m=4$ and $n=2$, the oriented circle diagram above represents an element in $I_{n,m}^{(1)}$ but not in $I_{n,m}^{(2)}$.}\label{fig:Ideals}
\end{figure}
In view of \eqref{eq:Inm}, \eqref{eq:2sidedIdeal}, we have $I_{n,m}=I_{n,m}^{(1)}+I_{n,m}^{(2)}$.
This implies that
\begin{align}
K_{n,m}^{\alg} = H_{m,2m}^{\alg}/I_{\Lambda_{n,m}}=(H_{m,2m}^{\alg}/I_{n,m}^{(1)})/(I_{n,m}^{(1)}+I_{n,m}^{(2)}/I_{n,m}^{(1)}) \label{eq:2QQ}
\end{align}
From \eqref{eq:bc}, we have $H_{m,2m}^{\alg}/I_{n,m}^{(1)}=H_{n,m+n}^{\alg}$ so \eqref{eq:2QQ} becomes
\begin{align}
K_{n,m}^{\alg}=H_{n,m+n}^{\alg}/J \label{eq:2Q}
\end{align}
where $J$ is spanned by
\begin{align}
\{\ul{b} \lambda \ol{a} \in H_{n,m+n}^{\alg}| \ul{b} \lambda \ol{a} \text{ is an oriented circle diagram and } \{m+1,\dots,m+n\} \nsubseteq \lambda^{-1}(\wedge)\} \label{eq:Jright}
\end{align}

We give an equivalent reformulation of \eqref{eq:2Q}.
Let $\be: \Lambda_{n,m} \to \Lambda_{n,m+n}^c$ be the inclusion
\begin{align}
 \be(\lambda)(a):=
 \left\{
 \begin{array}{ll}
  \lambda(a) &\text{ if }a \le m \\
  \wedge &\text{ if }a> m
 \end{array}
 \right. \label{eq:beFull}
\end{align}
For $\ul{b}\lambda \ol{a} \in K^{\alg}_{n,m}$, we define
\begin{align}
\be(\ul{b}\lambda \ol{a} )=\ul{\be(b)} \be(\lambda) \ol{\be(a)}  \in  H^{\alg}_{n,m+n}
\end{align}
Let $J$ be as in  \eqref{eq:Jright}. Then
\begin{align}
 \be: K_{n,m}^{\alg} \to  H^{\alg}_{n,m+n}/J & &\ul{b} \lambda \ol{a} \mapsto \be(\ul{b} \lambda \ol{a})+J \label{eq:be}
\end{align}
is an algebra isomorphism.

\section{More algebra isomorphisms}\label{s:AlgIsom}

We will next identify $K^{\symp}_{n,m}$ and $K_{n,m}^{\alg}$ as algebras for every $n,m$.
 The proof of Theorem \ref{t:AbouzaidSmith} in \cite{AbouzaidSmith19} relied in an essential way on the fact that the Floer product for a triple of Lagrangians meeting pairwise cleanly can be understood, via `plumbing models', as a possibly sign-twisted convolution product, and that for every triple of Lagrangians associated to compact weights, one could find non-vanishing Floer products which factored through products associated to triples with plumbing models. 
 By contrast, for the extended algebra, there are  triples of weights and corresponding Lagrangians for which the product of minimal degree generators cannot be written as a product of minimal degree generators between interpolating Lagrangians with plumbing models. Instead of mimicking the  strategy of \cite{AbouzaidSmith19}
in proving $H^{\symp}_{n,2n}=H^{\alg}_{n,2n}$, we will instead reduce the isomorphism of extended arc algebras (in stages)  to Theorem \ref{t:AbouzaidSmith} by restriction-type arguments similar to those appearing in the proof of Lemma \ref{l:sliding2}.  The crucial point, as explained in Section \ref{Sec:tqft_multiplication}, is that all the extended arc algebras can be understood as sub-quotients of the algebras $H_{n,2n}^{alg}$.

\subsection{The compact cases}

We first explain how to compute the algebra structure on $H^{\symp}_{n,m}$ for all $n,m$.
We assume $2n \le m$ because $H^{\symp}_{n,m}=0$ otherwise.
Let $\bc : \Lambda_{n,m}^c \to \Lambda_{m-n,2(m-n)}^c$ be the injection given by \eqref{eq:bcFull}.
Let
\begin{align}
 H_{n,m}^{\bc}:=\oplus_{\lambda_0,\lambda_1 \in \Lambda_{n,m}^c} HF(\ul{L}_{\bc(\lambda_0)},\ul{L}_{\bc(\lambda_1)}) \label{eq:coHc}
\end{align}
which is a subalgebra of $H^{\symp}_{m-n,2(m-n)}=H_{m-n,2(m-n)}^{\alg}$, so in particular, $H_{n,m}^{\bc}$ is formal.
We want to compare $H^{\symp}_{n,m}$ with $H_{n,m}^{\bc}$. Note that the Lagrangians underlying $H_{n,m}^{\symp}$ are diffeomorphic to $(S^2)^n$, whilst those relevant to $H_{n,m}^{\bc}$ are $(S^2)^{m-n}$, where $m-n \geq n$.

For each ordered pair $(\lambda_0, \lambda_1)$ such that $\lambda_0,\lambda_1 \in \Lambda_{n,m}^c$, 
we choose upper matchings $\ol{\bc(\lambda_0)}$ and $\ol{\bc(\lambda_1)}$
such that
\begin{align}
 &\text{for $a =1,\dots,m-2n$, the slope at $a+\sqrt{-1}$ of the matching path in $\ol{\bc(\lambda_0)}$ starting } \label{eq:LeftCondition}\\
&\text{from $a+\sqrt{-1}$ is larger than that of the corresponding path in $\ol{\bc(\lambda_1)}$}  \nonumber
\end{align}

\begin{remark}
 Condition \ref{eq:LeftCondition} is used to eliminate the existence of certain pseudo-holomorphic maps.
 For example, let $l_0,l_1,l_2$ be upper matching paths from $1+\sqrt{-1}$ to $2+\sqrt{-1}$ such that the slopes at $1+\sqrt{-1}$ are in decreasing order (Figure \ref{fig:Viterbo}).
 Let $L_i$ be the corresponding matching spheres.
 Let $u$ be a solution contributing to the multiplication map
 \begin{align}
  CF(L_1,L_2) \times CF(L_0,L_1) \to CF(L_0,L_2)
 \end{align}
 such that $\pi_E \circ v$ is holomorphic.
 We claim that  if the output of $u$ maps to $1+\sqrt{-1}$ under $\pi_E$,  so do all the inputs of $v$.
 This is because if $\pi_E \circ v$ restricted to the boundary labelled by $L_0$ is not a constant, 
 then the image of $v$ must have non-empty intersection with the unbounded region of $\bH \setminus (l_0 \cup l_1 \cup l_2)$ by the holomorphicity of $\pi_E \circ v$.
 This in turn implies that $\pi_E \circ v$ is not relatively compact by the open mapping theorem, a contradiction.
 As a result, one can show that $\pi_E \circ v$ restricted to the boundary labelled by $L_0$ is a constant.
 Inductively applying this argument, one can show that the restriction of $\pi_E \circ v$ to the whole boundary is constant, and hence $u$ itself is a constant map.
\end{remark}

\begin{figure}[h]
 \includegraphics{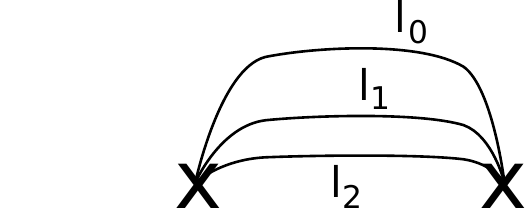}
 \caption{}\label{fig:Viterbo}
\end{figure}

The choices of matchings depend on the ordered pair we start with, cf. the discussion around Figure \ref{fig:tildeGeom}; to emphasise this dependence, we denote the matchings by $\ol{\bc(\lambda_0)}^{\lambda_0,\lambda_1}$
and $\ol{\bc(\lambda_1)}^{\lambda_0,\lambda_1}$ respectively.
Let $\cI_{\lambda_0,\lambda_1}$ be the subspace of $CF(\ul{L}_{\ol{\bc(\lambda_0)}^{\lambda_0,\lambda_1}},\ul{L}_{\ol{\bc(\lambda_1)}^{\lambda_0,\lambda_1}})$
generated by
\begin{align}
 \cB_{\lambda_0,\lambda_1}:=\{\ul{x} \in \cX(\ul{L}_{\ol{\bc(\lambda_0)}^{\lambda_0,\lambda_1}},\ul{L}_{\ol{\bc(\lambda_1)}^{\lambda_0,\lambda_1}})| \{1,\dots,m-2n\}+\sqrt{-1} \nsubseteq \pi_E(\ul{x}) \}
\end{align}
Let $\cG_{\lambda_0,\lambda_1}:=\cX(\ul{L}_{\ol{\bc(\lambda_0)}^{\lambda_0,\lambda_1}},\ul{L}_{\ol{\bc(\lambda_1)}^{\lambda_0,\lambda_1}}) \setminus \cB_{\lambda_0,\lambda_1}$.
We define
$\cI= \oplus \cI_{\lambda_0,\lambda_1}$, $\cB:=\oplus \cB_{\lambda_0,\lambda_1}$ and $\cG=\oplus \cG_{\lambda_0,\lambda_1}$.
We use the symbol $\cI$ because, as we will see in Lemma \ref{l:algIsom}, $\cI$ is actually an ideal.

\begin{lemma}\label{l:mu1preserveI}
 $\mu^1(\cI_{\lambda_0,\lambda_1}) \subset \cI_{\lambda_0,\lambda_1}$ so $\cI_{\lambda_0,\lambda_1}$ descends to a vector subspace of $H_{n,m}^{\bc}$, which we denote by $I_{\lambda_0,\lambda_1}$.
\end{lemma}

\begin{proof}
 It suffices to show that if the $\ul{x}_0$-coefficient of $\mu^1(\ul{x}_1)$ is non-zero for some $\ul{x}_0 \in \cG$, then $\ul{x}_1 \in \cG$.
 
 By definition, $\ul{x}_0 \in \cG$ implies that $\{1,\dots,m-2n\}+\sqrt{-1} \subset \pi_E(\ul{x}_0)$.
 Let $v:\Sigma \to E$ be a $J$-holomorphic map 
 contributing to the $\ul{x}_0$-coefficient of $\mu^1(\ul{x}_1)$  such that near an output puncture $\xi^0$, $v$ is asymptotic to the output lying above $1+\sqrt{-1}$.
 We project $v$ to $\bH^\circ$ by $\pi_E$ and apply the open mapping theorem.
 The condition \eqref{eq:LeftCondition} when $a=1$ forces at least one of the boundary components of $\Sigma$
 adjacent to $\xi^0$ to be mapped constantly to $1+\sqrt{-1}$ by $\pi_E \circ v$.
 Let the two boundary components of $\Sigma$
 adjacent to $\xi^0$ be $\partial^i \Sigma$ and $\partial^{j} \Sigma$, respectively, and suppose $\pi_E \circ v |_{\partial^j \Sigma}=1+\sqrt{-1}$. 
 As a result, the other puncture $\xi'$ that is adjacent to $\partial^j \Sigma$ is also mapped to $1+\sqrt{-1}$ by $\pi_E \circ v$.
 Let the other boundary component of $\Sigma$ that is adjacent to $\xi'$ be $\partial^{j'} \Sigma$.
 
 By the Lagrangian boundary conditions, $v(\partial^i \Sigma)$ and $v(\partial^{j'} \Sigma)$ are both contained in the Lagrangian components of 
 $\ul{L}_{\ol{\bc(\lambda_0)}^{\lambda_0,\lambda_1}}$ (or $\ul{L}_{\ol{\bc(\lambda_1)}^{\lambda_0,\lambda_1}}$).
 However, each of  $\ul{L}_{\ol{\bc(\lambda_0)}^{\lambda_0,\lambda_1}}$ and $\ul{L}_{\ol{\bc(\lambda_1)}^{\lambda_0,\lambda_1}}$ has only one component whose projection to $\bH^\circ$
 contains $1+\sqrt{-1}$. Therefore $\partial^i \Sigma=\partial^{j'} \Sigma$, and the two punctures adjacent to $\partial^i \Sigma$ are both mapped to $1+\sqrt{-1}$ under $\pi_E \circ v$.
 By the open mapping theorem, we conclude that the restriction of  $\pi_E \circ v$ to the connected component of $\Sigma$ that contains $\xi^0$ is a constant map
 and that connected component of $\Sigma$ is a bigon.
 
 The matching paths starting from $1+\sqrt{-1}$ and the  Lagrangian matching spheres lying above those paths are not Lagrangian boundary conditions of the restriction of $v$ to the other connected components of $\Sigma$. 
 We can therefore apply the previous reasoning inductively to $a=1,2,\dots,m-2n$.
 The conclusion is that any $v:\Sigma \to E$ contributing to $\ul{x}_0$-coefficient of $\mu^1(\ul{x}_1)$ contains $m-2n$ bigon components, each of which maps by a constant map to 
 $a+\sqrt{-1}$ under $\pi_E \circ v$, for $a=1,\dots,m-2n$, respectively.
 This implies that $\{1,\dots,m-2n\}+\sqrt{-1} \subset \pi_E(\ul{x}_1)$, and hence $\ul{x}_1 \in \cG$.
 \end{proof}

\begin{corollary}\label{c:asGvspace}
 $HF(\ul{L}_{\ol{\bc(\lambda_0)}},\ul{L}_{\ol{\bc(\lambda_1)}})/I_{\lambda_0,\lambda_1}=HF(\ul{L}_{\lambda_0},\ul{L}_{\lambda_1})$ as a graded vector space.
\end{corollary}

\begin{proof}
 For $j=0,1$, by removing the matching paths of $\ol{\bc(\lambda_j)}^{\lambda_0,\lambda_1}$ that contain $a+\sqrt{-1}$ for some $a=1,\dots,m-2n$ and translating 
 the remaining matching paths by $-(m-2n)$, we get an upper matching $\ol{\lambda_j}$ of $\lambda_j$.
 Therefore, there is an obvious bijective correspondence
 \begin{align}
  \cG_{\lambda_0,\lambda_1} \simeq \cX(\ul{L}_{\ol{\lambda_0}},\ul{L}_{\ol{\lambda_1}}) \label{eq:canonicalGen}
 \end{align}
 given by forgeting the elements in a tuple that lie above $\{1,\dots,m-2n\}+\sqrt{-1}$, and then translating the rest  by $-(m-2n)$.
 
 By the last paragraph of the proof of Lemma \ref{l:mu1preserveI}, there is a canonical isomorphism between the chain complexes
 \begin{align}
  CF(\ul{L}_{\ol{\bc(\lambda_0)}^{\lambda_0,\lambda_1}},\ul{L}_{\ol{\bc(\lambda_1)}^{\lambda_0,\lambda_1}})/\cI_{\lambda_0,\lambda_1} \cong CF(\ul{L}_{\ol{\lambda_0}},\ul{L}_{\ol{\lambda_1}}) \label{eq:canonicalChain}
 \end{align}
 which on generators is given by \eqref{eq:canonicalGen}, and on differentials is given by incorporating or removing the $m-2n$ constant bigon components from the proof of Lemma \ref{l:mu1preserveI}.
\end{proof}

\begin{lemma}\label{l:independent}
 The subspace $I_{\lambda_0,\lambda_1}$ is independent of the choices of $\ol{\bc(\lambda_0)}^{\lambda_0,\lambda_1}$
and $\ol{\bc(\lambda_1)}^{\lambda_0,\lambda_1}$, provided \eqref{eq:LeftCondition} is satisfied.
\end{lemma}

\begin{proof}
 Let $\bc(\lambda_0)^{a}$ and $\bc(\lambda_0)^{b}$ be two different choices of $\ol{\bc(\lambda_0)}^{\lambda_0,\lambda_1}$ such that 
 \eqref{eq:LeftCondition} is satisfied for the pairs $(\bc(\lambda_0)^{a}, \ol{\bc(\lambda_1)}^{\lambda_0,\lambda_1})$
 and $(\bc(\lambda_0)^{b},\ol{\bc(\lambda_1)}^{\lambda_0,\lambda_1})$.
 
By interpolating slopes,  there exists another choice $\bc(\lambda_0)^{c}$ of $\ol{\bc(\lambda_0)}^{\lambda_0,\lambda_1}$
 such that \eqref{eq:LeftCondition} is satisfied for the pairs $(\bc(\lambda_0)^{c},\bc(\lambda_0)^{a})$
 and $(\bc(\lambda_0)^{c},\bc(\lambda_0)^{b})$.
 Let $\theta$ be a continuation element (i.e the image of the identity element under a continuation map) of $HF(\ul{L}_{\bc(\lambda_0)^{c}},\ul{L}_{\bc(\lambda_0)^{a}})$.
 
 Let $I^a$ and $I^c$ be the respective $I_{\lambda_0,\lambda_1}$ for $HF(\ul{L}_{\bc(\lambda_0)^{a}},\ul{L}_{\ol{\bc(\lambda_1)}^{\lambda_0,\lambda_1}})$ and 
 $HF(\ul{L}_{\bc(\lambda_0)^{c}},\ul{L}_{\ol{\bc(\lambda_1)}^{\lambda_0,\lambda_1}})$.
 We need to show that the isomorphism
 \begin{align}
  \mu^2(-,\theta): HF(\ul{L}_{\bc(\lambda_0)^{a}},\ul{L}_{\ol{\bc(\lambda_1)}^{\lambda_0,\lambda_1}}) \to HF(\ul{L}_{\bc(\lambda_0)^{c}},\ul{L}_{\ol{\bc(\lambda_1)}^{\lambda_0,\lambda_1}})
 \end{align}
 sends $I^a$ to $I^c$.
 By Corollary \ref{c:asGvspace} and for dimension reasons, it suffices to show that the image of $I^a$ under $\mu^2(-,\theta)$ 
 is contained in $I^c$.
 The same will then be true when we replace $\bc(\lambda_0)^{a}$ by $\bc(\lambda_0)^{b}$,  so the result  will follow.
 
 The proof that $\mu^2(I^a,\theta) \subset I^c$ is similar to the proof of Lemma \ref{l:mu1preserveI}.
 Let $\cG^a$ and $\cG^c$ be the respective $\cG_{\lambda_0,\lambda_1}$ for $CF(\ul{L}_{\bc(\lambda_0)^{a}},\ul{L}_{\ol{\bc(\lambda_1)}^{\lambda_0,\lambda_1}})$ and 
 $CF(\ul{L}_{\bc(\lambda_0)^{c}},\ul{L}_{\ol{\bc(\lambda_1)}^{\lambda_0,\lambda_1}})$.
 By slight abuse of notation, we denote a chain level lift of $\theta$ to $CF(\ul{L}_{\bc(\lambda_0)^{c}},\ul{L}_{\bc(\lambda_0)^{a}})$ by $\theta$.
 It suffices to show that if the $\ul{x}_0$-coefficient of $\mu^2(\ul{x}_1,\theta)$ is non-zero for some $\ul{x}_0 \in \cG^c$, then $\ul{x}_1 \in \cG^a$.
 
 By definition, $\ul{x}_0 \in \cG^c$ implies that $\{1,\dots,m-2n\}+\sqrt{-1} \subset \pi_E(\ul{x}_0)$.
 Let $v:\Sigma \to E$ be a $J$-holomorphic map 
 contributing to the $\ul{x}_0$-coefficient of $\mu^2(\ul{x}_1,\theta)$  such that near an output puncture $\xi^0$, $v$ is asymptotic to the output lying above $1+\sqrt{-1}$.
 We project $v$ to $\bH^\circ$ by $\pi_E$ and apply the open mapping theorem.
 Note that both $(\bc(\lambda_0)^{a}, \ol{\bc(\lambda_1)}^{\lambda_0,\lambda_1})$ and $(\bc(\lambda_0)^{c},\bc(\lambda_0)^{a})$ satisfying condition \eqref{eq:LeftCondition}
 implies that $(\bc(\lambda_0)^{c},\ol{\bc(\lambda_1)}^{\lambda_0,\lambda_1})$ also satisfies condition \eqref{eq:LeftCondition}.
 This forces at least one of the boundary components of $\Sigma$
 adjacent to $\xi^0$ to be mapped constantly to $1+\sqrt{-1}$ by $\pi_E \circ v$.
 Let the two boundary component of $\Sigma$
 adjacent to $\xi^0$ be $\partial^i \Sigma$ and $\partial^{j} \Sigma$, respectively, and suppose $\pi_E \circ v |_{\partial^j \Sigma}=1+\sqrt{-1}$. 
 As before, the other puncture $\xi'$ that is adjacent to $\partial^j \Sigma$ is also mapped to $1+\sqrt{-1}$ by $\pi_E \circ v$.
 Let the other boundary component of $\Sigma$ that is adjacent to $\xi'$ be $\partial^{j'} \Sigma$.
 
 By examining the slopes of the matching paths at $1+\sqrt{-1}$ again, either $\partial^{i} \Sigma$ or $\partial^{j'} \Sigma$ is mapped to the constant  value $1+\sqrt{-1}$ by $\pi_E \circ v$.
 Say $\partial^{j'} \Sigma$ is mapped this way; then the next boundary components of $\Sigma$ and $\partial^{i} \Sigma$ 
 are mapped to components of the same Lagrangian tuple (i.e. $\ul{L}_{\bc(\lambda_0)^{c}}$, $\ul{L}_{\bc(\lambda_0)^{a}}$ or $\ul{L}_{\ol{\bc(\lambda_1)}^{\lambda_0,\lambda_1}}$) under $v$.
 As a result, these two boundary components co-incide, $\pi_E \circ v$ restricted to the component of $\Sigma$ containing $\xi^0$ is a constant map, and this component is a triangle.
 
 Applying this argument to all $a+\sqrt{-1}$ for $a=1,\dots,m-2n$, we conclude that 
 $\{1,\dots,m-2n\}+\sqrt{-1} \subset \pi_E(\ul{x}_1)$ and hence $\ul{x}_1 \in \cG^a$.
 \end{proof}

\begin{remark}
 Lemma \ref{l:independent} also proves that the continuation element $\theta$
 lies in the vector subspace spanned by the corresponding $\cG \subset CF(\ul{L}_{\bc(\lambda_0)^{c}},\ul{L}_{\bc(\lambda_0)^{a}})$.
\end{remark}


\begin{lemma}\label{l:algIsom}
 $I \subset H_{n,m}^{\bc}$  is an ideal. Moreover, there is an algebra isomorphism $H_{n,m}^{\bc}/I \simeq  H_{n,m}^{\symp}$.
\end{lemma}

\begin{proof}
By Lemma \ref{l:independent}, for $\lambda_0,\lambda_1,\lambda_2 \in \Lambda_{n,m}^c$, we can choose the upper matchings $\ol{\bc(\lambda_j)}$ such that
condition \eqref{eq:LeftCondition} is satisfied for both the pairs
$(\ol{\bc(\lambda_0)},\ol{\bc(\lambda_1)})$
and 
$(\ol{\bc(\lambda_1)},\ol{\bc(\lambda_2)})$.

In this case, the exact same argument as in the proof of Lemma \ref{l:independent} shows that 
if the $\ul{x}_0$-coefficient of $\mu^2(\ul{x}_2,\ul{x}_1)$ is non-zero for some $\ul{x}_0 \in \cG^{0,2}$, 
then $\ul{x}_1 \in \cG^{0,1}$ and $\ul{x}_2 \in \cG^{1,2}$, where $\cG^{i,j}$ is the respective $\cG$ for $CF(\ul{L}_{\ol{\bc(\lambda_i)}},\ul{L}_{\ol{\bc(\lambda_j)}})$.
We define $\cB^{i,j}$ similarly.
It implies that if 
$\ul{x}_1 \in \cB^{0,1}$, 
$\ul{x}_2 \in \cB^{1,2} \cup \cG^{1,2}$ (i.e. $\ul{x}_2$ is a basis element in $CF(\ul{L}_{\ol{\bc(\lambda_1)}},\ul{L}_{\ol{\bc(\lambda_2)}})$),
$\ul{x}_0 \in \cB^{0,2} \cup \cG^{0,2}$ and the $\ul{x}_0$-coefficient of $\mu^2(\ul{x}_2,\ul{x}_1)$ is non-zero,
then $\ul{x}_0 \in \cB^{0,2}$.
Similarly, if $\ul{x}_1 \in \cB^{0,1} \cup \cG^{0,1}$, $\ul{x}_2 \in \cB^{1,2}$,
$\ul{x}_0 \in \cB^{0,2} \cup \cG^{0,2}$ and the $\ul{x}_0$-coefficient of $\mu^2(\ul{x}_2,\ul{x}_1)$ is non-zero,
then $\ul{x}_0 \in \cB^{0,2}$.
This precisely says that $I$ is an ideal in $H_{n,m}^{\bc}$.  (We are working in the cohomological algebra, so only need to check closure under $\mu^2$.)

Moreover, we also know that any $u$ contributing to the $\ul{x}_0$-coefficient of $\mu^2(\ul{x}_2,\ul{x}_1)$
has $m-2n$ components of constant triangles.
Therefore, under the identification of  generators \eqref{eq:canonicalGen} in Corollary \ref{c:asGvspace}, 
 the moduli spaces of holomorphic triangles defining the products $\mu^2$ on the  two sides are also canonically identified. The result follows.
\end{proof}

Next, we want to understand the ideal $I$ in terms of the geometric basis of $H_{n,m}^{\bc}$.
Consider the cochain model (cf. \ref{eq:CohBasisCpt}, \eqref{eq:coHc}).
\begin{align}
 H_{n,m}^{\bc}=\oplus_{\lambda_0,\lambda_1 \in \Lambda_{n,m}^c} HF(\ul{L}_{\ol{\bc(\lambda_0)}},\ul{L}_{\widetilde{\bc(\lambda_1)}})=\oplus_{\lambda_0,\lambda_1 \in \Lambda_{n,m}^c} CF(\ul{L}_{\ol{\bc(\lambda_0)}},\ul{L}_{\widetilde{\bc(\lambda_1)}}).
\end{align}
Define $\cI^{arc}_{\lambda_0,\lambda_1}$ to be the vector subspace of $CF(\ul{L}_{\ol{\bc(\lambda_0)}},\ul{L}_{\widetilde{\bc(\lambda_1)}})$
generated by
\begin{align}
 \cB^{arc}_{\lambda_0,\lambda_1}:=\{\ul{x} \in \cX(\ul{L}_{\ol{\bc(\lambda_0)}},\ul{L}_{\widetilde{\bc(\lambda_1)}})| \{1,\dots,m-2n\}+\sqrt{-1} \nsubseteq \pi_E(\ul{x}) \}
\end{align}
Let $\cG^{arc}_{\lambda_0,\lambda_1}:=\cX(\ul{L}_{\ol{\bc(\lambda_0)}},\ul{L}_{\widetilde{\bc(\lambda_1)}}) \setminus \cB^{arc}_{\lambda_0,\lambda_1}$
and define $\cI^{arc}$, $\cB^{arc}$ and $\cG^{arc}$ accordingly.

\begin{lemma}\label{l:algIsom2}
 There is an algebra isomorphism from $\oplus HF(\ul{L}_{\ol{\bc(\lambda_0)}^{\lambda_0,\lambda_1}},\ul{L}_{\ol{\bc(\lambda_1)}^{\lambda_0,\lambda_1}})$
 to $\oplus CF(\ul{L}_{\ol{\bc(\lambda_0)}},\ul{L}_{\widetilde{\bc(\lambda_1)}})$
 sending $I$ to $\cI^{arc}$.
\end{lemma}

\begin{proof}
 Let $\theta_{\lambda_0,\lambda_1} \in HF^0(\ul{L}_{\ol{\bc(\lambda_1)}^{\lambda_0,\lambda_1}},\ul{L}_{\widetilde{\bc(\lambda_1)}})$ be a generator over $\mathbb{Z}$.
 Then, for a sign-consistent choice of $\theta_{\lambda_0,\lambda_1}$, the direct sum of 
 \begin{align}
  \mu^2(\theta_{\lambda_0,\lambda_1},-):HF(\ul{L}_{\ol{\bc(\lambda_0)}^{\lambda_0,\lambda_1}},\ul{L}_{\ol{\bc(\lambda_1)}^{\lambda_0,\lambda_1}}) 
  \to HF(\ul{L}_{\ol{\bc(\lambda_0)}^{\lambda_0,\lambda_1}},\ul{L}_{\widetilde{\bc(\lambda_1)}}) \label{eq:algIsom}
 \end{align}
 over all $(\lambda_0,\lambda_1)$ is an algebra isomorphism.
 
 Without loss of generality, we can assume that $\ol{\bc(\lambda_0)}^{\lambda_0,\lambda_1}=\ol{\bc(\lambda_0)}$.
 Note that
 we can choose $\widetilde{\bc(\lambda_1)}$ such that
 the pair $(\ol{\bc(\lambda_1)}^{\lambda_0,\lambda_1}, \widetilde{\bc(\lambda_1)})$ also satisfies \eqref{eq:LeftCondition}.
 By the same reasoning as in Lemmas \ref{l:independent} and \ref{l:algIsom}, the image of $\cI$ under \eqref{eq:algIsom} is contained in  $\cI^{arc}$.
 Since $\mu^2(\theta_{\lambda_0,\lambda_1},-)$ is an isomorphism on cohomology and
 the dimensions spanned by $I$ and $\cI^{arc}$ in  cohomology are the same, the image of $I$ is precisely $\cI^{arc}$.
\end{proof}

From Lemmas \ref{l:algIsom} and \ref{l:algIsom2}, we have an algebra isomorphism
$\oplus CF(\ul{L}_{\ol{\bc(\lambda_0)}},\ul{L}_{\widetilde{\bc(\lambda_1)}})/\cI^{arc} \simeq H_{n,m}^{\symp}$.
We want to show that this isomorphism respects the geometric basis.

For $\lambda_0, \lambda_1 \in \Lambda_{n,m}^c$, we consider the composition of quasi-isomorphisms
\begin{align}
CF(L_{\ol{\lambda}_0},L_{\tilde{\lambda}_1}) &\to CF(L_{\ol{\lambda}_0},L_{\ol{\lambda}_1})  
\to  CF(L_{\ol{\bc(\lambda_0)}},L_{\ol{\bc(\lambda_1)}})/\cI_{\lambda_0,\lambda_1} 
\to   CF(L_{\ol{\bc(\lambda_0)}},L_{\widetilde{\bc(\lambda_1)}})/\cI^{\alg}_{\lambda_0,\lambda_1} \label{eq:ComposeAll}
\end{align}
The first arrow is given by $\mu^2(\theta,-)$ for a continuation element $\theta \in CF(L_{\tilde{\lambda}_1},L_{\ol{\lambda}_1})$.
The second arrow is the inverse of the chain isomorphism \eqref{eq:canonicalChain}.
The last arrow is \eqref{eq:algIsom}, which is given by $\mu^2(\theta',-)$ for a continuation element $\theta' \in CF^0(L_{\ol{\bc(\lambda_1)}},L_{\widetilde{\bc(\lambda_1)}})$.

We can assume that $\ol{\lambda}_0$, $\tilde{\lambda}_1$ and $\ol{\lambda}_1$
are obtained, respectively, by removing the leftmost $m-2n$ matching paths of $\ol{\bc(\lambda_0)}$ ,$\widetilde{\bc(\lambda_1)}$ and $\ol{\bc(\lambda_1)}$, and translating by $-(m-2n)$.
Let $\theta^{-1}$ be a quasi-inverse of $\theta$.
We can choose $\theta' $ to be image of $\theta^{-1}$ under the canonical inclusion
$ CF(L_{\ol{\lambda}_1},L_{\tilde{\lambda}_1}) \to CF(L_{\ol{\bc(\lambda_1)}},L_{\widetilde{\bc(\lambda_1)}})$
In this case, the composition \eqref{eq:ComposeAll} coincides with the canonical
cochain isomorphism 
\begin{align}
 CF(L_{\ol{\lambda}_0},L_{\tilde{\lambda}_1}) &\to CF(L_{\ol{\bc(\lambda_0)}},L_{\widetilde{\bc(\lambda_1)}})/\cI^{\alg}_{\lambda_0,\lambda_1}\\
 \ul{x} &\mapsto \ul{y} \nonumber
\end{align}
characterised by $\pi_E(\ul{y})=\{1,\dots,m-2n\} \cup (\pi_E(\ul{x})+m-2n)$ (i.e. essentially the same map as in \eqref{eq:canonicalGen}, \eqref{eq:canonicalChain}).
As a consequence, there is an algebra isomorphism from $H_{n,m}^{\symp}$ to  $H_{n,m}^{\bc}/I$ respecting the geometric basis.

For the combinatorial arc algebra, we have
exactly the same quotient description for $H_{n,m}^{\alg}$ (see \eqref{eq:bc}).
To conclude, we have

\begin{proposition}\label{p:CompactCases}
The isomorphism $\Phi$ of Lemma \ref{l:gvspIsom} restricts to an algebra isomorphism from $H_{n,m}^{\symp}$ to $H_{n,m}^{\alg}$.
\end{proposition}

\subsection{All cases}

We show that $\Phi$ is an algebra isomorphism from $K^{\symp}_{n,m}$ to $K_{n,m}^{\alg}$ for all $n,m$ in this section.
The strategy is similar to the previous section.

Let $\be: \Lambda_{n,m} \to \Lambda_{n,m+n}^c$ be the inclusion given by \eqref{eq:beFull}.
Let
\begin{align}
 H^{\be}_{n,m}:=\oplus_{\lambda_0,\lambda_1 \in \Lambda_{n,m}} HF(\ul{L}_{\be(\lambda_0)},\ul{L}_{\be(\lambda_1)})
\end{align}
which is a subalgebra of $H^{\symp}_{n,m+n}=H_{n,m+n}^{\alg}$ so in particular, $H^{\be}_{n,m}$ is formal.
We want to compare $K^{\symp}_{n,m}$ with $H^{\be}_{n,m}$.

For each ordered pair $(\lambda_0, \lambda_1)$ such that $\lambda_0,\lambda_1 \in \Lambda_{n,m}$, 
we choose upper matchings $\ol{\be(\lambda_0)}$ and $\ol{\be(\lambda_1)}$
such that
\begin{align}
 &\text{for $a=1,\dots,n$, the slope at $m+a+\sqrt{-1}$ of the matching paths in $\ol{\be(\lambda_0)}$ that ends} \label{eq:RightCondition} \\
& \text{at $m+a+\sqrt{-1}$ is more negative than that of $\ol{\be(\lambda_1)}$}   \nonumber
\end{align}
Since the choices of matching depend on the ordered pair of weights, we denote the matchings by $\ol{\be(\lambda_0)}^{\lambda_0,\lambda_1}$
and $\ol{\be(\lambda_1)}^{\lambda_0,\lambda_1}$ respectively.
Let $J_{\lambda_0,\lambda_1}$ be the subspace of $CF(\ul{L}_{\ol{\be(\lambda_0)}^{\lambda_0,\lambda_1}},\ul{L}_{\ol{\be(\lambda_1)}^{\lambda_0,\lambda_1}})$
generated by
\begin{align}
 \cB_{\lambda_0,\lambda_1}:=\{\ul{x} \in \cX(\ul{L}_{\ol{\be(\lambda_0)}^{\lambda_0,\lambda_1}},\ul{L}_{\ol{\be(\lambda_1)}^{\lambda_0,\lambda_1}})| m+a+\sqrt{-1} \in \pi_E(\ul{x}) \text{ for some }a=1\dots,n\} \label{eq:IdealGen2}
\end{align}
Let $\cG_{\lambda_0,\lambda_1}:=\cX(\ul{L}_{\ol{\be(\lambda_0)}^{\lambda_0,\lambda_1}},\ul{L}_{\ol{\be(\lambda_1)}^{\lambda_0,\lambda_1}}) \setminus \cB_{\lambda_0,\lambda_1}$
and define $\cJ= \oplus \cJ_{\lambda_0,\lambda_1}$, $\cB:=\oplus \cB_{\lambda_0,\lambda_1}$ and 
$\cG=\oplus \cG_{\lambda_0,\lambda_1}$.

\begin{lemma}\label{l:mu1preserveI2}
 $\mu^1(\cJ_{\lambda_0,\lambda_1}) \subset \cJ_{\lambda_0,\lambda_1}$ so $\cJ_{\lambda_0,\lambda_1}$ descends to a vector subspace of $H_{n,m}^{\be}$, which we denote by $J_{\lambda_0,\lambda_1}$.
\end{lemma}

\begin{proof}
 
Let $\ul{x}_1 \in \cB$. We want to show that if the $\ul{x}_0$-coefficient of $\mu^1(\ul{x}_1)$ is non-zero then $\ul{x}_0 \in \cB$.

 By definition, $\ul{x}_1 \in \cG$ implies that $m+a+\sqrt{-1} \subset \pi_E(\ul{x}_1)$ for some $a=1,\dots,n$.
 Let $v:\Sigma \to E$ be a $J$-holomorphic map 
 contributing to the $\ul{x}_0$-coefficient of $\mu^1(\ul{x}_1)$  such that near an input puncture $\xi$, $v$ is asymptotic to the  element of $\ul{x}_1$ lying above $m+a+\sqrt{-1}$.
 We project $v$ to $\bH^\circ$ by $\pi_E$ and apply the open mapping theorem.
 The condition \eqref{eq:RightCondition} at $m+a+\sqrt{-1}$ forces at least one of the boundary components of $\Sigma$
 adjacent to $\xi$ to be mapped constantly to $m+a+\sqrt{-1}$ by $\pi_E \circ v$.
 It immediately implies that $m+a+\sqrt{-1} \in \pi_E(\ul{x}_0)$.
\end{proof}

\begin{corollary}\label{c:asGvspace2}
 $HF(\ul{L}_{\be(\lambda_0)},\ul{L}_{\be(\lambda_1)})/J_{\lambda_0,\lambda_1}=HF(\ul{L}_{\lambda_0},\ul{L}_{\lambda_1})$ as a graded vector space.
\end{corollary}

\begin{proof}
 For $j=0,1$, by forgetting the points $\{1,\dots,n\}+m+\sqrt{-1}$ and
`extending' the matching paths of $\be(\lambda_j)^{\lambda_0,\lambda_1}$ that contain $m+a+\sqrt{-1}$ for some $a=1,\dots,n$ to thimble paths, we get an upper matching $\ol{\lambda_j}$ of $\lambda_j$.
 There is then an obvious bijective correspondence
 \begin{align}
  \cG_{\lambda_0,\lambda_1} \simeq \cX(\ul{L}_{\ol{\lambda_0}},\ul{L}_{\ol{\lambda_1}}) \label{eq:canonicalGen2}
 \end{align}
 
Moreover, there is a canonical isomorphism between the chain complexes
 \begin{align}
  CF(\ul{L}_{\ol{\be(\lambda_0)}^{\lambda_0,\lambda_1}},\ul{L}_{\ol{\be(\lambda_1)}^{\lambda_0,\lambda_1}})/\cJ_{\lambda_0,\lambda_1} \cong CF(\ul{L}_{\ol{\lambda_0}},\ul{L}_{\ol{\lambda_1}}) \label{eq:canonicalChain2}
 \end{align}
 which on generators is given by \eqref{eq:canonicalGen2};  the differentials on the left and right agree by the proof of Lemma \ref{l:mu1preserveI2}.
\end{proof}

\begin{lemma}\label{l:independent2}
 The subspace $J_{\lambda_0,\lambda_1}$ is independent of the choices of $\ol{\be(\lambda_0)}^{\lambda_0,\lambda_1}$
and $\ol{\be(\lambda_1)}^{\lambda_0,\lambda_1}$, provided \eqref{eq:RightCondition} is satisfied.
\end{lemma}

\begin{proof}
 Let $\be(\lambda_0)^{a}$ and $\be(\lambda_0)^{b}$ be two different choices of $\ol{\be(\lambda_0)}^{\lambda_0,\lambda_1}$ such that 
 \eqref{eq:RightCondition} is satisfied for the pairs $(\be(\lambda_0)^{a}, \ol{\be(\lambda_1)}^{\lambda_0,\lambda_1})$
 and $(\be(\lambda_0)^{b},\ol{\be(\lambda_1)}^{\lambda_0,\lambda_1})$.
 
 There exists another choice $\be(\lambda_0)^{c}$ of $\ol{\be(\lambda_0)}^{\lambda_0,\lambda_1}$
 such that \eqref{eq:LeftCondition} is satisfied for the pairs $(\be(\lambda_0)^{c},\be(\lambda_0)^{a})$
 and $(\be(\lambda_0)^{c},\be(\lambda_0)^{b})$.
 Let $\theta$ be a continuation element of $HF(\ul{L}_{\be(\lambda_0)^{c}},\ul{L}_{\be(\lambda_0)^{a}})$.
 
 Let $J^a$ and $J^c$ be the respective $J_{\lambda_0,\lambda_1}$ for $HF(\ul{L}_{\be(\lambda_0)^{a}},\ul{L}_{\ol{\be(\lambda_1)}^{\lambda_0,\lambda_1}})$ and 
 $HF(\ul{L}_{\be(\lambda_0)^{c}},\ul{L}_{\ol{\be(\lambda_1)}^{\lambda_0,\lambda_1}})$.
 We need to show that the isomorphism
 \begin{align}
  \mu^2(-,\theta): HF(\ul{L}_{\be(\lambda_0)^{a}},\ul{L}_{\ol{\be(\lambda_1)}^{\lambda_0,\lambda_1}}) \to HF(\ul{L}_{\be(\lambda_0)^{c}},\ul{L}_{\ol{\be(\lambda_1)}^{\lambda_0,\lambda_1}})
 \end{align}
 sends $J^a$ to $J^c$.
 By Corollary \ref{c:asGvspace2} and a dimension count, it suffices to show that the image of $J^a$ under $\mu^2(-,\theta)$ 
 is contained in $J^c$.
 
 The proof that $\mu^2(J^a,\theta) \subset J^c$ is similar to the proof of Lemma \ref{l:mu1preserveI2}.
 Let $\cB^a$ and $\cB^c$ be the respective $\cB_{\lambda_0,\lambda_1}$ for $CF(\ul{L}_{\be(\lambda_0)^{a}},\ul{L}_{\ol{\be(\lambda_1)}^{\lambda_0,\lambda_1}})$ and 
 $CF(\ul{L}_{\be(\lambda_0)^{c}},\ul{L}_{\ol{\be(\lambda_1)}^{\lambda_0,\lambda_1}})$.
 By slight abuse of notation, we denote a chain level lift of $\theta$ to $CF(\ul{L}_{\be(\lambda_0)^{c}},\ul{L}_{\be(\lambda_0)^{a}})$ by $\theta$.
 It suffices to show that if 
$\ul{x}_1 \in \cB^a$ and
the $\ul{x}_0$-coefficient of $\mu^2(\ul{x}_1,\theta)$ is non-zero,
then $\ul{x}_0 \in \cB^c$.
 
 Let $v:\Sigma \to E$ be a $J$-holomorphic map 
 contributing to the $\ul{x}_0$-coefficient of $\mu^2(\ul{x}_1,\theta)$  such that near an input puncture $\xi$, $v$ is asymptotic to the element of $\ul{x}_1$ lying above $m+a+\sqrt{-1}$.
 We project $v$ to $\bH^\circ$ by $\pi_E$ and apply the open mapping theorem.
Since $(\be(\lambda_0)^{a}, \ol{\be(\lambda_1)}^{\lambda_0,\lambda_1})$ satisfies \eqref{eq:RightCondition},
 it forces at least one of the boundary components of $\Sigma$
 adjacent to $\xi$ to be mapped constantly to $m+a+\sqrt{-1}$ by $\pi_E \circ v$.
 Let the two boundary components of $\Sigma$
 adjacent to $\xi$ be $\partial^i \Sigma$ and $\partial^{j} \Sigma$, with Lagrangian labels $\ul{L}_{\be(\lambda_0)^{a}}$ and $\ul{L}_{\ol{\be(\lambda_1)}^{\lambda_0,\lambda_1}}$, respectively.

If $\pi_E \circ v |_{\partial^j \Sigma}=m+a+\sqrt{-1}$, then we are done because we must have $m+a+\sqrt{-1} \in \pi_E(\ul{x}_0)$. 

If $\pi_E \circ v |_{\partial^i \Sigma}=m+a+\sqrt{-1}$, then we denote the next boundary component of $\Sigma$
adjacent to $\partial^i \Sigma$ by $\partial^{i'} \Sigma$, which is equipped with the Lagrangian label $\ul{L}_{\be(\lambda_0)^{c}}$.
Using the fact that
$(\be(\lambda_0)^{c}, \ol{\be(\lambda_1)}^{\lambda_0,\lambda_1})$ satisfies \eqref{eq:RightCondition},
we conclude that either $\partial^j \Sigma$ or $\partial^{i'} \Sigma$ is mapped to $m+a+\sqrt{-1}$ under $\pi_E \circ v$.
In either case, we have $m+a+\sqrt{-1} \in \pi(\ul{x}_0)$.
\end{proof}

By very similar arguments, the analogues of Lemma \ref{l:algIsom} and \ref{l:algIsom2} hold.
In particular, we obtain algebra isomorphisms
\begin{align}
K_{n,m}^{\symp} \simeq H^{\be}_{n,m}/J \simeq \oplus CF(\ul{L}_{\ol{\be(\lambda_0)}},\ul{L}_{\widetilde{\be(\lambda_1)}})/\cJ^{arc}
\end{align}
where $\cJ^{arc}$ is generated by \eqref{eq:IdealGen2}.
There is also exactly the same quotient description for $K_{n,m}^{\alg}$ (see \eqref{eq:be}).
As a result, the analog of Proposition \ref{p:CompactCases} is also true.

\begin{proposition}\label{p:AllCases}
The map $\Phi$ of Lemma \ref{l:gvspIsom} gives an algebra isomorphism from $K_{n,m}^{\symp}$ to $K_{n,m}^{\alg}$.
\end{proposition}

\section{An nc-vector field}\label{s:ncField}

The algebra $K_{n,m}^{\symp}$ in principle carries a non-trivial $A_{\infty}$-structure. Seidel gave a necessary and sufficient criterion for formality of an $A_{\infty}$-algebra: it should admit a degree one Hochschild cohomology class which acts by the Euler field, see Theorem \ref{t:seidel-formality} (a proof is given in \cite{AbouzaidSmith16}).  Following the language of non-commutative geometry, we call a degree one Hochschild cocycle a \emph{non-commutative vector field} or nc-vector field: the motivating example is a global vector field $V \in H^0(\mathcal{T}_Z) \subset HH^1(D^b(\mathrm{Coh}(Z))$ on an algebraic variety $Z$.  
In this section, we apply the method from \cite{AbouzaidSmith16} to construct an nc-vector field by counting holomorphic discs with prescribed conormal-type conditions at infinity.
This will be well-defined even if $\ul{L}$ has some non-compact Lagrangian components (see Remark \ref{r:ncCom}).

\subsection{Moduli spaces of maps revisited}

We first explain how to modify the moduli spaces of maps used in \cite{AbouzaidSmith16} to construct an nc-vector field in our setting. 
Let $\cR^{d+1,h}_{(0,1)}$ be the moduli space of unit discs $S$ with $d+1$ boundary punctures $\{\xi_i\}_{i=0}^d$, 
$h$ ordered interior marked points $\mk(S)$ and $2$ more {\bf distinguished} ordered interior marked points $s_0,s_1$
such that $s_1$ lies on the hyperbolic geodesic arc between $s_0$ and $\xi_0$ (but $s_1 \neq s_0$ and $s_1 \neq \xi_0$).
We have $\dim(\cR^{d+1,h}_{(0,1)})=\dim(\cR^{d+1,h})+3$.
We call the $h$ ordered interior marked points type-1 interior marked points, and the $2$ distinguished ordered interior marked points type-2 interior marked points.
We denote $\mk(S) \cup \{s_0,s_1\}$ by $\mk(S)^+$.

We can compactify $\cR^{d+1,h}_{(0,1)}$ to $\overline{\cR}^{d+1,h}_{(0,1)}$ as in \cite[Section 3.6]{AbouzaidSmith16} (which treats the case $h=0$).
Informally, the compactification includes stable broken configurations as considered in \cite[Section 3.6]{AbouzaidSmith16}, with an extra condition
that nodal sphere components arise when some of interior marked points collide.
In particular, the codimension one boundary facets of $\overline{\cR}^{d+1,h}_{(0,1)}$ are given by 
(in all the moduli below, the subset $P \subset \{1,\dots,h\}$ remembers the type-1 ordered marked points that go to the same component)

\begin{enumerate}
 \item ($\{\xi_{i+1}, \dots, \xi_{i+j}$ move together) A nodal domain in which a collection of input boundary punctures bubble off:
 \begin{align}
  \coprod_{1 \le j \le d, 0\le i \le d-j, h_1+h_2=h, P \subset \{1,\dots,h\}, |P|=h_1} \cR^{d-j+1,h_1}_{(0,1)} \times \cR^{j+1,h_2} \label{eq:c1}
 \end{align}
 When $j=1$, we need $h_2>0$ so that the domain is stable.
 \item ($s_1 \to s_0$) A domain with a sphere bubble carrying the two type-2 interior marked points and some type-1 interior marked points attached to a disc carrying the remaining type-1 interior marked
 points and the $d+1$ boundary punctures.
 Letting $\eM_{0,3}^{h}$ denote the moduli space of spheres with $1$ interior node, $2$ type-2 marked points 
 and $h$ unordered type-1 interior marked points, and $\cR^{d+1,h}_1$ denote the moduli space of unit discs with $1$ interior node, $d+1$ boundary punctures and
$h$ type-1 interior marked points, this component is
\begin{align}
 \coprod_{h_1+h_2=h, P \subset \{1,\dots,h\}, |P|=h_1} \cR^{d+1,h_1}_1 \times \eM_{0,3}^{h_2} \label{eq:c2}
\end{align}
The attaching point is understood to be the node on the disc and on the sphere respectively.
 \item ($\{\xi_{d-l+1}, \dots,\xi_d,\xi_1,\dots, \xi_{i} \} \to \xi_0$) A nodal domain with two discs, one carrying both 
 type-2 interior marked points, some type-1 interior marked points and the boundary punctures $\{\xi_{i+1}, \dots, \xi_{d-l}\}$, the other carrying the remaining  type-1
 interior marked points and boundary punctures
 \begin{align}
  \coprod_{0 \le i+l \le d, h_1+h_2=h, P \subset \{1,\dots,h\}, |P|=h_1} \cR^{i+l+1,h_1} \times \cR^{d-i-l+1,h_2}_{(0,1)} \label{eq:c3}
 \end{align}
 When $i+l=0$, we need $h_1>0$ so that the domain is stable.
 \item ($\{s_1\}\cup \{\xi_{1}, \dots,\xi_i,\xi_{d-l+1},\dots, \xi_{d} \} \to \xi_0$)  A nodal 
 domain with two discs, one carrying $d-i-l$ input boundary punctures, one type-2 interior marked point and some type-1 interior marked points, the bubble carrying the second type-2 interior marked point,
 and the remaining type-1 interior marked points and boundary punctures:
 \begin{align}
  \coprod_{0 \le i+l \le d, h_1+h_2=h, P \subset \{1,\dots,h\}, |P|=h_1} \cR^{(i+l+1)+1,h_1}_1 \times \cR^{d-i-l+1,h_2}_{1} \label{eq:c4}
 \end{align}

\end{enumerate}


We pick a consistent choice of strip-like ends and marked-points neighborhoods for elements in $\overline{\cR}^{d+1,h}_{(0,1)}$ as in Section \ref{sss:EndsAndNeighborhoods}.
This time, we require each marked-points neighborhood to contain $s_0$ and $s_1$, and denote it by $\nu(\mk(S)^+)$.
For each $S \in \cR^{d+1,h}_{(0,1)}$, we can define $\cG_{\aff}(S)$ by \eqref{eq:GaugeS} but with $\nu(\mk(S))$ being replaced by $\nu(\mk(S)^+)$.
For a cylindrical Lagrangian label associating $\partial_j S$ to $\ul{L}_j$ for $j=0,\dots,d$
and a choice $A_S \in \cP_\aff(S, \{(A_j, \lambda_{\ul{L}_{j-1}}, \lambda_{\ul{L}_{j}})\}_{j=0}^d)$, we equip  $S$ with the additional data $(J,K)$ as in  \eqref{eq:FloerDataJ}
and \eqref{eq:FloerDataK}, again with $\nu(\mk(S))$ replaced by $\nu(\mk(S)^+)$. 

Let $D_0 \subset \Hilb^n(\overline{E})$ is the divisor of ideals whose support meets $D_E$.  We now assume that 
\begin{align}
 & D_0 \text{ is moveable, and the base locus of its linear system contains no rational curve, and} \label{eqn:moveable}\\
 & \text{there is a holomorphic volume form with simple poles on $D_0$ (so $c_1(\Hilb^n(\overline{E}))=PD(D_0)$)} \label{eq:simplepole}
 \end{align}
Let $D_0'$ be a divisor in $\Hilb^n(\overline{E})$ linearly equivalent to, but sharing no common irreducible component with, $D_0$.  
Let $B_0=D_0 \cap D_0'$ be the base locus of the corresponding pencil; 
we assume that it contains no rational curve. Note that these conditions will hold in the case of type $A$ Milnor fibres  (see \cite[Section 6]{AbouzaidSmith16}), and 
more generally \eqref{eqn:moveable} holds when $D_E \subset \overline{E}$ is moveable
and \eqref{eq:simplepole} holds when there is a holomorphic volume form on $\overline{E}$ with simple poles on $D_E$ (see \cite[Lemma 6.3]{AbouzaidSmith16}).

Given $\ul{x}_0 \in  \cX(H_0, \ul{L}_0,\ul{L}_d)$ and $\ul{x}_j \in \cX(H_j, \ul{L}_{j-1},\ul{L}_j)$ for $j=1,\dots,d$,
we define 
\begin{align}
\cR^{d+1,h}_{(0,1)}(\ul{x}_0;\ul{x}_d, \dots, \ul{x}_1) \label{eq:R(0,1)}
\end{align}
to be the moduli of all maps $u :S \to \Hilb^n(\overline{E})$ such that
\begin{align}\label{eq:FloerEquationNC}
    \left\{
\begin{array}{lll}
      \text{$[u] \cdot [D_r]=0$ and $[u] \cdot [D_0]=1$} \\
      u(\mk(S)) \subset D_{HC}, u(s_0) \in D_0, u(s_1) \in D_0' \\
      (Du|_z-X_{K}|_z)^{0,1}=0 \text{ with respect to } (J_z)_{u(z)} \text{ for } z \in S\\
      u(z) \in \Sym(\ul{L}_j) \text{ for } z \in \partial_j S \\
      \lim_{s \to \pm \infty}u( \epsilon^j(s,\cdot))=\ul{x}_j(\cdot) \text{ uniformly} \\
\end{array}
      \right.
\end{align}
where $J,K$ should be understood as their extension to $\Conf^n(\overline{E})$ (see Remark \ref{r:JKextension}).

Note that Lemma \ref{l:IPairing} and \ref{l:PositivityIntersection} remain true for $u \in \cR^{d+1,h}_{(0,1)}(\ul{x}_0;\ul{x}_d, \dots, \ul{x}_1)$
so when $h=I_{\ul{x}_0;\ul{x}_d, \dots, \ul{x}_1}$, \eqref{eq:FloerEquationNC}
implies that $Im(u) \cap D_r=\emptyset$ and $u$ intersect $D_{HC}$ transversally.
On the other hand, since $J=J^{[n]}_E$ near $s_0,s_1$, $[u] \cdot [D_0]=1$ implies that the intersection multiplicity of $u$ with $D_0$ at $s_0$ 
and with $D_0'$ at $s_1$  are both $1$. (We will see later, in the proof of Lemma \ref{l:posInt}, that every intersection between $u$ and $D_0$ contributes positively to their intersection number.)

We want to discuss the regularity and compactification of $\cR^{d+1,h}_{(0,1)}(\ul{x}_0;\ul{x}_d, \dots, \ul{x}_1)$ following 
Section \ref{sss:FloerData}, \ref{sss:Energy}, \ref{sss:Compactness} and \cite{AbouzaidSmith16}.

First, we pick a consistent choice of $(A_S,J,K)$ for all elements in $\overline{\cR}^{d+1,h}_{(0,1)}$ and all possible cylindrical Lagrangian labels and in/outputs.
In particular, this implies that if $S \in \overline{\cR}^{d+1,h}_{(0,1)}$
has a sphere component, then over that component $K \equiv 0$ and the complex structure is $J_{\overline{E}}^{[n]}$.
We can then introduce the following moduli  spaces of maps that are relevant for the compactification of $\cR^{d+1,h}_{(0,1)}(\ul{x}_0;\ul{x}_d, \dots, \ul{x}_1)$:
\begin{align}
 &\cR^{d+1,h}_1(\Hilb^n(\overline{E}), (d_0,d_r)| \ul{x}_0; \ul{x}_d,\dots,\ul{x}_1) \label{eq:R1bar}\\
 &\eM^h(\Hilb^n(\overline{E}) \setminus D_r|1) \label{eq:GW}
\end{align}


The moduli \eqref{eq:R1bar} consists of maps $u: S \to \Hilb^n(\overline{E})$ such that $S \in \cR_1^{d+1,h}$, and 
\begin{align}\label{eq:FloerEquationR1}
    \left\{
\begin{array}{lll}
      \text{$[u] \cdot [D_0]=d_0$, $[u] \cdot [D_r]=d_r$} \\
      u(\mk(S)) \subset D_{HC}  \\
      (Du|_z-X_{K}|_z)^{0,1}=0 \text{ with respect to } (J_z)_{u(z)} \text{ for } z \in S\\
      u(z) \in \Sym(\ul{L}_j) \text{ for } z \in \partial_j S \\
      \lim_{s \to \pm \infty}u( \epsilon^j(s,\cdot))=\ul{x}_j(\cdot) \text{ uniformly.} \\
\end{array}
      \right.
\end{align}
It has an evaluation map using the node to $\Hilb^n(\overline{E})$.
The virtual dimension of \eqref{eq:R1bar} is
\begin{align}
 d+2d_0+|\ul{x}_0|-\sum_{j=1}^d |\ul{x}_j|
\end{align}
where $2d_0$ comes from $c_1(\Hilb^n(\overline{E}))=PD(D_0)$ and $d$ comes from $\dim(\cR^{d+1}_1)$,

For $S \in \eM_{0,3}^h$, we can rigidify the domain and assume the type-2 marked points are $0,1$ and the node is $\infty$.
The moduli \eqref{eq:GW} consists of maps $u: S \to \Hilb^n(\overline{E}) \setminus D_r$ such that $S \in \eM_{0,3}^h$,
\begin{align}\label{eq:FloerEquationM}
    \left\{
\begin{array}{lll}
      \text{$[u] \cdot [D_0]=1$} \\ 
      u(\mk(S)) \subset D_{HC}, u(0) \in D_0, u(1) \in D_0' , \\
      \text{$u$ is $J^{[n]}_E$-holomorphic} \\
\end{array}
      \right.
\end{align}
It also has an evaluation map to $\Hilb^n(\overline{E})$ by evaluating at $\infty$.
The virtual dimension of \eqref{eq:GW} is
\begin{align}
 4n+2c_1([u])+2h-(4+2h)=4n-2
\end{align}
where $4n$ is the dimension of $\Hilb^n(\overline{E})$, $c_1([u])=[u] \cdot [D_0]=1$, $2h$ is the dimension of $\eM_{0,3}^h$ and $4+2h$ comes from the incidence conditions.

\begin{proposition}[cf. Lemma 3.18, 3.19 of \cite{AbouzaidSmith16}]\label{p:ncCompactification}
 Let $h=I_{\ul{x}_0;\ul{x}_d, \dots, \ul{x}_1}$.
 For generic consistent choice of $J,K$ satisfying \eqref{eq:FloerDataJ} and \eqref{eq:FloerDataK},
 every $u \in \cR^{d+1,h}_{(0,1)}(\ul{x}_0;\ul{x}_d, \dots, \ul{x}_1)$ is regular.
 Moreover, when $\cR^{d+1,h}_{(0,1)}(\ul{x}_0;\ul{x}_d, \dots, \ul{x}_1)$ is $0$ dimensional, it is compact.
 When $\cR^{d+1,h}_{(0,1)}(\ul{x}_0;\ul{x}_d, \dots, \ul{x}_1)$ is $1$ dimensional, it can be compactified by adding the following boundary strata, which are themselves regular.
 \begin{enumerate}
  \item (corresponding to \eqref{eq:c1}, \eqref{eq:c3})
  \begin{align}
   &\coprod_{\ul{x}, h_1+h_2=h, P \subset \{1,\dots,h\}, |P|=h_1}  \cR^{d-j+1,h_1}_{(0,1)}(\ul{x}_0;\ul{x}_d, \dots, \ul{x}_{i+j+1},\ul{x},\ul{x}_i,\dots, \ul{x}_1) 
   \times \cR^{j+1,h_2}(\ul{x};\ul{x}_{i+j}, \dots, \ul{x}_{i+1}) \label{eq:Cnc1}\\
   &\coprod_{\ul{x}, h_1+h_2=h, P \subset \{1,\dots,h\}, |P|=h_1}  \cR^{i+l+1,h_1}(\ul{x}_0;\ul{x}_d, \dots, \ul{x}_{d-l+1},\ul{x},\ul{x}_i,\dots, \ul{x}_1) 
   \times \cR^{d-i-l+1,h_2}_{(0,1)}(\ul{x};\ul{x}_{d-j}, \dots, \ul{x}_{i+1}) \label{eq:Cnc2}
  \end{align}
 \item (corresponding to \eqref{eq:c2})
 \begin{align}
  & \cR^{d+1,h}_1(\Hilb^n(\overline{E}), (0,0)| \ul{x}_0; \ul{x}_d,\dots,\ul{x}_1) \times_{\Hilb^n(\overline{E})} \eM^{0}(\Hilb^n(\overline{E}) \setminus D_r|1) \label{eq:Cnc3} \\
  & \cR^{d+1,h}_1(\Hilb^n(\overline{E}), (1,0)| \ul{x}_0; \ul{x}_d,\dots,\ul{x}_1) \times_{\Hilb^n(\overline{E})} B_0 \label{eq:Cnc4}
 \end{align}
 \end{enumerate}
where in \eqref{eq:Cnc3} and \eqref{eq:Cnc4} the fiber products are taken with respect to the evaluation maps, which are transversal.
\end{proposition}

We will see that $\eM^{h}(\Hilb^n(\overline{E}) \setminus D_r|1)$ is empty when $h>0$ (see Lemma \ref{l:chern1spheres} and Corollary \ref{c:pseudocycle})
which explains why it does not appear in \eqref{eq:Cnc3}. 
In the rest of this section, we establish some regularity and compactness statements and prove Proposition \ref{p:ncCompactification}.

\begin{remark}\label{r:ncCom}
 We will establish the compactness result below following the method in Section \ref{sss:Compactness}. The reason one should expect compactness to hold
 is because we are counting discs with interior marked points going to $D_0$ which corresponds to the vertical infinity of $E$, while non-compact Lagrangian components
 in $\ul{L}$ are only non-compact with respect to the horizontal infinity of $E$, so the method  in Section \ref{sss:Compactness} applies without substantial changes.
\end{remark}

\subsubsection{Moduli of Chern one spheres}

We first discuss $\eM^h(\Hilb^n(\overline{E}) \setminus D_r|1)$.
Let $C$ be a Chern number $1$ rational curve in $\overline{E}$ and $q_2,\dots,q_n \in \overline{E}$ be $n-1$ pairwise distinct points
such that at most one of $q_2,\dots,q_n$ lies in $C$.
The product determines a Chern number $1$ rational curve in $\Sym^n(\overline{E})$ which meets $\Delta_{\overline{E}}$ at at most one point.
Therefore, it uniquely lifts to an irreducible Chern number $1$ rational curve $\tilde{C}$ in $\Hilb^n(\overline{E})$.
We call $\tilde{C}$ a Chern number $1$ rational curve of product type.
In fact, these are essentially all the irreducible Chern number $1$ rational curves in $\Hilb^n(\overline{E})$.

\begin{lemma}[Lemma 6.4 of \cite{AbouzaidSmith16}]\label{l:chern1spheres}
 Let $C$ be an irreducible rational curve in $\Hilb^n(\overline{E})$ with Chern number $1$ and not contained in $D_{HC}$.
 Then $C$ is of product type.
\end{lemma}

\begin{proof}
 By assumption, $\pi_{HC}(C)$ is an irreducible Chern number $1$ rational curve in $\Sym^n(\overline{E})$.
 By the tautological correspondence, we get a (possibly disconnected) closed complex curve $C'$ in $\overline{E}$ of Chern number $1$.
 It means that $C'$ lies inside a finite union of fibers of $\pi_{\overline{E}}$.
 Having Chern number $1$ means that $[C'] \cdot [D_E]=1$ so $C'$ is a union of a Chern number $1$ rational curve $C''$ and $n-1$ points in $\overline{E}$.
 By the assumption that $C$ is not contained in $D_{HC}$ (or equivalently $\pi_{HC}(C)$ is not contained in $\Delta_{\overline{E}}$), we know that the points are pairwise disjoint
 and at most one of them lies inside $C''$.
\end{proof}

\begin{corollary}[Lemma 6.4 of \cite{AbouzaidSmith16}]\label{c:pseudocycle}
 The moduli $\eM^h(\Hilb^n(\overline{E}) \setminus D_r|1)$ is regular and 
 the evaluation map $\eM^h(\Hilb^n(\overline{E}) \setminus D_r|1) \to \Hilb^n(\overline{E})$ defines a pseudocycle.
 In other words, the image of the evaluation map can be compactified by adding a real codimension two (with respect to the image) subset.
\end{corollary}

\begin{proof}
When $h>0$, by Lemma \ref{l:chern1spheres}, $\eM^h(\Hilb^n(\overline{E}) \setminus D_r|1)$ is actually empty so there is nothing to prove.

When $h=0$, the regularity of $\eM^h(\Hilb^n(\overline{E}) \setminus D_r|1)$ follows from 
 the explicit description of Lemma \ref{l:chern1spheres} and
 the fact that every direct summand of the normal bundle of a product-type Chern one rational curve in 
 $\Hilb^n(\overline{E}) \setminus D_r$ has Chern number $\ge -1$.
 Therefore, automatic regularity for these somewhere injective curves applies.
 
 To define a pseudocycle, the codimension $2$ subset to be added is the union of the image
 of \emph{stable} Chern number $1$ rational curves in $\Hilb^n(\overline{E})$ that meet $D_r$ non-trivially, which is denoted by $B_r$ in \cite{AbouzaidSmith16}.
 It is of codimension $2$ because every such stable Chern number $1$ rational curve maps to a point in $\Sym^n(\bH^{\circ})$ that lies inside $\Delta_{\bH^{\circ}} \cap D'$ where $D'$
 is the divisor in $\Sym^n(\bH^{\circ})$ consisting of subschemes of $\bH^{\circ}$ whose support meets critical values of $\pi_{E}$.
\end{proof}

We denote the pseudocycle by $GW_1$.

\subsubsection{Moduli of discs}


The following regularity statement follows as in Lemma \ref{l:regularity}.

 \begin{lemma}[Regularity]
For generic consistent choice of $(J,K)$ satisfying \eqref{eq:FloerDataJ}, \eqref{eq:FloerDataK}, 
every element in the moduli spaces \eqref{eq:R(0,1)} and \eqref{eq:R1bar} is regular.
\end{lemma}

For an element $u$ in the moduli space \eqref{eq:R(0,1)} or \eqref{eq:R1bar}, we can define $E(u)$ by \eqref{eq:DefnEnergy}.
Even though $\overline{E}$ is not exact, we can still define $\omega^{geom}_K$ and $\omega^{top}_K$ on $\Conf^n(\overline{E}) \times S$.
They differ by $R_K$ which, by the same reasoning as in Lemma \ref{l:boundGeoEnergy2}, is uniformly bounded independent of $u$.
Moreover, the integration of $\omega^{top}_K$ over the graph of $u$ is bounded \emph{a priori} by the Lagrangian boundary conditions and the intersection numbers with $D_0,D_r$, because
$\Hilb^n(\overline{E}) \setminus (D_0 \cup D_r)$ is exact.
To conclude, we have

\begin{lemma}[Energy]
 For a fixed choice of $\{\ul{x}_j\}_{j=1}^d$ and a moduli of maps \eqref{eq:R(0,1)} or \eqref{eq:R1bar}, there exists $T>0$ such that for all elements $u$ in the moduli,
 we have $E(u)<T$.
\end{lemma}

We also have the corresponding statement for positivity of intersections.

\begin{lemma}[Positivity of intersection]\label{l:posInt}
 Let $u$ be an element in \eqref{eq:R(0,1)} or \eqref{eq:R1bar}.
 If $u(z) \in D_0$ (resp. $u(z) \in D_r$), then the intersection $u(z)$ 
 between $Im(u)$ and $D_0$ (resp. $D_r$) contribute positively to the algebraic intersection number.
\end{lemma}

\begin{proof}
 For both cases ($D_0$ and $D_r$), it follows from the fact that the Hamiltonian vector field in the perturbation term is tangent to (every stratum of) the divisor.
 For $D_0$, this is exactly Lemma \ref{l:1true}.
 For $D_r$, the argument is similar to that of Lemma \ref{l:PositivityIntersection} and \ref{l:3true}.
\end{proof}

\begin{proof}[Proof of Proposition \ref{p:ncCompactification}]
 Let $u_k:S_k \to \Hilb^n(\overline{E})$ be a sequence of maps in \eqref{eq:R(0,1)}.
 First consider the case that there is a subsequence of $S_k$ converging to $S \in \cR^{d+1,h}_{(0,1)}$.
 Without loss of generality, we can assume $S_k=S$ for all $k$.
 
 The analog of Lemma \ref{l:nobubble} holds.
 More precisely, if there is energy concentration at a point in $S$ but outside $\nu(\mk(S)^+)$, then it will produce a 
 sphere bubble or a disc bubble in $E^{\rceil}$ which intersects $D_E$.
 By Lemma \ref{l:posInt}, this will imply that for large $k$, the algebraic intersection number between $u_k$ and $D_0$ is greater than $1$.
 Contradiction.
 
Energy concentration cannot happen at a point outside $\mk(S)^+$.
 If energy concentration happens at $\mk(S)^+$, then it is either at $\{s_0,s_1\}$ or not.
 If not, the resulting sphere bubble (tree) has Chern number $0$ so it is contained in $\Hilb^n(E)$ (recall that $c_1(\Hilb^n(\overline{E}))=PD(D_0)$).
 Therefore, it intersects $D_r$ positively by Lemma \ref{l:rationalCurveinHilb} and gives a  contradiction.
 If energy concentration happens at $s_0$ or $s_1$, say $s_0$, then the bubble tree has either Chern number zero or Chern number one.  We get a contradiction in the former case as above.
 For the latter case, the Chern number $1$ bubble intersects $D_0'$.
 Note that $s_1$ is another point in the stable domain that is mapped to $D_0'$. Since $D_0$ and $D_0'$ are homologous, that implies that the algebraic intersection number between $u$ and $D_0$ is at least $2$ in total. This again gives a  
 contradiction.
 
 Energy concentration at strip like ends give parts of \eqref{eq:Cnc1} and \eqref{eq:Cnc2}.
 This finishes the bubbling analysis when $S_k$ has a subsequence converging to $S \in \cR^{d+1,h}_{(0,1)}$.
 
 Now, if $S_k$ has a subsequence converging to the boundary strata \eqref{eq:c1}, \eqref{eq:c3}, we get the remaining terms in  \eqref{eq:Cnc1} and \eqref{eq:Cnc2}.
 
 If $S_k$ has a subsequence converging to the boundary strata \eqref{eq:c4}, then for $k$ large, $u$ has algebraic intersection at least $2$ with $D_0$ because
 each disc component in the limit contributes at least one.
 
 Next, we consider the case that $S_k$ has a subsequence converging to the boundary strata \eqref{eq:c2}.
 Without loss of generality, we can assume $S_k=S$ lies in the boundary strata \eqref{eq:c2} for all $k$.
 The analog of Lemma \ref{l:nobubble} still holds.
 That is, there is no energy concentration at a point in $S \setminus \nu(\mk(S)^+)$, and hence no energy concentration at a point in $S \setminus \mk(S)^+$.
 If energy concentration happens at $\mk(S)$, we get a Chern $0$ bubble and the usual contradiction.
 If energy concentration happens at $s_0$ or $s_1$, say $s_0$, then at least one of the spheres (including the original one in the domain) has vanishing Chern number, so meets $D_r$.
 Therefore, the tree of spheres define a stable Chern number $1$ rational curve in $\Hilb^n(\overline{E})$ which lies in a real codimension $4$ subset of  $\Hilb^n(\overline{E})$ (see Corollary \ref{c:pseudocycle}).
 As a result, for a generic choice of data, there is no  such stable Chern number $1$ rational curve in the codimension $1$ boundary strata
 of $\cR^{d+1,h}_{(0,1)}(\ul{x}_0;\ul{x}_d,\dots,\ul{x}_1)$ when $\cR^{d+1,h}_{(0,1)}(\ul{x}_0;\ul{x}_d,\dots,\ul{x}_1)$ is $1$ dimensional.
 Energy concentration of solutions in \eqref{eq:c2} at strip-like ends is a higher codimension phenomenon. 
 
 When there is no energy concentration, the boundary strata \eqref{eq:c2} correspond precisely to \eqref{eq:Cnc3} and \eqref{eq:Cnc4}.
 
 Finally, when $S_k$ converges to some element in higher codimension strata, regularity reasoning as in Corollary \ref{c:v1} shows that it cannot exist for generic $(J,K)$.
 This completes the proof.
 \end{proof}



\subsection{Pure Lagrangians}

 An nc-vector field of an $A_{\infty}$ category $\eA$ is a cocycle $b \in CC^1(\eA,\eA)$.
 Given an nc-vector field $b$, a $b$-equivariant object is a pair $(L,c_L)$ such that $L \in Ob(\eA)$, $c_L \in hom_{\eA}^0(L,L)$ and $b^0|_L=\mu^1(c_L)$.
 In this case, $(L,c_L)$ is called a $b$-equivariant lift of $L$.
 Given two $b$-equivariant objects $(L,c_L)$ and $(L',c_{L'})$, the map
 \begin{align}
  a \mapsto b^1(a)-\mu^2(c_{L'},a)+\mu^2(a,c_L) \label{eq:bchain}
 \end{align}
 is a chain map from $hom_{\eA}(L,L')$ to itself.
 An nc-vector field is called \emph{pure} if every object $L \in Ob(\eA)$ admits a $b$-equivariant lift $(L,c_L)$
 and for every pair $(L,c_L), (L',c_{L'})$ of $b$-equivariant lifts, 
 the endomorphism on $H(hom_{\eA}(L,L'))$ induced by \eqref{eq:bchain} is given by 
 \begin{align}
  a \mapsto |a|a \label{eq:Euler}
 \end{align}
 for all pure degree elements $a \in H(hom_{\eA}(L,L'))$, where $|\cdot|$ denotes degree.  In other words, purity asserts that \eqref{eq:bchain} agrees with the Euler vector field.

\begin{theorem}[Seidel]\label{t:seidel-formality}
 If an $A_{\infty}$ category $\eA$ over a field of characteristic $0$ admits a pure nc-vector field, then $\eA$ is formal.
\end{theorem}

A proof is given in \cite{AbouzaidSmith16}.

\subsubsection{Construction}\label{ss:NCconstruction}

By Proposition \ref{p:ncCompactification}, we can define a Hochschild cochain 
\begin{align}
 \tilde{b} \in CC^*(\cFS^{cyl,n}(\pi_E),\cFS^{cyl,n}(\pi_E))
\end{align}
by counting rigid elements in $\cR^{d+1,h}_{(0,1)}(\ul{x}_0;\ul{x}_d, \dots, \ul{x}_1)$ and then divide it by $h!$.
The cochain $\tilde{b}$ is not closed due to the existence of the boundary strata \eqref{eq:Cnc3} and \eqref{eq:Cnc4}.

To compensate for these strata, we consider the closed-open maps
\begin{align}
 &CO: C^*(\Hilb^n(E) \setminus D_r) \to CC^*(\cFS^{cyl,n}(\pi_E),\cFS^{cyl,n}(\pi_E))\\
 &co: C^{lf}_{2n-2-j}(D_0,(D_0 \cap D_r) \cup D_0^{sing}) \to CC^j(\cFS^{cyl,n}(\pi_E),\cFS^{cyl,n}(\pi_E)) \text{ for }j \le 3
\end{align}
which are the direct analogs of the ones in \cite[Section 3.5]{AbouzaidSmith16}.
The only difference is that we added the type-1 interior marked points and hence use $\cR^{d+1,h}_1(\Hilb^n(\overline{E}), (d_0,d_r)| \ul{x}_0; \ul{x}_d,\dots,\ul{x}_1)$
for $(d_0,d_r)=(0,0)$ and $(1,0)$, respectively, to define $CO$ and $co$.

The boundary strata \eqref{eq:Cnc3} and \eqref{eq:Cnc4} correspond to
\begin{align}
 CO(GW_1) \quad \text{ and } \quad co([B_0]) 
\end{align}
respectively, where we recall $B_0 = D_0 \cap D_0'$.
Using the fact that both $CO$ and $co$ are chain maps, we get

\begin{proposition}[cf. Proposition 3.20 of \cite{AbouzaidSmith16}]\label{p:ncfield}
If both $GW_1$ and $[B_0]$ are null-homologous with primitives $gw_1$ and $\beta_0$, then
 the sum
 \begin{align}
 b:=\tilde{b}+CO(gw_1)+co(\beta_0)         
 \end{align}
 defines an nc-vector field for $\cFS^{cyl,n}(\pi_E)$.
\end{proposition}

We now specialise to the case where $E = A_{m-1}$ is the type $A$ Milnor fiber. 
It is shown in \cite[Lemma 6.7, 6.8]{AbouzaidSmith16} that both $GW_1$ and $[B_0]$ are null-homologous so Proposition \ref{p:ncfield} applies.
The Lagrangian spheres and thimbles in $E$ have trivial first cohomology and so (by exactness) trivial first Floer cohomology, so admit equivariant structures.

\begin{lemma}\label{l:Lpure}
 For each $\lambda \in \Lambda_{n,m}$, $\ul{L}_{\lambda}$ is pure.
\end{lemma}

\begin{proof}
The argument of \cite[Lemma 6.9]{AbouzaidSmith16} generalises to our situation: we briefly recall the key points, adapted to the current setup. 
By Lemma \ref{l:NoSideBubbling}, the rigid count of $\cR^{d+1,h}_{(0,1)}(\ul{x}_0)$ is $0$ when $h>0$.
When $h=0$, the domain is a disjoint union of $n$ discs which are mapped to different Lagrangian components of $\ul{L}_{\lambda}$.
Since the intersection with $D_0$ is $1$, exactly one of these $n$ discs hits $D_E$ transversely once, and the others are disjoint from $D_E$ and are hence constant maps by exactness.
The key point in the argument of \cite[Lemma 6.9]{AbouzaidSmith16} is that, over each point $t \in \gamma$ of a fibred Lagrangian $L_\gamma \subset \overline{E}$, 
there are exactly two holomorphic discs with boundary on $L_\gamma$ (the hemispheres of a $\mathbb{P}^1$-fiber of $\overline{E}$).
When $\gamma$ is a matching path, this algebraic count gives $a \mapsto 2a$ under the map \eqref{eq:bchain} for the degree $2$ generator $a \in HF(L_\gamma)$.
Moreover, by \cite[Lemma 2.12]{AbouzaidSmith16} and the paragraph thereafter, \eqref{eq:bchain} necessarily vanishes on degree $0$ generators of $HF(L_\gamma)$ whenever $HF^0(L_\gamma)$ has rank $1$ (which holds  when $\gamma$ is a matching path or a thimble path).
The result follows. 
\end{proof}

\section{Formality}\label{s:formality}

By Proposition \ref{p:AllCases}, we know that on the cohomological level, the extended symplectic arc algebra 
$K_{n,m}^{\symp}$ agrees with the extended arc algebra $K_{n,m}^{\alg}$.
In this section, we prove that the extended symplectic arc algebra is formal when working over a field of characteristic $0$, by proving that the nc-vector field constructed in \ref{ss:NCconstruction} is pure.
To ease notation, in this section, given a weight $\lambda$, we use $\ol{\lambda}$ and  $\ul{\lambda}$
to denote $\ol{\lambda}^{\alg}$ and  $\ul{\lambda}^{\alg}$, respectively.

\subsection{Some lemmata}

Let $\lambda_0, \lambda_1 \in \Lambda_{n,m}$.

\begin{lemma}[Cyclic module structure]\label{l:cyclic}
Let $x_{\min}$ be an oriented circle diagram $\ul{\lambda}_1 \eta_{\min} \ol{\lambda}_0$ which corresponds to the minimal degree generator of 
$HF^*(\ul{L}_{\lambda_0},\ul{L}_{\lambda_1})$.
Let $x$ be another oriented circle diagram $\ul{\lambda}_1 \eta \ol{\lambda}_0$.
Then there are two oriented circle diagrams $\ul{\lambda}_0 \eta_0 \ol{\lambda}_0$ and $\ul{\lambda}_1 \eta_1 \ol{\lambda}_1$
such that 
\begin{align}
 \mu^2(\ul{\lambda}_1 \eta_1 \ol{\lambda}_1,x_{\min})=x=\mu^2(x_{\min},\ul{\lambda}_0 \eta_0 \ol{\lambda}_0)
\end{align}

In particular, $HF^*(\ul{L}_{\lambda_0},\ul{L}_{\lambda_1})$ is a cyclic module over $HF^*(\ul{L}_{\lambda_i},\ul{L}_{\lambda_i})$ for both $i=0,1$.
\end{lemma}

\begin{proof}
We give a proof when $i=0$; the case $i=1$ can be dealt with similarly.

Recall that, schematically, the product in the combinatorial arc algebra is computed by stacking an oriented $\ul{\lambda}_0 \cup \ol{\lambda}_0$ on top of an oriented $\ul{\lambda}_1 \cup \ol{\lambda}_0$, 
and resolving / removing the middle levels by a sequence of elementary cobordisms, which induce maps by the action of an underlying TQFT. 

 The diagram $\ul{\lambda}_1 \cup \ol{\lambda}_0$ consists of circles and lines.
 Let $C_1,\dots,C_k$ be the circles, and define for $j=1,\dots,k$
 \begin{align}
  &\Gamma_{o,j}:=\{\gamma \in \ol{\lambda}_0 | \gamma \text{ is contained in the circle }C_j \}.
 \end{align}
 It is clear that $\Gamma_{o,j} \neq \emptyset$ for all $j$.
 Let $\Gamma_o:=\cup_j \Gamma_{o,j}$ and $\Gamma_l:=\ol{\lambda}_0 \setminus \Gamma_o$.
 In particular, if $\gamma \in \ol{\lambda}_0$ is a line, then $\gamma \in \Gamma_l$.
 
 The minimal degree generator $x_{\min}$ corresponds to orienting 
 all the circles in $\ul{\lambda}_1 \cup \ol{\lambda}_0$ counterclockwise (note that the orientations of the arcs 
 in $\ul{\lambda}_1 \cup \ol{\lambda}_0$ are uniquely determined by the rule \eqref{eq:orientations2}, namely, the remaining $\vee$ {\it must} occupy the 
 rightmost available positions).
  On the other hand, let $C_{s_1}, \dots, C_{s_t}$ be the circles in $\ul{\lambda}_1 \cup \ol{\lambda}_0$ that are oriented clockwise by $x$.
  For each $j \in \{s_1,\dots,s_t\}$, we pick a single $\gamma_j \in \Gamma_{o,j}$.
 Now we give $\ul{\lambda}_0 \cup \ol{\lambda}_0$ the orientation $y$ that is uniquely determined by the property that a circle $C$ in $\ul{\lambda}_0 \cup \ol{\lambda}_0$
 is oriented clockwise if and only if 
 $C$ contains $\gamma_j$ for some $j \in \{s_1,\dots,s_t\}$.
 
 One feature of the Khovanov TQFT (\cite{Khovanov}, \cite{Khovanov-arc}, which is recalled in \eqref{eq:TQFT1} and \eqref{eq:TQFT2}) is that  a counterclockwise circle times a counterclockwise (resp. clockwise) circle 
 is a counterclockwise (resp. clockwise) circle.
 The analog is true for the Khovanov-type TQFT from \cite{BS11, Stroppel-parabolic}.
 One can check, by applying the rules underlying the Khovanov-type TQFT from \cite{BS11, Stroppel-parabolic}, that our prescription ensures that we have $\mu^2(x_{\min},y)=x$ (see Figure \ref{fig:cyclicModule} for an example).
\end{proof}

\begin{figure}[h]
 \includegraphics{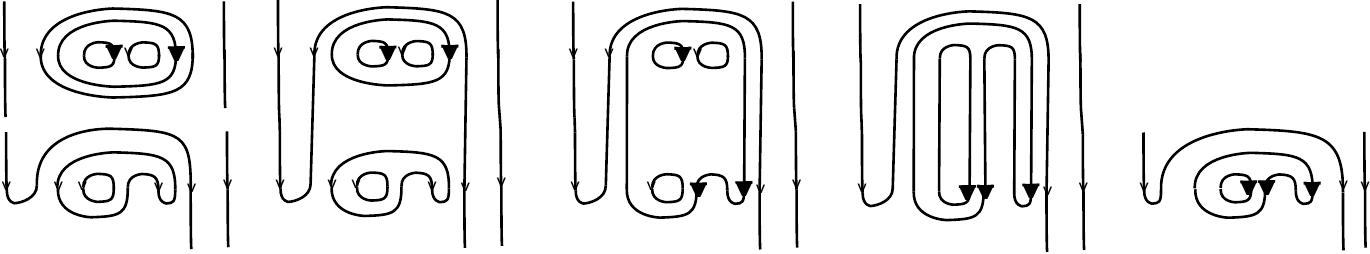}
 \caption{From left to right: $y$ (top) and $x_{\min}$ (bottom) before taking product; 
 successively applying two TQFT operations to the first picture; the product $x$. 
 Clockwise orientations on circles are indicated by $\blacktriangledown$.
 We choose $\gamma_j$ to be the outermost upper arcs in the $2$ circle components of $x$; this forces the indicated $\blacktriangledown$s on $y$.}\label{fig:cyclicModule}
\end{figure}

\begin{lemma}\label{l:constantmap}
 Let $ \lambda a$ be an oriented cap diagram and $b \lambda $ be an oriented cup diagram. 
 Then we have
 \begin{align}
  \mu^2(b \lambda \ol{\lambda} , \ul{\lambda} \lambda a)=b \lambda a \label{eq:CancellingProd}
 \end{align}
 In particular, the product map
 \begin{align}
  HF(L_{\lambda},L_b) \times HF(L_a, L_{\lambda}) \to HF(L_a,L_b) \label{eq:nontrivialProduct}
 \end{align}
 is nonzero.
\end{lemma}

\begin{proof}
 Note that $\ul{\lambda} \lambda a$, $b \lambda \ol{\lambda}$ and $b \lambda a $ are oriented circle diagrams so the LHS and RHS of \eqref{eq:CancellingProd} are well-defined.

Now we need to justify \eqref{eq:CancellingProd}.
If $\gamma \in \ul{\lambda}$ is a half-circle and the component in  $\ul{\lambda} \lambda a$ containing $\gamma$
is a circle $C$, then by the definition of $\ul{\lambda}$, we know that $C$ is oriented counterclockwise.
Similarly, if  $\gamma^* \in \ol{\lambda}$ is a half-circle and the component in  $b \lambda \ol{\lambda}$ containing $\gamma^*$
is a circle $C'$, then $C'$ is oriented counterclockwise.
In this case, when we do the resolution between $\gamma$ and $\gamma^*$ involved in the product, the outcome is also oriented counterclockwise.
In other words, the resolution cobordism move does not change the orientation (see Figure \ref{fig:Orientation unchanged1})

\begin{figure}[h]
 \includegraphics{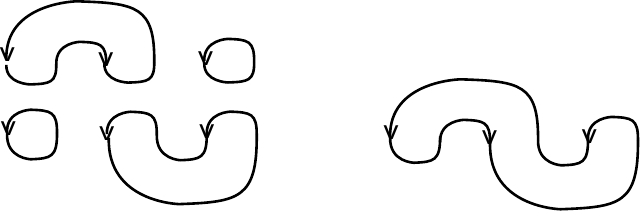}
 \caption{Orientation unchanged}\label{fig:Orientation unchanged1}
\end{figure}

On the other hand, if the component of $\ul{\lambda} \lambda a$ containing $\gamma$ is a line, then 
we need to first complete it to a circle before computing the product.
Moreover, the circle to which it is completed is oriented counterclockwise (see Figure \ref{fig:Orientation unchanged2}).
The same holds for $b \lambda \ol{\lambda}$.

Therefore, each type of resolution, and hence the product obtained from a composition of such, does not change the orientations.
It implies $\mu^2(b \lambda \ol{\lambda} , \ul{\lambda} \lambda a)=b \lambda a$.

By Proposition \ref{p:AllCases}, we have the algebra isomorphism $\Phi:K_{n,m}^{\symp} \to K_{n,m}^{\alg}$
so we have 
$$\mu^2(\Phi^{-1}(b \lambda \ol{\lambda}) , \Phi^{-1}(\ul{\lambda} \lambda a))=\Phi^{-1}(b \lambda a)$$
and \eqref{eq:nontrivialProduct} is nonzero.
\begin{figure}[h]
 \includegraphics{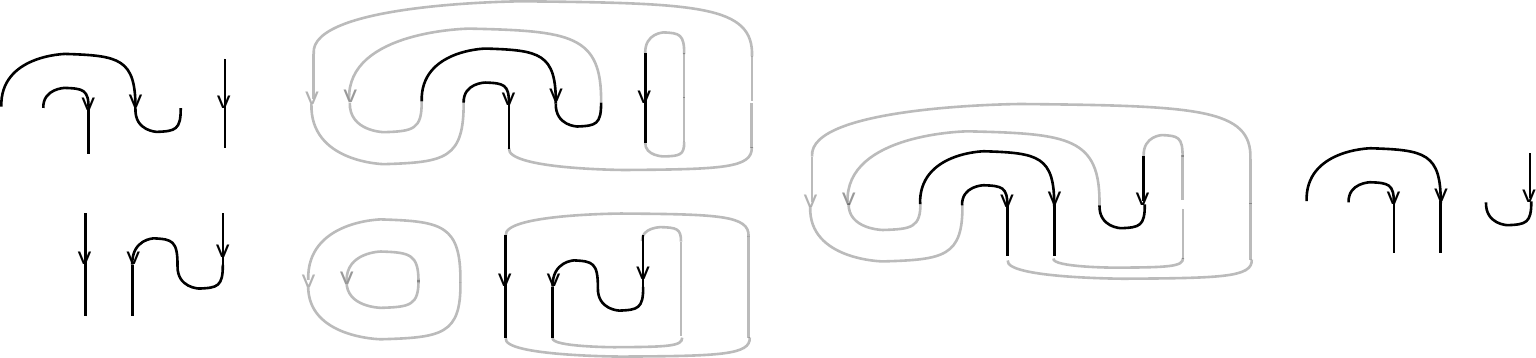}
 \caption{From left to right: 1) two oriented circle diagrams $\ul{\lambda} \lambda a$, $b \lambda \ol{\lambda}$, 
 2) closure of the oriented circle diagrams, all circles are oriented counterclockwise (see Section \ref{Sec:tqft_multiplication}), 3) taking product, 4) removing the completion (grey)}\label{fig:Orientation unchanged2}
\end{figure}
\end{proof}

The set of weights $\Lambda_{n,m}$ has a natural `Bruhat' partial ordering (see Section \ref{ss:directedsub}).
Intuitively, it is given by exchanging pairs $\vee$ and $\wedge$; one increases in the Bruhat order by moving $\vee$'s to the right, see \cite[Section 2]{BS11}.  Given $\lambda \in \Lambda_{n,m}$, there is a unique maximal element $w_{\lambda}$ in the set
\begin{align}
 \{w \in \Lambda_{n,m}\  |\  \ul{\lambda} w \ol{\lambda} \text{ is an oriented circle diagram}\}. \label{eq:maxElt}
\end{align}
Explicitly, the maximal element is given by putting a $\vee$ on the right end-points of all the half-circles and lines in $\ol{\lambda}$ (or $\ul{\lambda}$)
and a $\wedge$ otherwise.
In particular, it implies that $\lambda=w_{\lambda}$ only when $\ol{\lambda}$ has no half-circle component, that is, when $\lambda$ is the maximal weight in $\Lambda_{n,m}$.

\begin{lemma}\label{l:3choose1}
 Let $\lambda_0,\lambda_1 \in \Lambda_{n,m}$. At least one of the following statements holds.
\begin{enumerate}
 \item There is no $\eta \in \Lambda_{n,m}$ such that $\ul{\lambda_1} \eta \ol{\lambda_0}$ is an oriented circle diagram;
 \item There is some $\eta \in \Lambda_{n,m} \backslash \{\lambda_0, \lambda_1\}$ such that $\ul{\lambda_1} \eta \ol{\lambda_0}$ is an oriented circle diagram;
 \item $\lambda_0=w_{\lambda_1}$ or $\lambda_1=w_{\lambda_0}$.
\end{enumerate}

\end{lemma}

\begin{proof}
 Suppose neither $(1)$  nor $(2)$ hold. We want to prove  $(3)$.
 
 Since $(1)$ and $(2)$ fail, there exists $\eta$ such that 
 $\ul{\lambda_1} \eta \ol{\lambda_0}$ is an oriented circle diagram, but every such $\eta$ is given by $\lambda_0$ or $\lambda_1$.
 Without loss of generality, we assume that $\eta=\lambda_1$ so, in particular, we have $\lambda_0 \le \lambda_1$, since whenever $\alpha \lambda \beta$ is an oriented circle diagram one has $\alpha \leq \lambda$ and $\beta \leq \lambda$ in the Bruhat order, cf. \cite[Lemma 2.3]{BS11}.
 We want to prove that $\lambda_1=w_{\lambda_0}$.
 
 Notice that if $\ul{\lambda_1} \cup \ol{\lambda_0}$ has a circle component, then there is $\eta' \in \Lambda_{n,m}$
 such that $\eta' \neq \eta$ and $\ul{\lambda_1} \eta' \ol{\lambda_0}$ is an oriented circle diagram.
 It implies that $\eta =\lambda_1 < \eta'$ but we also have $\lambda_0 \le \eta$ so $\eta' \in \Lambda_{n,m} \setminus \{\lambda_0,\lambda_1\}$.
 Since we have assumed that $(2)$ fails, so we get a contradiction and hence, $\ul{\lambda_1} \cup \ol{\lambda_0}$ is a diagram with only line components (and there is only one orientation given by $\eta=\lambda_1$).
 
 Let $s_1<\dots <s_k$ be the integers which are the left end points of \emph{half-circles} of $\ol{\lambda_0}$.
 By definition, $\lambda_0(s_i)=\vee$ for all $i$.
 If $\lambda_1(s_i)=\wedge$ for all $i$, then we are done, because  then $\lambda_1$ has $\vee$ on all the right end points of half-circles and lines of $\ol{\lambda_0}$,
 and hence $\lambda_1=w_{\lambda_0}$.
 
 If this is not the case, then let $s_j$ be the largest among $s_1, \dots s_k$ such that $\lambda_1(s_j)=\vee$.
 Let $s_j+t$ be the right end point of the half-circle of $\ol{\lambda_0}$ with left end point $s_j$. 
 By definition of $\ol{\lambda_0}$, exactly half of $\{s_j+1,\dots,s_j+t-1\}$ are labelled by $\wedge$
 and half by $\vee$ with respect to $\lambda_0$.
 Moreover, all of them lie in a half-circle of  $\ol{\lambda_0}$.
 
 By the definition of $s_j$, for $l=1,\dots,t-1$, if $\lambda_0(s_j+l)=\vee$, then $s_j+j \in \{s_{j+1},\dots,s_k\}$ so we have $\lambda_1(s_j+l)=\wedge$.
 Moreover, for $\ul{\lambda_1} \eta \ol{\lambda_0}$ to have at least one orientation, we need that $\lambda_1(s_j+t)=\wedge$ and
 for  $l=1,\dots,t-1$, if $\lambda_0(s_j+l)=\wedge$, then $\lambda_1(s_j+l)=\vee$ (see Figure \ref{fig:3choose1}).
 Now, by the definition of $\ul{\lambda_1}$, we see that $\ul{\lambda_1} \eta \ol{\lambda_0}$ has a circle passing through $s_j$ and $s_j+t$.
 This contradicts the fact that $\ul{\lambda_1} \eta \ol{\lambda_0}$ does not have circle components.
 \begin{figure}[h]
  \includegraphics{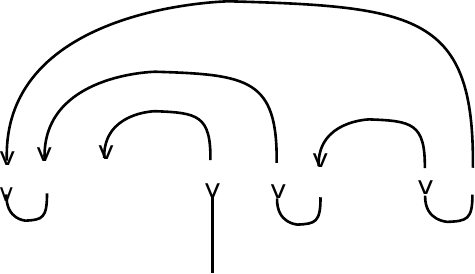}
  \caption{The weight and the cap diagram on top are $\lambda_0$ and $\ol{\lambda_0}$;
  the weight and the cup diagram on the bottom are $\lambda_1$ and $\ul{\lambda_1}$.}\label{fig:3choose1}
 \end{figure}
\end{proof}

\subsection{Consistent choice of equivariant structures\label{Sec:consistent}}

Let $b$ be the nc vector field constructed in Proposition \ref{p:ncfield}.

\begin{proposition}\label{p:bStructure}
 There is a choice of $b$-equivariant structures on $\{\ul{L}_{\lambda}\}_{\lambda \in \Lambda_{n,m}}$ such that $b$ is pure. 
\end{proposition}

\begin{proof}
By Lemma \ref{l:Lpure}, $\ul{L}_{\lambda}$ is pure  for each $\lambda \in \Lambda_{n,m}$.

 Any partial order admits a total order refinement. We choose a total ordering $(\le^T)$ on $\Lambda_{n,m}$ that is compatible with the Bruhat partial ordering $\le$.
 We make choices of $b$-equivariant structure on the $\ul{L}_{\lambda}$ with decreasing order of $\lambda$.
 
 Let $\eta \in \Lambda_{n,m}$ and suppose that a $b$-equivariant structure on $\ul{L}_{\lambda}$ is chosen for all $\lambda >^T \eta$, 
 such that $b$ is pure.
 We want to choose a  $b$-equivariant structure on $\ul{L}_{\eta}$ such that $b$ is pure on $\{L_{\lambda}\}_{\lambda  \ge^T \eta}$.
 
 We choose a  $b$-equivariant structure on $\ul{L}_{\eta}$ such that the minimal degree element of
 $HF(\ul{L}_{\eta}, \ul{L}_{w_\eta})$
 has weight equal to the degree.
 By Lemma \ref{l:cyclic}, $HF(\ul{L}_{\eta}, \ul{L}_{w_\eta})$ is cyclic as a module over $HF(\ul{L}_{w_\eta},\ul{L}_{w_\eta})$ 
 so the weight equals the degree for all pure degree elements in $HF(\ul{L}_{\eta}, \ul{L}_{w_\eta})$.
 
 By Lemma \ref{l:constantmap}, we have
 \begin{align}
  \mu^2(\ul{\eta}w_\eta\ol{w_\eta}, \ul{w}_{\eta}w_\eta \ol{\eta})=\ul{\eta}w_\eta \ol{\eta}
 \end{align}
This implies that $x:= \ul{\eta} w_\eta \ol{w}_{\eta}$ in $HF(\ul{L}_{w_\eta},\ul{L}_{\eta})$ has weight equal to its  degree.

Let $x_{\min}$ be the minimal degree generator of $HF(\ul{L}_{w_\eta},\ul{L}_{\eta})$.
By Lemma \ref{l:cyclic}, we know that $x=\mu^2(x_{\min},y)$ for some $y \in HF(\ul{L}_{w_\eta}, \ul{L}_{w_\eta})$.
It means that $x_{\min}$ also has weight equal to its degree, and hence all elements of $HF(\ul{L}_{w_\eta},\ul{L}_{\eta})$
have weight equal to their degrees.

Now, suppose that $\lambda >^T \eta$ but $\lambda \neq w_{\eta}$, and purity holds for all of  $HF(L_{\lambda'}, L_\eta)$, $HF(L_\eta, L_{\lambda'})$
such that $\lambda' >^T \lambda$.
By Lemma \ref{l:3choose1}, $\lambda, \eta$ satisfies either $(1)$ or $(2)$ of Lemma \ref{l:3choose1}.
If $(1)$ holds, then $HF(L_{\lambda}, L_\eta)=HF(L_\eta,L_{\lambda})=0$ so purity is automatic.
If $(2)$ holds, let $\lambda'> \lambda, \eta$ and $\lambda' \neq \lambda, \eta$ such that $\ol{\eta} \lambda' \ul{\lambda}$
is an oriented circle diagram.

By Lemma \ref{l:constantmap}, we have
 \begin{align}
  \mu^2(\ul{\lambda}\lambda' \ol{\lambda'}, \ul{\lambda'} \lambda' \ol{\eta})=\ul{\lambda} \lambda' \ol{\eta} \label{eq:hahaha}
 \end{align}
 This means that $\ul{\lambda} \lambda' \ol{\eta}$ in $HF(\ul{L}_{\eta},\ul{L}_{\lambda})$ has weight equal to its  degree. 
 By Lemma \ref{l:cyclic} again, we know the minimal degree elements in $HF(\ul{L}_{\eta},\ul{L}_{\lambda})$ have weight equal to degree, hence so do all elements in 
 $HF(\ul{L}_{\eta},\ul{L}_{\lambda})$.
 The same is true for $HF(\ul{L}_{\lambda},\ul{L}_{\eta})$ because Lemma \ref{l:constantmap} also implies that
 \begin{align}
  \mu^2(\ul{\eta} \lambda' \ol{\lambda'}, \ul{\lambda'} \lambda' \ol{\lambda} )=\ul{\eta} \lambda'  \ol{\lambda}
 \end{align}
 By induction over the total order, the result follows.
 \end{proof}

 \begin{corollary}
  Over a field of characteristic $0$, the $A_{\infty}$-endomorphism algebra of the direct sum of objects in the collection $\{\ul{L}_{\lambda}\}_{\lambda \in \Lambda_{n,m}}$ is formal.
 \end{corollary}

 \begin{proof}
  This follows from Theorem \ref{t:seidel-formality} and Proposition \ref{p:bStructure}.
 \end{proof}

\begin{proof}[Proof of Theorem \ref{t:main}]
Recall the notation for the tuple of Lagrangian thimbles $\ul{T}^I$ from Section \ref{s:FScategories}.
 One can argue inductively to show that
 $\ul{T}^{(\lambda)^{-1}(\vee)}$ can be generated by $\{\ul{L}_{\ul{\lambda}}|\lambda  \in \Lambda_{n,m}\}$ using the exact triangles constructed in Section \ref{sss:CanEmbed}
 (compare with Proposition \ref{p:cellMod}).
 
 We now give the  details.
 When $n=1$, $\{\ul{L}_{\ul{\lambda}}|\lambda  \in \Lambda_{n,m}\}$ consists of 
 one Lefschetz thimble at the right most critical point, and all the matching spheres over matching paths 
 whose end points are consecutive critical values.
 This collection clearly generates $\ul{T}^{(\lambda)^{-1}(\vee)}$, by applying Dehn twists in the matching spheres to the Lefschetz thimble. We now argue inductively. 
 Let $k \in \mathbb{N}_+$. We assume that for all $n<k$, for all $m \ge n$ and for all $\lambda_0 \in \Lambda_{n,m}$, the Lagrangian tuple
 $\ul{T}^{(\lambda_0)^{-1}(\vee)}$ can be generated by the collection $\{\ul{L}_{\ul{\lambda}}|\lambda  \in \Lambda_{n,m}\}$.
 Now we consider the case $n=k$, $m \ge n$ and let $\lambda  \in \Lambda_{n,m}$.
 
When $\lambda^{-1}(\vee)=\{m-n+1,\dots,m\}$, $\lambda$ is the maximal weight and we have  $\ul{L}_{\ul{\lambda}}=\ul{T}^{\lambda^{-1}(\vee)}$.

If $\lambda$ is not the maximal weight, then there is $k \in \{1,\dots, m\}$
such that $\lambda(k)=\vee$ and $\lambda(k+1)=\wedge$.
Let $\lambda''$ be the weight given by `swapping positions $k$ and $k+1$'. That is
\begin{align}
 \lambda''(i)=
 \left\{
 \begin{array}{ccc}
  \lambda(i) &\text{ if }i \neq k,k+1 \\
  \wedge     &\text{ if }i = k \\
  \vee       &\text{ if }i = k+1
 \end{array}
 \right.
\end{align}
Let $\lambda' \in \Lambda_{n-1,m-2}$ be the weight given by forgetting the positions $k,k+1$.That is
\begin{align}
 \lambda'(i)=
 \left\{
 \begin{array}{ccc}
  \lambda(i) &\text{ if }i <k \\
  \lambda(i+2) &\text{ if }i \ge k
 \end{array}
 \right.
\end{align}
We define $\cup_{k,k+1} \lambda'$ to be the admissible tuple in $\{im(z) \le 1\}$ consisting of $1$ matching path connecting 
$k+\sqrt{-1}$ and $k+1+\sqrt{-1}$ and $n-1$ thimble paths from $i+\sqrt{-1}$ to $i$ for each $i \in \lambda^{-1}(\vee) \setminus \{ k\}$.
The notation suggests that it is obtained from applying a `cup functor' to $\lambda'$ in position $k,k+1$ in the sense of \cite{AbouzaidSmith19}.

By Corollary \ref{c:Triangle}, we have the following exact triangle
\begin{align}
 \ul{T}^{(\lambda'')^{-1}(\vee)}[-1] \to \ul{L}_{\cup_{k,k+1} \lambda'} \to \ul{T}^{(\lambda)^{-1}(\vee)} \to \ul{T}^{(\lambda'')^{-1}(\vee)} \label{eq:ExactT}
\end{align}
Note that the second term $\ul{L}_{\cup_{k,k+1} \lambda'}$ can be generated by $\{\ul{L}_{\ul{\lambda}}|\lambda  \in \Lambda_{n,m}\}$
by applying the `cup functor' $\cup_{k,k+1}$ to the induction hypothesis.
More precisely, 
by induction hypothesis, the collection 
$$\{\ul{L}_{\ul{\lambda}}|\lambda  \in \Lambda_{n-1,m-2}\}$$
is sufficient to generate $\ul{L}_{\Gamma}$ for all admissible tuple $\Gamma=\{\gamma_1,\dots,\gamma_n\}$ such that 
for all $i$, $\gamma_i$ is a thimble path from $q_i+\sqrt{-1}$ to $q_i$ for some $q_i \in \{1,\dots,m\}$. 
Together with Corollary \ref{c:Triangle} (which implies that the `cup functor' takes exact triangles to exact triangles), we know that
the collection
$$\{\ul{L}_{\ul{\lambda}}|\lambda  \in \Lambda_{n,m}, \lambda(k)=\vee, \lambda(k+1)=\wedge\}$$ 
is sufficient to generate $\ul{L}_{\Gamma}$ for all admissible tuple $\Gamma=\{\gamma_1,\dots,\gamma_n\}$
such that (see Figure \ref{fig:cupfunctor})
\begin{enumerate}
 \item $\gamma_i \subset \{im(z) \le 1\}$ for all $i$
 \item $\gamma_1$ is a matching path joining $k+\sqrt{-1}$ and $k+1+\sqrt{-1}$
 \item for $i>1$, $\gamma_i$ is a thimble path from $q_i+\sqrt{-1}$ to $q_i$ for some $q_i \in \{1,\dots,m\}$. 
\end{enumerate}

\begin{figure}[h]
  \includegraphics{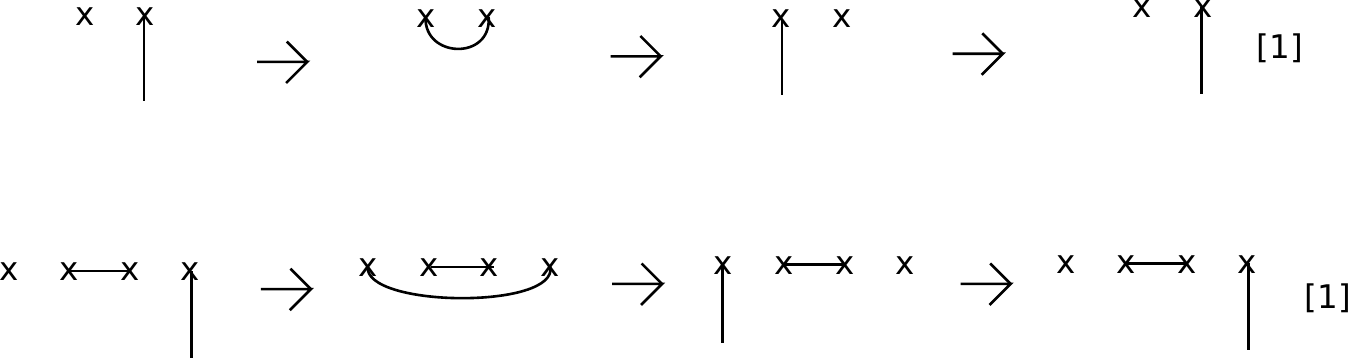}
  \caption{The first row represents an exact triangle in the induction hypothesis.
  The second row is the exact triangle after applying the `cup functor' $\cup_{2,3}$ to it.
  }\label{fig:cupfunctor}
 \end{figure}

On the other hand, the first and third term of \eqref{eq:ExactT} has the property that $\lambda''>\lambda$.
Since $\ul{L}_{\ul{\lambda}}=\ul{T}^{\lambda^{-1}(\vee)}$ when $\lambda$ is maximal, 
by \eqref{eq:ExactT} and  a second induction with respect to the Bruhat ordering starting from the maximal weight, we know that
$\ul{T}^{(\lambda)^{-1}(\vee)}$ can be generated by $\{\ul{L}_{\ul{\lambda}}|\lambda  \in \Lambda_{n,m}\}$.

 Conversely, $\ul{L}_{\ul{\lambda}}$ can be generated by $\{\ul{T}^{(\lambda)^{-1}(\vee)}|\lambda \in \Lambda_{n,m}\}$, either again by directly applying the exact triangles, or by appeal to Proposition \ref{p:generation}.
 Therefore, by Proposition \ref{c:directedSubCat2}, we have quasi-equivalences
 \begin{align}
 D^{\pi}\cFS(\pi_{n,m}) \simeq \rperf(\End(\oplus \ul{T}^{(\lambda)^{-1}(\vee)})) \simeq \rperf(K^{\symp}_{n,m})=\rperf(K^{\alg}_{n,m})
 \end{align}
\end{proof}

\begin{remark} Theorem \ref{t:main} and Claim \ref{p:Serre} together give a new proof of a result of \cite{BBM}, namely that the Serre functor of the constructible derived category of the Grassmannian is given by the action of the centre of the braid group.
\end{remark}

There is an algebra isomorphism $K^{\alg}_{n,m} \simeq (K^{\alg}_{n,m})^{op}$,
given on basis elements by sending an oriented circle diagram $\beta \lambda \alpha$ to 
its reflection $\alpha^r \lambda \beta^r$ (see \cite[Equation (3.9)]{BS11}).
Together with Theorem \ref{t:main}, this implies that
\begin{corollary}\label{c:flip}
For all $\lambda, \lambda' \in \Lambda_{n,m}$, we have
\begin{align}
 HF(\ul{L}_{\lambda},\ul{L}_{\lambda'})=HF(\ul{L}_{\lambda'},\ul{L}_{\lambda})
\end{align}
as graded vector spaces.
\end{corollary}
This is already non-trivial, because $\ul{L}_{\lambda},\ul{L}_{\lambda'}$ are not compact Lagrangians,  so it is not \emph{a priori} 
clear why they should satisfy this Poincar\'e duality type of equality (because of the non-trivial wrapping in constructing the category, the underlying cochain groups are typically \emph{not} equal).

The algebra isomorphism $K^{\alg}_{n,m} \simeq (K^{\alg}_{n,m})^{op}$ can be understood geometrically: we give a sketch of the argument.
By a hyperbolic isometry between the upper half plane and the unit disc, we can assume the target of $\pi_E$ is the unit disc $B$.
We also assume that the critical values lie on the real line, and that infinity in the upper half plane is mapped to $\sqrt{-1} \in \partial B$.
For example, it can be achieved by applying the inverse of the hyperbolic isometry $f:\overline{B}_1 \to \bH$ in \eqref{eq:dictionary}.
Let $\iota_B:B \to B$ be the reflection along the real line, which is an anti-symplectic involution of $B$.
We recall that for the construction of $(E,\omega_E)$ in Section \ref{ss:AmMilnorFiber}, $\omega_M$ is restricted from a product symplectic form $\omega_X$.
Therefore, by appropriately choosing $\omega_R$ in Lemma \ref{l:omegaR}, we can assume that there is an 
anti-symplectic involution $\iota_E:E \to E$ covering $\iota_B$.

Let $u:S \to \cY_E$ be a solution that contributes to the $A_{\infty}$ structure of $\cFS^{cyl,n}(\pi_E)$.
Let $\iota_S:S \to \ol{S}$ be reflection in the real diameter of the disc, and $\ol{S}$ the image of $S$ (equipped with the push-forward data by $\iota_S$).
Let $\iota_{\cY_E}: \cY_E \to \cY_E$ be the anti-symplectic involution induced by $\iota_E$.
In this case, $\iota_{\cY_E}  \circ u \circ \iota_S: \ol{S} \to \cY_E$ is tautologically a solution to the push-forward equation, which is itself a perturbed pseudo-holomorphic equation.
Note, however, that the ordering of the Lagrangian boundary conditions is reversed on $\partial \ol{S}$.

This illustrates that, for appropriate choices of Floer data, the involution $\iota_E$
induces an $A_{\infty}$-equivalence
\begin{align}
 \cFS^{cyl,n}(\pi_E) &\simeq (\cFS^{cyl,n}(\pi_E)^{\dagger})^{op} \label{eq:A_equivalence} \\
 \ul{L} &\mapsto \iota_E(\ul{L}) \nonumber
\end{align}
where  $\cFS^{cyl,n}(\pi_E)^{\dagger}$ is quasi-equivalent to $\cFS^{cyl,n}(\pi_E)$ but
the wrapping at infinity for $CF(?,??)$ in $\cFS^{cyl,n}(\pi_E)^{\dagger}$
is defined to be wrapping the latter entry clockwise, with stop at $-\sqrt{-1}$, whilst the wrapping in $\cFS^{cyl,n}(\pi_E)$  
is wrapping the former entry counterclockwise with stop at $\sqrt{-1}$ (see Figure \ref{fig:reflection}).
For more details of (quasi-)equivalences of the kind of \eqref{eq:A_equivalence}, see \cite[Appendix]{Nick11}.

 \begin{figure}[h]
  \includegraphics{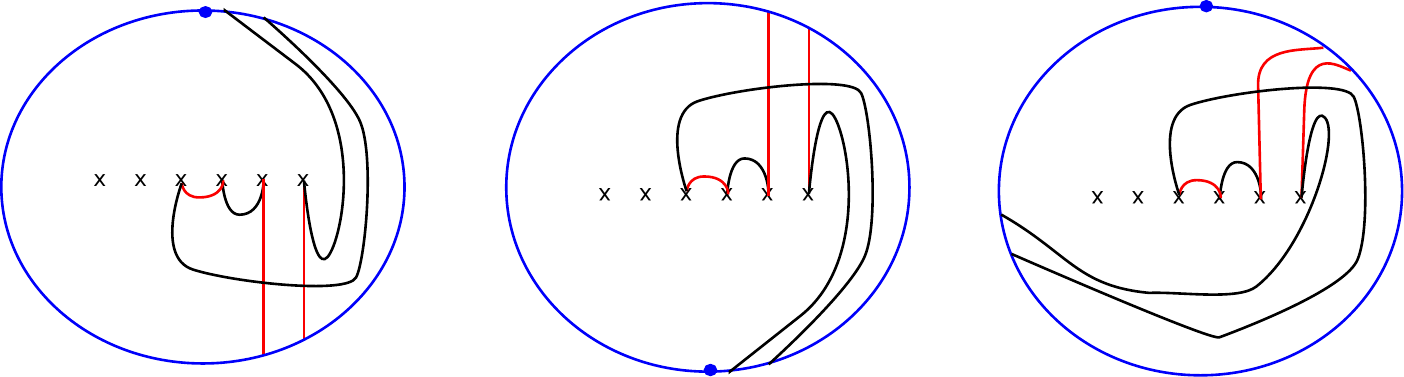}
  \caption{$CF(\ul{L}_{\lambda},\ul{L}_{\lambda'})$ (left), $CF(\iota_E (\ul{L}_{\lambda'}),\iota_E (\ul{L}_{\lambda}))$ in $\cFS^{cyl,n}(\pi_E)^{\dagger}$ (middle)
  and $CF(\Psi \circ \iota_E (\ul{L}_{\lambda'}),\Psi \circ \iota_E (\ul{L}_{\lambda}))$ (right). The dot on the boundary indicates the stop.
  }\label{fig:reflection}
 \end{figure}

To identify $\cFS^{cyl,n}(\pi_E)^{\dagger}$ with $\cFS^{cyl,n}(\pi_E)$, we need to pick a quasi-equivalence $\Psi$
between them.
We choose the one that corresponds to moving the `stop' $-\sqrt{-1}$ for $\cFS^{cyl,n}(\pi_E)^{\dagger}$
to $\sqrt{-1}$ clockwise along the boundary of $B$ and keeping a compact region containing the critical points unchanged (see Figure \ref{fig:reflection}).
This choice of quasi-equivalence has the property that 
\begin{align}
 \Psi \circ \iota_E (\ul{L}_{\ul{\lambda}}) = \ul{L}_{\ol{\lambda}} \label{eq:reSym}
\end{align}
for all $\lambda \in \Lambda_{n,m}$.


Applying $\Psi \circ \iota_E $ to the hom space between the Lagrangians associated to weights in $\Lambda_{n,m}$, we get Corollary \ref{c:flip}:
\begin{align}
 HF(\ul{L}_{\lambda},\ul{L}_{\lambda'})=HF(\Psi \circ \iota_E (\ul{L}_{\lambda'}),\Psi \circ \iota_E (\ul{L}_{\lambda}))=HF(\ul{L}_{\lambda'},\ul{L}_{\lambda})
\end{align}
where the first equality comes from \eqref{eq:A_equivalence}
and the second one comes from \eqref{eq:reSym}.
Similarly, applying $\Psi \circ \iota_E $ to the multiplication maps, we can deduce $K^{\symp}_{n,m} \simeq (K^{\symp}_{n,m})^{op}$,
so \eqref{eq:A_equivalence} can be viewed as a geometric interpretation of the isomorphism $K^{\alg}_{n,m} \simeq (K^{\alg}_{n,m})^{op}$.

\begin{proof}[Proof of Corollary \ref{c:projDual}]
 The statement follows from Theorem \ref{t:main} and the fact that $K^{\alg}_{n,m}=(K^{\alg}_{m-n,m})^{op}$.
 The  geometric origin of the latter equality is the Schubert-cells compatible identification of the Grassmannians $Gr(n,m)=Gr(m-n,m)$ but, following \cite[Equation (3.10)]{BS11}, it is not hard to give  a concrete isomorphism in diagrammatic terms. 
 We present their isomorphism here using our notations.
 
 Let $PD: \Lambda_{n,m} \to \Lambda_{m-n,m}$ be the bijection
 \begin{align}
  PD(\lambda)(a)=
  \left\{
  \begin{array}{cc}
   \wedge &\text{ if } \lambda(m+1-a)=\vee\\
   \vee &\text{ if }   \lambda(m+1-a)=\wedge
  \end{array}
 \right.
 \end{align}
 In words, $PD$ is obtained by rotating $\lambda$ by $\pi$ (see Figure \ref{fig:rotation}).

 \begin{figure}[h]
  \includegraphics{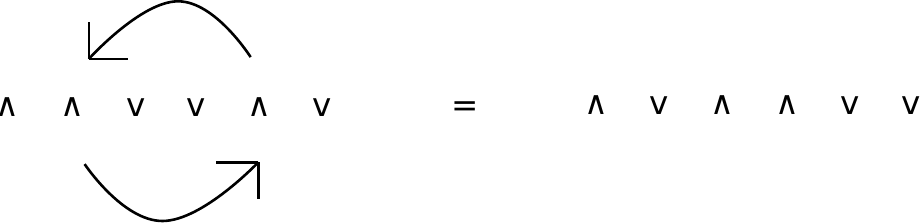}
  \caption{Left: a weight $\lambda$, Right: the weight $PD(\lambda)$}\label{fig:rotation}
 \end{figure}

 Each basis element of $K^{\alg}_{n,m}$ is an oriented circle diagram, for which we can perform a $\pi$-rotation to obtain
 another  oriented circle diagram, and hence a basis element of $(K^{\alg}_{m-n,m})^{op}$.
 This gives a bijection between basis elements of $K^{\alg}_{n,m}$ and $(K^{\alg}_{m-n,m})^{op}$.
 Since clockwise cup becomes clockwise cap, and vice versa, this bijection preserves grading.
 Because of the diagrammatic nature of the defining TQFT, it is immediate that this induces an algebra isomorphism, as observed in  \cite[Equation (3.10)]{BS11}.
\end{proof}

\section{Dictionary between Lagrangians and modules}\label{s:Dictionary}

In this section, we summarise the dictionary between 
certain objects in $\cFS^{cyl,n}(A_{m-1})$ (or their Yoneda images as modules over $K^{\symp}_{n,m}$) and modules over $K^{\alg}_{n,m}$.  None of our results rely on this dictionary: it is simply meant to help orient the interested reader. Given that, we make free use of Claim \ref{p:Serre} in this section.

Modules in this section are right modules, which differs from the convention used in \cite{BS11}.
However, as remarked previously, there is a canonical algebra isomorphism $K^{\alg}_{n,m} \simeq (K^{\alg}_{n,m})^{op}$, which identifies the
right $K^{\alg}_{n,m}$-modules used below with the left $K^{\alg}_{n,m}$-modules used in \cite{BS11}.

For an $A_\infty$ or dg category/algebra $\cC$, we use $\rperf \cC$ and $\cC \lperf$ to denote the dg category of perfect right, and respectively left, $\cC$ modules.
We use $[\cC,\cD]$ denote the dg category of $\cC\text{-}\cD$ bimodules.
Since $K^{\alg}_{n,m}$ is homologically smooth, every proper (cohomologically finite) module or bimodule is perfect.

\subsection{Indecomposable projectives}

For each $\lambda \in \Lambda_{n,m}$, we define 
$P(\lambda)$ to be the submodule of $K^{\alg}_{n,m}$ generated by the oriented circle diagrams of the form
\begin{align}
 P(\lambda):= \oplus \mathbb{K} \ul{\lambda}^{\alg} \mu \alpha
\end{align}
The collection of all $P(\lambda)$ (together with their grading shifts) is the set of all indecomposable projective modules of $K^{\alg}_{n,m}$.
Under the isomorphism $K^{\symp}_{n,m} \simeq K^{\alg}_{n,m}$, $P(\lambda)$
is the same as the Yoneda embedding of $\ul{L}_{\ul{\lambda}}$
as a right $K^{\symp}_{n,m}$ module.  Thus, \emph{the indecomposable projectives are the Lagrangians associated to weights.}

By \cite[Theorem 5.3]{BS11}, the set of all indecomposable injective modules of $K^{\alg}_{n,m}$ is given by
\begin{align}
 P(\lambda)^*:= \Hom( P(\lambda), \mathbb{K})
\end{align}
where the right module structure is given by $fa(m):=f(ma^*)$ for $a \in K^{\alg}_{n,m}$,  $m \in P(\lambda)$ and $f \in P(\lambda)^*$.
In other words, $P(\lambda)^*$ is obtained by pulling back the left module $\Hom(P(\lambda), \mathbb{K})$  via the algebra isomorphism $K^{\alg}_{n,m}=(K^{\alg}_{n,m})^{op}$.
On the symplectic side, up to a grading shift, $\Hom( P(\lambda), \mathbb{K})$ is given by the Yoneda embedding of $\tau^{-1}(\ul{L}_{\ul{\lambda}})$
as a left $K^{\symp}_{n,m}$ module, where $\tau$ is the global monodromy (see Claim \ref{p:Serre}), 
so $P(\lambda)^*$ corresponds to the pull-back of  $\tau^{-1}(\ul{L}_{\ul{\lambda}})$ via the algebra isomorphism $K^{\symp}_{n,m}=(K^{\symp}_{n,m})^{op}$.

In view of the geometric explanation of the equivalence $K^{\symp}_{n,m}=(K^{\symp}_{n,m})^{op}$ given in \eqref{eq:A_equivalence}, $P(\lambda)^*$ is given by the right Yoneda embedding of
$\Psi \circ \iota_E \circ \tau^{-1}(\ul{L}_{\ul{\lambda}})$.


\begin{figure}[h]
  \includegraphics{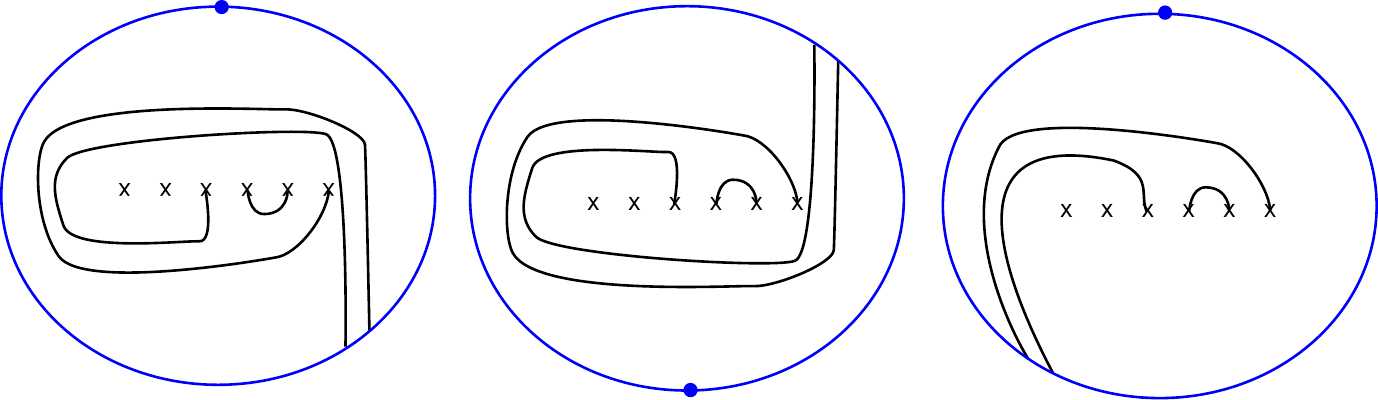}
  \caption{An example of $\tau^{-1}(\ul{L}_{\ul{\lambda}})$ (left), $\iota_E \circ \tau^{-1}(\ul{L}_{\ul{\lambda}})$ (middle)
  and 
  $\Psi \circ \iota_E \circ \tau^{-1}(\ul{L}_{\ul{\lambda}})$ (right)}\label{fig:injective}
 \end{figure}


 %

\subsection{Standard modules} 

For each $\lambda \in \Lambda_{n,m}$, there is a standard module $V(\lambda)$ defined in \cite[Equation (5.11)]{BS11}.
As a graded vector space, $V(\lambda)$ is generated by oriented cap diagrams of the form $\lambda \alpha$.
The collection of all $V(\lambda)$ gives a full exceptional collection for the dg category of right $K^{\alg}_{n,m}$ modules
and $\Ext^i(V(\lambda_0),V(\lambda_1)) \neq 0$ for some $i$ only if $\lambda_0 < \lambda_1$, where $<$ is with respect to the Bruhat partial ordering.

An inductive construction of a projective resolution of $V(\lambda)$ is described in \cite[Theorem 5.3]{BS2}.
By identifying $P(\lambda)$ with the Yoneda image of $\ul{L}_{\ul{\lambda}}$ and imitating that inductive construction, we have

\begin{proposition}\label{p:cellMod}
 The standard module $V(\lambda)$, as an $A_{\infty}$ module over $K^{\alg}_{n,m}=K^{\symp}_{n,m}$, is quasi-isomorphic to the Yoneda image of $\ul{T}^{\lambda^{-1}(\vee)}$ up to grading shift.
\end{proposition}

\begin{proof}[Sketch of proof]
We want to compare the projective resolution of the standard module $V(\lambda)$
and the iterated mapping cone decomposition of $\ul{T}^{\lambda^{-1}(\vee)}$
(cf. the proof of Theorem \ref{t:main}).

When $\lambda^{-1}(\vee)=\{m-n+1,\dots,m\}$, $\lambda$ is the maximal weight and $P(\lambda)=V(\lambda)$. 
It corresponds to $\ul{L}_{\ul{\lambda}}=\ul{T}^{\lambda^{-1}(\vee)}$.

If $\lambda$ is not the maximal weight, then there is $k \in \{1,\dots, m\}$
such that $\lambda(k)=\vee$ and $\lambda(k+1)=\wedge$.
Let $\lambda''$, $\lambda'$ and $\cup_{k,k+1} \lambda'$ be as in the proof of Theorem \ref{t:main}.

The construction of the projective resolution of $V(\lambda)$ comes from iteratively applying the exact sequence of modules in \cite[Equation 5.8]{BS2}.
On the other hand, by Corollary \ref{c:Triangle}, we have the exact triangle of Yoneda modules
\begin{align}
 \ul{T}^{(\lambda'')^{-1}(\vee)}[-1] \to \ul{L}_{\cup_{k,k+1} \lambda'} \to \ul{T}^{(\lambda)^{-1}(\vee)} \to \ul{T}^{(\lambda'')^{-1}(\vee)} \label{eq:ProjRes}
\end{align}
Therefore, it is sufficient to prove that 
the exact sequence of modules in \cite[Equation 5.8]{BS2} corresponds to \eqref{eq:ProjRes}.
 This correspondence can be proved by induction on $m$, induction on the partial ordering of weights (from high to low; note that $\lambda'' > \lambda$)
and the fact that $HF(\ul{T}^{(\lambda'')^{-1}(\vee)}[-1],\ul{L}_{\cup_{k,k+1} \lambda'})$ has rank one, so up to quasi-isomorphism there is only one non-trivial mapping cone.
We leave the details to readers.
\end{proof}

Thus, \emph{the standard modules are the Lefschetz thimbles.}

\begin{remark}
In constrast to the collection of indecomposable projective modules,  
the collection of standard modules does not have formal endomorphism algebra when $n>1$.
A minimal model for the algebra when $n=2$ is given in \cite{Klamt}, which conjectured that this minimal model is not formal; we give a proof of that conjecture in the Appendix.
\end{remark}

By \cite[Theorem 5.3]{BS11}, the set of all costandard  modules of $K^{\alg}_{n,m}$ are given by
\begin{align}
 V(\lambda)^*:= \Hom( V(\lambda), \mathbb{K})
\end{align}
where the right module structure is given by $fa(m):=f(ma^*)$ for $a \in K^{\alg}_{n,m}$,  $m \in V(\lambda)^*$ and $f \in V(\lambda)^*$.
In other words, up to grading shift, it is given by the Yoneda embedding of $\tau^{-1}(\ul{T}^{\lambda^{-1}(\vee)})$
as a right $(K^{\symp}_{n,m})^{op}$ module after pull-back via the algebra isomorphism $K^{\symp}_{n,m}=(K^{\symp}_{n,m})^{op}$.

In view of  \eqref{eq:A_equivalence} again, $V(\lambda)^*$ is given by the right Yoneda embedding of
$\Psi \circ \iota_E \circ \tau^{-1}(\ul{T}^{\lambda^{-1}(\vee)})$.


 \begin{figure}[h]
  \includegraphics{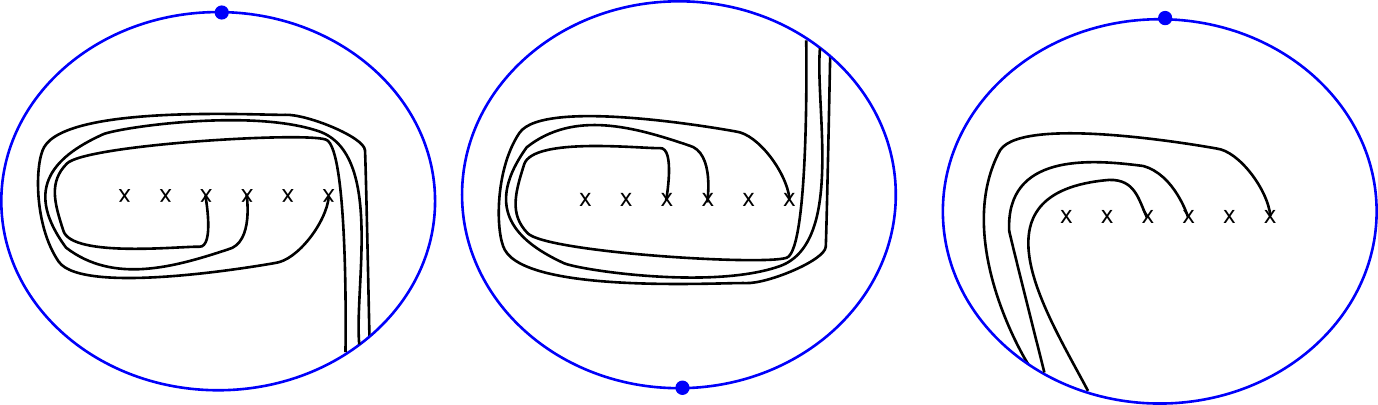}
  \caption{An example of $\tau^{-1}(\ul{T}^{\lambda^{-1}(\vee)})$ (left), $\iota_E \circ \tau^{-1}(\ul{T}^{\lambda^{-1}(\vee)})$ (middle)
  and 
  $\Psi \circ \iota_E \circ \tau^{-1}(\ul{T}^{\lambda^{-1}(\vee)})$ (right)}\label{fig:costandard}
 \end{figure}



In fact, the $\Psi \circ \iota_E \circ \tau^{-1}(\ul{T}^{\lambda^{-1}(\vee)})$ form the right Koszul duals of the thimbles 
(i.e. $HF(\ul{T}^{\lambda_0^{-1}(\vee)},\Psi \circ \iota_E \circ \tau^{-1}(\ul{T}^{\lambda_1^{-1}(\vee)}))=\mathbb{K}$ if and only if $\lambda_0=\lambda_1$, and equal to $0$ otherwise).
It is well-known that costandard modules are right Koszul dual to standard modules (see for example \cite{BK18}).

\subsection{Irreducible modules}

For each $\lambda \in \Lambda_{n,m}$, we define 
$L(\lambda)$ to be the submodule of $K^{\alg}_{n,m}$ generated by the signle oriented circle diagram
\begin{align}
 L(\lambda):= \mathbb{K} \ul{\lambda} \lambda \ol{\lambda}
\end{align}
The collection of all $L(\lambda)$ is the set of all irreducible modules of $K^{\alg}_{n,m}$.

It is not immediately clear in our context which Lagrangians correspond to the  irreducible modules.
In view of the work of \cite{EL17},  and the `topological model' for the compact core of the space $\cY_{n,m}$ arising from its description as a quiver variety (see e.g. \cite{AbouzaidSmith19}), it is reasonable to believe 
that one can enlarge the category  $\cFS^{cyl,n}(E)$ to allow Lagrangian discs that intersect exactly one of the $\{ \ul{L}_{\ul{\lambda}} \}_{\lambda}$ once and are disjoint from the others. These would be the analogues of `cocore discs' over the components of the skeleton of a plumbing of cotangent bundles, and the Yoneda images of these Lagrangian discs are natural candidates for the irreducible modules.


\section{Symplectic annular Khovanov homology}\label{s:Annular}

This section gives an application of our main results to link homology theories in the sense of Khovanov.   
Suppose  $m=2n$. A braid $\beta \in Br_n$ in the $n$-string braid group defines bimodules $P_{\beta}^{(j)}$ over each of the extended arc algebras $K_{j,n}^{\alg}$ with $0\leq j \leq n$. 
One can combine a theorem of Roberts \cite{Roberts13} with deep recent work of Beliakova, Putyra and Wehrli \cite{BPW19} 
to infer that there is a spectral sequence $\oplus_j HH_*(P_\beta^{(j)}) \Rightarrow Kh(\kappa(\beta))$, where $\kappa(\beta)$ is  the link closure of $\beta$, and where $Kh(\kappa)$ is the Khovanov homology \cite{Khovanov} of the link $\kappa \subset S^3$, constructed as a categorification of the Jones polynomial.  It is by now classical that Khovanov homology can be understood as a certain morphism group $\mathrm{Ext}_{\rperf H_{n,2n}^{\alg}}(P, (\beta\times\mathrm{id})(P))$ in the derived category of modules over the \emph{compact} arc algebra $H_{n,2n}^{\alg}$, see \cite{Khovanov-arc}, where $\beta \times \mathrm{id} \in Br_{2n}$ belongs to the $2n$-string braid group and $P$ is a particular projective.  From the viewpoint of the (extended) arc algebras, the existence of such a spectral sequence is rather mysterious. The purpose of this section is to give a transparent account of why it should exist, starting from a new semi-orthogonal decomposition of $per\!f$-$K_{n,2n}^{\alg}$ which we shall derive from the geometric viewpoint afforded by Theorem \ref{t:main}.

We assume $m=2n$ in Sections \ref{s:Annular} and \ref{s:Ann2Kh} unless stated otherwise.

\subsection{The annular symplectic theory}
 
Let $\hs \in \Lambda_{n,m}$ be the weight such that $\hs^{-1}(\vee)=\{1,\dots,n\}$.
The corresponding tuple $\ul{L}_{\hs}$ is called the `horseshoe' Lagrangian tuple
and the Lagrangian $\Sym(\ul{L}_{\hs}) \subset \cY_{n,m}$ is the horseshoe Lagrangian from \cite{SeidelSmith,AbouzaidSmith16}.
Let $Br_n$ be the braid group on $n$ strands.
The paper \cite{SeidelSmith} associates to each element $\beta \in Br_n$ a symplectomorphism $\phi^{(n)}_{\beta}: \cY_{n,m} \to \cY_{n,m}$, and defines   the \emph{symplectic Khovanov cohomology} of the braid closure $\kappa = \kappa(\beta)$ of $\beta$ to be 
\begin{align}
 Kh^{\symp}(\kappa(\beta)) := HF(\Sym(\ul{L}_{\hs}), \phi^{(n)}_{\beta}(\Sym(\ul{L}_{\hs})))
\end{align}
The main result of \cite{SeidelSmith} is that $Kh^{\symp}(\kappa(\beta))$ is a link invariant, i.e. it is independent of the representation of $\kappa$ as a braid closure.

One can give an equivalent definition of symplectic Khovanov cohomology using $\ul{L}_{\hs}$ instead of its product $\Sym(\ul{L}_{\hs})$. 
Let $E:=A_{m-1}$ denote the Milnor fiber.  
To each simple braid element $\sigma_i$ we associate the symplectomorphism $\phi_{\sigma_i}: E \to E$ given by the Dehn twist along the $i^{th}$ matching sphere (lying above 
the line joining $i+\sqrt{-1}$ and $i+1+\sqrt{-1}$). This defines a representation $Br_m \to \pi_0\mathrm{Symp}_{ct}(E)$.
Let $Br_n \hookrightarrow Br_m$ be the embedding of the left $n$ strands. 
We have the restricted representation
$Br_n \to \pi_0\mathrm{Symp}_{ct}(E)$
so each braid $\beta \in Br_n$ determines a symplectomorphism $\phi_{\beta}$ of $E$ that acts as the identity on the `right half'. 

Then \cite{Manolescu06}
\begin{align}
 Kh^{\symp}(\kappa(\beta))=HF(\ul{L}_{\hs}, \phi_{\beta}(\ul{L}_{\hs})).
\end{align}

Let $E_{1/2}$ be the $A_{n-1}$-Milnor fiber with its standard Lefschetz fibration $\pi_{E_{1/2}}$.
We can define the symplectomorphism $\phi_{\beta}$ of $E_{1/2}$ for $\beta \in Br_n$ accordingly.
Then $\phi_{\beta}$ induces an $A_{\infty}$ endofunctor of $\cFS^{cyl,j}(\pi_{E_{1/2}})$, and hence an $A_{\infty}$-bimodule $P_{\beta}^{(j)}$ over $\cFS^{cyl,j}(\pi_{E_{1/2}})$.

\begin{definition}\label{defn:AnnSympKh}
The \emph{symplectic annular Khovanov homology} of the braid $\beta \in Br_n$ is the direct sum of Hochschild homology groups
\begin{align}
 AKh^{\symp}(\beta):=\oplus_{j=0}^n HH_*(\cFS^{cyl,j}(\pi_{E_{1/2}}),P_{\beta}^{(j)})
\end{align}
where, by convention, $\cFS^{cyl,j}(\pi_{E_{1/2}})=\mathbb{K}$ and $P_{\beta}^{(j)}=\mathbb{K}$ is the diagonal bimodule when $j=0$ (cf. Remark \ref{r:nm}).
\end{definition}

Symplectic annular Khovanov homology is clearly a braid invariant. It is not an invariant of the link closure of the braid.  Its definition is motivated by the fact that combinatorial annular Khovanov homology, originally defined by a diagrammatic calculus for links in a solid torus \cite{APS04}, is itself isomorphic to a direct sum of 
Hochschild homologies of  braid bimodules over the extended arc algebras, a recent theorem of \cite{BPW19} establishing a conjecture due to \cite{AGW}.
Indeed, Theorem \ref{t:main} and the same formality arguments for bimodules as in \cite{AbouzaidSmith19} would prove that

\begin{proposition} \label{p:equal}
 $AKh^{\symp}(\beta)$ is isomorphic to annular Khovanov homology when $\mathbb{K}$ has characteristic $0$.
\end{proposition}

We will prove Theorem \ref{t:sseq} by constructing a spectral sequence $AKh^{\symp}(\beta) \Rightarrow Kh^{\symp}(\kappa(\beta))$.

\begin{remark}\label{r:fixedFloer}
One could also define a braid invariant as the direct sum of fixed point Floer homologies of $\phi_\beta$ on $\cY_{j,n}$ over $j \in \{0,\ldots, n\}$ (where `partial wrapping' would remove fixed points at infinity).
It seems likely that this geometric definition would recover that of Definition \ref{defn:AnnSympKh}, but establishing such an isomorphism is beyond the scope of this paper. See \cite{Seidel2.5} for closely related results.
\end{remark}

\begin{remark}
 There is an $\liesl_2$ action on $AKh^{\symp}(\beta)$, by combining  \cite{GLW18} and Proposition \ref{p:equal}, which implies a rank inequality 
 $$\rank(HH_*(\cFS^{cyl,j}(\pi_{E_{1/2}}),P_{\beta}^{(j)})) \le \rank(HH_*(\cFS^{cyl,j+1}(\pi_{E_{1/2}}),P^{(j+1)}_{\beta}))$$
 for all $j< n/2 $.
Coupled with Remark \ref{r:fixedFloer}, this predicts a non-trivial existence result for periodic points of the symplectomorphisms of $\cY_{j,n}$ associated to braids. 
It would be interesting to construct the $\liesl_2$ action symplectic-geometrically.
\end{remark}

\subsection{Embedding of algebras}
Let $E:=A_{m-1}$ and $\pi_E$ be the standard Lefschetz fibration.
Let $W_1:=\{re(z) < n+\frac{1}{2}\}$ and $W_2:=\{re(z) > n+\frac{1}{2}\}$.
Note that $\cFS^{cyl,j}_{W_1}(\pi_E)=\cFS^{cyl,j}(\pi_{E_{1/2}})$.
As explained in Section \ref{sss:CanEmbed}, there is a faithful $A_{\infty}$ functor $\sqcup \ul{K}: \cFS^{cyl,j}_{W_1}(\pi_E) \to \cFS^{cyl,n}(\pi_E)$
for each object $\ul{K}$ of $\cFS^{cyl,n-j}_{W_2}(\pi_E)$.

By essentially the same argument, we can prove that 
\begin{align}
 HF(\ul{L_0} \sqcup \ul{K_0}, \ul{L_1} \sqcup \ul{K_1})=HF(\ul{L_0},\ul{L_1}) \otimes HF(\ul{K_0},\ul{K_1})
\end{align}
for $\ul{L}_i \in \cFS^{cyl,j}_{W_1}(\pi_E)$ and $\ul{K}_i \in \cFS^{cyl,j}_{W_2}(\pi_E)$.
Varying either the $\ul{L}_i$ or $\ul{K}_i$, Floer multiplication amongst the corresponding groups also respects the tensor product structure.
Let $0 \le j \le n$ be an integer.
For each $\lambda \in \Lambda_{j,n}$, one can define an object $\ul{L}_{\ul{\lambda}}$ in  $\cFS^{cyl,j}_{W_1}(\pi_E)$
by forgetting the right $n$ critical values (see Figure \ref{fig:Embedding}). 
Our convention is $\cFS^{cyl,j}_{W_1}(\pi_E)=\mathbb{K}$ when $j=0$, i.e. that the category contains a unique object whose endomorphism algebra is the ground field concentrated in degree zero;  the object  $\ul{L}_{\ul{\lambda}}$ is by definition the unique non-zero object in the category in this case.
For the other extreme, when $j=n$, the object  $\ul{L}_{\ul{\lambda}}$ is the tuple of Lagrangian thimbles (cf. Corollary \ref{c:nm}).
Similarly, for each $\mu \in \Lambda_{n-j,m-n}$, one can define an object $\ul{K}_{\ul{\mu}}$ in  $\cFS^{cyl,n-j}_{W_2}(\pi_E)$ by forgetting the left $n$ critical values.
We have an algebra isomorphism
\begin{align}
 \bigoplus \, HF(\ul{L}_{\lambda_0} \sqcup \ul{K}_{\mu_0}, \ul{L}_{\lambda_1} \sqcup \ul{K}_{\mu_1}) =K^{\symp}_{j,n} \times K^{\symp}_{n-j,m-n} \label{eq:DGproduct}
\end{align} 
where the sum is over all $\lambda_0,\lambda_1 \in \Lambda_{j,n}$ and $\mu_0,\mu_1 \in \Lambda_{n-j,m-n}$.

\begin{figure}[h]
  \includegraphics{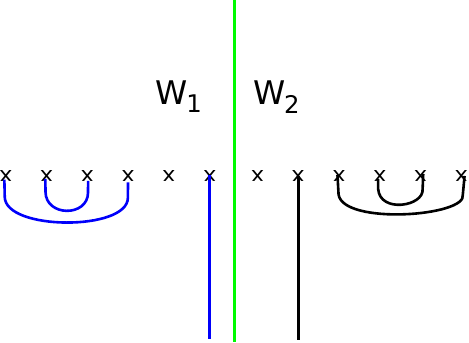}
  \caption{Lagrangian tuples $\ul{L}_{\ul{\lambda}}$ (blue, left) and $\ul{K}_{\ul{\mu}}$ (black, right)}\label{fig:Embedding}
 \end{figure}

Let $A$ be the $A_{\infty}$ full subcategory of $\cFS^{cyl,n}(\pi_E)$ consisting of objects  $\ul{L}_{\lambda} \sqcup \ul{K}_{\mu}$ over all 
$\lambda \in \Lambda_{j,n}$, $\mu\in \Lambda_{n-j,m-n}$ and $j=0,\dots,n$.
For each integer $0 \le j \le n$, we have the full subcategory $A_j$ of $A$  given by the objects  $\ul{L}_{\lambda} \sqcup \ul{K}_{\mu}$ over all 
$\lambda \in \Lambda_{j,n}$ and $\mu\in \Lambda_{n-j,m-n}$.

\begin{lemma}\label{l:productDG}
For each integer $0 \le j \le n$, the endomorphism algebra of the objects in $A_j$ is formal, and hence
 quasi-isomorphic to the right hand side of \eqref{eq:DGproduct}.
\end{lemma}

\begin{proof} This can be seen as a version of the K\"unneth theorem, i.e. 
that the $A_{\infty}$ structure on a product Lagrangian is quasi-isomorphic to the tensor product of a dg model for each factor (see \cite{Amorim} for the relevant theory).

Alternatively, and more concretely, in our case we observe that each $\ul{L}_{\lambda} \sqcup \ul{K}_{\mu}$ is pure with respect to the nc vector field for the same reason in Lemma \ref{l:Lpure}.
Moreover,  the cohomological algebra of the collection of $\ul{L}_{\lambda} \sqcup \ul{K}_{\mu}$ splits as a tensor product,
so we can repeat the strategy from Section \ref{s:formality}. We next give the details of this second approach. 

We choose a total ordering $(\le^T_1)$ and $(\le^T_2)$ for $\Lambda_{j,n}$ and $\Lambda_{n-j,m-n}$, respectively, 
both of which refine the Bruhat partial ordering.
We can define a total order $(\le^T)$ on the pairs $(\lambda,\mu) \in \Lambda_{j,n} \times \Lambda_{n-j,m-n}$ by declaring that 
\begin{align}
 (\lambda_0,\mu_0)<^T(\lambda_1,\mu_1) \text{ if and only if } (\lambda_0<^T_1\lambda_1) \text{ or } (\lambda_0=\lambda_1 \text{ and } \mu_0<^T_2\mu_1)
\end{align}
Now, we can run the induction in the proof of Proposition \ref{p:bStructure}.

For the rest of the proof, we use $L_{\lambda,\mu}$ to denote $\ul{L}_{\lambda} \sqcup \ul{K}_{\mu}$.
Let $(\eta_1,\eta_2 ) \in \Lambda_{j,n} \times \Lambda_{n-j,m-n}$ and suppose that a $b$-equivariant structure on $\ul{L}_{\lambda,\mu}$
has been chosen for all $(\lambda,\mu) >^T (\eta_1,\eta_2)$ such that $b$ is moreover pure.
We want to choose a $b$-equivariant structure on $\ul{L}_{\eta_1,\eta_2}$ such that $b$ is pure on 
$\{\ul{L}_{\lambda,\mu}\}_{(\lambda,\mu) \ge^T (\eta_1,\eta_2)}$.

Recall the notation for the maximal element $w_{\lambda}$ associated to $\lambda \in \Lambda_{j,n}$ from \eqref{eq:maxElt}.  We choose a  $b$-equivariant structure on $\ul{L}_{\eta_1,\eta_2}$ such that the minimal degree element of
 $HF(\ul{L}_{\eta_1,\eta_2}, \ul{L}_{w_{\eta_1},w_{\eta_2}})$
 has weight equal to the degree.
 By Lemma \ref{l:cyclic}, $HF(\ul{L}_{\eta_1,\eta_2}, \ul{L}_{w_{\eta_1},w_{\eta_2}})$ 
 is cyclic as a module over $HF(\ul{L}_{w_{\eta_1},w_{\eta_2}},\ul{L}_{w_{\eta_1},w_{\eta_2}})$ 
 so weight equals degree for all pure degree elements in $HF(\ul{L}_{\eta_1,\eta_2}, \ul{L}_{w_{\eta_1},w_{\eta_2}})$. By Lemma \ref{l:cyclic} and \ref{l:constantmap}, and exactly the same argument  as in Proposition \ref{p:bStructure}, we know 
that all elements of $HF(\ul{L}_{w_{\eta_1},w_{\eta_2}},\ul{L}_{\eta_1,\eta_2})$
have weight equal to their degrees.

Now, suppose that $(\lambda,\mu) >^T (\eta_1,\eta_2)$ but $(\lambda,\mu) \neq (w_{\eta_1},w_{\eta_2})$, 
and purity holds for all of  $HF(L_{\lambda',\mu'}, L_{\eta_1,\eta_2})$, $HF(L_{\eta_1,\eta_2}, L_{\lambda',\mu'})$
such that $(\lambda',\mu') >^T (\lambda,\mu)$.

\begin{claim}\label{c:exist}
 If $HF(L_{\lambda,\mu}, L_{\eta_1,\eta_2}) \neq 0$, then there exists $(\lambda',\mu')  >^T (\lambda,\mu)$
 such that both $\ul{\lambda} \lambda' \ol{\eta_1}$ and $\ul{\mu} \mu' \ol{\eta_2}$ are oriented circle diagrams.
\end{claim}

Assuming Claim \ref{c:exist}, 
one can show that purity holds also for  $HF(L_{\lambda,\mu}, L_{\eta_1,\eta_2})$, $HF(L_{\eta_1,\eta_2}, L_{\lambda,\mu})$ exactly as in  the last two paragraphs of the proof of Proposition \ref{p:bStructure}.
Therefore, the result follows by induction.
\end{proof}

\begin{proof}[Proof of Claim \ref{c:exist}]
 By assumption, we have
 \begin{align}
  (w_{\eta_1},w_{\eta_2})>^T (\lambda,\mu) >^T (\eta_1,\eta_2) \label{eq:ineq}
 \end{align}
By the definition of $>^T$, we have $\lambda=\eta_1$ or $\lambda >^T_1 \eta_1$.

If $\lambda=\eta_1$, then $\mu >^T_2 \eta_2$.
If $\mu \neq w_{\eta_2}$, then  by (2) of Lemma \ref{l:3choose1},
there exists $\mu' >^T_2 \mu, \eta_2$ such that $\ul{\mu} \mu' \ol{\eta_2}$ is an oriented circle diagram.
We can take $\lambda'=\eta_1$ so that we have $(\lambda',\mu')  >^T (\lambda,\mu)$ and 
both $\ul{\lambda} \lambda' \ol{\eta_1}$ and $\ul{\mu} \mu' \ol{\eta_2}$ are oriented circle diagrams.

If $\lambda=\eta_1$ and $\mu = w_{\eta_2}$, then by \eqref{eq:ineq}, we have $\eta_1 \neq w_{\eta_1}$ 
(in fact, $\nu=w_{\nu}$ only when $\nu$ is the maximal weight, see \eqref{eq:maxElt}).
In this case, we can take $(\lambda',\mu')=(w_{\eta_1},w_{\eta_2})$.

If $\lambda >^T_1 \eta_1$,  then either $\lambda=w_{\eta_1}$ or $\lambda \neq w_{\eta_1}$.
If $\lambda \neq w_{\eta_1}$, then by (2) of Lemma \ref{l:3choose1}, 
there exists $\lambda' >^T_1 \lambda, \eta_1$ such that $\ul{\lambda} \lambda' \ol{\eta_1}$ is an oriented circle diagram.
By $HF(L_{\lambda,\mu}, L_{\eta_1,\eta_2}) \neq 0$, there exists $\mu'$ such that $\ul{\mu} \mu' \ol{\eta_2}$ is an oriented circle diagram.
We can pick any such $\mu'$ because $(\lambda',\mu')  >^T (\lambda,\mu)$ no matter which $\mu'$ we pick.

If $\lambda >^T_1 \eta_1$ but $\lambda = w_{\eta_1}$, then we pick $\lambda'=\lambda=w_{\eta_1}$.
Since $(w_{\eta_1},w_{\eta_2})>^T (\lambda,\mu)$, we have $\mu <^T_2 w_{\eta_2}$.
If $\eta_2 \neq w_{\mu}$, then by (2) of Lemma \ref{l:3choose1}, 
there exists $\mu' >^T_2 \mu, \eta_2$ such that $\ul{\mu} \mu' \ol{\eta_2}$ is an oriented circle diagram so we have done.
If $\eta_2 = w_{\mu}$, then $\mu <^T_2 \eta_2$ so we can take $\mu'=\eta_2$
becasue it implies that $(\lambda',\mu')=(w_{\eta_1},\eta_2)  >^T (w_{\eta_1},\mu)=(\lambda,\mu)$
and both $\ul{\lambda} \lambda' \ol{\eta_1}$ and $\ul{\mu} \mu' \ol{\eta_2}$ are oriented circle diagrams.

This completes the proof.
\end{proof}

It is not clear whether the endomorphism algebra of the direct sum of all the objects in $A$ is formal or not. 
However, the dg category of perfect modules $\rperf A$ is formal in the following  weaker sense.

\begin{lemma}\label{l:switchingProjective}
$\rperf A$ is quasi-equivalent to $\rperf K_{n,m}^{\alg}$.
\end{lemma}

\begin{proof}
 By iteratively applying exact triangles, we see that the objects in $A_j$ can generate thimbles $\ul{T}=\{T_1,\dots,T_n\}$ in $\cFS^{cyl,n}(\pi_E)$ such that $j$ of the $T_i$'s are contained in 
 $\pi_E^{-1}(W_1)$, and the remaining $n-j$ are contained in $\pi_E^{-1}(W_2)$.
 Therefore, every thimble of $\cFS^{cyl,n}(\pi_E)$ can be generated by objects in $A$, so the result follows. 
\end{proof}

\begin{lemma}\label{l:semiOrtho}
Let $X_0 \in A_{j_0}$ and $X_1 \in A_{j_1}$.
If $HF(X_0,X_1) \neq 0$, then $j_0 \ge j_1$.
As a result, $\langle A_n,\dots,A_0 \rangle$ is a semi-orthogonal decomposition of $A$.
\end{lemma}

\begin{proof}
 It is clear that $CF(X_0,X_1)=0$ if $j_0 < j_1$, because of the definition of Floer cochains via positive isotopies along $\mathbb{R} = \partial \mathbb{H}$. 
\end{proof}

 \begin{theorem}\label{t:semiOrtho}
  $\rperf K_{n,m}^{\alg}$ admits a semi-orthogonal decomposition $\langle \rperf A_n,\dots, \rperf A_0 \rangle$, where for each $j$, $\rperf A_j$ is quasi-equivalent to 
  $\rperf(K_{j,n}^{\alg} \times K_{n-j,m-n}^{\alg})$.
 \end{theorem}

 \begin{proof}
  This immediately follows from Lemmas \ref{l:productDG}, \ref{l:switchingProjective} and \ref{l:semiOrtho}. 
 \end{proof}
 Theorem \ref{t:semiOrtho} categorifies the identity ${m \choose n}  = \sum_{j=0}^n {n \choose j} {m-n  \choose n-j}$ for ranks of Grothendieck groups.

 \begin{remark}
  Even though we are primarily interested in the case $m=2n$ in this section, Theorem \ref{t:semiOrtho} holds for all $n<m$.   
  This algebraic result seems to be new and may be of independent interest.  \end{remark}

  We need another full subcategory of $\cFS^{cyl,n}(\pi_E)$ that is analogous to $A$.
  Let
  \begin{align}
   W^!_2:=\{n+\frac{1}{2}<re(z) <m+\frac{1}{2} \text{ and }im(z) <2\} \cup \{re(z) <m+\frac{1}{2} \text{ and }im(z) <\frac{1}{2}\}
  \end{align}
  and $W^!_1$ be the complement of $W^!_2$.
 For $\lambda \in \Lambda_{j,n}$ and $\mu \in \Lambda_{n-j,m-n}$, we can consider the corresponding
 upper Lagrangian tuples 
 $\ul{L}^!_{\ol{\lambda}} \in \cFS^{cyl,j}_{W^!_1}(\pi_E)$
 and  
 $\ul{K}^!_{\ol{\mu}} \in \cFS^{cyl,n-j}_{W^!_2}(\pi_E)$ (see Figure \ref{fig:AnotherEmbedding}).

\begin{figure}[h]
  \includegraphics{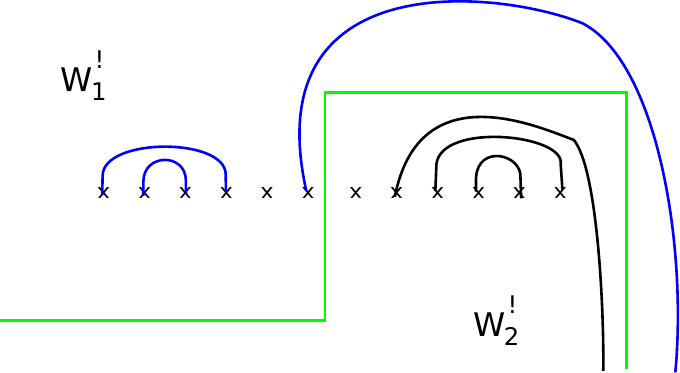}
  \caption{Lagrangian tuples $\ul{L}^!_{\ol{\lambda}}$ (blue) and $\ul{K}^!_{\ol{\mu}}$ (black)}\label{fig:AnotherEmbedding}
 \end{figure}

 Let $A^!$ be the $A_{\infty}$ full subcategory of $\cFS^{cyl,n}(\pi_E)$ consisting of objects  $\ul{L}^!_{\lambda} \sqcup \ul{K}^!_{\mu}$ over all 
$\lambda \in \Lambda_{j,n}$, $\mu\in \Lambda_{n-j,m-n}$ and $j=0,\dots,n$.
For each integer $0 \le j \le n$, we have the full subcategory $A^!_j$ of $A^!$  given by the objects  $\ul{L}^!_{\lambda} \sqcup \ul{K}^!_{\mu}$ over all 
$\lambda \in \Lambda_{j,n}$ and $\mu\in \Lambda_{n-j,m-n}$.
The following Koszulness property justifies the notation $(-)^!$.

\begin{lemma}\label{l:Koszul dual}
Let $X^! \in A^!_i$ and $X \in A_j$.
If $HF(X^!,X) \neq 0$, then $i=j$.
\end{lemma}

\begin{proof}
 The direction of wrapping shows 
 that $CF(X^!,X)$ is non-zero only if $i=j$.
\end{proof}

On the other hand, $A^!$ shares many features with $A$.
By the same arguments as in Lemma \ref{l:productDG} and \ref{l:switchingProjective}, we know that 
$\rperf A^!$ is quasi-equivalent to $\rperf K_{n,m}^{\alg}$ and
for $j=0,\dots,n$, the category 
$A^!_j$ is formal with endomorphism algebra
\begin{align}
\bigoplus \, HF(\ul{L}^!_{\lambda_0} \sqcup \ul{K}^!_{\mu_0}, \ul{L}_{\lambda_1}^! \sqcup \ul{K}^!_{\mu_1}) =K^{\symp}_{j,n} \times K^{\symp}_{n-j,m-n} \label{eq:DGproduct2}
\end{align}
The analogues of Lemma \ref{l:semiOrtho2} and Theorem \ref{t:semiOrtho} are

\begin{lemma}\label{l:semiOrtho2}
Let $X_0 \in A^!_{j_0}$ and $X_1 \in A^!_{j_1}$.
If $HF(X_0,X_1) \neq 0$, then $j_0 \le j_1$.
As a result, $\langle A^!_0,\dots,A^!_n \rangle$ is a semi-orthogonal decomposition of $A^!$.
\end{lemma}

 \begin{theorem}\label{t:semiOrtho2}
  $\rperf K_{n,m}^{\alg}$ admits a semi-orthogonal decomposition $\langle \rperf A^!_0,\dots, \rperf A^!_n \rangle$, where for each $j$, $\rperf A^!_j$ is quasi-equivalent to 
  $\rperf(K_{j,n}^{\alg} \times K_{n-j,m-n}^{\alg})$.
 \end{theorem}

 One consequence of Theorems \ref{t:semiOrtho} and \ref{t:semiOrtho2}
 is that, when $m=2n$, $\rperf(K_{j,n} \times K_{n-j,m-n})$ is the same as $[(K_{n-j,n})^{op},K_{j,n}]$.
 By the isomorphism $(K_{n-j,n})^{op}=K_{j,n}$ (see Corollary \ref{c:projDual}), this in turn is isomorphic to $[K_{j,n},K_{j,n}]$. 
 Similarly, $(K_{j,n} \times K_{n-j,n})\lperf$ is the same as $[K_{j,n},(K_{n-j,n})^{op}]$, which in turn is isomorphic to $[K_{j,n},K_{j,n}]$. 

\subsection{A Beilinson-type spectral sequence}\label{ss:digression}

We next explain why a semi-orthogonal decomposition of an $A_{\infty}$ category $\cC$  induces a spectral sequence with target a given morphism group in $\cC$. 
From now on, for an object $X$, we use $X^r$ and $X^l$ to denote its right and left Yoneda modules, respectively.

 Let $\cC$ be a split-closed trianglated $A_{\infty}$ category with a semi-orthogonal decomposition
 \begin{align}
  \cC =\langle \cC_n,\dots,\cC_0 \rangle \label{eq:semiC}
 \end{align}
A Koszul dual semi-orthogonal decomposition is a semi-orthogonal decomposition
 \begin{align}
  \cC =\langle \cC^!_0,\dots,\cC^!_n \rangle
 \end{align}
such that for $X \in \cC^!_i$ and $Y \in \cC_k$
\begin{align}
 H(hom_\cC(X,Y)) \neq 0 \text{ only if } i=k \label{eq:cKoszul}
\end{align}
Without loss of generality, we can assume each summand in the semi-orthogonal decomposition is split-closed trianglated.  Note that Koszul dual semi-orthogonal decompositions always exist \cite{Kuznetsov}.

Let $\iota_j:\cC_j \to \cC$ and $\iota_j^!:\cC^!_j \to \cC$ be the embeddings. 
Let $\pi_j:\cC \to \cC_j$ and $\pi_j^!:\cC \to \cC^!_j$ be the projection functors (i.e. the unique functors such that $\pi_j \circ \iota_j=id_{\cC_j}$ and $\pi_j \circ \iota_i=0$ for $i \neq j$).
These functors induce pull back functors on modules.

\begin{lemma}
 Up to quasi-isomorphism, the left Yoneda embedding $\cC^!_j \to \cC\lperf$ factors through $(\pi_j)^*:\cC_j\lperf \to \cC\lperf$.
\end{lemma}

\begin{proof}
 Let $X \in \cC^!_j$.
 By \eqref{eq:cKoszul}, we have $H(hom_\cC(X,X'))=0$ for all $X' \in \cC_i$ such that $i \neq j$.
 It implies that, up to quasi-isomorphism, we have 
 \begin{align}
  X^l= (\pi_j)^* \circ (\iota_j)^*(X^l)
 \end{align}
  so the result follows.
\end{proof}

Since Yoneda embedding and $(\pi_j)^*$ are both cohomologically full and faithful, so is the functor 
\begin{align}
(\iota_j)^*((-)^l):  \cC^!_j \to \cC_j\lperf \label{eq:ssEmbedding}
\end{align}
 By \eqref{eq:cKoszul} again, we can see that
\begin{align}
 (\iota_j)^*((-)^l):  \cC^!_i \to \cC_j\lperf \label{eq:0functor}
\end{align}
is the $0$ functor when $i \neq j$.

\begin{lemma}
\eqref{eq:ssEmbedding} is essentially surjective, so induces a quasi-equivalence $\cC^!_j \to \cC_j\lperf$.
\end{lemma}

\begin{proof}
Combining  the fact that \eqref{eq:0functor} is the $0$ functor and that both Yoneda embedding and $(\iota_j)^*$ are essentially surjective, 
we know that \eqref{eq:ssEmbedding} is essentially surjective, so the result follows.
\end{proof}

Let $\psi_j: \cC_j\lperf \to  \cC^!_j $ be a quasi-inverse.

\begin{lemma}\label{l:decomposeX}
 Let $X \in \cC$ be quasi-isomorphic to an iterated mapping cone of the form
 \begin{align}
  X=\Cone(\dots \Cone(X_0 \to X_1) \dots X_n) \label{eq:coneX}
 \end{align}
 where $X_j \in \cC^!_j$ for $j=0,\dots,n$.
 Then $\psi_j \circ (\iota_j)^*(X^l)$ is quasi-isomorphic to $X_j$.
\end{lemma}

\begin{proof}
 By applying the functor $\psi_j \circ (\iota_j)^*((-)^l)$
 to the iterated mapping cone \eqref{eq:coneX}, and using the fact that \eqref{eq:0functor} is a $0$ functor, we see that only the object $X_j$ contributes
 and hence
 \begin{align}
  \psi_j \circ (\iota_j)^*(X^l) \simeq \psi_j \circ (\iota_j)^*(X_j^l). \label{eq:singleXj}
 \end{align}
 By definition, $\psi_j \circ (\iota_j)^*((-)^l)$ is quasi-isomorphic to the identity functor on $\cC_j$ so 
 the right hand side of \eqref{eq:singleXj} is in turn  quasi-isomorphic to $X_j$.
 \end{proof}

\begin{proposition}\label{p:spectral}
 Let $X,Y \in \cC$.
 There is a spectral sequence to $H(hom_{\cC}(X,Y))$ from
 \begin{align}
  \oplus_{j=0}^n   \, H((\iota^!_j)^*(Y^r)   \otimes_{\cC^!_j} \psi_j \circ (\iota_j)^*(X^l))
 \end{align}
where $\psi_j \circ (\iota_j)^*(X^l)$, as an object in $\cC^!_j$, is regarded as a left $\cC^!_j$ module via the Yoneda embedding.
 
\end{proposition}

\begin{proof}
 By \eqref{eq:semiC}, $X$ is quasi-isomorphic to an iterated mapping cone of the form \eqref{eq:coneX}.
 An expression of $X$ as such a mapping cone induces a filtration on $Y^r(X)$, with direct sum of graded pieces being
 \begin{align}
  \oplus_{j=0}^n Y^r(X_j). \label{eq:filtration}
 \end{align}
 Thus we have a spectral sequence from the cohomology of \eqref{eq:filtration} to $H(hom_{\cC}(X,Y))$.
 
 On the other hand, we have quasi-isomorphisms of cochain complexes
 \begin{align}
  Y^r(X_j)&= (\iota^!_j)^*(Y^r)(X_j) \\
  &=(\iota^!_j)^*(Y^r)(\psi_j \circ (\iota_j)^*(X^l)) \\
  &=(\iota^!_j)^*(Y^r)   \otimes_{\cC^!_j} \psi_j \circ (\iota_j)^*(X^l)
 \end{align}
where the second quasi-isomorphism comes from Lemma \ref{l:decomposeX}, and the other quasi-isomorphisms come from the definitions of the objects involved.
 \end{proof}

\begin{remark}
 In the extreme case where each $\cC_j$ (and $\cC_j^!$) is an exceptional object, the semi-orthogonal decomposition
 is a full exceptional collection and Proposition \ref{p:spectral} reduces to the well-known Beilison type spectral sequence (cf. \cite[Section (5l)]{SeidelBook}).
\end{remark}

\section{From annular to ordinary Khovanov homology}\label{s:Ann2Kh}

In this section, we explain how to apply Proposition \ref{p:spectral} 
to the semi-orthogonal decompositions in Theorem \ref{t:semiOrtho}, \ref{t:semiOrtho2} to obtain
a spectral sequence from annular to ordinary (symplectic) Khovanov homology.

Recall that for an $A_{\infty}$ or dg category/algebra $\cC$,
the Hochschild homology of a $\cC\text{-}\cC$ bimodule $P$ is defined to be
\begin{align}
 HH_*(\cC,P):= H(\Delta_{\cC} \otimes_{\cC\text{-}\cC} P)
\end{align}
where $\Delta_{\cC}$ is the diagonal bimodule and $\otimes_{\cC\text{-}\cC}$ is the derived tensor product over $[\cC,\cC]$.
Equivalently, $\Delta_{\cC}$ defines a right $\cC \otimes \cC^{op}$ module
and $P$ defines a left $\cC \otimes \cC^{op}$ module, and we have 
\begin{align}
 HH_*(\cC,P):= H(\Delta_{\cC} \otimes_{\cC \otimes \cC^{op}} P)
\end{align}
The fact that Hochschild homology groups appear in the first page of a spectral sequence to symplectic Khovanov cohomology arises from a relation between the horseshoe Lagrangian and the diagonal bimodule, that we explain next. More precisely, in section \ref{ss:diagonalBimod}, we compute $(\iota_j)^*(\ul{L}_{\hs}^l)$ and $(\iota^!_j)^*(\phi_\beta(\ul{L}_{\hs})^r)$
for the embeddings $\iota_j: A_j \to \cFS^{cyl,n}(\pi_E)$ and $\iota^!_j: A^!_j \to \cFS^{cyl,n}(\pi_E)$.
Then, we describe the quasi-inverse $\psi_j$ and complete the proof of Theorem \ref{t:sseq} in Section \ref{ss:proofofss}.

\subsection{Horseshoe Lagrangian tuple and diagonal bimodule}\label{ss:diagonalBimod}


Recall that we have the pull-back functor
\begin{align}
 (\iota_j)^*:  \cFS^{cyl,n}(\pi_E)\lperf \to  A_j\lperf \simeq [K_{j,n},K_{j,n}]
\end{align}
The following is the key technical result.

\begin{proposition}\label{p:diagonalbimod}
 Up to grading shift, $(\iota_j)^* (\ul{L}_{\hs}^l)$ is quasi-isomorphic to the diagonal $K_{j,n}$-bimodule.
\end{proposition}

We are going to divide the proof into two steps.
First, we will show that $(\iota_j)^* (\ul{L}_{\hs}^l)$ coincides with the diagonal bimodule (up to grading shift) on the cohomological level.
Then, we will explain how to adapt the strategy in Section \ref{s:formality} to prove that this bimodule is formal.
For all the Lagrangian tuples in the proof, we continue to use our grading convention \eqref{eq:FloerGrading} and the cost is that
we will see eventually that $(\iota_j)^* (\ul{L}_{\hs}^l)$ corresponds to the diagonal bimodule shifted by $n-j$ instead of the diagonal bimodule.
This arises from the fact that there are $n-j$ matching paths of $(\iota_j)^* (\ul{L}_{\hs}^l)$ being oriented clockwise in the sense of
the algebraic extended arc algebra (see Remark \ref{r:GradingExplain}).

Let $\mu \in \Lambda_{n-j,n}$ and define 
$1 \le c_{\mu,1} < \dots < c_{\mu,n-j} \le n$ as in the beginning of Section \ref{ss:weights}.
Let $c_{\mu,i_1} < \dots < c_{\mu,i_s}$ be all the good points.
Let 
\begin{align}
 T&:= \{n+1 -c_{\mu,i_l}|l=1,\dots,s\} \\
 S&:= \{1,\dots,n\} \cup (T+\frac{1}{3}) \cup (T+\frac{2}{3})
\end{align}
We define $(\mu^+)',(\mu^-)':S \to \{ \vee, \wedge \}$ by
\begin{align}
 (\mu^+)'(a)=
 \left\{
 \begin{array}{ccc}
  PD(\mu)(a) &\text{ if }a \in \{1,\dots,n\} \setminus T\\
  \vee    &\text{ if }a \in T \\
  \wedge       &\text{ if }a \in (T+\frac{1}{3}) \cup (T+\frac{2}{3})
 \end{array}
 \right.
\end{align}
and
\begin{align}
 (\mu^-)'(a)=
 \left\{
 \begin{array}{ccc}
  \vee &\text{ if }a \in T+\frac{1}{3}\\
  \wedge     &\text{ otherwise }
 \end{array}
 \right.
\end{align}
There is a unique order preserving bijective map (the order induced as a subset of $\mathbb{R}$) $f:\{1,\dots,n+2s\} \to S$.
Let $\mu^+:=(\mu^+)' \circ f \in \Lambda_{j+s,n+2s}$ and $\mu^-:=(\mu^-)' \circ f \in \Lambda_{s,n+2s}$.
For $\lambda \in \Lambda_{j,n}$, we define $(\lambda \sqcup \mu^-)': S \to \{ \vee, \wedge \}$ by
\begin{align}
 (\lambda \sqcup \mu^-)'(a)=
 \left\{
 \begin{array}{ccc}
  \lambda(a) &\text{ if }a \in \{1,\dots,n\}\\
  (\mu^-)'(a)     &\text{ if }a \in (T+\frac{1}{3}) \cup (T+\frac{2}{3})
 \end{array}
 \right.
\end{align}
and let $\lambda \sqcup \mu^-:=(\lambda \sqcup \mu^-)' \circ f \in \Lambda_{j+s,n+2s}$.
Note that both $\mu^+$ and $\lambda \sqcup \mu^-$ lie in $\Lambda_{j+s,n+2s}$.

\begin{lemma}\label{l:gvisom}
For $\lambda \in \Lambda_{j,n}$ and $\mu \in \Lambda_{n-j,n}$, we have vector space isomorphisms 
\begin{align}
  HF^{k+n-j-s}(\ul{L}_{\hs}, \iota_j(\ul{L}_{\lambda},\ul{K}_{\mu}))= HF^k(\ul{L}_{\mu^+},\ul{L}_{\lambda \sqcup \mu^-}) \label{eq:gvisom}
 \end{align}
 for all $k$.
(Here $n-j-s$ is the number of thimble paths in $\ul{K}_{\mu}$.)
\end{lemma}

The importance of Lemma \ref{l:gvisom} is that the right hand side of 
\eqref{eq:gvisom} is part of the symplectic extended arc algebra. We will see later how to use \eqref{eq:gvisom},
and the identification of the symplectic and algebraic extended arc algebras (Proposition \ref{p:AllCases}), to compute the bimodule
multiplication map of $(\iota_j)^* (\ul{L}_{\hs}^l)$.

\begin{proof}[Proof of Lemma \ref{l:gvisom}]
We use an upper matching to represent  $\ul{L}_{\hs}$ and a lower matching to represent
$\ul{K}_{\mu}$ (see the first picture of Figure \ref{fig:horseshoeBimod}).
Now, we apply a symplectomorphism to `move the right half to the top' and 
make all the matching paths of $\ul{L}_{\hs}$ be straight line segments (see the second picture of Figure \ref{fig:horseshoeBimod}).
For each matching path of $\ul{K}_{\mu}$, we apply a further symplectomorphism to `bend it to the right' and make the two end points become 
$n+1 -c_{\mu,i_l}+\frac{1}{3}+ \sqrt{-1}$ and  $n+1 -c_{\mu,i_l}+\frac{2}{3}+ \sqrt{-1}$ (see the third picture of Figure \ref{fig:horseshoeBimod}).
Then, we remove all critical values contained in the thimble paths of $\ul{K}_{\mu}$ and all the paths that contain these critical values 
(see the fourth picture of Figure \ref{fig:horseshoeBimod}).
Finally, for those matching paths of $\ul{L}_{\hs}$ that do not intersect with a matching path of $\ul{K}_{\mu}$, we bend them to the right and turn them into thimble paths 
(see the fifth picture of Figure \ref{fig:horseshoeBimod}).

 \begin{figure}[h]
  \includegraphics{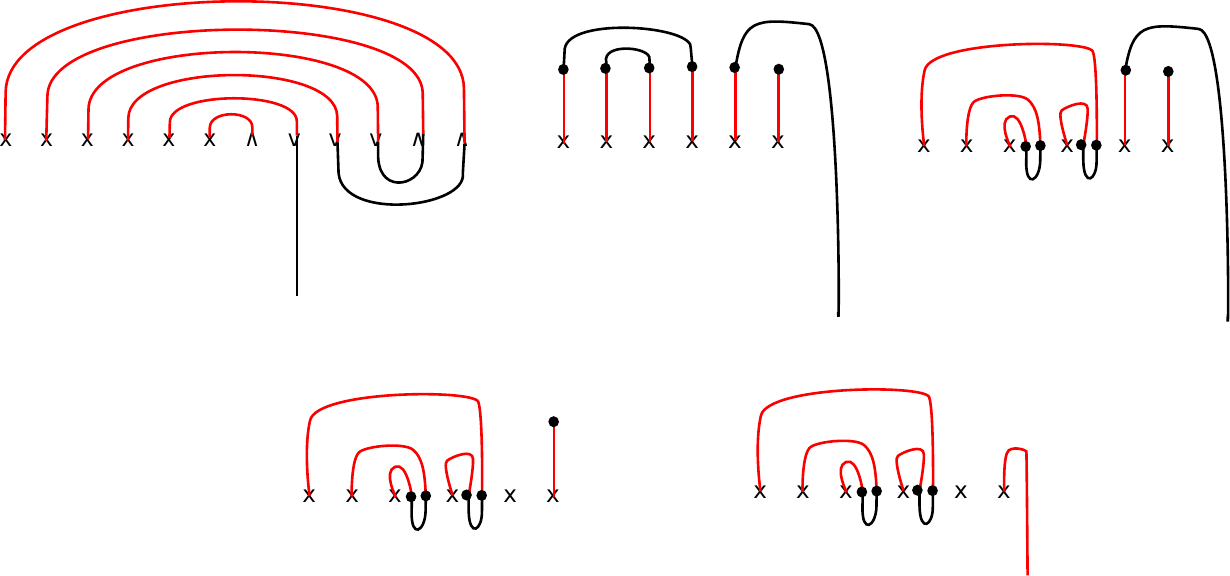}
  \caption{Turning $\ul{L}_{\hs}$ (red) and $\ul{K}_{\mu}$ (black) to $\ul{L}_{\mu^+}$ (red) and $\ul{L}_{\mu^-}$ (black), respectively}\label{fig:horseshoeBimod}
 \end{figure}

After this procedure, the set of critical values (the crosses and black dots in Figure \ref{fig:horseshoeBimod}) is exactly $S$, and we identify it with $\{1,\dots,n+2s\}$ by $f$.
The Lagrangian tuples $\ul{L}_{\hs}$ and $\ul{K}_{\mu}$ become $\ul{L}_{\mu^+}$ and $\ul{L}_{\mu^-}$ respectively.

If we now add in a lower matching of $\ul{L}_{\lambda}$ that does not intersect with  $\ul{L}_{\mu^-}$,
then together with $\ul{L}_{\mu^-}$ we obtain a representative of $\ul{L}_{\lambda \sqcup \mu^-}$ (see the second picture of Figure \ref{fig:horseshoeBimod2}).
Moreover, for degree reasons (either all generators have odd degree, or all have even degree),
both the cochain model $CF(\ul{L}_{\hs}, \iota_j(\ul{L}_{\lambda},\ul{K}_{\mu}))$
and $CF(\ul{L}_{\mu^+},\ul{L}_{\lambda \sqcup \mu^-})$ have vanishing differentials.
Furthermore, there is a identification of the generators by sending a generator $\ul{x}$ of 
$CF(\ul{L}_{\mu^+},\ul{L}_{\lambda \sqcup \mu^-})$ to the generator $\ul{y}$ of $CF(\ul{L}_{\hs}, \iota_j(\ul{L}_{\lambda},\ul{K}_{\mu}))$
such that $\pi_E(\ul{y})$ is the union of $\pi_E(\ul{x})$ and the end points of the thimble paths of $\ul{K}_{\mu}$ (Figure \ref{fig:horseshoeBimod2}).
This identification increases the grading by the number of points added, which is $n-j-s$.
It completes the proof.
\end{proof}

 \begin{figure}[h]
  \includegraphics{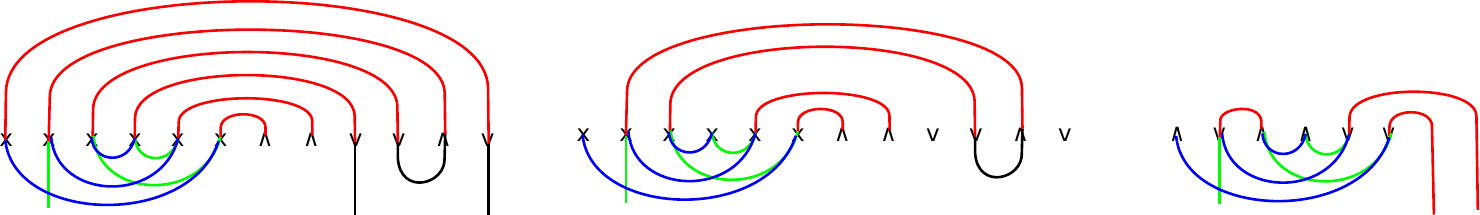}
  \caption{$CF(\ul{L}_{\hs}, \iota_j(\ul{L}_{\lambda},\ul{K}_{\mu}))$ (left) and $CF(\ul{L}_{\mu^+},\ul{L}_{\lambda \sqcup \mu^-})$ (right)}\label{fig:horseshoeBimod2}
 \end{figure}

To prove Lemma \ref{l:gvisom}, it suffices to consider a lower matching of 
$\ul{L}_{\lambda}$ to give a representative of $\ul{L}_{\lambda \sqcup \mu^-}$.
However, for this representative, there may be nested circles in the union of matching and thimble paths of 
$\ul{L}_{\mu^+}$ and $\ul{L}_{\lambda \sqcup \mu^-}$, which will be inconvenient later on.
By possibly sliding the lower matching of $\ul{L}_{\lambda}$ cross some matchings of $\ul{L}_{\mu^-}$ (and $\ul{L}_{\lambda}$ itself),
we can choose a representative $\ul{L}_{\widetilde{\lambda \sqcup \mu^-}}$ of $\ul{L}_{\lambda \sqcup \mu^-}$ with no nested circles (see the second picture of Figure \ref{fig:diffRep}).
This corresponds to choosing another representative $\ul{L}_{\widetilde{\lambda \sqcup \mu}}$ of
$\iota_j(\ul{L}_{\lambda},\ul{K}_{\mu})$, by sliding across some matching paths of $\ul{K}_{\mu}$ (see the first picture of Figure \ref{fig:diffRep}).
For these representatives, we still have the (grading shifting) cochain isomorphism with $0$ differential
\begin{align}
  CF^{*+n-j-s}(\ul{L}_{\hs}, \ul{L}_{\widetilde{\lambda \sqcup \mu}})= CF^*(\ul{L}_{\mu^+},\ul{L}_{\widetilde{\lambda \sqcup \mu^-}}) \label{eq:Cochainisom}
 \end{align}
 which is again given by sending a generator $\ul{x}$ of 
$CF(\ul{L}_{\mu^+},\ul{L}_{\widetilde{\lambda \sqcup \mu^-}})$ to the generator $\ul{y}$ of $CF(\ul{L}_{\hs}, \ul{L}_{\widetilde{\lambda \sqcup \mu}})$
such that $\pi_E(\ul{y})$ is the union of $\pi_E(\ul{x})$ and the end points of the thimble paths of $\ul{K}_{\mu}$.
The geometric generators on the right hand side of \eqref{eq:Cochainisom} are geometric basis elements of 
$K^{\symp}_{j+s,n+2s}$, which are identified with the corresponding diagrammatic basis elements of $K^{\alg}_{j+s,n+2s}$
under the map $\Phi$ of Proposition \ref{p:AllCases}.

\begin{figure}[h]
  \includegraphics{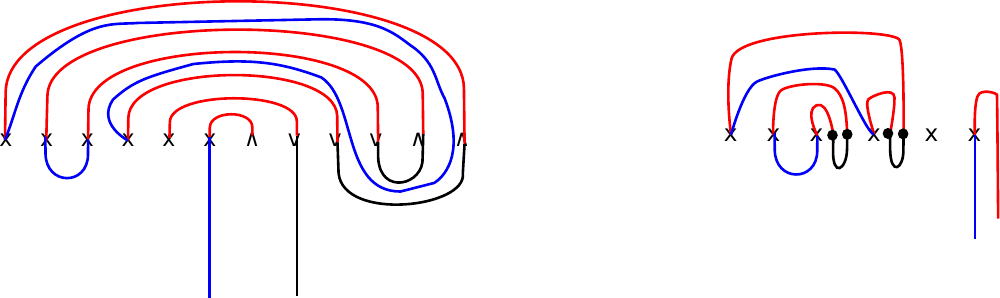}
  \caption{$CF(\ul{L}_{\hs}, \ul{L}_{\widetilde{\lambda \sqcup \mu}})$ (left) and $CF(\ul{L}_{\mu^+},\ul{L}_{\widetilde{\lambda \sqcup \mu^-}})$ (right)}\label{fig:diffRep}
 \end{figure}

Next, we want to compare the multiplication map
\begin{align}
 CF(\ul{L}_{\widetilde{\lambda_0 \sqcup \mu}}, \ul{L}_{\widetilde{\lambda_1 \sqcup \mu}}) \times CF(\ul{L}_{\hs}, \ul{L}_{\widetilde{\lambda_0 \sqcup \mu}})
 \to CF(\ul{L}_{\hs}, \ul{L}_{\widetilde{\lambda_1 \sqcup \mu}}) \label{eq:BimodMult}
\end{align}
and 
\begin{align}
 CF(\ul{L}_{\widetilde{\lambda_0 \sqcup \mu^-}}, \ul{L}_{\widetilde{\lambda_1 \sqcup \mu^-}}) \times CF(\ul{L}_{\mu^+},\ul{L}_{\widetilde{\lambda_0 \sqcup \mu^-}})
 \to CF(\ul{L}_{\mu^+},\ul{L}_{\widetilde{\lambda_1 \sqcup \mu^-}}) \label{eq:BimodMult2}
\end{align}
The second and third terms of \eqref{eq:BimodMult} and \eqref{eq:BimodMult2} are identified via \eqref{eq:Cochainisom},
and the identification of the first terms is defined similarly (i.e. adding to a generator $\ul{x}$ of 
$CF(\ul{L}_{\widetilde{\lambda_0 \sqcup \mu^-}}, \ul{L}_{\widetilde{\lambda_1 \sqcup \mu^-}})$ the intersection points of 
$CF(K_{\mu}, K_{\mu})$ that lie above those critical values that are contained in the thimble paths of $K_\mu$, see Figure \ref{fig:diffWeights},
note that this identification is degree-preserving, because the points being added now have degree $0$).

\begin{figure}[h]
  \includegraphics{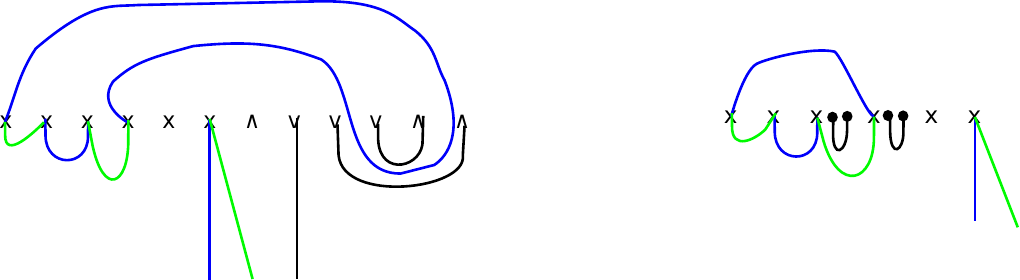}
  \caption{$CF(\ul{L}_{\widetilde{\lambda_0 \sqcup \mu}}, \ul{L}_{\widetilde{\lambda_1 \sqcup \mu}})$ (left) and $CF(\ul{L}_{\widetilde{\lambda_0 \sqcup \mu^-}}, \ul{L}_{\widetilde{\lambda_1 \sqcup \mu^-}})$ (right)}\label{fig:diffWeights}
 \end{figure}

By the open mapping theorem, all solutions $u$ contributing to the map 
\eqref{eq:BimodMult} must include $n-j-s$ constant triangles which, under $\pi_E$, map to the $n-j-s$ critical values that are contained in the thimble paths of $K_\mu$.
Removing these constant triangles (and removing the thimble paths of $K_\mu$), we can identify the moduli spaces defining the maps \eqref{eq:BimodMult} and \eqref{eq:BimodMult2},
so the maps \eqref{eq:BimodMult} and \eqref{eq:BimodMult2} agree under the identification above.

Let $CF(\ul{L}_{\lambda_0 \sqcup \mu^-}^\dagger,  \ul{L}_{\lambda_1 \sqcup \mu^-}^\dagger)$ be a cochain model for the pair 
$(\ul{L}_{\lambda_0 \sqcup \mu^-},  \ul{L}_{\lambda_1 \sqcup \mu^-})$ which 
contains the matching paths of $\mu^-$ used above, and such that 
the union of matching and thimble paths has no nested circles (see the first picture of Figure \ref{fig:split}).
In particular, as the matching paths in $\mu^-$ are not nested, 
the Floer cochain complex canonically splits
\begin{align}
 CF(\ul{L}_{\lambda_0 \sqcup \mu^-}^\dagger,  \ul{L}_{\lambda_1 \sqcup \mu^-}^\dagger)=CF(\ul{L}_{\lambda_0}^\dagger,  \ul{L}_{\lambda_1}^\dagger) \otimes CF(\ul{L}_{\mu^-},\ul{L}_{\mu^-}) \label{eq:split}
\end{align}
where, for $j=0,1$, $\ul{L}_{\lambda_j}^\dagger$ is obtained by removing the matchings corresponding to $\mu^-$ (see the second picture of Figure \ref{fig:split}).
Moreover, as there is no nesting, the geometric 
generators are geometric basis elements of $K^{\symp}_{j,n} \otimes K^{\symp}_{s,2s}$.
By Proposition \ref{p:AllCases}, these geometric generators can in turn be identified with diagrammatic basis elements in 
$K^{\alg}_{j,n} \otimes K^{\alg}_{s,2s}$.
Let
\begin{align}
 \kappa:CF(\ul{L}_{\lambda_0 \sqcup \mu^-}^\dagger,  \ul{L}_{\lambda_1 \sqcup \mu^-}^\dagger) \to CF(\ul{L}_{\widetilde{\lambda_0 \sqcup \mu^-}}, \ul{L}_{\widetilde{\lambda_1 \sqcup \mu^-}})
\end{align}
be a continuation map (the induced map on cohomology is independent of choices).
By Proposition \ref{p:AllCases}, we can completely describe
\begin{align}
 \mu^2(\kappa(-),-):CF(\ul{L}_{\lambda_0 \sqcup \mu^-}^\dagger,  \ul{L}_{\lambda_1 \sqcup \mu^-}^\dagger) \times CF(\ul{L}_{\mu^+},\ul{L}_{\widetilde{\lambda_0 \sqcup \mu^-}})
 \to CF(\ul{L}_{\mu^+},\ul{L}_{\widetilde{\lambda_1 \sqcup \mu^-}}) \label{eq:BimodMult3}
\end{align}
using the algebraic extended arc algebra.
In particular, by applying \eqref{eq:split} to the first term of \eqref{eq:BimodMult3}, we completely understand the map 
\begin{align}
 \mu^2(\kappa(- \otimes e_{\mu^-}),-):CF(\ul{L}_{\lambda_0}^\dagger,  \ul{L}_{\lambda_1}^\dagger) \times CF(\ul{L}_{\mu^+},\ul{L}_{\widetilde{\lambda_0 \sqcup \mu^-}})
 \to CF(\ul{L}_{\mu^+},\ul{L}_{\widetilde{\lambda_1 \sqcup \mu^-}}) \label{eq:BimodMult3.5}
\end{align}
where $e_{\mu^-} \in CF(\ul{L}_{\mu^-},\ul{L}_{\mu^-})$ is the unit (cf. Lemma \ref{l:chainRep}).

\begin{figure}[h]
  \includegraphics{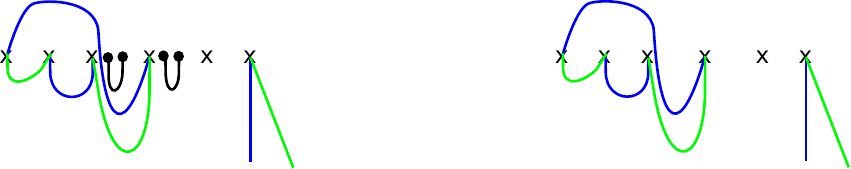}
  \caption{$CF(\ul{L}_{\lambda_0 \sqcup \mu^-}^\dagger,  \ul{L}_{\lambda_1 \sqcup \mu^-}^\dagger)$ (left) and $CF(\ul{L}_{\lambda_0}^\dagger,  \ul{L}_{\lambda_1}^\dagger)$ (right)}\label{fig:split}
 \end{figure}

On the other hand, we now consider
\begin{align}
 CF(\ul{L}_{\widetilde{\lambda_0}}, \ul{L}_{\widetilde{\lambda_1}}) \times CF(\ul{L}_{\ol{PD(\mu)}},\ul{L}_{\widetilde{\lambda_0}}) \to CF(\ul{L}_{\ol{PD(\mu)}},\ul{L}_{\widetilde{\lambda_1}}) \label{eq:BimodMult4}
\end{align}
where, for $j=0,1$, $\ul{L}_{\widetilde{\lambda_j}}$ is defined by removing the matching paths (and their end points) of
$\ul{L}_{\widetilde{\lambda_j \sqcup \mu^-}}$ that correspond to $\mu^-$, and 
$\ul{L}_{\ol{PD(\mu)}}$ is an upper matching such that there are no nested circles in the union of the matching paths 
of $\ul{L}_{\ol{PD(\mu)}}$ and $\ul{L}_{\widetilde{\lambda_j}}$ (see Figure \ref{fig:PDside}).

\begin{figure}[h]
  \includegraphics{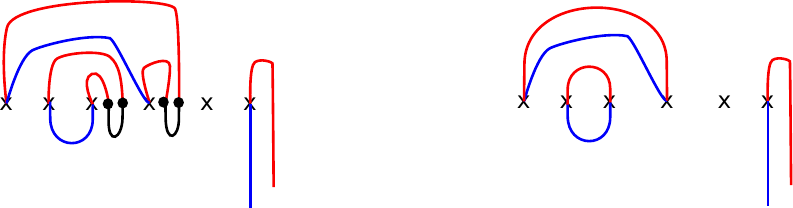}
  \caption{$CF(\ul{L}_{\mu^+},\ul{L}_{\widetilde{\lambda \sqcup \mu^-}})$ (left) and $CF(\ul{L}_{\ol{PD(\mu)}},\ul{L}_{\widetilde{\lambda_0}})$ (right)}\label{fig:PDside}
 \end{figure}

Analogously, we have a continuation map
\begin{align}
 \kappa':CF(\ul{L}_{\lambda_0}^\dagger,  \ul{L}_{\lambda_1}^\dagger) \to CF(\ul{L}_{\widetilde{\lambda_0}}, \ul{L}_{\widetilde{\lambda_1}})
\end{align}
and hence we can define
\begin{align}
 \mu^2(\kappa'(-),-):CF(\ul{L}_{\lambda_0}^\dagger,  \ul{L}_{\lambda_1}^\dagger) \times CF(\ul{L}_{\ol{PD(\mu)}},\ul{L}_{\widetilde{\lambda_0}}) \to CF(\ul{L}_{\ol{PD(\mu)}},\ul{L}_{\widetilde{\lambda_1}}) \label{eq:BimodMult5}
\end{align}
which we also know completely because the geometric generators are all geometric basis elements and we can apply Proposition \ref{p:AllCases}.

For $j=0,1$, we have a (grading shifting) cochain isomorphism (both sides have $0$ differential as usual)
\begin{align}
 CF^{*+s}(\ul{L}_{\mu^+},\ul{L}_{\widetilde{\lambda_j \sqcup \mu^-}}) \to CF^*(\ul{L}_{\ol{PD(\mu)}},\ul{L}_{\widetilde{\lambda_j}}) \label{eq:CoChainId}
\end{align}
giving by forgetting the points of a generator $\ul{x}$ that lie above $(T+\frac{1}{3}+\sqrt{-1}) \cup (T+\frac{2}{3}+\sqrt{-1})$.
The grading shift arises from removing the grading contribution of the points lying above $(T+\frac{1}{3}+\sqrt{-1}) \cup (T+\frac{2}{3}+\sqrt{-1})$.
There are two cases. If a point lying above $(T+\frac{1}{3}+\sqrt{-1})$ is removed, then its grading contribution, which is $1$, is removed and the rest of the gradings are unchanged.
If a point lying above $(T+\frac{2}{3}+\sqrt{-1})$ is removed, then its grading contribution, which is $2$, is removed
but there is grading change from $1$ to $2$ for the point lying above $T+\sqrt{-1}$.  Indeed, since there is a matching path from $T+\sqrt{-1}$ to $(T+\frac{1}{3}+\sqrt{-1})$
and another matching path from $(T+\frac{1}{3}+\sqrt{-1})$ to $(T+\frac{2}{3}+\sqrt{-1})$, a generator has a point lying above $(T+\frac{2}{3}+\sqrt{-1})$ if and only if it has a point lying above $T+\sqrt{-1}$.
Therefore, in both cases, each point removal results in a decrease of overall grading by $1$.

The outcome of this discussion is the following:

\begin{lemma}\label{l:oneMultId}
 Under the cochain identifications \eqref{eq:CoChainId}, the maps \eqref{eq:BimodMult3.5} and \eqref{eq:BimodMult5} are identified.
\end{lemma}

\begin{proof}
 After applying Proposition \ref{p:AllCases} to the geometric basis elements for all the Floer cochain groups involved, the maps 
 \eqref{eq:BimodMult3.5} and \eqref{eq:BimodMult5} can be computed by the algebraic extended arc algebra
 in $K^{\alg}_{j+s,n+2s}$ and in $K^{\alg}_{j,n}$, respectively.
 
 The upshot is that, for \eqref{eq:BimodMult3.5}, elements of the form $-\otimes e_{\mu^-}$
 correspond to orienting the $s$ circles in $\ul{\lambda_1 \sqcup \mu^-}^{\alg} \cup \ol{\lambda_0 \sqcup \mu^-}^{\alg}$
 associated to $\mu^-$ counterclockwise.
 Applying TQFT type multiplication to these circles does nothing \cite[Equation (3.4)]{BS11}, and the rest of the diagrammatic calculus operations (for \ref{eq:BimodMult3.5} and \eqref{eq:BimodMult5}) 
 are obviously identified.
\end{proof}

As a consequence of Lemma \ref{l:gvisom}, the identification between \eqref{eq:BimodMult} and \eqref{eq:BimodMult2}, and Lemma \ref{l:oneMultId}, we get:

\begin{corollary}\label{c:DiBiMult1}
 We have canonical vector space isomorphisms 
 \begin{align}
  HF^{*+n-j}(\ul{L}_{\hs}, \iota_j(\ul{L}_{\lambda},\ul{K}_{\mu})) = HF^*(\ul{L}_{PD(\mu)},\ul{L}_{\lambda})
 \end{align}
 and under these isomorphisms, the multiplication maps
 \begin{align}
\mu^2(-\otimes e_K,-)  : HF(\ul{L}_{\lambda_0},\ul{L}_{\lambda_1}) \times HF(\ul{L}_{\hs}, \iota_j(\ul{L}_{\lambda_0},\ul{K}_{\mu}))
 \to HF(\ul{L}_{\hs}, \iota_j(\ul{L}_{\lambda_1},\ul{K}_{\mu})) \label{eq:1}
 \end{align}
and
\begin{align}
 \mu^2(-,-):HF(\ul{L}_{\lambda_0},  \ul{L}_{\lambda_1}) \times HF(\ul{L}_{PD(\mu)},\ul{L}_{\lambda_0})
 \to HF(\ul{L}_{PD(\mu)},\ul{L}_{\lambda_1}) 
\end{align}
are identified.
Here, for \eqref{eq:1}, 
\begin{align}
 -\otimes e_K: HF(\ul{L}_{\ul{\lambda_0}},\ul{L}_{\ul{\lambda_1}}) \to  HF(\ul{L}_{\ul{\lambda_0}},\ul{L}_{\ul{\lambda_1}}) \otimes HF(\ul{K}_{\mu},\ul{K}_{\mu}) \simeq HF(\iota_j(\ul{L}_{\lambda_0},\ul{K}_{\mu}),\iota_j(\ul{L}_{\lambda_1},\ul{K}_{\mu}))
\end{align}
is the obvious map where $e_K$ is the cohomological unit of $HF(\ul{K}_{\mu},\ul{K}_{\mu})$.
\end{corollary}

\begin{remark}\label{r:GradingExplain}
 Another way to interpret the grading shift is as follows.
 We have seen from Lemma \ref{l:gvspIsom} (or Proposition \ref{p:AllCases}) that the map from the symplectic extended arc algebra to the algebraic extended arc algebra
 is given by putting a $\vee$ to the points $\pi_E(\ul{x})$, where $\ul{x}$ is a geometric basis element.
 Each generator $\ul{x}$ of $CF(\ul{L}_{\hs}, \ul{L}_{\widetilde{\lambda \sqcup \mu}})$ (see the first picture of Figure \ref{fig:diffRep})
 must have $j$ points lying above $\{1,\dots,n\}$ and $n-j$ points lying above $\{n+1,\dots,2n\}$.
 That corresponds to making $j$ counterclockwise caps and $n-j$ clockwise caps for the oriented cap diagram associated to $\ul{L}_{\hs}$.
 This is where the extra $n-j$ in the grading comes from.
\end{remark}

By completely analogous reasoning,
\begin{corollary}\label{c:DiBiMult2}
We have canonical vector space isomorphisms 
 \begin{align}
  HF^{*+n-j}(\ul{L}_{\hs}, \iota_j(\ul{L}_{\lambda},\ul{K}_{\mu})) = HF^*(\ul{L}_{PD(\lambda)},\ul{K}_{\mu})
 \end{align}
 and under these isomorphisms, the multiplication maps
 \begin{align}
\mu^2(e_L \otimes -,-) :  HF(\ul{K}_{\mu_0},\ul{K}_{\mu_1}) \times HF(\ul{L}_{\hs}, \iota_j(\ul{L}_{\lambda},\ul{K}_{\mu_0}))
 \to HF(\ul{L}_{\hs}, \iota_j(\ul{L}_{\lambda},\ul{K}_{\mu_1})) \label{eq:2}
 \end{align}
and
\begin{align}
 \mu^2(-,-):HF(\ul{K}_{\mu_0},\ul{K}_{\mu_1}) \times HF(\ul{L}_{PD(\lambda)},\ul{K}_{\mu_0})
 \to HF(\ul{L}_{PD(\lambda)},\ul{K}_{\mu_1})  \label{eq:beforeDual}
\end{align}
are identified.
Here, for \eqref{eq:2}, 
\begin{align}
 e_L\otimes -:HF(\ul{K}_{\mu_0},\ul{K}_{\mu_1}) \to  HF(\ul{L}_{\ul{\lambda}},\ul{L}_{\ul{\lambda}}) \otimes HF(\ul{K}_{\mu_0},\ul{K}_{\mu_1}) \simeq HF(\iota_j(\ul{L}_{\lambda},\ul{K}_{\mu_0}),\iota_j(\ul{L}_{\lambda},\ul{K}_{\mu_1}))
\end{align}
is the obvious map where $e_L$ is the cohomological unit of $HF(\ul{L}_{\ul{\lambda}},\ul{L}_{\ul{\lambda}})$.
\end{corollary}

Note that, after applying $PD$ (which is grading preserving), \eqref{eq:beforeDual} would become
\begin{align}
 \mu^2(-,-): HF(\ul{L}_{PD(\mu_0)},\ul{L}_{\lambda}) \times HF(\ul{L}_{PD(\mu_1)},\ul{L}_{PD(\mu_0)})
 \to HF(\ul{L}_{PD(\mu_1)},\ul{L}_{\lambda}) \label{eq:afterDual}
\end{align}

By Corollary \ref{c:DiBiMult1} and \ref{c:DiBiMult2}, and the fact that 
every basis element of $HF(\iota_j(\ul{L}_{\lambda_0},\ul{K}_{\mu_0}),\iota_j(\ul{L}_{\lambda_1},\ul{K}_{\mu_1}))$
can be written as the product of 
\begin{align}
 x \otimes e_K &\in HF(\iota_j(\ul{L}_{\lambda_0},\ul{K}_{\mu_0}), \iota_j(\ul{L}_{\lambda_1},\ul{K}_{\mu_0})) \text{ and }\\
 e_L \otimes y &\in HF(\iota_j(\ul{L}_{\lambda_1},\ul{K}_{\mu_0}), \iota_j(\ul{L}_{\lambda_1},\ul{K}_{\mu_1}))
\end{align}
for some $x,y$, we now know that the multiplication map
\begin{align}
HF(\iota_j(\ul{L}_{\lambda_0},\ul{K}_{\mu_0}),\iota_j(\ul{L}_{\lambda_1},\ul{K}_{\mu_1})) \times HF(\ul{L}_{\hs}, \iota_j(\ul{L}_{\lambda_0},\ul{K}_{\mu_0})) 
\to HF(\ul{L}_{\hs}, \iota_j(\ul{L}_{\lambda_1},\ul{K}_{\mu_1})) \label{eq:bimodMult}
\end{align}
can be identified with the left and right multiplication maps
\begin{align}
\mu^2(-,\mu^2(-,-)): HF(\ul{L}_{\lambda_0},  \ul{L}_{\lambda_1}) \times HF(\ul{L}_{PD(\mu_0)},\ul{L}_{\lambda_0}) \times HF(\ul{L}_{PD(\mu_1)},\ul{L}_{PD(\mu_0)})
\to HF(\ul{L}_{PD(\mu_1)},\ul{L}_{\lambda_1})   \label{eq:bimodMult2}
\end{align}
In other words, we have now shown that $(\iota_j)^* (\ul{L}_{\hs}^l)$ coincides with the diagonal bimodule (shifted by $n-j$) on the cohomological level.
Having this, we can now finish the proof of Proposition \ref{p:diagonalbimod}

\begin{proof}[Completion of the proof of Proposition \ref{p:diagonalbimod}]
 In the previous paragraph, we have shown that 
 $(\iota_j)^* (\ul{L}_{\hs}^l)$ coincides with the diagonal $K^{\symp}_{j,n}\text{-}K^{\symp}_{j,n}$ bimodule on the cohomological level.
 It remains to show that $(\iota_j)^* (\ul{L}_{\hs}^l)$ is formal as a bimodule over $K^{\symp}_{j,n}$, 
 or formal as a left module over $A_j$.

 As in the proof of formality in Section \ref{s:formality}, it suffices to show that we can choose a consistent choice of $b$-equivariant structures
 for the collection of Lagrangians in $A_j$ together with the single additional Lagrangian $\ul{L}_{hs}$.
 As explained in Lemma \ref{l:productDG}, the existence of a consistent choice of $b$-equivariant structures
 for the collection of Lagrangians in $A_j$ follows from Section \ref{s:formality} so we now only need to 
 argue the consistency for the pairs
 $(\ul{L}_{\hs}, \iota_j(\ul{L}_{\lambda},\ul{K}_{\mu}))$.
 
 The strategy is as before.
 First observe that, by the identification between \eqref{eq:bimodMult} and \eqref{eq:bimodMult2} and Lemma \ref{l:cyclic}, we know that
 $HF(\ul{L}_{\hs}, \iota_j(\ul{L}_{\lambda},\ul{K}_{\mu}))$ is a cyclic module over $HF( \iota_j(\ul{L}_{\lambda},\ul{K}_{\mu}),  \iota_j(\ul{L}_{\lambda},\ul{K}_{\mu}))$.

 Then, we can run the proof of Proposition \ref{p:bStructure}.
 We choose a $b$-equivariant structure on $\ul{L}_{\hs}$ such that the rank $1$ vector space
 $HF(\ul{L}_{\hs}, \iota_j(\ul{L}_{\lambda_{max}},\ul{K}_{\mu_{max}}))$
 is pure, where $\lambda_{max}$ and $\mu_{max}$ are the unique maximal weight in $\Lambda_{j,n}$ and $\Lambda_{n-j,n}$, respectively.
 Fixing $\ul{K}_{\mu_{max}}$, we want to inductively argue that 
 $HF(\ul{L}_{\hs}, \iota_j(\ul{L}_{\lambda},\ul{K}_{\mu_{max}}))$
 is pure for all $\lambda$.
 Let $\lambda$ be such that the purity of $HF(\ul{L}_{\hs}, \iota_j(\ul{L}_{\lambda'},\ul{K}_{\mu_{max}}))$ has been verified for all $\lambda' >^T \lambda$, where $>^T$ is the total ordering of weights in Proposition \ref{p:bStructure}; 
 then we  need to verify purity for $HF(\ul{L}_{\hs}, \iota_j(\ul{L}_{\lambda},\ul{K}_{\mu_{max}}))$.
 By Lemma \ref{l:constantmap}, and the identification between \eqref{eq:bimodMult} and \eqref{eq:bimodMult2}, there exists a weight $\lambda' >^T \lambda$ such that
 we have a non-trivial product (this is the analog of \eqref{eq:hahaha})
 \begin{align}
  HF(\iota_j(\ul{L}_{\lambda'},\ul{K}_{\mu_{max}}),\iota_j(\ul{L}_{\lambda},\ul{K}_{\mu_{max}})) \times HF(\ul{L}_{\hs}, \iota_j(\ul{L}_{\lambda'},\ul{K}_{\mu_{max}})) 
\to HF(\ul{L}_{\hs}, \iota_j(\ul{L}_{\lambda},\ul{K}_{\mu_{max}}))
 \end{align}
Together with cyclicity, it implies that $HF(\ul{L}_{\hs}, \iota_j(\ul{L}_{\lambda},\ul{K}_{\mu_{max}}))$ is pure.
Thus $HF(\ul{L}_{\hs}, \iota_j(\ul{L}_{\lambda},\ul{K}_{\mu_{max}}))$ is pure for all $\lambda$.

Similarly, we can fix $\ul{L}_{\lambda}$ and apply the inductive argument to 
$HF(\ul{L}_{\hs}, \iota_j(\ul{L}_{\lambda},\ul{K}_{\mu}))$ with varying $\mu$.
All together, we conclude that $HF(\ul{L}_{\hs}, \iota_j(\ul{L}_{\lambda},\ul{K}_{\mu}))$ is pure for all $\lambda,\mu$ and hence 
$(\iota_j)^* (\ul{L}_{\hs}^l)$ is formal as a left module over $A_j$.
\end{proof}

\begin{lemma}\label{l:diagonalbimodBETA}
Up to grading shift, $(\iota_j)^* (\phi_{\beta}(\ul{L}_{\hs})^l)$ is quasi-isomorphic to $P_\beta^{(j)}$.
\end{lemma}

\begin{proof}
 We can apply the same method as in the proof of Proposition \ref{p:diagonalbimod} to 
 $(\phi_\beta)^* (\iota_j)^* (\phi_{\beta}(\ul{L}_{\hs}))^l$, where $(\phi_\beta)^*: \rperf A_j \to \rperf A_j$
 is the pull-back functor under $\phi_\beta$.
 This is because, geometrically, $CF(\phi_\beta(\ul{L}_{\hs}), \phi_\beta(\ul{L}) \sqcup \ul{K})$ can be canonically identified with
 $CF(\ul{L}_{\hs}, \ul{L} \sqcup \ul{K})$, as can all holomorphic polygons contributing to the $A_{\infty}$ structure amongst tuples of such Floer groups.  It may be worth emphasising that at this point in the argument we use the  fact that, working with pullback data for Hamiltonians, almost complex structures, etc, the symplectomorphism $\phi_{\beta}$ gives a canonical identification of moduli spaces of polygons;  the corresponding isomorphism, from the algebraic viewpoint, is much less transparent.
 
 In consequence, $(\phi_\beta)^* (\iota_j)^* (\phi_{\beta}(\ul{L}_{\hs}))^l$ is also quasi-isomorphic to the diagonal bimodule, and hence
 $(\iota_j)^* (\phi_{\beta}(\ul{L}_{\hs}))^l$ is quasi-isomorphic to $(\phi_{\beta}^{-1})^*\Delta=P_\beta^{(j)}$.
\end{proof}
 
We have the analogous results for the embedding $\iota^!_j: A^!_j \to \cFS^{cyl,n}(\pi_E)$ and the induced pull back
\begin{align}
 (\iota^!_j)^*:  \rperf \cFS^{cyl,n}(\pi_E) \to  \rperf A^!_j \simeq [K_{j,n},K_{j,n}]
\end{align}
 
\begin{proposition}\label{p:diagonalbimod2}
 Up to grading shift, $(\iota^!_j)^* (\ul{L}_{\hs}^r)$ is quasi-isomorphic to the diagonal bimodule.
\end{proposition}

\begin{proof}[Sketch of proof]
 A completely parallel argument applies (see Figure \ref{fig:RightYoneda} and \ref{fig:RightYoneda2} for illustration).
\end{proof}

\begin{figure}[h]
  \includegraphics{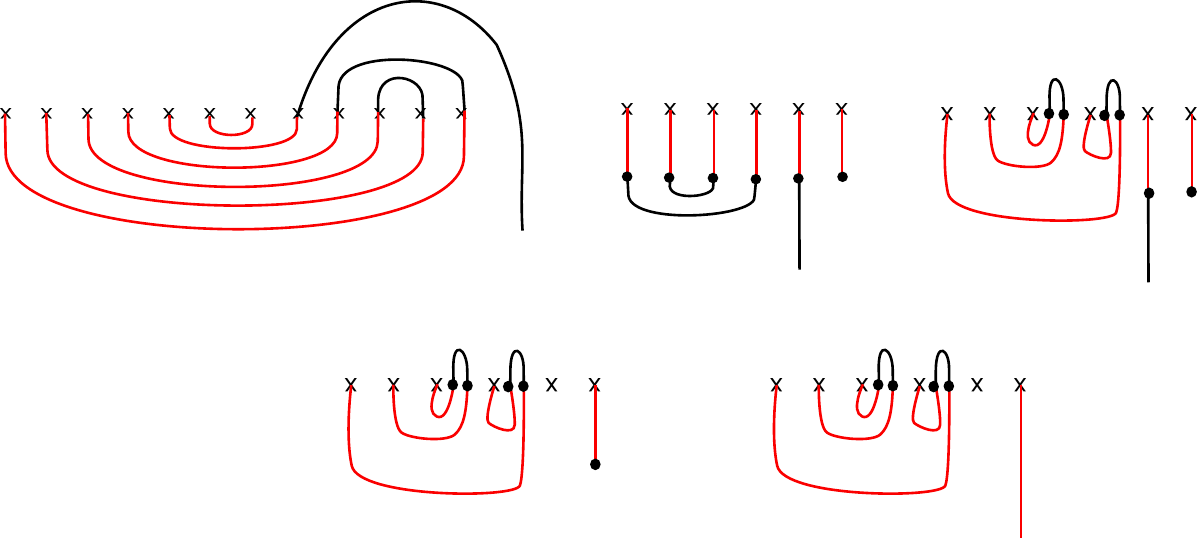}
  \caption{The corresponding figure of Figure \ref{fig:horseshoeBimod}}\label{fig:RightYoneda}
 \end{figure}
 
 \begin{figure}[h]
  \includegraphics{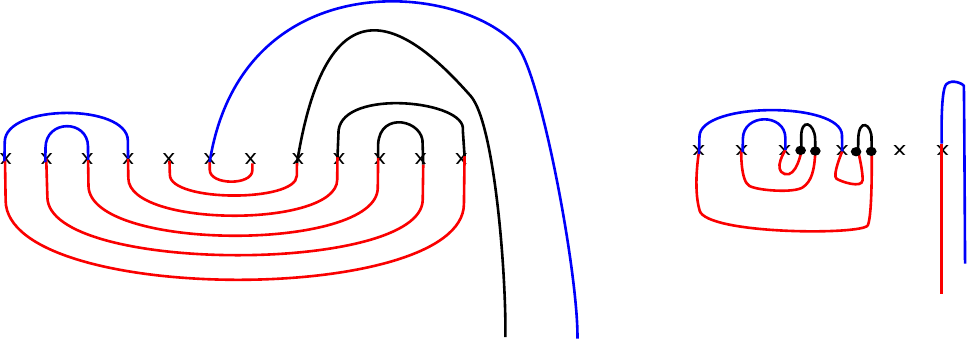}
  \caption{The corresponding figure of Figure \ref{fig:horseshoeBimod2}}\label{fig:RightYoneda2}
 \end{figure}

If one keeps track of the grading and follows our grading convention, then 
$(\iota^!_j)^* (\ul{L}_{\hs}^r)$ is quasi-isomorphic to the diagonal bimodule shifted by $n-j$ (cf. Remark \ref{r:GradingExplain}).

 \begin{lemma}\label{l:diagonalbimodBETA2}
 Up to grading shift, $(\iota_j)^* (\phi_{\beta}(\ul{L}_{\hs})^r)$ is quasi-isomorphic to $P_\beta^{(j)}$.
\end{lemma}

\begin{proof}
 The proof repeats that of Lemma \ref{l:diagonalbimodBETA}, but replacing Proposition \ref{p:diagonalbimod} by Proposition \ref{p:diagonalbimod2}.
\end{proof}

\subsection{Spectral sequence to Khovanov homology}\label{ss:proofofss}

Having Proposition \ref{p:diagonalbimod} and Lemma \ref{l:diagonalbimodBETA2}, the last piece of information we need
to understand the $E_1$-page of the spectral sequence in Proposition \ref{p:spectral}
is a description of a quasi-inverse
\begin{align}
 \psi_j: A_j\lperf \to A^!_j
\end{align}
of the functor $(\iota_j)^*(-)^l:A^!_j \to A_j\lperf$.
We are going to describe $\psi_j$ geometrically.

 For $\lambda \in \Lambda_{j,n}$ and $\mu \in \Lambda_{n-j,m-n}$, we consider the
 upper Lagrangian tuples 
 $\ul{L}_{\ol{\lambda}} \in \cFS^{cyl,j}_{W_1}(\pi_E)$
 and  
 $\ul{K}_{\ol{\mu}} \in \cFS^{cyl,n-j}_{W_2}(\pi_E)$ (see Figure \ref{fig:ThirdEmbedding}).
The disjoint union 
$\ul{L}_{\ol{\lambda}} \sqcup \ul{K}_{\ol{\mu}}$ gives an object in $\cFS^{cyl,n}(\pi_E)$.
On the other hand, we have $\ul{L}^!_{\ol{\lambda}} \sqcup \ul{K}^!_{\ol{\mu}} \in A^!_j$ which is in particular an object in $\cFS^{cyl,n}(\pi_E)$.

 \begin{figure}[h]
  \includegraphics{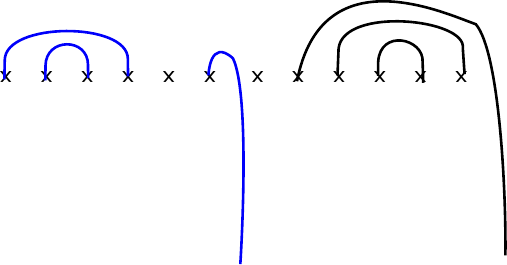}
  \caption{Lagrangian tuples $\ul{L}_{\ol{\lambda}}$ (blue) and $\ul{K}_{\ol{\mu}}$ (black)}\label{fig:ThirdEmbedding}
 \end{figure}

\begin{lemma}\label{l:switchingObject}
For $\lambda \in \Lambda_{j,n}$ and $\mu \in \Lambda_{n-j,m-n}$, we have a quasi-isomorphism of objects
 \begin{align}
  (\iota_j)^*(\ul{L}_{\ol{\lambda}} \sqcup \ul{K}_{\ol{\mu}})^l \simeq (\iota_j)^*(\ul{L}^!_{\ol{\lambda}} \sqcup \ul{K}^!_{\ol{\mu}})^l \label{eq:quasiinverse}
 \end{align}
in $A_j\lperf$.
\end{lemma}

\begin{proof}
 For any object $\ul{L}_{\ul{\lambda'}} \sqcup \ul{K}_{\ul{\mu'}} \in A_j$, there is an obvious isomorphism (see Figure \ref{fig:SecondThirdEmbedding})
 \begin{align*}
  CF(\ul{L}_{\ol{\lambda}} \sqcup \ul{K}_{\ol{\mu}},\ul{L}_{\ul{\lambda'}} \sqcup \ul{K}_{\ul{\mu'}} ) 
  &=CF(\ul{L}_{\ol{\lambda}}, \ul{L}_{\ul{\lambda'}}) \otimes CF(\ul{K}_{\ol{\mu}},\ul{K}_{\ul{\mu'}} ) \\
  &=CF(\ul{L}^!_{\ol{\lambda}}, \ul{L}_{\ul{\lambda'}}) \otimes CF(\ul{K}^!_{\ol{\mu}},\ul{K}_{\ul{\mu'}} ) \\
  &=CF(\ul{L}^!_{\ol{\lambda}} \sqcup \ul{K}^!_{\ol{\mu}}, \ul{L}_{\ul{\lambda'}} \sqcup \ul{K}_{\ul{\mu'}} )
 \end{align*}
 Moreover, when $X_0,\dots,X_d \in A_j$, for a careful choice of Floer data,
 one can actually make the $A_{\infty}$ structural maps
 \begin{align}
  CF(X_{d-1},X_d) \times \dots CF(X_0,X_1) \times CF(\ul{L}_{\ol{\lambda}} \sqcup \ul{K}_{\ol{\mu}},X_0) &\to CF(\ul{L}_{\ol{\lambda}} \sqcup \ul{K}_{\ol{\mu}},X_d) \\
  CF(X_{d-1},X_d) \times \dots CF(X_0,X_1) \times CF(\ul{L}^!_{\ol{\lambda}} \sqcup \ul{K}^!_{\ol{\mu}},X_0) &\to CF(\ul{L}_{\ol{\lambda}} \sqcup \ul{K}_{\ol{\mu}},X_d) 
 \end{align}
 completely coincide.
 It implies the result.
\end{proof}

 \begin{figure}[h]
  \includegraphics{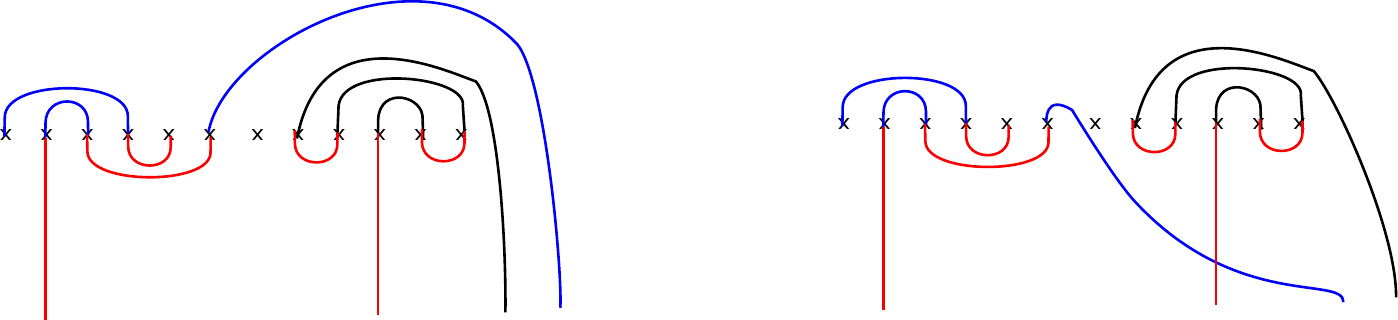}
  \caption{
  Floer cochain $CF(\ul{L}^!_{\ol{\lambda}} \sqcup \ul{K}^!_{\ol{\mu}}, \ul{L}_{\ul{\lambda'}} \sqcup \ul{K}_{\ul{\mu'}} )$ (left) and
  $ CF(\ul{L}_{\ol{\lambda}} \sqcup \ul{K}_{\ol{\mu}},\ul{L}_{\ul{\lambda'}} \sqcup \ul{K}_{\ul{\mu'}} )$ (right).
  The intersection between the positive wrapping of $\ul{L}_{\ol{\lambda}}$ and $\ul{K}_{\ul{\mu'}}$ cannot contribute generator to 
  $CF(\ul{L}_{\ol{\lambda}} \sqcup \ul{K}_{\ol{\mu}},\ul{L}_{\ul{\lambda'}} \sqcup \ul{K}_{\ul{\mu'}} )$ so the Floer cochain splits.}\label{fig:SecondThirdEmbedding}
 \end{figure}

\begin{remark}
 By taking morphisms with an appropriate object in $A_i$ for $i \neq j$, one can see 
 that $(\ul{L}_{\ol{\lambda}} \sqcup \ul{K}_{\ol{\mu}})^l$
 is not quasi-isomorphic to
 $(\ul{L}^!_{\ol{\lambda}} \sqcup \ul{K}^!_{\ol{\mu}})^l$.
\end{remark}

By Lemma \ref{l:switchingObject}, we know that the quasi-inverse $\psi_j$ is given, on the object level, by 
\begin{align}
 \psi_j: (\iota_j)^*(\ul{L}_{\ol{\lambda}} \sqcup \ul{K}_{\ol{\mu}})^l \mapsto \ul{L}^!_{\ol{\lambda}} \sqcup \ul{K}^!_{\ol{\mu}}
\end{align}
A further inspection of the proof of Lemma \ref{l:switchingObject} shows that the quasi-isomorphism \eqref{eq:quasiinverse} is functorial when we vary
$\lambda \in \Lambda_{j,n}$ and $\mu \in \Lambda_{n-j,m-n}$.
Together with the fact that $\ul{L}_{\ol{\lambda}} \sqcup \ul{K}_{\ol{\mu}}$ is quasi-isomorphic to 
$\ul{L}_{\ul{\lambda}} \sqcup \ul{K}_{\ul{\mu}}$ in $\cFS^{cyl,n}(\pi_E)$ (so the former can be regarded as an object in $A_j$),
we have a commutative diagram
\[
\begin{tikzcd}
A_j \arrow[r,"\simeq"] \arrow[d,"(-)^l"]& K_{j,n} \times K_{n-j,n} \arrow[r,"\simeq"]&   A^!_j \arrow[lld, "(\iota_j)^*(-)^l"]\\
A_j \lperf    & &
\end{tikzcd}
\]
where the horizontal isomorphisms are the ones in \eqref{eq:DGproduct}, \eqref{eq:DGproduct2}, which in particular send 
$\ul{L}_{\lambda} \sqcup \ul{K}_{\mu} \in A_j$ to  $\ul{L}^!_{\lambda} \sqcup \ul{K}^!_{\mu} \in A^!_j$.
This implies that $\psi_j$ sits inside the commutative diagram
\[
\begin{tikzcd}
A_j \arrow[r,"\simeq"] \arrow[d,"(-)^l"]& K_{j,n} \times K_{n-j,n} \arrow[r,"\simeq"]&   A^!_j\\
A_j \lperf \arrow[rru, swap, "\psi_j"]   & &
\end{tikzcd}
\]

\begin{proof}[Proof of Theorem \ref{t:sseq}]
By Proposition \ref{p:spectral}, we have a spectral sequence to 
$  Kh(\kappa(\beta))$
from 
\begin{align}
 \oplus_{j=0}^n   H((\iota^!_j)^*(\phi_\beta(\ul{L}_{\hs})^r)   \otimes_{A^!_j} \psi_j \circ (\iota_j)^*(\ul{L}_{\hs}^l))
\end{align}
By Proposition \ref{p:diagonalbimod}, Lemma \ref{l:diagonalbimodBETA2} and the commutative diagram above,
we have, up to grading shift,
\begin{align}
 H((\iota^!_j)^*(\phi_\beta(\ul{L}_{\hs})^r)   \otimes_{A^!_j} \psi_j \circ (\iota_j)^*(\ul{L}_{\hs}^l))
 &=H( P^{(j)}_\beta \otimes_{K_{j,n}-K_{j,n}} \Delta_{K_{j,n}} )\\
 &=HH_*(K_{j,n}, P^{(j)}_\beta)
\end{align}
which is what we want.
 \end{proof}

\appendix

\section{Non-formality of standard modules/thimbles\label{Appendix}}

Here we explain that the endomorphism algebra of the Lefschetz thimbles in $\cFS(\pi_{n,m})$ may be non-formal when $n>1$.
The non-formality follows rather directly from the minimal model computed in the thesis of Klamt \cite{Klamt}. 
Since this is not required for our main results, our treatment is somewhat brief. 

When $(n,m) = (1,m)$ the algebra is known to be formal; by duality, the same holds when $n=m-1$. On the other hand, if $n>1$ and $(n,m) \neq (m-1,m)$, the algebra associated to $(n,m) = (2,4)$ occurs as an $A_{\infty}$-subalgebra of the algebra associated to $(n,m)$.  It therefore suffices to consider the case $(n,m)=(2,4)$. In that case,  there are six Lefschetz thimbles (up to grading shift).
We denote them by $P(4|3)$, $P(4|2)$, $P(4|1)$, $P(3|2)$, $P(3|1)$ and $P(2|1)$, where $P(n|m)$ corresponds to $\ul{T}^{\{1,2,3,4\} \setminus \{n,m\}}$. 

The cohomological algebra is given by the following quiver with relations, by Proposition \ref{p:cellMod} and \cite[Section 5.2.8, 5.2.9]{Klamt}.

\[
\begin{tikzcd}
P(4|3) \arrow[d,"F^{(4|3)}_{(4|2)}"] \arrow[shift right, swap]{d}{Id^{(4|3)}_{(4|2)}} \arrow[bend left=30, swap]{ddrrrr}{K^{(4|3)}_{(2|1)}} \arrow[shift left,bend left=30]{ddrrrr}{G^{(4|3)}_{(2|1)}} &  &   & &\\
P(4|2) \arrow[d,"F^{(4|2)}_{(4|1)}"] \arrow[shift right, swap]{d}{Id^{(4|2)}_{(4|1)}}    
\arrow[rr,"\widetilde{F}^{(4|2)}_{(3|2)}"] \arrow[shift right, swap]{rr}{Id^{(4|2)}_{(3|2)}} &
& P(3|2)   \arrow[d,"F^{(3|2)}_{(3|1)}"] \arrow[shift right, swap]{d}{Id^{(3|2)}_{(3|1)}} &   &\\
P(4|1) \arrow[rr,"\widetilde{F}^{(4|1)}_{(3|1)}"] \arrow[shift right, swap]{rr}{Id^{(4|1)}_{(3|1)}} & & 
P(3|1) \arrow[rr,"\widetilde{F}^{(3|1)}_{(2|1)}"] \arrow[shift right, swap]{rr}{Id^{(3|1)}_{(2|1)}} & &P(2|1)   
\end{tikzcd}
\]

The subscript and superscript of the name of the arrow indicate the target and the source of the arrow, respectively.
The definition of $Id^{(k|l)}_{(a|b)}$, $F^{(k|l)}_{(a|b)}$, $\widetilde{F}^{(k|l)}_{(a|b)}$, $G^{(k|l)}_{(a|b)}$, $K^{(k|l)}_{(a|b)}$
for various $k,l,a,b$ can be found in Theorem 5.16, 5.17, 5.18 and 5.19 in \cite{Klamt}.
The complete list of relations of the product structure can be found in \cite[Table 5.25]{Klamt}.
For example, $\widetilde{F}^{(4|1)}_{(2|1)}=Id^{(3|1)}_{(2|1)} \widetilde{F}^{(4|1)}_{(3|1)}$
and $F^{(4|2)}_{(4|1)}F^{(4|3)}_{(4|2)}=0$.
We remark that $J^{(k|l)}_{(a|b)}$ in the table is defined in Theorem 5.26
and $J^{(4|2)}_{(3|1)}$ equals to $F^{(3|2)}_{(3|1)} \widetilde{F}^{(4|2)}_{(3|2)}$
and $\widetilde{F}^{(4|1)}_{(3|1)} F^{(4|2)}_{(4|1)}$ up to sign.
Moreover, the dimension of $\Ext(P(k|l),P(a|b))$ can be found in \cite[Section 5.2.6]{Klamt}.

For the $A_{\infty}$ structure, Klamt computed a minimal model with $\mu^3 \neq 0$ but $\mu^d=0$ for $d > 3$.
The list of non-zero $\mu^3$ is given in \cite[Table 8.4]{Klamt}.

\begin{lemma}\label{l:nonformal}
 When $(n,m)=(2,4)$, the endomorphism algebra of Lefschetz thimbles is not formal.
\end{lemma}

\begin{proof}
Let $\cB$ be the $A_{\infty}$ endomorphism algebra of Lefschetz thimbles
and $B:=H(\cB)$ be the underlying cohomological algebra.
If $\cB$ were formal, then the $\mu^3$ of the minimal model Klamt found would define the zero class in 
$HH^3(B,B[-1])$.
It means that $\mu^3=d_{CC} \tau$ for some $\tau \in CC^2(B,B[-1])$, where $d_{CC}$ is the Hochschild differential.

From \cite[Table 8.4]{Klamt}, we have 
 \begin{align}
  \mu^3(\widetilde{F}^{(4|1)}_{(2|1)}, F^{(4|2)}_{(4|1)},F^{(4|3)}_{(4|2)})= \pm K^{(4|3)}_{(2|1)}
 \end{align}
On the other hand, up to sign,
 \begin{align}
   &d_{CC} \tau(\widetilde{F}^{(4|1)}_{(2|1)}, F^{(4|2)}_{(4|1)},F^{(4|3)}_{(4|2)}) \nonumber \\
  =&\widetilde{F}^{(4|1)}_{(2|1)} \tau( F^{(4|2)}_{(4|1)},F^{(4|3)}_{(4|2)})+\tau(\widetilde{F}^{(4|1)}_{(2|1)}, F^{(4|2)}_{(4|1)}) F^{(4|3)}_{(4|2)}  \label{eq:dCC}+ \tau(\widetilde{F}^{(4|1)}_{(2|1)} F^{(4|2)}_{(4|1)},F^{(4|3)}_{(4|2)})+\tau(\widetilde{F}^{(4|1)}_{(2|1)}, F^{(4|2)}_{(4|1)} F^{(4|3)}_{(4|2)}) 
 \end{align}
By \cite[Table 5.25]{Klamt}, we have $\widetilde{F}^{(4|1)}_{(2|1)} F^{(4|2)}_{(4|1)}=F^{(4|2)}_{(4|1)} F^{(4|3)}_{(4|2)}=0$
so the last two terms of \eqref{eq:dCC} vanish.

One can check from \cite[Section 5.2.6]{Klamt} that the dimension of $\Ext(P(2|1),P(4|3))$ is $4$ and it is generated by 
$Id^{(4|3)}_{(2|1)}$, $F^{(4|3)}_{(2|1)}=\widetilde{F}^{(4|3)}_{(2|1)}$, $G^{(4|3)}_{(2|1)}$ and $K^{(4|3)}_{(2|1)}$.
On the other hand, one can check from \cite[Table 5.25]{Klamt} that
for any $a \in \Ext(P(4|1),P(4|3))$ and $b \in \Ext(P(2|1),P(4|2))$, the sum
$\widetilde{F}^{(4|1)}_{(2|1)} a+b F^{(4|3)}_{(4|2)}$
lie in the subspace spanned by $F^{(4|3)}_{(2|1)}$.
Therefore, $K^{(4|3)}_{(2|1)}$ cannot be written as the sum of the first two terms of \eqref{eq:dCC} so it does not equal to
$d_{CC} \tau (\widetilde{F}^{(4|1)}_{(2|1)}, F^{(4|2)}_{(4|1)},F^{(4|3)}_{(4|2)})$
for any $\tau$ in $CC^2(B,B[-1])$.
Contradiction and the result follows.
\end{proof}

\begin{remark}
 Lemma \ref{l:nonformal} shows that the strategy from Section \ref{s:formality} cannot be used to prove that the collection of thimbles is formal.  In other words, the cohomological algebra of the thimbles 
is not compatible with the existence of such a consistent choice of equivariant structure. It seems hard to pinpoint why one would expect the collection 
of
Lagrangians $\{\ul{L}_{\lambda}\}_{\lambda \in \Lambda_{n,m}}$ to be better in this respect. 
\end{remark}

\end{document}